\documentclass[11pt,a4paper]{article}
\usepackage[margin=1in]{geometry}
\usepackage[authoryear,round]{natbib}
\usepackage[T1]{fontenc}
\usepackage[utf8]{inputenc}
\usepackage{amsmath,amssymb,amsthm,mathtools}
\usepackage{thmtools,thm-restate}
\usepackage{booktabs,tabularx,longtable,array,multirow}
\usepackage{algorithm}
\usepackage[noend]{algpseudocode}
\usepackage{graphicx,xcolor}
\usepackage{microtype}
\usepackage{caption}
\usepackage{subcaption}
\usepackage{enumitem}
\usepackage{url}
\usepackage{etoc}
\usepackage{xurl}  % long URLs in the bibliography break inside words
\usepackage[colorlinks=true,linkcolor=blue!50!black,citecolor=blue!50!black,
  urlcolor=blue!50!black,
  pdftitle={The Complexity of Convex Optimization with Mismatched Geometry}]{hyperref}
\usepackage[capitalise,nameinlink]{cleveref}

\declaretheorem[name=Proposition,numberwithin=section]{proposition}
\declaretheorem[name=Lemma,sibling=proposition]{lemma}

\declaretheorem[name=Definition,sibling=proposition,style=definition]{definition}

\crefname{algorithm}{Algorithm}{Algorithms}
\crefname{equation}{}{}
\Crefname{equation}{Equation}{Equations}

\newcommand{\R}{\mathbb{R}}
\newcommand{\N}{\mathbb{N}}
\newcommand{\E}{\mathbb{E}}
\newcommand{\Prob}{\mathbb{P}}

\newcommand{\nrm}[1]{\lVert#1\rVert}
\newcommand{\ip}[2]{\left\langle#1,#2\right\rangle}
\newcommand{\opt}{x^\star}
\newcommand{\fstar}{f^\star}

\newcommand{\cF}{\mathcal{F}}

\newcommand{\cE}{\mathcal{E}}
\newcommand{\cA}{\mathcal{A}}
\newcommand{\cG}{\mathcal{G}}
\newcommand{\cB}{\mathcal{B}}
\newcommand{\tO}{\widetilde{O}}
\DeclareMathOperator*{\argmin}{arg\,min}
\DeclareMathOperator*{\argmax}{arg\,max}
\DeclareMathOperator{\dist}{dist}
\DeclareMathOperator{\conv}{conv}

\DeclareMathOperator{\tr}{tr}
\DeclareMathOperator{\sign}{sign}

\DeclareMathOperator{\spn}{span}
\DeclareMathOperator{\proj}{proj}

\DeclareMathOperator{\diam}{diam}
\DeclareMathOperator{\vol}{vol}
\DeclareMathOperator{\Lip}{Lip}
\newcommand{\ceilp}[1]{\bigl\lceil#1\bigr\rceil_{+}}
\newcommand{\1}{\mathbf{1}}
\providecommand{\noopsort}[1]{}% sort key for the concurrent preprints in references.bib
\newcommand{\alg}[1]{\textsf{#1}}
\newcommand{\SPC}{\alg{SP-Cut}}
\newcommand{\SPL}{\alg{SP-Level}}
\newcommand{\SPA}{\alg{SP-Accel}}
\newcommand{\SPT}{\alg{SP-Taylor}}
\newcommand{\SPV}{\alg{SP-VI}}
\newcommand{\SPR}{\alg{SP-Restart}}
\newcommand{\SPP}{\alg{SP-Prox}}
\newcommand{\CGC}{\alg{CG-Cut}}
\newcommand{\CLA}{\alg{CLA}}

\newcommand{\AGDE}{\alg{AGD-E}}
\newcommand{\AGDEnt}{\alg{AGD-Ent}}
\newcommand{\PG}{\alg{PG}}
\newcommand{\PSG}{\alg{PSG}}
\newcommand{\MDE}{\alg{MD-Ent}}
\newcommand{\EG}{\alg{EG}}

\newcommand{\sub}[1]{\textup{\textsc{#1}}}
\newcommand{\st}[1]{\operatorname{st}(#1)} % Steiner point
\newcommand{\hs}[1]{h_{#1}}                % support function
\newcommand{\gw}[1]{\omega(#1)}            % Gaussian width (root mean square)
\newcommand{\bN}{\mathsf{b}_N}
\newcommand{\cn}{\mathsf{c}_n}
\newcommand{\PN}{P_N}
\newcommand{\PT}[1]{P_{#1}}                % path budget with a general horizon
\newcommand{\flo}{\underline{f}}
\newcommand{\fhi}{\overline{f}}
\newcommand{\Dnu}{D_\nu}

\setlist{leftmargin=*}
\newcolumntype{Y}{>{\raggedright\arraybackslash}X}
\newcolumntype{C}{>{\centering\arraybackslash}X}

\algrenewcommand\algorithmiccomment[1]{\hfill$\triangleright$\ #1}
\makeatletter
\providecommand{\theHALG@line}{}
\renewcommand{\theHALG@line}{\thealgorithm.\arabic{ALG@line}}
\makeatother

\title{The Complexity of Convex Optimization with Mismatched Geometry}
\author{\small
\textbf{Timofei Loginov}\textsuperscript{1}, \textbf{Alexander Gasnikov}\textsuperscript{2}, \textbf{Yuriy Dorn}\textsuperscript{3},\\
\small\textbf{Aleksandr Shestakov}\textsuperscript{4}, \textbf{Nazarii Tupitsa}\textsuperscript{4},\\
\small\textbf{Osman Osmanov}\textsuperscript{2}, \textbf{Darina Dvinskikh}\textsuperscript{5}\\[3pt]
\normalfont\small\textsuperscript{1} Moscow Independent Research Institute of Artificial Intelligence (miriai.org)\\
\normalfont\small\textsuperscript{2} Innopolis University (innopolis.university)\\
\normalfont\small\textsuperscript{3} AI Institute MSU (msu.ru)\\
\normalfont\small\textsuperscript{4} Mohamed bin Zayed University of Artificial Intelligence (mbzuai.ac.ae)\\
\normalfont\small\textsuperscript{5} HSE University (hse.ru)}
\date{}

\begin{document}
\maketitle
\etocdepthtag.toc{mtmain}
\begin{abstract}
Optimal first-order methods on non-Euclidean domains such as the $\ell_1$ ball $B_1^n(R)=\{x\in\mathbb R^n:\|x\|_1\le R\}$ pair the prox-function with the norm in which smoothness is measured. When the gradient is $L$-Lipschitz in the Euclidean norm only, the accelerated method with a Euclidean prox-setup reduces the functional gap $f(x_N)-\min_{B_1^n(R)}f$ to $O(LR^2/N^2)$ after $N$ first-order queries (an entropic $\ell_1$ prox-setup replaces $L$ by the $\ell_1\to\ell_\infty$ constant $L_1\le L$, at the cost of a factor $\log n$ and with the same exponent). The lower bound of Guzm\'an and Nemirovski is of order $LR^2/N^3$, and whether the upper bound can be brought down to that order is a question of A. S. Nemirovski. We show that for this mismatched problem the minimax value of the gap is of order $LR^2/N^3$ up to logarithmic factors, for deterministic and for randomized methods alike, and already in dimension proportional to the number of queries. The upper bound is attained by a Steiner-point level method: every supporting hyperplane seen so far is kept as a level cut, and the next query is made near the Steiner point of the resulting localization polytope. The analysis rests on a single geometric fact: along nested subsets of the ball the Steiner points travel a distance that is polylogarithmic in the dimension, in contrast to $\sqrt n$ for the Euclidean ball. All queries stay feasible, and a randomized selector keeps the internal work polynomial. The same geometry yields optimal rates for nonsmooth objectives, H\"older gradients, higher-order oracles and Lipschitz monotone operators, the last with a matching deterministic lower bound. For convex quadratics, a curvature-learning method attains the optimal rate on every $\ell_p$ ball with $1\le p<2$ and without logarithmic loss. Experiments confirm the predicted $N^{-3}$ behaviour on objectives that are hard for Euclidean methods.
\end{abstract}

\section{Introduction}
\label{sec:intro}

Consider minimizing a convex function $f$ over the $\ell_1$ ball
$B_1^n(R)=\{x\in\R^n:\nrm x_1\le R\}$, or over the probability simplex, when all
we know about the smoothness of $f$ is a bound in the Euclidean norm,
\[
 \nrm{\nabla f(x)-\nabla f(y)}_2\le L\nrm{x-y}_2 .
\]
This situation is common. The $\ell_1$ ball and the simplex are the standard
feasible sets for sparse and probabilistic models, while for a least-squares
loss the natural smoothness constant is the largest eigenvalue of the Gram
matrix, a Euclidean quantity.

First-order methods for constrained problems are built around a
\emph{prox-setup}: a prox-function that is strongly convex with respect to the
same norm in which the gradient is Lipschitz
\citep{nemirovski1983,bental2001lectures,beck2003mirror,gasnikov2026nesterov};
methods for relatively smooth functions follow the same principle
\citep{hanzely2021}. With the Euclidean prox-function, accelerated projected
gradient reduces the optimality gap to $O(LR^2/N^2)$ after $N$ gradient
queries. The entropy prox-function is adapted to the geometry of the
$\ell_1$ ball, but it needs the constant $L_1\le L$ of $\nabla f$ as a map from
$(\R^n,\nrm\cdot_1)$ to $(\R^n,\nrm\cdot_\infty)$ and gives
$O(L_1R^2\log n/N^2)$ \citep{juditsky2011first}; when $L_1$ is close to $L$,
this is no better.

We compare rates by their \emph{exponent}: a bound of order $N^{-\alpha}$, up
to logarithmic factors, has exponent $\alpha$. Both methods above have exponent
two. The lower bound of \citet{guzman2015lower} for this class is of order
$LR^2/N^3$ when $n>N$, so it has exponent three. Which exponent is correct was
posed as an open problem by \citet{guzman2015open}; \citet{vorontsova2019}
attribute the question to A.~S.~Nemirovski. The gap does not come from a poor
choice of constants: the lower bound holds already for functions whose smallest
Euclidean and $\ell_1$ smoothness constants coincide (\cref{app:equal-constants}).

Two known approaches each capture half of what is needed. Prox methods
accelerate, but their progress is measured in the norm of the prox-function,
and in the Euclidean norm the $\ell_1$ ball is as large as the Euclidean ball
of the same radius. Cutting-plane methods, such as the method of centers of
gravity or the volumetric method, use the shape of the feasible set: they keep
every cut and query a center of the remaining localization set
\citep{levin1965,vaidya1996,lee2015cutting,jiang2020cutting}. Their guarantees
come from volume reduction and require order $n\log(1/\varepsilon)$ queries,
which says nothing in the regime $N\le n$ where the lower bound lives. Bundle
and level methods combine cuts with prox steps and are optimal when the
prox-function matches the smoothness norm
\citep{lemarechal1995bundle,lan2015bundle}. For convex quadratics,
\citet{ouyang2026cubic} recently obtained the rate $O(LR^2/N^3)$ on the $\ell_1$
ball by learning the Hessian from gradient differences; that argument does not
extend beyond quadratics.

\paragraph{Our approach.} We show that the exponent is three. The method is an
accelerated level method that keeps all its cuts, but replaces the prox step by
a move of its center to the \emph{Steiner point} of the current localization
polytope, the average of the maximizers of a random Gaussian linear function
over the polytope \citep{shephard1968,schneider2014}; as in every accelerated
method, the queries are convex combinations of this center and the current
iterate. \citet{bubeck2020chasing} used this
point for chasing nested convex bodies, because it moves little when the body
shrinks. On the $\ell_1$ ball this effect is strong: along any nested sequence of
$N$ subsets of $B_1^n(R)$, the Steiner points travel a total Euclidean distance
of order $R\sqrt{\log N\log n}$, whereas inside a Euclidean ball the bound is of
order $R\sqrt{n\log N}$. An accelerated level method needs exactly such a bound.
Every query either permits an accelerated step, whose error term is the squared
displacement of the center, or moves the center far; the short path limits both
kinds of queries, and balancing them gives the error $LR^2/N^3$ up to
logarithmic factors.

\paragraph{Contributions.}
\begin{itemize}[leftmargin=1.2em,itemsep=2pt,topsep=2pt]
\item \emph{The smooth case.} For convex functions with Euclidean $L$-Lipschitz
gradient on $B_1^n(R)$, the minimax error after $N$ first-order queries is
$LR^2/N^3$ up to the factor $\log(2N)\log(2n)$, for deterministic and for
randomized methods, when $n$ is at least proportional to $N$ (\cref{sec:setting}). The
upper bound is attained by an explicit method whose queries all lie in the
feasible set; its randomized version solves polynomially many linear programs.
\item \emph{Other problem classes.} Keeping the geometric argument and changing
the one-step inequality, we obtain methods with the optimal exponent for
nonsmooth objectives ($1$ instead of $1/2$), H\"older continuous gradients and
higher-order oracles, and with the exponent $2$ instead of $1$, optimal for
deterministic methods, for monotone variational inequalities and
convex--concave saddle-point problems. Restarts give the strongly convex and
strongly monotone versions.
\item \emph{Lower bounds.} We extend the deterministic first-order lower bounds
of \citet{guzman2015lower} to randomized methods and higher-order oracles, and
prove lower bounds for quadratics, smooth strongly convex functions and
Lipschitz variational inequalities. All of them match the rates above up to
logarithmic factors, and they also hold for methods that take prox steps with
any known prox-function, for instance the entropy.
\item \emph{Quadratics.} A curvature-learning method removes the logarithmic
factors for convex quadratics and extends the result of \citet{ouyang2026cubic}
from the $\ell_1$ ball to every $\ell_p$ ball with $1\le p<2$, with the optimal
exponent $1+2/p$.
\end{itemize}
\Cref{tab:intro} compares these rates with those of standard methods, which
measure all distances in the Euclidean norm, and with the lower bounds known
before. Lower bounds proved for the Euclidean ball do not apply on
$B_1^n(R)$: our upper bounds are below them. Experiments on instances that are
hard for Euclidean methods confirm the short Steiner paths and show fewer oracle
calls for a given accuracy, at a much higher cost per call
(\cref{sec:experiments}).

\begin{table}[!htb]
\centering
\caption{Known bounds and ours on $B_1^n(R)$ with Euclidean constants; new
results are in bold.}
\label{tab:intro}
\footnotesize
\setlength{\tabcolsep}{2.2pt}
\renewcommand{\arraystretch}{1.3}
\renewcommand{\crefpairconjunction}{, }\renewcommand{\crefmiddleconjunction}{, }%
\renewcommand{\creflastconjunction}{, }
\newcommand{\nw}[1]{\boldsymbol{#1}}
\resizebox{\linewidth}{!}{%
\begin{tabular}{@{}llllll@{}}
\toprule
 & \multicolumn{2}{c}{Upper bound} & \multicolumn{2}{c}{Lower bound} & \\
\cmidrule(lr){2-3}\cmidrule(lr){4-5}
Class & known & ours & known & ours & Results\\
\midrule
\multicolumn{6}{@{}l}{\emph{Convex minimization: error after $N$ queries}}\\
$G$-Lipschitz $f$ & $O(\frac{GR}{\sqrt N})$ & $\nw{\tO(\frac{GR}{N})}$
 & $\Omega(\frac{GR}{N})^\dagger$ & $\nw{\Omega(\frac{GR}{N})}$
 & \labelcref{thm:spcut,thm:first-order-lower,thm:rand-nonsmooth}\\
$\nu$-H\"older $\nabla f$ & $O(\frac{L_{1,\nu}R^{1+\nu}}{N^{(1+3\nu)/2}})$
 & $\nw{\tO(\frac{L_{1,\nu}R^{1+\nu}}{N^{1+2\nu}})}$
 & $\Omega(\frac{L_{1,\nu}R^{1+\nu}}{N^{1+2\nu}})^\dagger$
 & $\nw{\Omega(\frac{L_{1,\nu}R^{1+\nu}}{N^{1+2\nu}})}$
 & \labelcref{thm:holder,thm:first-order-lower,thm:rand-tensor}\\
$L$-Lipschitz $\nabla f$ & $O(\frac{LR^2}{N^2})$ & $\nw{\tO(\frac{LR^2}{N^3})}$
 & $\Omega(\frac{LR^2}{N^3})^\dagger$ & $\nw{\Omega(\frac{LR^2}{N^3})}$
 & Thm.~\ref{thm:overview}\\
H\"older $D^kf$ & $O(\frac{L_{k,\nu}R^s}{N^{(3s-2)/2}})$
 & $\nw{\tO(\frac{L_{k,\nu}R^s}{N^{2s-1}})}$ & ---
 & $\nw{\Omega(\frac{L_{k,\nu}R^s}{N^{2s-1}})}$
 & \labelcref{thm:tensor,thm:coordinate-tensor-lower,thm:rand-tensor}\\
quadratic, $p=1$ & $O(\frac{LR^2}{N^3})$ & $O(\frac{LR^2}{N^3})$ & ---
 & $\nw{\Omega(\frac{LR^2}{N^3})^\dagger}$
 & \labelcref{thm:cla,thm:power-chain}\\
quadratic, $p>1$ & $O(\frac{LR^2}{N^2})$ & $\nw{O(\frac{LR^2}{N^{1+2/p}})}$ & ---
 & $\nw{\Omega(\frac{LR^2}{N^{1+2/p}})^\dagger}$
 & \labelcref{thm:cla,thm:power-chain}\\
\multicolumn{6}{@{}l}{\emph{Strongly convex minimization: queries for error $\varepsilon$}}\\
$G$-Lipschitz $f$ & $O(\frac{G^2}{\mu_1\varepsilon})$
 & $\nw{\tO(\frac{G}{\sqrt{\mu_1\varepsilon}})}$ & --- & --- & \labelcref{thm:restart}\\
$\nu$-H\"older $\nabla f$
 & $O(\varrho^{\frac1{1+3\nu}})$
 & $\nw{\tO(\varrho^{\frac1{2+4\nu}})}$
 & --- & --- & \labelcref{thm:restart}\\
$L$-Lipschitz $\nabla f$ & $O(\kappa_1^{1/2}\log\frac1\varepsilon)$
 & $\nw{\tO(\kappa_1^{1/3}\log\frac1\varepsilon)}$ & ---
 & $\nw{\widetilde\Omega(\kappa_1^{1/3}\log\frac1\varepsilon)^{\dagger\ddagger}}$
 & \labelcref{thm:restart,thm:strong-map}\\
H\"older $D^kf$
 & $O(\bar\kappa^{\frac2{3s-2}}+\log\log\frac1\varepsilon)$
 & $\nw{\tO(\bar\kappa^{\frac1{2s-1}}+\log\frac1\varepsilon)}$
 & --- & --- & \labelcref{cor:tensor-objective-restart}\\
quadratic & $O(\kappa_p^{1/2}\log\frac1\varepsilon)$
 & $\nw{O(\kappa_p^{p/(p+2)}\log\frac1\varepsilon)}$ & ---
 & $\nw{\Omega(\kappa_p^{p/(p+2)})^\dagger}$
 & \labelcref{cor:cla-restart}, \labelcref{thm:sc-constant-reduction}\\
\multicolumn{6}{@{}l}{\emph{Monotone variational inequalities: weak gap after
$N$ calls}}\\
$L$-Lipschitz $F$ & $O(\frac{LR^2}{N})$ & $\nw{\tO(\frac{LR^2}{N^2})}$ & ---
 & $\nw{\Omega(\frac{LR^2}{N^2})^\dagger}$ & \labelcref{thm:vi,thm:vi-lower}\\
$\nu$-H\"older $F$ & $O(\frac{L_{0,\nu}R^{1+\nu}}{N^{(1+\nu)/2}})$
 & $\nw{\tO(\frac{L_{0,\nu}R^{1+\nu}}{N^{1+\nu}})}$ & --- & --- & \labelcref{thm:holder-tensor-vi}\\
\multicolumn{6}{@{}l}{\emph{Strongly monotone variational inequalities: calls
for $\nrm{x-x^\star}_1\le\varepsilon$}}\\
$L$-Lipschitz $F$
 & $O(\kappa_1\log\frac R\varepsilon)$ & $\nw{\tO(\kappa_1^{1/2}\log\frac R\varepsilon)}$
 & --- & --- & \labelcref{thm:vi}\\
\bottomrule
\end{tabular}}
\par\smallskip
{\footnotesize\raggedright
$L_{k,\nu}$ is the $\nu$-H\"older constant of $D^kf$ (of $F$ for $k=0$) and
$s=k+\nu$, with $\nu<1$ and $s>2$ in the strongly convex H\"older rows; $\mu_p$
is the modulus of strong convexity or monotonicity in $\ell_p$,
$\kappa_p=L/\mu_p$, $\bar\kappa=L_{k,\nu}R^{s-2}/\mu_1$,
$\varrho=L_{1,\nu}^2\mu_1^{-1-\nu}\varepsilon^{\nu-1}$; the quadratic rows are
on $B_p^n(R)$, $1\le p<2$. Known upper bounds are those of Euclidean methods and,
for $p=1$, of \citet{ouyang2026cubic}; known lower bounds are those of
\citet{guzman2015lower,guzman2015open}. $\tO$ and $\widetilde\Omega$ hide
polylogarithmic factors; lower bounds need $n$ at least proportional to $N$;
$^\dagger$: deterministic methods only; $^\ddagger$: if
$\kappa_1\ge2n\ge2\kappa_1^{1/3}$; ---: none known.\par}
\end{table}

\section{Setting and main result}
\label{sec:setting}

For $1\le p\le\infty$, $\nrm\cdot_p$ is the $\ell_p$ norm on $\R^n$,
$B_p^n(R)=\{x:\nrm x_p\le R\}$, and $e_1,\dots,e_n$ are the coordinate vectors.
Unless stated otherwise, regularity constants
are Euclidean: $f$ is \emph{$L$-smooth} if
$\nrm{\nabla f(x)-\nabla f(y)}_2\le L\nrm{x-y}_2$ for all $x,y$. We write
$\fstar=\min_Kf$ for the feasible set $K$, and $\log$ is the natural
logarithm.

\paragraph{Oracle model.} A method for minimizing $f$ over a known set $K$
chooses query points $x_1,\dots,x_N\in\R^n$ one after another, each depending on
the previous answers; at every query the \emph{first-order oracle} returns
$f(x_t)$ and $\nabla f(x_t)$. After $N$ queries the method outputs a point
$\widehat x\in K$. Only oracle calls are counted, while computations with known
objects, such as linear programs over $K$, are free; this is the
information-based complexity of \citet{nemirovski1983}. A query is
\emph{feasible} if it lies in $K$. Let $\cF_L$ be the class of convex $L$-smooth
functions on $\R^n$. The minimax error of deterministic methods on the $\ell_1$
ball is
\[
 \cE_N^{\mathrm{det}}=\inf_{\text{methods}}\ \sup_{f\in\cF_L}
 \Bigl[f(\widehat x)-\min_{B_1^n(R)}f\Bigr],
\]
and $\cE_N^{\mathrm{ran}}$ is defined in the same way for randomized methods,
with the error averaged over the internal randomness of the method.

\begin{restatable}[Minimax rate for smooth functions]{theorem}{thmoverview}
\label{thm:overview}
For every $N\ge1$,
\[
 \cE_N^{\mathrm{det}}\ge\frac{LR^2}{16(N+1)^3}\quad(n\ge N+1),\qquad
 \cE_N^{\mathrm{ran}}\ge\frac{LR^2}{2^{50}N^3}\quad(n\ge32N+1),
\]
and, in every dimension $n$,
\[
 \cE_N^{\mathrm{ran}}\le\cE_N^{\mathrm{det}}
 \le\frac{2^{35}\,LR^2\log(2N)\log(2n)}{N^3}.
\]
The upper bound is attained by a method whose queries all lie in $B_1^n(R)$.
\end{restatable}

Thus the answer to the question of \cref{sec:intro} is $N^{-3}$: when $n$ is at
least proportional to $N$, the minimax error is $LR^2/N^3$ up to the factor
$\log(2N)\log(2n)$, and randomization does not change the exponent. Some
condition on the dimension is necessary, because in a fixed dimension
cutting-plane methods converge linearly. The deterministic lower bound is the
construction of \citet{guzman2015lower} with an explicit constant; the
randomized lower bound and the upper bound are new. The upper bound holds, with
the same proof, on every polytope contained in a translate of $B_1^n(R)$, in
particular on the simplex. Its constant is large, so \cref{thm:overview} is a
complexity result rather than a practical iteration count.

\Cref{sec:method} presents the method and proves the upper bound up to routine
steps, \cref{sec:lower-main} explains the lower bounds, and
\cref{sec:extensions} shows how the same construction gives the other rows of
\cref{tab:intro}.

\section{The method: short Steiner paths and an accelerated level method}
\label{sec:method}

Throughout this section $K$ is a nonempty polytope, given explicitly by linear
inequalities, such that $K\subseteq c+B_1^n(R)$ for some known $c\in\R^n$ and
$R>0$. The main example is $K=B_1^n(R)$ with $c=0$.

\subsection{Steiner points move little}
\label{sec:steiner-main}

Let $Z$ be a standard Gaussian vector in $\R^n$. For a nonempty compact convex
set $C\subseteq\R^n$, the maximizer $x_C(Z)$ of the linear function
$x\mapsto\ip Zx$ over $C$ is unique with probability one, and the
\emph{Steiner point} of $C$ is its expectation,
\begin{equation}\label{eq:steiner-def-main}
 \st C=\E\,x_C(Z)\ \in C .
\end{equation}
This is the classical Steiner point of convex geometry written in Gaussian form
\citep{shephard1968,schneider2014}. It belongs to $C$ because it is an average
of points of $C$, and it can be estimated by averaging the solutions of a few
linear programs. The property we need is its stability: when $C$ shrinks, the
point $\st C$ moves little.

\begin{restatable}[Short Steiner paths]{lemma}{lemsteinerpath}
\label{lem:steiner-path}
Let $T\ge1$ and let $C_0\supseteq C_1\supseteq\dots\supseteq C_T$ be nonempty
compact convex subsets of $c+B_1^n(R)$. Then
\begin{equation}\label{eq:steiner-path}
 \sum_{t=1}^T\nrm{\st{C_t}-\st{C_{t-1}}}_2\ \le\
 4R\sqrt{(1+\log T)(1+\log(2n))} .
\end{equation}
\end{restatable}

\begin{proof}[Proof sketch]
Let $h_C(z)=\max_{x\in C}\ip zx$ be the support function of $C$. Gaussian
integration by parts gives $\st C=\E[Z\,h_C(Z)]$. Let $v_t$ be the unit vector
in the direction of $\st{C_{t-1}}-\st{C_t}$ and
$A(Z)=\max\{0,\ip{v_1}Z,\dots,\ip{v_T}Z\}$. The sets are nested, so every
difference $h_{C_{t-1}}-h_{C_t}$ is nonnegative, and the differences telescope:
\begin{equation}\label{eq:path-step1}
 \sum_{t=1}^T\nrm{\st{C_{t-1}}-\st{C_t}}_2
 =\E\sum_{t=1}^T\ip{v_t}Z\bigl(h_{C_{t-1}}(Z)-h_{C_t}(Z)\bigr)
 \le\E\Bigl[A(Z)\bigl(h_{C_0}(Z)-h_{C_T}(Z)\bigr)\Bigr].
\end{equation}
Since $C_T\subseteq C_0\subseteq c+B_1^n(R)$, we have
$0\le h_{C_0}(Z)-h_{C_T}(Z)\le2R\nrm Z_\infty$, and the Cauchy--Schwarz
inequality gives
\begin{equation}\label{eq:path-step2}
 \E\Bigl[A(Z)\bigl(h_{C_0}(Z)-h_{C_T}(Z)\bigr)\Bigr]
 \ \le\ 2R\,\bigl(\E A(Z)^2\bigr)^{1/2}\bigl(\E\nrm Z_\infty^2\bigr)^{1/2}.
\end{equation}
Both factors are maxima of Gaussian variables with variance at most one:
$\E A(Z)^2\le2(1+\log T)$ and $\E\nrm Z_\infty^2\le2(1+\log(2n))$. The full proof
is in \cref{app:steiner}.
\end{proof}

The dimension enters only through $\log n$, because an $\ell_1$ ball has small
Gaussian width: its support function is $R\nrm Z_\infty$, of order
$R\sqrt{\log n}$. For a Euclidean ball the same argument gives
$2R\sqrt{2(1+\log T)\,n}$, and the experiments of \cref{sec:experiments} show
that movement of order $\sqrt n$ indeed occurs there. Neither logarithm in
\eqref{eq:steiner-path} can be dropped in general: along the faces
$\conv\{Re_1,\dots,Re_j\}$, $j=n,n-1,\dots,1$, of the simplex, the Steiner
points travel a distance of order $R\log n$ (\cref{prop:path-tight}). In the
algorithms we use the slightly larger \emph{path budget}
\begin{equation}\label{eq:main-P}
 \PN=4R\sqrt{\bN\cn},\qquad
 \bN=1+\lceil\log_2N\rceil,\quad \cn=1+\lceil\log_2(2n)\rceil ,
\end{equation}
which bounds the left-hand side of \eqref{eq:steiner-path} for every chain of
length $T\le N$.

\subsection{An accelerated level method with Steiner centers}
\label{sec:level-main}

Fix a target level $\ell\in\R$. A query at $x$ defines the \emph{level cut}
$\{u:f(x)+\ip{\nabla f(x)}{u-x}\le\ell\}$; by convexity it contains every point
$u$ with $f(u)\le\ell$. Intersecting $K$ with the cuts produces nested
localization polytopes $C_0\supseteq C_1\supseteq\cdots$, and if one of them is
empty, then $\min_Kf>\ell$. Accelerated level methods \citep{lan2015bundle}
combine such cuts with the three sequences $x_t,y_t,z_t$ of an accelerated
method and choose $z_t$ by a prox step onto the current polytope.
\Cref{alg:splevel} chooses instead a point $z_t$ close to the Steiner point of
$C_t$. This point is not selected to decrease any distance, so each iteration,
or \emph{trial}, checks how far the center moved: a short move leads to the
accelerated update, while a long move is \emph{rejected}, although its cut is
kept. We call a point $z\in C$
with $\nrm{z-\st C}_2\le\eta$ an \emph{$\eta$-approximate Steiner point} of $C$.

\begin{algorithm}[htbp]
\caption{\SPL$(\ell,N)$: accelerated level method with Steiner centers}
\label{alg:splevel}
\begin{algorithmic}[1]
\Require polytope $K\subseteq c+B_1^n(R)$; first-order oracle of $f$;
constant $L$; level $\ell$; horizon $N$
\State $\eta_N\gets\PN/(4N)$, $\delta_N\gets4\PN/N$
\Comment{center accuracy and acceptance threshold}
\State $C_0\gets K$, $A_0\gets0$; $y_0\gets$ any point of $K$;
$z_0\gets$ an $\eta_N$-approximate Steiner point of $K$
\For{$t=1,\dots,N$}
 \State $a_t\gets\bigl(1+\sqrt{1+4LA_{t-1}}\bigr)/(2L)$, \ $B_t\gets A_{t-1}+a_t$
 \Comment{so that $La_t^2=B_t$}
 \State $x_t\gets(A_{t-1}y_{t-1}+a_tz_{t-1})/B_t$; query $f(x_t)$ and
 $g_t=\nabla f(x_t)$
 \State $C_t\gets C_{t-1}\cap\{u:f(x_t)+\ip{g_t}{u-x_t}\le\ell\}$
 \If{$C_t=\varnothing$} \Return ``$\min_Kf>\ell$'' \EndIf
 \State $z_t\gets$ an $\eta_N$-approximate Steiner point of $C_t$
 \If{$\nrm{z_t-z_{t-1}}_2\le\delta_N$}
 \Comment{accept}
  \State $y_t\gets(A_{t-1}y_{t-1}+a_tz_t)/B_t$, \ $A_t\gets B_t$
 \Else
 \Comment{reject}
  \State $y_t\gets y_{t-1}$, \ $A_t\gets A_{t-1}$
 \EndIf
\EndFor
\State \Return $y_N$
\end{algorithmic}
\end{algorithm}

\begin{restatable}[Level method]{proposition}{thmsplevel}
\label{thm:splevel}
Let $K$ be a nonempty polytope given by linear inequalities with
$K\subseteq c+B_1^n(R)$, and let $f$ be convex with $L$-Lipschitz Euclidean
gradient on $K$. For every level $\ell\in\R$ and horizon $N\ge1$,
\cref{alg:splevel} makes at most $N$ first-order queries, all in $K$, and either
certifies that $\min_Kf>\ell$ or returns $y_N\in K$ with
\begin{equation}\label{eq:level-error}
 f(y_N)\ \le\ \ell+\frac{48L\PN^2}{N^3}
 \ =\ \ell+\frac{768\,LR^2\,\bN\cn}{N^3}.
\end{equation}
\end{restatable}

\begin{proof}[Proof sketch]
The points $x_t$, $y_t$ and $z_t$ are convex combinations of points of $K$, so
all queries are feasible. If $\min_Kf\le\ell$, a minimizer satisfies every level
cut, so an empty $C_t$ certifies $\min_Kf>\ell$. Assume that all $C_t$ are
nonempty. Two facts drive the proof.

\emph{The energy inequality.} For $A\ge0$, $a>0$, $B=A+a$, $x=(Ay+az)/B$ and
$y^+=(Ay+az^+)/B$, if $z^+$ lies in the level cut at $x$, then
\begin{equation}\label{eq:main-energy}
 B\,[f(y^+)-\ell]\ \le\ A\,[f(y)-\ell]+\frac{La^2}{2B}\nrm{z^+-z}_2^2 .
\end{equation}
It follows from the descent lemma between $x$ and $y^+$, after the linear term
is split into the convexity inequality at $x$, evaluated at $y$ with weight $A$,
and the level cut at $z^+$ with weight $a$. The rule $La_t^2=B_t$ makes the last
coefficient equal to $\frac12$.

\emph{The path bound.} Let $d_t=\nrm{z_t-z_{t-1}}_2$. The polytopes $C_t$ are
nested, so \cref{lem:steiner-path} and the accuracy $\eta_N$ of the centers
give, for accepted and rejected trials together,
\begin{equation}\label{eq:actual-path}
 \sum_{t=1}^Nd_t\ \le\ \PN+2N\eta_N\ =\ \tfrac32\PN .
\end{equation}
A rejected trial has $d_t>\delta_N=4\PN/N$, so fewer than $3N/8$ trials are
rejected and more than $N/2$ are accepted. Summing \eqref{eq:main-energy} over
the accepted trials, where $d_t\le\delta_N$, gives
\[
 A_N\,[f(y_N)-\ell]\ \le\ \frac12\sum_{t\ \mathrm{accepted}}d_t^2
 \ \le\ \frac{\delta_N}2\sum_{t=1}^Nd_t\ \le\ \frac{3\PN^2}{N},
\]
and the weights of more than $N/2$ accepted steps satisfy $A_N\ge N^2/(16L)$.
Dividing by $A_N$ gives \eqref{eq:level-error}; \cref{app:level} has the
details.
\end{proof}

The mechanism behind the rate is visible in the last display. As in every
accelerated method, the weight $A_N$ grows like $N^2/L$. For a prox step the
squared displacements would sum to order $R^2$; here the path bound makes them
sum to $O(\PN^2/N)$, which gains one power of $N$.

\subsection{Finding the level and computing the centers}
\label{sec:bracket-main}

The level method needs a target level, while $\fstar=\min_Kf$ is unknown. One
query at a point $x_0\in K$ brackets $\fstar$: with $g_0=\nabla f(x_0)$ and a
vertex $v$ minimizing $\ip{g_0}u$ over $u\in K$, convexity and smoothness give
\begin{equation}\label{eq:bracket-main}
 f(x_0)+\ip{g_0}{v-x_0}\ \le\ \fstar\ \le\ f(v)\ \le\
 f(x_0)+\ip{g_0}{v-x_0}+\tfrac L2\nrm{v-x_0}_2^2 .
\end{equation}
The complete method, \SPA{} (\cref{alg:spaccel} in \cref{app:level}), bisects
this bracket. At the midpoint level it runs \cref{alg:splevel} with the smallest
horizon $N=2^j$ whose guarantee \eqref{eq:level-error} is at most a quarter of
the current width. A certificate $\min_Kf>\ell$ halves the bracket; otherwise
the returned point shrinks it by the factor $3/4$. The widths decrease
geometrically and the horizons are minimal powers of two, so the horizons are
dominated by a geometric series: the total number of queries is at most
$1+24N(\varepsilon)$, where $N(\varepsilon)$ is the horizon that the rule
assigns to the width $\varepsilon$, which bounds every horizon used
(\cref{thm:main}). The search for the level therefore costs only a constant
factor, and choosing $\varepsilon$ from the budget gives the upper bound of
\cref{thm:overview}.

It remains to compute the centers. Computing the Steiner point of a known
polytope costs no oracle calls, and a finite, though exponential-time,
procedure approximates it to any accuracy (\cref{prop:finite-selector}). A
practical estimator uses sampling: for independent Gaussian vectors
$Z_1,\dots,Z_m$, the average of the maximizers of $\ip{Z_i}u$ over $u\in C$ lies in
$C$, is an unbiased estimate of $\st C$, and concentrates at the rate
$R/\sqrt m$. With $m$ polynomial in $N(\varepsilon)$ and in the logarithm of the
inverse failure probability, all guarantees hold with high probability, and the
method solves polynomially many linear programs (\cref{thm:random-selector}).

\section{Lower bounds}
\label{sec:lower-main}

\paragraph{Deterministic methods.} The hard functions of
\citet{guzman2015lower} are maxima of signed coordinates, built by a resisting
oracle. At the $t$-th query $x_t$, the oracle takes, among the coordinates it has
not used yet, one with the largest $|x_{t,i}|$, sets
$v_t=\sign(x_{t,i})\,e_i$, and answers according to
\[
 \varphi_t(x)=\max_{j\le t}\bigl\{\ip{v_j}x-(j-1)\tau\bigr\}
\]
with a small offset $\tau>0$. Near every past query, each form revealed later
lies at least $\tau$ below the current maximum, so every completion of $\varphi_t$
gives the same answers, and one fixed function reproduces the whole
interaction. After $N$ queries one coordinate is still unused. Giving it the
sign of the corresponding coordinate of the output $\widehat x$ defines
$v_{N+1}$ and the final function $\varphi=\varphi_{N+1}$ with $\varphi(\widehat x)\ge-N\tau$,
while the feasible point $-R\sum_{j\le N+1}v_j/(N+1)$ has value $-R/(N+1)$. Smoothing $\varphi$ by a Moreau
envelope makes it $L$-smooth and gives the bound $LR^2/[16(N+1)^3]$ of
\cref{thm:overview}. For order-$k$ oracles we smooth by convolution with a
compactly supported product kernel, whose derivative bounds do not depend on
the dimension; this gives the exponent $2(k+\nu)-1$. The bounds allow queries
anywhere in $\R^n$ and any computation between queries, so they also apply to
methods that take prox steps with a known prox-function, for example the
entropy on the simplex.

\paragraph{Randomized methods.} Against a randomized method, the hard function
has to be chosen before the interaction, as a random function
\citep{diakonikolas2020lower,agarwal2018higher}. We use a smoothed maximum
$\max_i(x_i-a_i)$ with independent exponential shifts $a_i$. Given the answers
so far, the shifts that have not been revealed are independent and, by the
memoryless property, exponential above known thresholds. Because the smoothing
parameter is small compared with the mean shift, one query reveals at most two
new coordinates on average, a constant fraction of the coordinates remains
hidden after $N$ queries when $n\ge32N+1$, and the output cannot locate the
minimizer. This gives the randomized bound of \cref{thm:overview} and its
analogue for higher-order oracles. \Cref{app:lower} contains the proofs.

\section{The same geometry for other problem classes}
\label{sec:extensions}

The proof of \cref{thm:splevel} uses smoothness only through the energy
inequality \eqref{eq:main-energy} and the geometry only through
\cref{lem:steiner-path}. Keeping the second and replacing the first gives
methods for the other rows of \cref{tab:intro}, except the quadratic ones, which
use a different mechanism; \crefrange{app:classes}{app:quadratics} contain the
statements and proofs.

\paragraph{Nonsmooth objectives.} Let $f$ be convex and $G$-Lipschitz, with a
subgradient $g_t$ returned at each query. The method \SPC{} queries an
approximate Steiner point $x_t$ of the current polytope and keeps only the
points $u$ with $\ip{g_t}{x_t-u}\ge\varepsilon$. Such deep cuts may remove every
minimizer; their role is to force the Steiner point to move. While the polytope
is nonempty,
the new Steiner point lies in it, hence at distance at least $\varepsilon/G$
from $x_t$, and \cref{lem:steiner-path} allows only $O(G\PN/\varepsilon)$ such
moves. Once the polytope is empty, linear-programming duality gives weights
$\lambda_t\ge0$, $\sum_t\lambda_t=1$, with
\begin{equation}\label{eq:main-mixture}
 \max_{u\in K}\sum_t\lambda_t\ip{g_t}{x_t-u}<\varepsilon ,
\end{equation}
and by convexity the average $\sum_t\lambda_tx_t$ is $\varepsilon$-optimal. The
certificate uses the stored subgradients only, and the error after $N$ queries is
$\tO(GR/N)$.

\paragraph{H\"older gradients and higher-order oracles.} If $\nabla f$ is
$\nu$-H\"older with constant $L_{1,\nu}$, the remainder $\frac L2\nrm{y^+-x}_2^2$ of
the descent lemma behind \eqref{eq:main-energy} becomes
$L_{1,\nu}\nrm{y^+-x}_2^{1+\nu}/(1+\nu)$, and the same counting gives the exponent
$1+2\nu$. With an oracle of order $k$, each trial computes an approximate
stationary point $v$ over $K$ of a regularized Taylor model of $f$
\citep{nesterov2006cubic}, queries $f$ at $v$, and uses the supporting
hyperplane of $f$ at $v$, not that of the model, as the level cut; the exponent
becomes $2(k+\nu)-1$. The model may be nonconvex, since the stationarity of $v$
is certified by one linear program. For a Hessian with Lipschitz constant $L_{2,1}$ this gives $\tO(L_{2,1}R^3/N^5)$,
compared with $N^{-7/2}$ for optimal Euclidean methods
\citep{arjevani2019,kovalev2022tensor}.

\paragraph{Strong convexity.} If $f$ is also $\mu_1$-strongly convex in $\ell_1$,
a point with error $\varepsilon$ lies within $\ell_1$ distance
$\sqrt{2\varepsilon/\mu_1}$ of the minimizer. Restarting the method on the
intersection of $K$ with shrinking $\ell_1$ balls, which are again polytopes
inside translates of $\ell_1$ balls, gives $\tO(\kappa_1^{1/3}\log(1/\varepsilon))$
queries for smooth $f$, $\kappa_1=L/\mu_1$, and the other counts of
\cref{tab:intro}. There is a subtlety: a function that is $L$-smooth in $\ell_2$
and $\mu_1$-strongly convex in $\ell_1$ always has $\kappa_1\ge n$. The complexity is
governed by the \emph{slack} $\Lambda=\kappa_1-n$, since one projected gradient
step localizes the minimizer in an $\ell_1$ ball whose radius depends on
$\Lambda$ rather than on $\kappa_1$. For every $\Lambda\ge8$ we determine the
deterministic complexity up to polylogarithmic factors: with
$\bar n=\min\{n,\Lambda^{1/3}\}$, it is $\min\{\bar n,(LR^2/\varepsilon)^{1/3}\}$
plus $\bar n\log_+(LR^2/(\bar n^3\varepsilon))$, where $\log_+=\max\{0,\log\}$
(\cref{app:strong}).

\paragraph{Variational inequalities and saddle-point problems.} For a monotone
operator $F$ on $K$, the accuracy of a point $\bar w$ is its weak gap
$\max_{u\in K}\ip{F(u)}{\bar w-u}$, and for a convex--concave function on a product
of polytopes it is the saddle gap. The certificate \eqref{eq:main-mixture} carries
over with $F(w_t)$ and $w_t$ in place of $g_t$ and $x_t$ and bounds both gaps.
For an $L$-Lipschitz $F$, the method \SPV{} makes an extragradient step
$w_t=\proj_K(z_t-F(z_t)/(2L))$, with the Euclidean projection $\proj_K$ onto
$K$, from the current approximate Steiner point $z_t$ and cuts deeply with
$F(w_t)$. A nonempty cut at depth $\varepsilon$ forces the
next center to move by at least $\frac23\sqrt{\varepsilon/L}$, so with
$\varepsilon$ of order $L\PN^2/N^2$ the method stops within $N$ trials. The weak
gap after $2N$ operator calls is therefore $\tO(LR^2/N^2)$, against $O(LR^2/N)$
for extragradient and Mirror-Prox \citep{korpelevich1976,nemirovski2004prox}. A
skew-symmetric affine operator whose columns are revealed two per query shows
that the exponent two is optimal for deterministic methods; the lower bound
applies also to bilinear games on products of $\ell_1$ balls and of simplices.
The same scheme handles H\"older operators and Taylor models of higher order,
and restarts give the strongly monotone version.

\paragraph{Quadratics.} For a convex quadratic, gradient differences give
Hessian--vector products, and \citet{ouyang2026cubic} accelerates by learning a
positive semidefinite lower model of the Hessian through symmetric rank-one (SR1) updates, combined
with a Huber regularizer. Replacing the Huber function by a power--Huber
function, quadratic near zero and equal to $|t|^p/p$ up to a constant elsewhere,
extends the method, which we call \CLA{} (curvature-learning acceleration), to
every $\ell_p$ ball with $1\le p<2$. It gives the error $O(LR^2/N^{1+2/p})$ with a
constant depending only on $p$ and without logarithmic factors, which matches a
lower bound (\cref{app:quadratics}). The mechanism needs the residual between the function
and its learned model to remain a positive semidefinite quadratic, and this
fails already in dimension one for nonquadratic functions.

\section{Experiments}
\label{sec:experiments}

The experiments check the two ingredients of the method, short Steiner paths and
the level method, and compare the methods with standard ones on instances that
are hard for Euclidean methods. All instances live on the unit $\ell_1$ ball
(games: on a product of two such balls), their Euclidean constants are at most
one, and errors are computed from known minimizers. Steiner points are
estimated by averaging the solutions of $m$ linear programs with random Gaussian
objectives, solved with HiGHS; the $m$ directions are shared by all centers of
a run, which works better than fresh directions at these sample sizes. \SPA{}
doubles its horizons adaptively, and \SPC{} and \SPV{} set the cut depth from
their certificate, because the theoretical parameters are far too conservative
at these budgets. The baselines are accelerated projected gradient
with the Euclidean and the entropy prox-function (\AGDE, \AGDEnt), projected
gradient (\PG), Frank--Wolfe variants, a Euclidean bundle-level method,
projected subgradient (\PSG), entropic mirror descent (\MDE) and extragradient
(\EG). \Cref{app:experiments} describes the instances, the implementations and
a larger controlled study.

\begin{figure}[htbp]
\centering
\includegraphics[width=\linewidth]{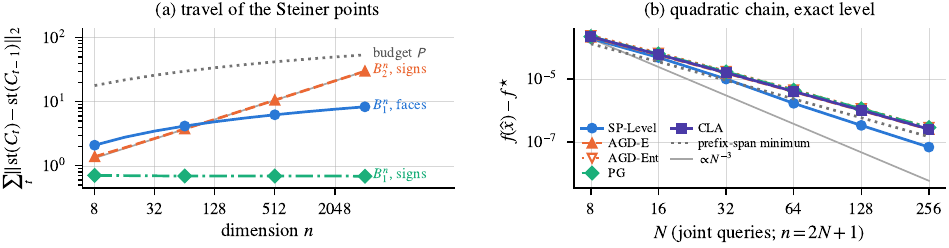}
\caption{(a) Total movement of Steiner points along nested chains in the unit
$\ell_1$ and $\ell_2$ balls. (b) Error on the rotated quadratic chain, $n=2N+1$.}
\label{fig:main-mechanism}
\end{figure}

\paragraph{Steiner paths.} \Cref{fig:main-mechanism}(a) shows the quantity
bounded in \cref{lem:steiner-path} for the chain
$C_t=K\cap\{x:x_1\le0,\dots,x_t\le0\}$, $t\le n$ (``signs''). In the Euclidean
ball the Steiner points travel about $0.47\sqrt n$. In the $\ell_1$ ball they
travel about $0.69$ for every $n$ between $8$ and $4096$, a factor $43$ less at
$n=4096$. The chain of simplex faces from \cref{sec:steiner-main} (``faces'')
travels $8.34$ at $n=4096$, in agreement with its $R\log n$ growth. The dotted
curve is the path budget $\PN$ at $N=n$.

\paragraph{The level method.} \Cref{fig:main-mechanism}(b) uses Nesterov's
worst-case quadratic chain, rotated so that the minimizer is $5$-sparse, in
dimension $n=2N+1$. On this family Euclidean and entropic methods remove only a
small part of the initial gap; the dotted line is the best error attainable in
the span of the first $N$ chain coordinates. \SPL, run at the exact level
$\ell=\fstar$ with $m=256$, has a fitted slope $-2.3$ over $N=8,\dots,256$, and at
$N=256$ its error is $0.27$ times that of \AGDE. These budgets are far below the regime of
\cref{thm:splevel}, whose constant keeps the guarantee above the initial gap up
to $N\approx3\cdot10^7$.

\paragraph{The complete method.} In a study with fixed dimension $n=65$,
$5$-sparse minimizers and $64$ queries, \SPA, which also searches for the level,
leaves $0.23$ of the initial gap on the quadratic chain and $0.11$ on a
nonquadratic one (medians over three instances). A feasible accelerated method
leaves $0.71$ and $0.74$, and the bundle-level method $0.39$ and $0.58$. The
advantage nearly disappears at $n=129$ ($0.85$ against $0.87$ on one instance),
and with $64$ directions on the budget-indexed chains of \cref{fig:main-mechanism}(b)
the complete method does not improve on its initial point
(\cref{app:experiments-full}). It is also paid for in time: a run of \SPA{}
takes 74--94 seconds, mostly in linear programs, against milliseconds for
accelerated gradient.

\paragraph{Curvature learning.} On a quadratic chain in dimension $n=2N+1$ with
the $1$-sparse minimizer $0.9e_1$, the error of \CLA{} drops by a factor $8.7$
between $N=400$ and $N=800$, a local slope of $-3.1$ that matches the predicted
$N^{-3}$, while the error of \AGDE{} drops by $4.0$; at $N=800$ \CLA{} is $3.6$
times more accurate. The gain depends on the sparsity: at $n=65$ and $N=128$ (so
$N>n$), \CLA{} improves on \AGDE{} by a factor $6$ for a $1$-sparse minimizer
but not for a $5$-sparse or a dense one.

\paragraph{Nonsmooth problems and games.} On $\ell_\infty$ regression with
$n=4N+1$, the error of \SPC{} decays like $N^{-1.8}$, against $N^{-0.7}$ for
\PSG{} and $N^{-0.3}$ for \MDE, and its certificate stays within a factor $1.25$
of the true error for $N\ge64$. On bilinear games with $n=2Q$, the saddle gap of
\SPV{} after $Q=256$ operator calls is $8\cdot10^{-7}$, against
$5.6\cdot10^{-3}$ for \EG. The bundle-level method, which, unlike \SPC, uses
function values, is more accurate than \SPC{} on the nonsmooth instances of the
controlled study.

Overall, the path bound is visible already in small dimension, the level
method and \CLA{} decay faster than the Euclidean rate on the chains, and the
deep-cut methods need far fewer calls than subgradient and extragradient
methods. Every Steiner-point method is slower in wall-clock time by several
orders of magnitude, because each center costs $m$ linear programs; a faster
estimator of Steiner points is the main practical open problem.

\section{Conclusion}
\label{sec:conclusion}

When the gradient is Lipschitz in the Euclidean norm and the feasible set is an
$\ell_1$ ball, the minimax error of first-order methods in dimension at least
proportional to $N$ is $LR^2/N^3$ up to logarithmic factors, not the rate $LR^2/N^2$ of methods whose prox-function
matches the smoothness norm. The reason is geometric: Steiner points of nested
subsets of the $\ell_1$ ball move little, and an accelerated level method with
Steiner centers turns this stability into acceleration. The same mechanism
doubles the exponent for nonsmooth objectives and for monotone variational
inequalities, and gives optimal exponents for H\"older and higher-order classes.
Four questions remain open: whether the logarithmic factors are necessary;
whether Steiner points can be approximated deterministically in polynomial time,
or fast enough to make the methods competitive in running time; whether the
exponent two for monotone operators is also optimal for randomized methods; and
which lower bounds hold in the five rows of \cref{tab:intro} that have none.
Concurrent work is discussed in \cref{app:concurrent}.

\label{main-text-end}
\clearpage
\subsection*{Reproducibility statement}
Each result states its assumptions, and complete proofs are in
\crefrange{app:steiner}{app:quadratics}. \Crefrange{alg:splevel}{alg:cla-certified}
specify our methods, and \cref{app:experiments} gives the instances, parameters,
seeds, software and hardware of all experiments.

\subsection*{AI use statement}
Generative AI tools were used for language editing and to improve the clarity of
the manuscript. We reviewed all resulting changes and take full responsibility
for the final content.

\bibliography{references}
\bibliographystyle{plainnat}
\clearpage
\appendix
\etocdepthtag.toc{mtappendix}
\etocsettagdepth{mtmain}{none}
\etocsettagdepth{mtappendix}{subsection}
\etocsettocstyle{\section*{Contents of the appendix}}{}
\tableofcontents
\clearpage
\section{Notation and conventions}
\label{app:notation}

\paragraph{Norms and sets.}
Throughout, $n\ge1$ is the dimension, $e_1,\dots,e_n$ are the coordinate
vectors, $\1$ is the all-ones vector, $\ip xy$ is the Euclidean inner product,
$1\le p\le2$, and $q$ is the conjugate exponent, $1/p+1/q=1$ ($q=\infty$ for
$p=1$). We write
$B_p^n(R)=\{x\in\R^n:\nrm{x}_p\le R\}$, $B_p(c,R)=c+B_p^n(R)$, and
$\Delta_n(R)=\{x\in\R^n:x\ge0,\ \sum_ix_i=R\}$ for the simplex of radius $R$,
with $\Delta_n=\Delta_n(1)$. The feasible set $K$ is nonempty, convex and
compact. For the Steiner-point methods it is a \emph{known polytope}: a polytope
given to the method by linear inequalities and contained in a translate
$c+B_1^n(R)$ of an $\ell_1$ ball; the main examples are $B_1^n(R)$ and
$\Delta_n(R)$. Only the objective is accessed through the oracle. If $x$ has at
most $d$ nonzero coordinates, then
\begin{equation}\label{eq:sparse-inclusion}
 \nrm{x}_2\le\nrm{x}_p\le d^{1/p-1/2}\nrm{x}_2 ,
\end{equation}
and we write $\beta_p=2/p-1\in[0,1]$, so that $\nrm{x}_p^2\le n^{\beta_p}\nrm{x}_2^2$
on $\R^n$. The support function of a nonempty compact convex set $C$ is
$\hs{C}(z)=\max_{x\in C}\ip zx$, and $\proj_C$ is the Euclidean projection onto
$C$.

\paragraph{Regularity constants.}
Unless stated otherwise, $L$ is a Euclidean smoothness constant,
$\nrm{\nabla f(x)-\nabla f(y)}_2\le L\nrm{x-y}_2$. We write $L_1$ for the
Lipschitz constant of $\nabla f$ from $(\R^n,\nrm\cdot_1)$ to
$(\R^n,\nrm\cdot_\infty)$. These are upper bounds known to the method. For an integer $k\ge1$ and $0\le\nu\le1$, $L_{k,\nu}$ is a
H\"older constant of the $k$th derivative in the Euclidean operator norm,
\begin{equation}\label{eq:holder-tensor}
 \nrm{D^kf(x)-D^kf(y)}\le L_{k,\nu}\nrm{x-y}_2^{\nu},\qquad
 \nrm{T}=\sup_{\nrm{h_1}_2,\dots,\nrm{h_k}_2\le1}|T[h_1,\dots,h_k]|
\end{equation}
(for $\nu=0$ the right-hand side is read as $L_{k,0}$ for $x\ne y$). Thus
$L_{1,1}=L$. A nonsmooth function is $G$-Lipschitz if
$|f(x)-f(y)|\le G\nrm{x-y}_2$. A function is $\mu_p$-strongly convex with
respect to $\nrm\cdot_p$ on $K$ if
$f(y)\ge f(x)+\ip{g}{y-x}+\frac{\mu_p}2\nrm{y-x}_p^2$ for all $x,y\in K$ and
$g\in\partial f(x)$, and $\kappa_p=L/\mu_p$. For a differentiable convex $h$,
$D_h(u,z)=h(u)-h(z)-\ip{\nabla h(z)}{u-z}$ is its Bregman divergence. We use
the integer logarithmic factors
\begin{equation}\label{eq:log-factors}
 \bN=1+\lceil\log_2N\rceil,\qquad
 \cn=1+\lceil\log_2(2n)\rceil,
\end{equation}
which are computed by comparing powers of two and satisfy $\bN\ge1+\log N$ and
$\cn\ge1+\log(2n)$. We also write $t_+=\max\{0,t\}$, $\log_+t=\max\{0,\log t\}$ and
$\ceilp{t}=\max\{0,\lceil t\rceil\}$.

\paragraph{Oracles.}
At a query point $x$, the \emph{first-order oracle} (also called the joint
oracle) returns $(f(x),\nabla f(x))$, the \emph{gradient oracle} returns
$\nabla f(x)$ only, and the \emph{oracle of order $k$} returns $f(x)$ and
$D^jf(x)$ for $j\le k$. A \emph{subgradient oracle} returns one element of
$\partial f(x)$, with or without $f(x)$ (each result says which), and the
\emph{operator oracle} of a variational inequality returns $F(x)$. Queries are
\emph{global} (anywhere in $\R^n$) unless they are called \emph{feasible} (in
$K$). All upper bounds use feasible queries, except the finite-difference
products for quadratics in \cref{app:quadratics}, and all lower bounds allow
global queries.

\paragraph{What is counted.}
A query is one call of the oracle at one point, and every query is counted,
including the queries of trials that a level method rejects, the second query of
each trial of \SPT{} and \SPV, finite-difference probes, and the queries that
compute an initial bracket; for \SPP{} every gradient of the inner loops is
counted. Computations with known functions and sets are not counted. These
include the minimization of an explicitly given model over $K$, Euclidean
projections, linear programs over known polytopes, centroids and approximate
Steiner points of known polytopes, matrix computations with stored vectors, and
comparisons of real numbers. This is the information-based complexity model of
\citet{nemirovski1983}. It does not measure arithmetic cost; \cref{app:level}
states which internal operations are finite or polynomial.

\paragraph{Minimax error.}
For a class $\cF$ the minimax error of deterministic methods is
$\cE_N^{\mathrm{det}}(\cF;K)=\inf_{\cA}\sup_{f\in\cF}[f(\widehat x_{\cA,f})-\min_Kf]$,
where $\cA$ ranges over deterministic methods with at most $N$ queries and
$\widehat x_{\cA,f}\in K$ is the output. In $\cE_N^{\mathrm{ran}}(\cF;K)$ the
methods may be randomized and the error is replaced by its expectation over the
internal randomness of the method; our randomized lower bounds exhibit, for
every method, a fixed function on which the expected error is large. We write
$\fstar=\min_Kf$.

\paragraph{Dimension and mismatch.}
The deterministic lower bounds of \cref{thm:first-order-lower,thm:coordinate-tensor-lower}
hold for $n\ge N+1$ (\cref{cor:equal-constants}: $n\ge N+2$; quadratics:
$n\ge3N+2$; \cref{thm:vi-lower}: $n\ge4N+6$), and the randomized ones for
$n\ge32N+1$ (smooth and higher order) or $n\ge2^{p+2}N$ (nonsmooth, $8N$ on the
$\ell_1$ ball); the upper bounds hold in every dimension. A convex quadratic on $\R^n$ is identified by $n+1$
gradient queries (\cref{app:quadratics}), so lower bounds for
quadratics require $N<n$. For every function the smallest valid $L_1$ is at most
the smallest valid $L$, which is at most $n$ times the former; both extremes
occur, for $f=\frac12\nrm x_2^2$ and $f=\frac12(\1^\top x)^2$. When $L\asymp nL_1$,
the entropic rate $L_1R^2\log n/N^2$ is better than $LR^2\log(2N)\log(2n)/N^3$ in
the regime $N\le n$; the new rate matters when $L$ and $L_1$ are comparable, and
by \cref{cor:equal-constants} the lower bound persists when the two smallest
constants coincide.

\begin{table}[htbp]
\centering
\caption{Main symbols.}
\label{tab:notation}
\small
\begin{tabularx}{\linewidth}{@{}p{.25\linewidth}Y@{}}
\toprule
Symbol & Meaning\\
\midrule
$n$, $N$ & dimension; number of queries, or horizon of one call of a level method\\
$B_p^n(R)$, $B_p(c,R)$, $\Delta_n(R)$ & $\ell_p$ ball of radius $R$ centered at $0$ or $c$; simplex of radius $R$\\
$p$, $q$, $\beta_p$ & $1\le p\le2$, $1/p+1/q=1$, $\beta_p=2/p-1$\\
$K$; $C$, $C_t$ & feasible set; localization polytopes\\
$\hs C$, $\st C$, $\gw C$ & support function, Steiner point, Gaussian width (\cref{app:steiner})\\
$L$, $L_1$ & Euclidean and $\ell_1\to\ell_\infty$ gradient Lipschitz constants\\
$L_{k,\nu}$, $G$ & H\"older constant of $D^kf$; Lipschitz constant of $f$\\
$s$ & $k+\nu$, the order of H\"older smoothness\\
$\mu_p$, $\kappa_p$, $\Lambda$ & strong convexity modulus in $\nrm\cdot_p$; $L/\mu_p$; slack $\kappa_1-n$ (\cref{app:strong})\\
$\bar\kappa$ & $L_{k,\nu}R^{s-2}/\mu_1$, condition number of order $s$ (\cref{app:restart})\\
$\bN$, $\cn$, $\PN$ & $1+\lceil\log_2N\rceil$, $1+\lceil\log_2(2n)\rceil$, path budget $4R\sqrt{\bN\cn}$\\
$\eta_N$, $\delta_N$ & center accuracy $\PN/(4N)$ and acceptance threshold $4\PN/N$ of \SPL\\
$\ell$; $\flo$, $\fhi$, $w$ & level; ends and width of the bracket for $\fstar$\\
$\varphi_t$, $\varphi$ & prefix maximum and final hard function of the resisting oracle (\cref{sec:lower-main})\\
$A$, $a$, $B$; $x,y,z$ & weights ($B=A+a$) and sequences of a level method\\
$F$, $T_z$, $r$, $M$ & operator; operator model at $z$, its error exponent and constant (\cref{app:vi})\\
$\psi_{\tau,p}$, $\Psi_I$ & power--Huber function and regularizer with promoted set $I$ (\cref{app:quadratics})\\
\bottomrule
\end{tabularx}
\end{table}

The algorithms are \SPL{} (\cref{alg:splevel}), \SPA{} (\cref{alg:spaccel}),
\SPC{} (\cref{alg:spcut}), \SPT{} (\cref{alg:sptaylor}), \SPP{}
(\cref{alg:spprox}), \CGC{} (\cref{alg:cgcut}), \SPR{} (restarts,
\cref{app:restart}), \SPV{} (\cref{alg:spvi}) and \CLA{} (\cref{alg:cla}).

\section{Steiner points and short paths}
\label{app:steiner}

This section proves the geometric facts used by all Steiner-point methods: the
basic properties of the Steiner point, the path bound of
\cref{lem:steiner-path} in a general form, and its tightness.

\paragraph{Steiner point and Gaussian width.}
Let $Z\sim\mathcal N(0,I_n)$. For a nonempty compact convex $C\subseteq\R^n$ the
Steiner point is
\begin{equation}\label{eq:steiner-def}
 \st{C}=\E\bigl[Z\,\hs{C}(Z)\bigr]=\E\bigl[\nabla\hs{C}(Z)\bigr]\in C,
\end{equation}
where the equality of the two expectations is Gaussian integration by parts
and $\nabla\hs{C}(Z)$ is the almost surely unique maximizer of $\ip Zx$ over
$x\in C$, so that \eqref{eq:steiner-def} agrees with
\eqref{eq:steiner-def-main} (\cref{lem:steiner-basic}). Since $\E\nrm Z_2^2=n$ and $\hs C$ is positively
homogeneous, $\E[Z\hs C(Z)]=n\,\E[\theta\,\hs C(\theta)]$ for $\theta$ uniform on
the unit sphere, so \eqref{eq:steiner-def} is the classical Steiner point
\citep{schneider2014,shephard1968}; \citet{bubeck2020chasing} use this spherical
form for nested convex body chasing. We measure the size of a set by its
root-mean-square Gaussian width
\begin{equation}\label{eq:gaussian-width}
 \gw{C}=\bigl(\E\,\hs{C}(Z)^2\bigr)^{1/2}.
\end{equation}
It depends on the position of $C$; the bounds below use translates $K-c$. For
instance, $\hs{B_1^n(R)}(Z)=R\nrm Z_\infty$ and $\hs{B_2^n(R)}(Z)=R\nrm Z_2$, so
\begin{equation}\label{eq:width-two-balls}
 \gw{B_1^n(R)}\le R\sqrt{2(1+\log(2n))},
 \qquad \gw{B_2^n(R)}=R\sqrt n ,
\end{equation}
by \cref{lem:gaussian-max} below and $\E\nrm Z_2^2=n$.

\begin{lemma}[Basic properties]\label{lem:steiner-basic}
Let $C\subseteq\R^n$ be nonempty, compact and convex. Then $\hs{C}$ is convex
and Lipschitz continuous with constant $\max_{x\in C}\nrm x_2$, differentiable
almost everywhere with $\nabla\hs{C}(z)\in C$ at every point of
differentiability, and
\[
 \st{C}=\E[Z\hs{C}(Z)]=\E[\nabla\hs{C}(Z)]\in C.
\]
If $C$ is symmetric about $c$, then $\st{C}=c$; in particular $\st{B_1^n(R)}=0$.
\end{lemma}

\begin{proof}
Let $\rho=\max_{x\in C}\nrm x_2$. As a maximum of linear functions, $\hs C$ is
convex, and $|\hs C(z)-\hs C(z')|\le\max_{x\in C}|\ip{z-z'}{x}|\le\rho\nrm{z-z'}_2$.
By Rademacher's theorem $\hs C$ is differentiable almost everywhere. Let $z$ be a
point of differentiability and let $x\in C$ be any maximizer of $\ip z\cdot$ over
$C$. For every $v$ and every real $t$,
$\hs C(z+tv)\ge\ip{z+tv}{x}=\hs C(z)+t\ip vx$; dividing by $t>0$ and by $t<0$ and
letting $t\to0$ gives $\ip{\nabla\hs C(z)}v=\ip vx$ for all $v$, so
$\nabla\hs C(z)=x\in C$ (in particular the maximizer is unique at such $z$).

\emph{Integration by parts.} Fix $i$ and all coordinates of $z$ except $z_i=u$,
and put $\varphi(u)=\hs C(z)$. The function $\varphi$ is $\rho$-Lipschitz, hence
absolutely continuous on bounded intervals with $|\varphi'|\le\rho$ almost
everywhere, and $|\varphi(u)|\le\rho\nrm z_2$. For $a>0$,
\[
 \int_{-a}^au\varphi(u)e^{-u^2/2}\,du=\bigl[-\varphi(u)e^{-u^2/2}\bigr]_{-a}^{a}
 +\int_{-a}^a\varphi'(u)e^{-u^2/2}\,du .
\]
The boundary term tends to $0$ as $a\to\infty$ because $\varphi$ grows at most
linearly, and both integrands are integrable on $\R$, so
$\int u\varphi(u)e^{-u^2/2}du=\int\varphi'(u)e^{-u^2/2}du$. Since
$\nrm{z\hs C(z)}_2\le\rho\nrm z_2^2$ is Gaussian-integrable, Fubini's theorem
gives $\E[Z_i\hs C(Z)]=\E[\partial_i\hs C(Z)]$, where $\partial_i\hs C$ is the
partial derivative, which exists almost everywhere and coincides with the $i$th
component of $\nabla\hs C$ at the points of differentiability. Hence
$\E[Z\hs C(Z)]=\E[\nabla\hs C(Z)]$.

\emph{Membership.} The random vector $V=\nabla\hs C(Z)$ is bounded and takes
values in the closed convex set $C$ almost surely. For every closed half-space
$\{x:\ip ax\le b\}\supseteq C$ we have $\ip a{\E V}=\E\ip aV\le b$, and $C$ is
the intersection of the closed half-spaces containing it, so $\E V\in C$.

\emph{Symmetry.} Translation gives $\hs C(z)=\ip cz+\hs{C-c}(z)$ and
$\E[Z\ip cZ]=c$, so $\st C=c+\st{C-c}$. If $C-c=-(C-c)$, then
$\hs{C-c}(-z)=\hs{C-c}(z)$, and since $-Z$ has the law of $Z$,
$\st{C-c}=\E[Z\hs{C-c}(Z)]=\E[(-Z)\hs{C-c}(-Z)]=-\st{C-c}$, so $\st{C-c}=0$.
\end{proof}

\begin{lemma}[Maxima of Gaussian variables]\label{lem:gaussian-max}
Let $X_1,\dots,X_T$ be centered Gaussian variables, not necessarily independent,
with variances at most one, and let $A=\max\{0,X_1,\dots,X_T\}$. Then
$\E A^2\le2(1+\log T)$. In particular $\E\nrm Z_\infty^2\le2(1+\log(2n))$.
\end{lemma}

\begin{proof}
Each $X_t$ satisfies $\Prob(X_t>u)\le e^{-u^2/2}$ for $u\ge0$, so the union bound
gives $\Prob(A>u)\le\min\{1,Te^{-u^2/2}\}$. With $u_0=\sqrt{2\log T}$,
\[
 \E A^2=\int_0^\infty2u\,\Prob(A>u)\,du
 \le u_0^2+\int_{u_0}^\infty2uTe^{-u^2/2}\,du=2\log T+2 .
\]
For the second claim apply the first to the $2n$ variables $\pm Z_i$, whose
maximum is $\nrm Z_\infty$.
\end{proof}

The path bound holds for nested subsets of any compact convex set; the domain
enters only through the width of a centered copy.

\begin{lemma}[Short Steiner paths, general form]\label{lem:steiner-path-general}
Let $T\ge1$, let $K\subseteq\R^n$ be nonempty, compact and convex, and let
$C_0\supseteq C_1\supseteq\dots\supseteq C_T$ be nonempty compact convex subsets
of $K$. Then, for every $c\in\R^n$,
\begin{equation}\label{eq:steiner-path-general}
 \sum_{t=1}^T\nrm{\st{C_t}-\st{C_{t-1}}}_2\ \le\ 2\sqrt{2(1+\log T)}\;\gw{K-c}.
\end{equation}
\end{lemma}

\begin{proof}
Let $v_t$ be the unit vector along $\st{C_{t-1}}-\st{C_t}$, or $v_t=0$ if the
two points coincide, and put $A(Z)=\max\{0,\ip{v_1}Z,\dots,\ip{v_T}Z\}$.

\emph{Step 1: telescoping.} By \cref{lem:steiner-basic} and linearity of
expectation,
\[
 \sum_{t=1}^T\nrm{\st{C_{t-1}}-\st{C_t}}_2
 =\sum_{t=1}^T\ip{v_t}{\st{C_{t-1}}-\st{C_t}}
 =\E\sum_{t=1}^T\ip{v_t}Z\bigl(\hs{C_{t-1}}(Z)-\hs{C_t}(Z)\bigr).
\]
The sets are nested, so every difference $\hs{C_{t-1}}(Z)-\hs{C_t}(Z)$ is
nonnegative, and $\ip{v_t}Z\le A(Z)$. Hence
\begin{equation}\label{eq:path-telescope}
 \sum_{t=1}^T\nrm{\st{C_{t-1}}-\st{C_t}}_2
 \le\E\Bigl[A(Z)\sum_{t=1}^T\bigl(\hs{C_{t-1}}(Z)-\hs{C_t}(Z)\bigr)\Bigr]
 =\E\bigl[A(Z)\bigl(\hs{C_0}(Z)-\hs{C_T}(Z)\bigr)\bigr].
\end{equation}

\emph{Step 2: the range of the support functions.} Fix $y\in C_T$. Since
$C_0\subseteq K$, we have $\hs{C_0}(Z)\le\ip cZ+\hs{K-c}(Z)$ and
$\hs{C_T}(Z)\ge\ip yZ\ge\ip cZ-\hs{K-c}(-Z)$. Therefore
\begin{equation}\label{eq:path-range}
 0\le\hs{C_0}(Z)-\hs{C_T}(Z)\le\hs{K-c}(Z)+\hs{K-c}(-Z),
\end{equation}
and since $-Z$ has the law of $Z$, the second moment of the right-hand side is
at most $4\gw{K-c}^2$.

\emph{Step 3: Cauchy--Schwarz.} By \eqref{eq:path-telescope},
\eqref{eq:path-range} and the Cauchy--Schwarz inequality,
\[
 \sum_{t=1}^T\nrm{\st{C_{t-1}}-\st{C_t}}_2
 \le\bigl(\E A(Z)^2\bigr)^{1/2}\cdot2\gw{K-c}.
\]
Each $\ip{v_t}Z$ is Gaussian with variance at most one, so
$\E A(Z)^2\le2(1+\log T)$ by \cref{lem:gaussian-max}.
\end{proof}

\begin{proof}[Proof of \cref{lem:steiner-path}]
Apply \cref{lem:steiner-path-general} with $K=c+B_1^n(R)$. Then $K-c=B_1^n(R)$, and
$\gw{B_1^n(R)}\le R\sqrt{2(1+\log(2n))}$ by \eqref{eq:width-two-balls}.
\end{proof}

The Gaussian vector appears only in the proof: \eqref{eq:steiner-path} is a
deterministic bound for every fixed chain, and it requires neither
full-dimensional sets nor central cuts. The lemma is a variant of the mean-width
bound of \citet[Lemma~3.2 and Theorem~2.1]{bubeck2020chasing}; the Gaussian form
gives explicit constants and a bound that is sharper for chains shorter than
$n$. For the Euclidean ball, \eqref{eq:steiner-path-general} gives
$2\sqrt{2(1+\log T)}\,R\sqrt n$. The path budget
$\PN=4R\sqrt{\bN\cn}$ of \eqref{eq:main-P} dominates the right-hand side of
\eqref{eq:steiner-path} for every $T\le N$, because $\bN\ge1+\log N$ and
$\cn\ge1+\log(2n)$. Its two logarithmic factors cannot be removed from a bound
that depends on the domain only through its width, even for the simplex.

\begin{proposition}[Tightness of the path bound]\label{prop:path-tight}
Let $n\ge2$, $R>0$ and $C_j=\conv\{Re_1,\dots,Re_j\}$, $1\le j\le n$. Then
$C_n\supseteq C_{n-1}\supseteq\dots\supseteq C_1$ are nonempty compact convex
subsets of $\Delta_n(R)\subseteq B_1^n(R)$, $\st{C_j}=\frac Rj\sum_{i\le j}e_i$, and
\[
 R(H_n-1)\le\sum_{j=2}^n\nrm{\st{C_j}-\st{C_{j-1}}}_2\le RH_{n-1},\qquad
 H_n=\sum_{j=1}^n\frac1j ,
\]
while $\gw{\Delta_n(R)}\le R\sqrt{2(1+\log(2n))}$. Moreover, for every rule that
selects a point of the current set (deterministic or randomized), there is a
nested chain of $n$ faces of $\Delta_n(R)$, revealed one at a time, along which
the selected points move by at least $R(H_n-1)$ in expectation.
\end{proposition}

\begin{proof}
For $Z\sim\mathcal N(0,I_n)$ the maximizer of $\ip Zx$ over $C_j$ is $Re_i$ with
$i$ the index of the largest of $Z_1,\dots,Z_j$, almost surely unique, and by
symmetry each index has probability $1/j$; \cref{lem:steiner-basic} gives the
formula for $\st{C_j}$. Then
$\nrm{\st{C_j}-\st{C_{j-1}}}_2^2=R^2/j^2+(j-1)R^2(\frac1{j-1}-\frac1j)^2=R^2/(j(j-1))$,
and $1/j\le1/\sqrt{j(j-1)}\le1/(j-1)$ gives the two bounds. The width bound
follows from $\hs{\Delta_n(R)}(Z)=R\max_iZ_i$, $(\max_iZ_i)^2\le\nrm Z_\infty^2$
and $\E\nrm Z_\infty^2\le2(1+\log(2n))$ (\cref{lem:gaussian-max}). For the last claim, let the
current set be the face $C_S=\conv\{Re_i:i\in S\}$, $|S|=j\ge2$, and let $x$ be
the selected point, so $x_i\ge0$ and $\sum_{i\in S}x_i=R$. Remove an index $i\in S$
chosen uniformly at random (independently of the selector); the next selected
point $y$ lies in $C_{S\setminus\{i\}}$, so $y_i=0$ and
$\E[\nrm{x-y}_2\mid x]\ge\frac1j\sum_{i\in S}x_i=R/j$. Summing over
$j=n,\dots,2$ gives $R(H_n-1)$; averaging over the random order of removal shows
that some fixed order gives at least this expected movement (for a deterministic
rule the choice $i\in\argmax_{i\in S}x_i$ gives the bound directly).
\end{proof}

Thus at $T=n-1$ both factors of $4R\sqrt{(1+\log T)(1+\log(2n))}$ are of the
right order for the simplex, and no online rule does better on such chains.
Whether these logarithms are necessary in the oracle complexity of smooth
minimization is open.

\section{The accelerated level method}
\label{app:level}

This section proves \cref{thm:splevel}, analyzes the bracket search and proves
\cref{thm:overview}, and then turns to the computation of Steiner points: an
exact finite procedure, a randomized estimator with polynomially many linear
programs, and a version for inexact oracles in finite precision. The last
subsection extends the method to other domains and smoothness metrics.
Throughout, $K\subseteq c+B_1^n(R)$ is a known polytope, and $\PN$, $\bN$, $\cn$
are as in \eqref{eq:main-P}.

\subsection{Proof of the guarantee of the level method}
\label{app:level-proof}

We first isolate the one-step inequality, because \cref{app:holder,app:tensor}
reuse it with other remainders.

\begin{lemma}[One accepted step]\label{lem:level-step}
Let $f$ be convex with $L$-Lipschitz Euclidean gradient on $K$, let $A\ge0$,
$a>0$, $B=A+a$, let $y,z,z^+\in K$, and set $x=(Ay+az)/B$ and
$y^+=(Ay+az^+)/B$. If $f(x)+\ip{\nabla f(x)}{z^+-x}\le\ell$, then
\eqref{eq:main-energy} holds.
\end{lemma}

\begin{proof}
Since $y^+-x=(a/B)(z^+-z)$, the descent inequality between $x$ and $y^+$ gives
$Bf(y^+)\le Bf(x)+a\ip g{z^+-z}+\frac{La^2}{2B}\nrm{z^+-z}_2^2$ with
$g=\nabla f(x)$. By $Bx=Ay+az$ its first two terms equal
$A[f(x)+\ip g{y-x}]+a[f(x)+\ip g{z^+-x}]\le Af(y)+a\ell$ (convexity at $x$ and
the level cut), and subtracting $B\ell=A\ell+a\ell$ gives \eqref{eq:main-energy},
also when $A=0$.
\end{proof}

The proof uses only the descent inequality between $x$ and $y^+$, convexity at
$x$ and the level cut at $z^+$. Replacing $\frac L2\nrm{y^+-x}_2^2$ by any
remainder $\phi(\nrm{y^+-x}_2)$ therefore replaces
$\frac{La^2}{2B}\nrm{z^+-z}_2^2$ by $B\,\phi(\tfrac aB\nrm{z^+-z}_2)$.

\thmsplevel*

\begin{proof}
All iterates are convex combinations of points of $K$, so every query is
feasible. If $\min_Kf\le\ell$, a minimizer $u_*$ satisfies
$f(x_t)+\ip{g_t}{u_*-x_t}\le f(u_*)\le\ell$ for every $t$, hence $u_*\in C_t$ and
no localization set is empty; an empty $C_t$ therefore certifies $\min_Kf>\ell$.

Assume $C_1,\dots,C_N$ are nonempty. The selected centers satisfy $\nrm{z_t-\st{C_t}}_2\le\eta_N$ and the sets $C_t$
are nested in $K$, so \cref{lem:steiner-path} and the triangle inequality give
\eqref{eq:actual-path}, with $2N\eta_N=\frac12\PN$. Every rejected trial has
$d_t>\delta_N=4\PN/N$, so their number $r$ obeys $4r\PN/N<\frac32\PN$, that is
$r<3N/8$, and $k=N-r>5N/8>N/2$ trials are accepted. On an accepted trial the
membership $z_t\in C_t$ is the hypothesis of \cref{lem:level-step},
and the choice $La_t^2=B_t$ makes $La_t^2/(2B_t)=\frac12$, so
$A_t[f(y_t)-\ell]\le A_{t-1}[f(y_{t-1})-\ell]+\frac12d_t^2$; on a rejected trial
$A_t=A_{t-1}$ and $y_t=y_{t-1}$. Telescoping from $A_0=0$, bounding
$d_t\le\delta_N$ on the accepted trials and using \eqref{eq:actual-path} gives
\begin{equation}\label{eq:energy-sum}
 A_N[f(y_N)-\ell]\ \le\ \frac12\sum_{t\ \mathrm{accepted}}d_t^2
 \ \le\ \frac{\delta_N}{2}\sum_{t=1}^Nd_t\ \le\ \frac{3\PN^2}{N}.
\end{equation}

Number the accepted updates $j=1,\dots,k$. The recursion $La_j^2=A_j$ with
$A_j=A_{j-1}+a_j$ means $a_j=(1+\sqrt{1+4LA_{j-1}})/(2L)$, so
$\sqrt{LA_j}=\frac12(1+\sqrt{1+4LA_{j-1}})\ge\sqrt{LA_{j-1}}+\frac12$ and
$A_N=A_k\ge k^2/(4L)\ge N^2/(16L)$. Dividing \eqref{eq:energy-sum} by $A_N$
proves \eqref{eq:level-error}.
\end{proof}

\subsection{Bracket search and the proof of \cref{thm:overview}}
\label{app:bracket}

\Cref{alg:spaccel} needs no knowledge of $\fstar=\min_Kf$. One query at $x_0\in K$
gives the bracket $\flo\le\fstar\le f(v)\le\fhi$ with
\begin{equation}\label{eq:bracket}
 v\in\argmin_{u\in K}\ip{g_0}{u},\qquad
 \flo=f(x_0)+\ip{g_0}{v-x_0},\qquad
 \fhi=\flo+\tfrac L2\nrm{v-x_0}_2^2 ,
\end{equation}
where $g_0=\nabla f(x_0)$; the first inequality is the tangent lower bound, and
the last one follows from smoothness at $x_0$. Since
$\nrm{v-x_0}_2\le\nrm{v-x_0}_1\le2R$, the width is at most $2LR^2$ for a
general $K\subseteq c+B_1^n(R)$ and at most $LR^2/2$ for $K=B_1^n(R)$ and
$x_0=0$. No query at $v$ is needed. \Cref{alg:spaccel} bisects the bracket, calling
\cref{alg:splevel} at the midpoint with the smallest power-of-two horizon whose
guaranteed error \eqref{eq:level-error} is at most a quarter of the width:
\begin{equation}\label{eq:horizon-rule}
 N(w)=\min\Bigl\{N=2^j:\ \frac{48L\PN^2}{N^3}\le\frac w4\Bigr\}
 =\min\Bigl\{N=2^j:\ N^3\ge\frac{3072LR^2}{w}\,\bN\cn\Bigr\}.
\end{equation}
A negative certificate raises the lower end to $\ell$ (width $w/2$); a positive
answer stores the output and lowers the upper end to $\ell+w/4$ (width $3w/4$).

\begin{algorithm}[htbp]
\caption{\SPA: bracket search for the level}
\label{alg:spaccel}
\begin{algorithmic}[1]
\Require known polytope $K\subseteq c+B_1^n(R)$; first-order oracle of an
$L$-smooth convex $f$; accuracy $\varepsilon>0$
\State query $(f(x_0),g_0)$ at a point $x_0\in K$ ($x_0=0$ if $K=B_1^n(R)$);
compute $v,\flo,\fhi$ by \eqref{eq:bracket}
\While{$w:=\fhi-\flo>\varepsilon$}
 \State $\ell\gets(\flo+\fhi)/2$;\quad $N\gets N(w)$ by \eqref{eq:horizon-rule}
 \State run \cref{alg:splevel} with $(\ell,N)$, restarting from $C_0=K$, $A_0=0$
 \If{it certifies $\min_Kf>\ell$} $\flo\gets\ell$
 \Else\ $v\gets y_N$;\quad $\fhi\gets\ell+w/4$
 \EndIf
\EndWhile
\State\Return $v$ \Comment{$f(v)-\min_Kf\le\fhi-\flo\le\varepsilon$}
\end{algorithmic}
\end{algorithm}

\begin{restatable}[Bracket search]{proposition}{thmmain}
\label{thm:main}
Let $f$ be convex with $L$-Lipschitz Euclidean gradient on $K=B_1^n(R)$.
\begin{enumerate}[label=(\alph*)]
\item For every $0<\varepsilon<LR^2/2$, \cref{alg:spaccel} returns $v\in K$ with
$f(v)-\min_Kf\le\varepsilon$ after at most
\begin{equation}\label{eq:main-count}
 1+24\,N(\varepsilon)
 \ =\ O\Bigl(1+\Bigl[\frac{LR^2}{\varepsilon}\,\log(2n)\,
 \log\Bigl(2+\frac{LR^2}{\varepsilon}\log(2n)\Bigr)\Bigr]^{1/3}\Bigr)
\end{equation}
feasible first-order queries; explicitly, $N(\varepsilon)<[32\Theta(3+\log_2\Theta)]^{1/3}$
with $\Theta=3072LR^2\cn/\varepsilon$.
\item For every budget $T\ge25$ there is a choice of $\varepsilon$ (explicit in
the proof) for which the method makes at most $T$ queries and returns
$v$ with
\begin{equation}\label{eq:main-rate}
 f(v)-\min_Kf\ \le\ \min\Bigl\{\frac{LR^2}2,\ \frac{3072\cdot48^3\,LR^2\,\mathsf b_T\cn}{(T-1)^3}\Bigr\}
 =O\Bigl(\frac{LR^2\log(2T)\log(2n)}{T^3}\Bigr).
\end{equation}
For $T<25$ the point $v$ of \eqref{eq:bracket} has error at most $LR^2/2$.
\end{enumerate}
The same statements hold on every known polytope $K\subseteq c+B_1^n(R)$, in
particular on $\Delta_n(R)$, with $LR^2/2$ replaced by $2LR^2$, an upper bound on
the initial bracket width of \eqref{eq:bracket}; the horizon rule
\eqref{eq:horizon-rule} and the count $1+24N(\varepsilon)$ are unchanged.
\end{restatable}

\begin{proof}
We first treat $K=B_1^n(R)$, $x_0=0$. Write $w$ for the current width
$\fhi-\flo$ and $N(w)$ for the horizon \eqref{eq:horizon-rule}, which exists
because for $N=2^j$ one has $\bN=1+j$ and $2^{3j}/(1+j)\to\infty$.

\emph{Invariants.} By \eqref{eq:bracket}, $\flo\le\min_Kf\le f(v)\le\fhi$ holds
initially with $w=\frac L2\nrm v_2^2\le LR^2/2$. If $w\le\varepsilon$ the method
returns $v$ after one query. Suppose the invariant holds at the start of a phase with
width $w>\varepsilon$ and level $\ell=(\flo+\fhi)/2$. If the level method
returns the certificate, $\min_Kf>\ell$, so $\flo\gets\ell$ keeps the invariant
and the new width is $w/2$. Otherwise \cref{thm:splevel} and the choice of
$N(w)$ give $f(y_N)\le\ell+48L\PN^2/N^3\le\ell+w/4$, so $v\gets y_N$,
$\fhi\gets\ell+w/4$ keeps the invariant and the new width is $3w/4$. The loop
terminates because each phase multiplies the width by at most $3/4$, and at
termination $f(v)-\min_Kf\le\fhi-\flo\le\varepsilon$. Every query is feasible
by \cref{thm:splevel}.

\emph{Phase count.} Let $w_0>w_1>\dots>w_{m-1}>\varepsilon$ be the widths at
the beginnings of the $m$ phases and $N_j=N(w_j)$. Each phase multiplies the
width by at most $3/4$, so $w_j\ge(4/3)^{m-1-j}w_{m-1}$, and $N(\cdot)$
is nonincreasing, so $N_j\le N_{m-1}\le N(\varepsilon)$. Every $N_j$ exceeds
$1$: the condition in \eqref{eq:horizon-rule} at $N=1$ reads
$1\ge3072LR^2\cn/w_j$, which fails because $\cn\ge2$ and $w_j\le LR^2/2$.
Hence $N_j/2$ is a power of two that violates \eqref{eq:horizon-rule}, and
using $\mathsf b_{N_j/2}\le\mathsf b_{N_j}\le\mathsf b_{N_{m-1}}$,
\[
 \frac{N_j^3}{8}<\frac{3072LR^2}{w_j}\mathsf b_{N_{m-1}}\cn,
 \qquad
 N_{m-1}^3\ge\frac{3072LR^2}{w_{m-1}}\mathsf b_{N_{m-1}}\cn .
\]
Dividing, $N_j^3<8(w_{m-1}/w_j)N_{m-1}^3\le8(3/4)^{m-1-j}N_{m-1}^3$, i.e.\
$N_j<2(3/4)^{(m-1-j)/3}N_{m-1}$. Summing the geometric series,
\[
 \sum_{j=0}^{m-1}N_j<\frac{2}{1-(3/4)^{1/3}}N_{m-1}<24\,N(\varepsilon),
\]
because $(11/12)^3=1331/1728>3/4$ gives $(3/4)^{1/3}<11/12$. With the initial
query this proves the count $1+24N(\varepsilon)$ in (a).

\emph{Explicit horizon.} Let $\Theta=3072LR^2\cn/\varepsilon>1$,
$J=\lceil\log_2\Theta\rceil$ and $q=\lceil\log_2(2^{J+1}(J+2))/3\rceil$. Then
$q\le(J+1+\log_2(J+2))/3+1$ and $\log_2(J+2)\le J+1$ give $q+1\le2(J+2)$, hence
$(2^q)^3\ge2^{J+1}(J+2)\ge\Theta(q+1)$, which is \eqref{eq:horizon-rule} for
$N=2^q$ and $w=\varepsilon$ (note $\mathsf b_{2^q}=q+1$). Thus
$N(\varepsilon)\le2^q$ and, since $2^{3q}<8\cdot2^{J+1}(J+2)$ and $2^J<2\Theta$,
\[
 N(\varepsilon)^3<32\Theta(J+2)\le32\Theta(3+\log_2\Theta),
\]
which gives the explicit bound and the order in \eqref{eq:main-count}.

\emph{Fixed horizon.} Let $T\ge25$, let $Q$ be the largest power of two with
$24Q\le T-1$, and $\varepsilon_Q=3072LR^2\mathsf b_Q\cn/Q^3$. If
$\varepsilon_Q\ge LR^2/2$, return the point $v$ of \eqref{eq:bracket}, whose
error is at most $LR^2/2$ after one query. Otherwise run the method with
accuracy $\varepsilon_Q$: $Q$ satisfies \eqref{eq:horizon-rule} for
$w=\varepsilon_Q$ with equality, so $N(\varepsilon_Q)\le Q$, the count is at
most $1+24Q\le T$, and the error is at most $\varepsilon_Q$. By maximality of
$Q$, $48Q>T-1$, and $\mathsf b_Q\le\mathsf b_T$, so
$\varepsilon_Q\le3072\cdot48^3LR^2\mathsf b_T\cn/(T-1)^3$, which is \eqref{eq:main-rate}.

\emph{General polytopes.} For $K\subseteq c+B_1^n(R)$ the level method is
unchanged: \cref{lem:steiner-path} applies to chains in $c+B_1^n(R)$, and
$z_0$ is an $\eta_N$-approximate Steiner point of $K$ computed as in
\cref{app:steiner-computation}. The bracket \eqref{eq:bracket} has width
$\frac L2\nrm{v-x_0}_2^2\le\frac L2(2R)^2=2LR^2$, the condition $N_j>1$ still
holds because $w_j\le2LR^2<6144LR^2$, and the rest of the argument is the same
with $LR^2/2$ replaced by $2LR^2$.
\end{proof}

\begin{proof}[Proof of \cref{thm:overview}]
The lower bound for deterministic methods is \cref{thm:first-order-lower}(b)
with $p=1$ and $\nu=1$, and the lower bound for randomized methods is
\cref{thm:rand-smooth}. For the upper bound, run \cref{alg:spaccel} with the
deterministic selector of \cref{prop:finite-selector} and let $T=N$ in
\cref{thm:main}(b) if $N\ge25$: the method is deterministic, makes at most $N$
feasible queries, and its error is at
most
\[
 \frac{3072\cdot48^3LR^2\mathsf b_N\cn}{(N-1)^3}
 \ \le\ \frac{3072\cdot48^3\cdot8\cdot4}{\log^22}\cdot\frac{LR^2\log(2N)\log(2n)}{N^3}
 \ \le\ \frac{2^{35}LR^2\log(2N)\log(2n)}{N^3},
\]
since $\mathsf b_N\le2\log(2N)/\log2$, $\cn\le2\log(2n)/\log 2$ and $N-1\ge N/2$.
For $N\le24$ the point $v$ of \eqref{eq:bracket} has error at most $LR^2/2$
after one query, which is at most $2^{35}LR^2\log(2N)\log(2n)/N^3$ for every
$n\ge1$, because $\log(2N)\log(2n)\ge\log^22$ and $24^3/(2\log^22)<14387<2^{35}$.
Randomized methods include deterministic ones, so
$\cE_N^{\mathrm{ran}}\le\cE_N^{\mathrm{det}}$.
\end{proof}

The level cuts use function values; \cref{app:gradient-only} shows that the exponent is
unchanged when only $\nabla f$ is returned, with two additional logarithmic
factors.

\subsection{Computing approximate Steiner points}
\label{app:steiner-computation}

Every localization set is a known polytope $C=K\cap\{\text{stored half-spaces}\}$,
and the methods need, for a prescribed $\eta>0$, a point $x(C,\eta)\in C$ with
$\nrm{x(C,\eta)-\st{C}}_2\le\eta$. We give a deterministic and a randomized
implementation.

\begin{restatable}[Deterministic finite selector]{proposition}{propfiniteselector}
\label{prop:finite-selector}
For a known nonempty polytope $C\subseteq c+B_1^n(R)$ and $\eta>0$, a point
$x(C,\eta)$ can be computed by finitely many linear programs over $C$, one
Euclidean projection onto $C$, and rational arithmetic with certified
approximations of $e^{-r}$ and $\sqrt{2\pi}$. The Gaussian integral
\eqref{eq:steiner-def} is truncated to a cube $[-a,a]^n$ and approximated by a
midpoint rule whose nodes require $\hs C$ at rational points, and the result is
projected onto $C$. The number of nodes is finite but grows exponentially in $n$.
\end{restatable}

\begin{proof}
Since $C\subseteq c+B_1^n(R)$, by translation we may assume $c=0$, so
$\nrm x_2\le\nrm x_1\le R$ on $C$ and $|\hs C(z)|\le R\nrm z_\infty$. The
polytope $C$ is given by finitely many known linear inequalities (the $2^n$
facets $\ip\sigma x\le R$, $\sigma\in\{-1,1\}^n$, of the ball, or the facets of
$K$, together with the stored cuts). Enumerating the nondegenerate systems of
$n$ active constraints, solving the linear systems and checking the remaining
constraints yields all vertices of $C$, decides whether $C$ is empty, and
evaluates $\hs C$ at any rational point; lower-dimensional $C$ are included,
since the constraints defining the affine hull are among those enumerated. The
Euclidean projection $\proj_C(w)$ is computed by projecting $w$ onto the affine
hull of every face (enumerating linearly independent subsets of active
constraints, the empty set included), keeping the feasible candidates and
choosing the closest one. This is correct: if $\proj_C(w)$ lies in the relative
interior of a face $C'$, then $w-\proj_C(w)$ lies in the normal cone of $C$ at
$\proj_C(w)$, which is spanned by the active normals, so $\proj_C(w)$ is the
projection of $w$ onto the affine hull of $C'$, a candidate in the list.

\emph{Truncation.} Let $B=\nrm Z_\infty$. For $a\ge1$, the tail bound
$\Prob(B>u)\le2ne^{-u^2/2}$ gives
$\E[B^2\1_{\{B>a\}}]=a^2\Prob(B>a)+\int_a^\infty2u\Prob(B>u)du\le2n(a^2+2)e^{-a^2/2}$,
and since $\nrm{z\hs C(z)}_2\le R\nrm z_2\nrm z_\infty\le R\sqrt n\,\nrm z_\infty^2$,
\[
 \bigl\|\E[Z\hs C(Z)\1_{\{B>a\}}]\bigr\|_2\le2Rn^{3/2}(a^2+2)e^{-a^2/2}.
\]
Choose the smallest even integer $a\ge2$ with $2Rn^{3/2}(a^2+2)2^{-a^2/2}\le\eta/4$
(a condition with integer powers only, which implies the bound with
$e^{-a^2/2}$ since $e>2$).

\emph{Quadrature.} On the cube $[-a,a]^n$ the integrand
$\mathcal I(z)=z\hs C(z)\varphi_n(z)$, $\varphi_n(z)=(2\pi)^{-n/2}e^{-\nrm z_2^2/2}$, is
Lipschitz with constant $J=R(2\sqrt n\,a+na^3)$, since $|\hs C|\le Ra$,
$\Lip(\hs C)\le R$, $\nrm z_2\le\sqrt n\,a$, $\varphi_n\le1$ and
$\nrm{\nabla\varphi_n}_2\le\sqrt n\,a$ on the cube. Split the cube into $m^n$
equal subcubes of side $h=2a/m$ with the smallest integer $m$ such that
$J\sqrt n\,h(2a)^n\le\eta/8$; the midpoint rule then approximates the truncated
integral with Euclidean error at most $\eta/8$, and at every node $\hs C$ is
evaluated by a linear program. The Gaussian density at the finitely many
rational nodes is needed only to absolute accuracy
$\theta=\eta/(8(2a)^nR\sqrt n\,a^2)$, which changes the quadrature by at most
$\eta/8$ because $\nrm{z\hs C(z)}_2\le R\sqrt n\,a^2$ on the cube. Certified
rational approximations of $e^{-r}$ for rational $r$ and of $\sqrt{2\pi}$ to any
positive accuracy are obtained from convergent series with explicit remainders.
The resulting finite sum $w$ satisfies $\nrm{w-\st C}_2\le\eta/4+\eta/8+\eta/8=\eta/2$.

\emph{Projection.} Return $x(C,\eta)=\proj_C(w)\in C$. Since $\st C\in C$
(\cref{lem:steiner-basic}) and the projection is nonexpansive,
$\nrm{\proj_C(w)-\st C}_2=\nrm{\proj_C(w)-\proj_C(\st C)}_2\le\eta/2\le\eta$.
\end{proof}

\begin{restatable}[Randomized selector]{proposition}{thmrandomselector}
\label{thm:random-selector}
Let $C\subseteq c+B_1^n(R)$ be a known nonempty polytope, let
$Z_1,\dots,Z_m$ be independent standard Gaussian vectors, let
$V_i\in\argmax_{u\in C}\ip{Z_i}u$ (a linear program) and
$\bar x=\frac1m\sum_iV_i$. Then $\bar x\in C$ on every outcome, $\E V_i=\st C$,
and for $m\ge4R^2/\eta^2$,
$\Prob\{\nrm{\bar x-\st C}_2>\eta\}\le\exp(-m\eta^2/(8R^2))$.
Consequently, let $\delta\in(0,1)$ and let every selector call in
\cref{alg:spaccel} use
$m_N=\lceil32N^2(1+\log(\bar S/\delta))/(\bN\cn)\rceil$ directions, where $N$ is
the horizon of the current level call and $\bar S=2+48N(\varepsilon)$ bounds the
number of selector calls. Then with probability at least $1-\delta$ all
guarantees of \cref{thm:main} hold, the number of oracle queries is bounded by
\eqref{eq:main-count} on every outcome, and the total number of linear programs
is $O\bigl(N(\varepsilon)+N(\varepsilon)^3\log(N(\varepsilon)/\delta)/(\mathsf b_{N(\varepsilon)}\cn)\bigr)$,
polynomial in $n$, $N(\varepsilon)$ and $\log(1/\delta)$.
\end{restatable}

\begin{proof}
By translation assume $c=0$, so $\nrm{V_i}_2\le R$. Fix a measurable
tie-breaking rule for the maximizer; by \cref{lem:steiner-basic} the maximizer
is almost surely unique and equals $\nabla\hs C(Z_i)$, so $\E V_i=\st C$, and
$\bar x\in C$ on every outcome by convexity. The $V_i$ are independent with
$\E\nrm{V_i-\st C}_2^2\le R^2$, so $\E\nrm{\bar x-\st C}_2^2\le R^2/m$ and
$\E\nrm{\bar x-\st C}_2\le R/\sqrt m\le\eta/2$ when $m\ge4R^2/\eta^2$. Changing
one $V_i$ changes $\bar x$ by at most $2R/m$ in norm, hence changes
$X=\nrm{\bar x-\st C}_2$ by at most $2R/m$. The bounded-differences inequality
of \citet{mcdiarmid1989} gives
$\Prob\{X>\E X+t\}\le\exp(-2t^2/(m(2R/m)^2))=\exp(-mt^2/(2R^2))$; with
$t=\eta/2\le\eta-\E X$ this is the stated bound $\exp(-m\eta^2/(8R^2))$.

Now consider \cref{alg:spaccel} with every selector call replaced by this
average with $m_N=\lceil32N^2(1+\log(\bar S/\delta))/(\bN\cn)\rceil$ directions,
where $N$ is the horizon of the current level call and $\eta_N=\PN/(4N)=R\sqrt{\bN\cn}/N$.
Then $m_N\eta_N^2/(8R^2)\ge4(1+\log(\bar S/\delta))$ and $m_N\ge4R^2/\eta_N^2$, so
conditionally on the past (which fixes the current polytope, while the new
directions are independent of it) each selector call fails, i.e.\ produces a
point farther than $\eta_N$ from the Steiner point, with probability at most
$e^{-4}(\delta/\bar S)^4\le\delta/\bar S$. The algorithm is defined on every outcome.
After $N$ nonempty trials the level method returns the stored $y_N$ (also when
$A_N=0$), and the bracket is updated by the prescribed formulas, so the widths
shrink by the factors $1/2$ or $3/4$ whether or not the upper end is valid. The
phase-count argument of \cref{app:bracket} uses only the widths, so it bounds
the number of oracle queries by $1+24N(\varepsilon)$ on every outcome. Each
level call makes one selector call per trial and one for $z_0$,
so the number of selector calls is less than twice the query bound
$S=1+24N(\varepsilon)$; with $\bar S=2S$, the union bound over the
selector calls in their order of occurrence gives probability at least $1-\delta$
that all of them succeed. On this event every chain of selected points is an
$\eta_N$-accurate approximation of the Steiner points of a nested chain, so
\eqref{eq:actual-path} holds for it, and all statements of
\cref{thm:splevel,thm:main} apply without change.

\emph{Number of linear programs.} All horizons used are powers of two at most
$N(\varepsilon)$, and $N\mapsto N^2/\bN$ is increasing on powers of two (the
ratio of consecutive values is $4(j+1)/(j+2)\ge2$), so every $m_N$ is at most
$m^*=\lceil32N(\varepsilon)^2(1+\log(2S/\delta))/(\mathsf
b_{N(\varepsilon)}\cn)\rceil$. Counting one linear minimization for the bracket,
at most $S$ emptiness tests and, over the fewer than $2S$ selector calls, at
most $2Sm^*$ support linear programs, the total is at most
$1+S(2m^*+1)=O(N(\varepsilon)+N(\varepsilon)^3\log(N(\varepsilon)/\delta)/(\mathsf b_{N(\varepsilon)}\cn))$.
Each linear program maximizes a linear function over
$\{(x,u):-u\le x\le u,\ \sum_iu_i\le R\}$ intersected with at most
$N(\varepsilon)$ stored cuts.
\end{proof}

\subsection{Inexact oracles and finite precision}
\label{app:finite-bit}

\Cref{thm:random-selector} assumes exact oracle answers and linear programs
solved exactly over the reals, with real Gaussian directions.
On $K=B_1^n(R)$ these assumptions can be removed. Let $R,L,\varepsilon$ be
rational and let a query at a rational $x\in K$ return rational
$(\widetilde f_x,\widetilde g_x)$ with
\begin{equation}\label{eq:inexact-oracle}
 |\widetilde f_x-f(x)|\le\alpha:=\frac{\varepsilon}{256},\qquad
 \nrm{\widetilde g_x-\nabla f(x)}_2\le\beta:=\frac{\varepsilon}{512R},
\end{equation}
where the errors may be biased and may depend on the whole past, including the
random bits drawn so far but not later ones. The answers lie on
fixed grids $\widetilde f_x\in2^{-k_f}\mathbb Z$, $\widetilde g_x\in2^{-k_g}\mathbb Z^n$
whose exponents $k_f,k_g$ are part of the input; componentwise accuracy
$\varepsilon/(512Rn)$ suffices for the second condition.
The \emph{rounded level method} is \cref{alg:splevel} with three changes. The
weights of the accepted steps are $a_j=j/(2L)$, $A_j=j(j+1)/(4L)$, where $j$ counts
accepted steps, so no square root is needed. The level cut at $x$ is the
outward-rounded half-space $\{u:\ip{\widetilde g_x}{u}\le c_x\}$ with
$c_x=\rho\lceil(\ell+\tau-\widetilde f_x+\ip{\widetilde g_x}x)/\rho\rceil$,
$\tau=\alpha+2R\beta=\varepsilon/128$, $\rho=\varepsilon/64$; it contains the
true level set $\{f\le\ell\}\cap K$. The acceptance test is
$d_t^2\le\delta_N^2$, a comparison of rationals. The guarantee becomes
$f(y_N)\le\ell+48L\PN^2/N^3+\varepsilon/32$. The \emph{rounded bracket
search} queries $x_0=0$, takes a vertex $v\in\argmin_{u\in K}\ip{\widetilde g_0}u$,
uses the bracket $\flo=\widetilde f_0+\ip{\widetilde g_0}v-\tau_0$,
$\fhi=\flo+LR^2/2+2\tau_0$, $\tau_0=\alpha+R\beta=3\varepsilon/512$, and the
horizon rule with $w/8$ in place of $w/4$ in \eqref{eq:horizon-rule}, i.e.\
$N'(w)=\min\{N=2^j:N^3\ge6144LR^2\bN\cn/w\}$.

\begin{restatable}[Inexact oracle and finite-bit implementation]{proposition}{thmfinitebit}
\label{thm:finite-bit}
Let $K=B_1^n(R)$ and let the oracle satisfy \eqref{eq:inexact-oracle}.
\begin{enumerate}[label=(\alph*)]
\item With any selector that returns rational points $z_t\in C_t$ with
$\nrm{z_t-\st{C_t}}_2\le\eta_N$ (\cref{prop:finite-selector} is one), the rounded
bracket search returns $v\in K$ with $f(v)-\min_Kf\le\varepsilon$ after at most
$S=1+24N'(\varepsilon)$ queries,
$N'(\varepsilon)=O(1+[\frac{LR^2}{\varepsilon}\log(2n)\log(2+\frac{LR^2}{\varepsilon}\log(2n))]^{1/3})$;
the errors do not accumulate with the number of steps.
\item For a known nonempty rational polytope $C\subseteq c+B_1^n(R)$ with
rational $c$, rational $\eta\in(0,R]$ and rational $\gamma\in(0,1)$, a rational
point $\widehat s\in C$ with $\Prob\{\nrm{\widehat s-\st C}_2\le\eta\}\ge1-\gamma$
can be computed from finitely many uniformly random bits with
$m=\lceil32(R/\eta)^2(1+\lceil\log_2(1/\gamma)\rceil)\rceil$ rational linear
programs over $C$, in a number of bit operations polynomial in $n$, $R/\eta$,
$\log(1/\gamma)$ and the bit lengths of $R$, $\eta$ and of the description of
$C$; the vertices of $C$ are not enumerated.
\item Consequently, for rational $\delta\in(0,1)$, the rounded bracket search
with the selectors of (b) (accuracy $\eta=R/N$ and failure probability
$\delta/(2S)$ per selector call) returns a point with error at most
$\varepsilon$ with probability at least $1-\delta$, makes at most $S$ queries on
every outcome, solves at most $66S^3(1+\lceil\log_2(2S/\delta)\rceil)$ rational
linear programs plus at most $S$ emptiness tests, and its total number of bit
operations is polynomial in $n$, $S$, $\log(1/\delta)$, $k_f$, $k_g$, the bit
lengths of $R,L,\varepsilon,\delta$ and the bit length of the first oracle
answer.
\end{enumerate}
\end{restatable}

The dependence on the first answer is unavoidable, since the affine part of $f$
may be arbitrarily large. The stored cut constants are rounded to the grid
$\rho\mathbb Z$ and bounded in terms of the first answer, so the bit lengths of
the localization polytopes stay polynomial. The proposition concerns a randomized
implementation in the bit model, and its running time is polynomial in
the number of queries $S=\widetilde\Theta((LR^2/\varepsilon)^{1/3})$, not in
$\log(1/\varepsilon)$. A deterministic polynomial-time selector and a practically
fast one remain open (\cref{sec:conclusion}).

\begin{lemma}[Rounded level method]\label{lem:rounded-level}
Let $K=B_1^n(R)$, let the oracle satisfy \eqref{eq:inexact-oracle}, and let
$\tau=\alpha+2R\beta=\varepsilon/128$, $\rho=\varepsilon/64$. The rounded level
method with level $\ell$ and horizon $N$ makes at most $N$ queries in $K$, and
either all its cuts have empty intersection with $K$, in which case
$\min_Kf>\ell$, or it returns $y_N\in K$ with
$f(y_N)\le\ell+48L\PN^2/N^3+\varepsilon/32$, provided every selected point is
within $\eta_N$ of the Steiner point of its polytope.
\end{lemma}

\begin{proof}
For $x,u\in K$, $\nrm{u-x}_2\le2R$ and the Cauchy--Schwarz inequality give
\begin{equation}\label{eq:rounded-tangent}
 \bigl|\widetilde f_x+\ip{\widetilde g_x}{u-x}-f(x)-\ip{\nabla f(x)}{u-x}\bigr|\le\tau .
\end{equation}
\emph{The cut contains the level set.} If $f(u)\le\ell$, then
$f(x)+\ip{\nabla f(x)}{u-x}\le f(u)\le\ell$, so
$\ip{\widetilde g_x}u\le\ell+\tau-\widetilde f_x+\ip{\widetilde g_x}x\le c_x$ by
\eqref{eq:rounded-tangent} and the upward rounding. Hence an empty intersection
certifies $\min_Kf>\ell$.

\emph{Effect of the rounding.} If $\ip{\widetilde g_x}{z'}\le c_x$ with $z'\in K$,
then, since $c_x<\ell+\tau-\widetilde f_x+\ip{\widetilde g_x}x+\rho$,
\eqref{eq:rounded-tangent} gives
$f(x)+\ip{\nabla f(x)}{z'-x}\le\ell+2\tau+\rho=\ell+\varepsilon/32$.

\emph{Accepted steps.} Number the accepted steps $j=1,2,\dots$ and let
$a_j=j/(2L)$, $A_j=j(j+1)/(4L)$, so that $A_j=A_{j-1}+a_j$ and
$La_j^2/(2A_j)=j/(2(j+1))\le\frac12$. \Cref{lem:level-step} with the level
$\ell+\varepsilon/32$ (which the selected point $z'$ satisfies by the previous
paragraph) gives, for the accepted step with displacement $d=\nrm{z'-z}_2$,
\[
 A_j[f(y_j)-\ell]\le A_{j-1}[f(y_{j-1})-\ell]+\tfrac12d^2+a_j\varepsilon/32 ;
\]
telescoping and $\sum_{i\le j}a_i=A_j$ give
$f(y_j)-\ell\le\frac1{2A_j}\sum_{\rm accepted}d^2+\varepsilon/32$, without any
factor depending on the number of steps. The path bound
\eqref{eq:actual-path}, the counting of rejections and $A_k\ge k^2/(4L)\ge N^2/(16L)$
are as in the proof of \cref{thm:splevel}, and give the claim.
\end{proof}

We first prove part (b) of \cref{thm:finite-bit}.

\begin{proof}[Proof of \cref{thm:finite-bit}\textup{(b)}]
Translate by the rational vector $c$ so that $C\subseteq B_2(0,R)$. Write
$C=\{\mathsf Ty+c':y\in\widetilde C\}$ with a bounded rational polytope
$\widetilde C=\{y\in\R^d:\mathsf My\le\mathsf q\}$ given by $p$ inequalities and
rational $\mathsf T,c'$ (for the localization polytopes of our methods, $d=2n$,
$\widetilde C$ is the lifted description with variables $(x,u)$, $\mathsf T$ the
projection to $x$, $c'=0$, and $p\le3n+1+S$). Every vertex of $\widetilde C$ is
determined by $d$ linearly independent active inequalities, so $\widetilde C$
has at most $p^d=:V$ vertices and $C$
is the convex hull of at most $V$ points, the vertex images. For
$Z\sim\mathcal N(0,I_n)$ let $X(Z)$ be the maximizer of $\ip Zx$ over $C$; it is
almost surely a unique vertex image and $\E X(Z)=\st C$.

\emph{Rounded directions.} Let $\widehat Z$ be a rational vector coupled with $Z$
so that $\Prob\{\nrm{Z-\widehat Z}_2>h\}\le\zeta$, and let $\widehat X\in C$ be
any rational optimizer of $\ip{\widehat Z}x$ over $C$ (an LP over $\widetilde C$
with objective $\mathsf T^\top\widehat Z$). We claim $\Prob\{\widehat X\ne X(Z)\}\le\zeta+V^2h$.
Suppose that $\nrm{Z-\widehat Z}_2\le h$ and $v:=X(Z)\ne\widehat X$. The optimal face
of $\ip{\widehat Z}\cdot$ contains $\widehat X$ and is the convex hull of the
vertex images it contains, all with the same value of $\ip{\widehat Z}\cdot$;
since $\widehat X\ne v$, this face contains a vertex image $u\ne v$ with
$\ip{\widehat Z}{u-v}\ge0$, while $\ip Z{v-u}>0$ by uniqueness. Thus
$0<\ip Z{(v-u)/\nrm{v-u}_2}\le\nrm{Z-\widehat Z}_2\le h$. For a fixed pair of
distinct vertex images the left-hand side is a standard Gaussian, which falls
in an interval of length $h$ with probability at most $h/\sqrt{2\pi}<h$; the
union bound over at most $V^2$ pairs proves the claim. Consequently
$\nrm{\E\widehat X-\st C}_2\le\E\nrm{\widehat X-X(Z)}_2\le2R(\zeta+V^2h)$, and
the choice $\zeta=\eta/(8R)$, $h=\eta/(8RV^2)$ makes the bias at most $\eta/2$;
note $\log(1/h)=O(\log(R/\eta)+d\log p)$.

\emph{Finite Gaussian generator.} Let $J=\lceil\log_2(2n/\zeta)\rceil$,
$t_0=\min\{t\in\N:t^2\ge2J\}$, $h'=h/n$, $\vartheta=2^{-(t_0^2+2)}$, and let
$k_b\ge1$ be an integer with $2^{-k_b}\le\vartheta h'$ (we may take $\zeta$
and $h$ to be the largest powers of two below the values prescribed above, so
that all these quantities have bit length polynomial in the bit lengths of $R$,
$\eta$ and $\log V$). Then
$\Prob\{\nrm Z_\infty>t_0\}\le2ne^{-t_0^2/2}\le2n2^{-J}\le\zeta$, and the standard
normal density is at least $\vartheta$ on $[-t_0,t_0]$ (because $\sqrt{2\pi}<4$ and
$e^{t_0^2/2}<2^{t_0^2}$), so the truncated quantile
$\Phi^{-1}_{t_0}(u)=\max\{-t_0,\min\{t_0,\Phi^{-1}(u)\}\}$ is $1/\vartheta$-Lipschitz on $[0,1]$.
For each coordinate draw $k_b$ uniform bits, forming an integer
$I_i\in\{0,\dots,2^{k_b}-1\}$, put $U_i=(I_i+\frac12)2^{-k_b}$ and compute a
rational $\widehat Z_i\in[-t_0,t_0]$ with $|\widehat Z_i-\Phi^{-1}_{t_0}(U_i)|\le h'/2$. For
the analysis couple $I_i$ with an independent uniform $\Theta_i\in[0,1]$, so
that $\widetilde U_i=(I_i+\Theta_i)2^{-k_b}$ is uniform and
$Z_i=\Phi^{-1}(\widetilde U_i)$ is standard Gaussian; then
$|\widehat Z_i-\Phi^{-1}_{t_0}(\widetilde U_i)|\le h'/2+2^{-k_b-1}/\vartheta\le h'$, and on
the event $\nrm Z_\infty\le t_0$ we have $\Phi^{-1}_{t_0}(\widetilde U_i)=Z_i$ and
$\nrm{\widehat Z-Z}_2\le\sqrt n\,h'\le h$. This is the coupling required
above. To compute $\widehat Z_i$ we bisect $[l_0,r_0]=[-t_0,t_0]$: at the midpoint
$t$ compute a rational $\widetilde\varphi$ with $|\widetilde\varphi-\Phi(t)|\le\theta:=\vartheta h'/16$
and set $r_0=t$ if $\widetilde\varphi-\theta>U_i$, $l_0=t$ if
$\widetilde\varphi+\theta<U_i$, and otherwise return $t$ (then
$|\Phi(t)-U_i|\le2\theta$ and $|t-\Phi^{-1}_{t_0}(U_i)|\le2\theta/\vartheta\le h'/8$); return
the midpoint as soon as $r_0-l_0\le h'$. Monotonicity keeps $\Phi^{-1}_{t_0}(U_i)\in[l_0,r_0]$,
so at most $\lceil\log_2(2t_0/h')\rceil+1$ evaluations of $\Phi$ are needed,
and each is a truncation of the everywhere convergent series
$\Phi(t)=\frac12+\frac1{\sqrt{2\pi}}\sum_{j\ge0}\frac{(-1)^jt^{2j+1}}{2^jj!(2j+1)}$
with an explicit remainder, together with a rational approximation of
$1/\sqrt{2\pi}$ of certified accuracy; all of this costs a number of bit
operations polynomial in $t_0^2$, $\log(1/\theta)$ and the bit length of $t$.
Since $t_0^2=O(1+\log(n/\zeta))$ and $k_b=O(1+\log(n/\zeta)+\log(n/h))$, the
generator uses $nk_b$ random bits and polynomial work per direction.

\emph{Averaging.} Draw $m$ independent directions, solve the $m$ linear
programs, and return $\widehat s=\frac1m\sum_{i\le m}\widehat X_i\in C$ (a convex
combination of points of $C$; no projection is needed). Let
$\bar s=\E\widehat X_1$. As $\E\nrm{\widehat X_1-\bar s}_2^2\le R^2$,
$\E\nrm{\widehat s-\bar s}_2\le R/\sqrt m\le\eta/4$ for $m\ge16R^2/\eta^2$, and
changing one sample changes $\nrm{\widehat s-\bar s}_2$ by at most $2R/m$, so the
bounded-differences inequality \citep{mcdiarmid1989} gives
$\Prob\{\nrm{\widehat s-\bar s}_2>\eta/2\}\le\exp(-m\eta^2/(32R^2))\le\gamma$ for the
stated $m$. With the bias bound, $\nrm{\widehat s-\st C}_2\le\eta$ with
probability at least $1-\gamma$. Each linear program has rational data of
polynomial bit length, is solved exactly in polynomial time
\citep{khachiyan1979,karmarkar1984}, and its basic optimal solutions have
polynomial bit length by Cramer's rule; the average of $m$ of them keeps
polynomial length. If the polytope depends on earlier oracle answers and
random choices, we condition on that history. The polytope and its vertex
images are then fixed, the new bits are independent of it (the oracle errors
may depend on the past but not on bits drawn later), and the bounds hold
conditionally with the same $p,d,R$.
\end{proof}

\begin{proof}[Proof of \cref{thm:finite-bit}\textup{(a)} and \textup{(c)}]
\emph{(a)} Let $\tau_0=\alpha+R\beta=3\varepsilon/512$ and let $v$ be a rational
vertex $\pm Re_i$ minimizing $\ip{\widetilde g_0}\cdot$ over $K$. For every
$u\in K$, convexity and \eqref{eq:inexact-oracle} with $\nrm u_2\le R$ give
$f(u)\ge f(0)+\ip{\nabla f(0)}u\ge\widetilde f_0+\ip{\widetilde g_0}u-\tau_0\ge\widetilde f_0+\ip{\widetilde g_0}v-\tau_0=\flo$,
so $\flo\le\min_Kf$; and smoothness at $0$ gives
$f(v)\le f(0)+\ip{\nabla f(0)}v+LR^2/2\le\flo+2\tau_0+LR^2/2=\fhi$. If
$w=\fhi-\flo\le\varepsilon$, return $v$. Otherwise $\varepsilon<w=LR^2/2+3\varepsilon/256$,
so $\varepsilon<128LR^2/253$ and $w<LR^2$. In a phase of width
$w>\varepsilon$ with level $\ell$ and horizon $N=N'(w)$,
\cref{lem:rounded-level} either certifies $\min_Kf>\ell$ or returns $y_N$ with
$f(y_N)\le\ell+w/8+\varepsilon/32\le\ell+w/4$. Hence the bracket updates of
\cref{alg:spaccel} keep the invariant $\flo\le\min_Kf\le f(v)\le\fhi$ and
multiply $w$ by $\frac12$ or $\frac34$. Every horizon exceeds $1$ because
$w<LR^2$, and the geometric summation of \cref{app:bracket} applies without
change to the rule $N'$, giving at most $1+24N'(\varepsilon)$ queries; the explicit
order of $N'(\varepsilon)$ is obtained as in \cref{app:bracket} with $6144$ in
place of $3072$. All comparisons are between rational numbers.

\emph{(b)} was proved above.

\emph{(c)} The width argument in the proof of \cref{thm:random-selector}, and
with it the bound $S$ on the number of queries, does not depend on the accuracy
of the selected points and applies here. As there, one selector call is made
per trial and one per level call for $z_0$, so fewer than $2S$ calls are made;
with failure probability $\delta/(2S)$ each, all are $\eta_N$-accurate with
probability at least $1-\delta$, and then (a) applies. A call in a level call
with horizon $N$ uses the rational accuracy $\eta=R/N$, which is at most
$\eta_N=R\sqrt{\bN\cn}/N$ and at most $R$, so its points are $\eta_N$-accurate.
Since $(R/\eta)^2=N^2\le S^2$, such a call solves at most
$32S^2(1+\lceil\log_2(2S/\delta)\rceil)+1$ linear programs; multiplying by $2S$
gives the stated total, and the emptiness tests, one per trial, add at most $S$.

\emph{Bit lengths.} On $K$, $\nrm{\nabla f(x)}_2\le\nrm{\nabla f(0)}_2+LR$ and
$|f(x)|\le|f(0)|+R\nrm{\nabla f(0)}_2+LR^2/2$, all levels lie in the initial
bracket, and $|\ip{\widetilde g_x}x|\le R(\nrm{\nabla f(0)}_2+LR+\beta)$. Since
the answers lie on the grids $2^{-k_f}\mathbb Z$, $2^{-k_g}\mathbb Z^n$, the
coefficients $\widetilde g_t$ of the stored cuts have bit length
$O(n(k_g+\log(2+\nrm{\nabla f(0)}_2+LR+\beta)))$. The constants
$c_t\in\rho\mathbb Z$ satisfy $|c_t|\le|\ell|+|\widetilde f_{x_t}|+|\ip{\widetilde g_{x_t}}{x_t}|+\tau+\rho$,
hence have bit length $O(\log(2+|f(0)|+R\nrm{\nabla f(0)}_2+LR^2)+\mathrm{bits}(\rho))$,
\emph{independently of the bit length of the query point} $x_t$; here
$\mathrm{bits}(\rho)$ is the bit length of the rational $\rho=\varepsilon/64$. The lifted
localization polytopes therefore have polynomial descriptions, their vertices
and the selected averages have polynomial bit length, the accepted-step
weights $a_j/A_j=2/(j+1)$ are rational, and the iterates $x,y$, being convex
combinations with these weights of selected points, grow in length only
additively over at most $S$ steps. Together with (b) this gives the polynomial
bound on the total number of bit operations.
\end{proof}

\subsection{Other domains and smoothness metrics}
\label{app:general-geometry}

For a known polytope $K$ with $V$ vertices contained in $B_2(c,R)$ with $c\in K$,
the union bound and the tail integration of \cref{lem:gaussian-max}
give $\gw{K-c}\le R\sqrt{2(1+\log V)}$, so \cref{thm:general-geometry} below, with
$H=I$ and $\bar\omega=R\sqrt{2(1+\log V)}$, gives the bounds of \cref{thm:main}
with $1+\log(2n)$ replaced by $1+\log V$. For the $\ell_p$ ball with $1<p<2$ and
$q=p/(p-1)$, $\gw{B_p^n(R)}=R(\E\nrm Z_q^2)^{1/2}$ is at most
$R(n\,\E|Z_1|^q)^{1/q}$ and at most $Rn^{1/q}(\E\nrm Z_\infty^2)^{1/2}$, and
$(\E|Z_1|^q)^{2/q}=\Theta(q)$; \cref{thm:general-geometry} then gives the error
$O(LR^2n^{2-2/p}\min\{q,\log(2n)\}\log(2N)/N^3)$, uniformly in $1<p<2$. This
bound matches the lower bound $LR^2/N^{1+2/p}$ of \cref{thm:first-order-lower}
up to logarithms when $N+1\le n=O(N)$, but not when $n\gg N$; in that regime
the general question for $1<p<2$ remains open (for quadratics it is settled in
\cref{app:quadratics}). For the Euclidean ball the bound is
$O(LR^2n\log(2N)/N^3)$, which does not improve on $LR^2/N^2$ when $N\le n$, as
the Euclidean lower bound requires.

More generally, neither \cref{lem:steiner-path} nor the level method is specific
to the $\ell_1$ ball or to the Euclidean metric. The domain enters
through the Gaussian width of its image under the square root of the
smoothness metric. Let $H\succ0$ be a known symmetric matrix,
$\nrm x_H=\sqrt{x^\top Hx}$, $\nrm g_{H^{-1}}=\sqrt{g^\top H^{-1}g}$, $S=H^{1/2}$,
and let $K$ be a known nonempty compact convex set. For a nonempty compact
convex $C\subseteq K$ define the \emph{$H$-Steiner point} $\operatorname{st}_H(C)=S^{-1}\st{SC}\in C$.
We assume that the stored sets $C=K\cap\{\text{half-spaces}\}$ admit an
emptiness test and $\eta$-accurate feasible approximations of $\operatorname{st}_H(C)$ in the
norm $\nrm\cdot_H$, and that linear functions can be minimized over $K$ and the
certificate program $\min_\lambda\{\sum_j\lambda_j\ip{g_j}{x_j}+\hs K(-\sum_j\lambda_jg_j)\}$
of \cref{alg:spcut} (a convex program, a linear program for polytopes) can be
solved. These are computations with known sets, which are not counted
(\cref{app:notation}); for polytopes, \cref{app:steiner-computation} provides
them after the change of variables $x\mapsto Sx$, with $R$ replaced by the
Euclidean radius of $S(K-c)$.

\begin{restatable}[General known geometry]{proposition}{thmgeneralgeometry}
\label{thm:general-geometry}
Let $\bar\omega\ge\gw{S(K-c)}$ be known for some $c\in\R^n$ (the best choice is
$c=\operatorname{st}_H(K)$), let $P^H_N=2\sqrt2\,\bar\omega\sqrt{\bN}$, and run \SPC, \SPL{}
and \SPA{} with $\nrm\cdot_H$ in place of $\nrm\cdot_2$, $\operatorname{st}_H$ in place of
$\st{\cdot}$, and $P^H_N$ in place of $\PN$ (cut depth $4GP^H_N/N$;
$\eta_N=P^H_N/(4N)$, $\delta_N=4P^H_N/N$; horizon rule
$N_H(w)=\min\{N=2^j:N^3\ge1536L\bar\omega^2\bN/w\}$; initial bracket
$[\flo,\flo+W_0]$ with $\flo$ from \eqref{eq:bracket} and $W_0=\pi L\bar\omega^2$).
\begin{enumerate}[label=(\alph*)]
\item If $f$ is convex on $K$ and the subgradient oracle returns $g_x$ with
$\nrm{g_x}_{H^{-1}}\le G$, then \SPC{} with horizon $N$ returns $\widehat x\in K$
with $f(\widehat x)-\min_Kf<4GP^H_N/N=8\sqrt2\,G\bar\omega\sqrt{\bN}/N$ after at
most $N$ queries.
\item If $f$ is convex and differentiable on a neighborhood of $K$ with
$f(v)\le f(u)+\ip{\nabla f(u)}{v-u}+\frac L2\nrm{v-u}_H^2$ for $u,v\in K$ (for
instance if $\nrm{\nabla f(v)-\nabla f(u)}_{H^{-1}}\le L\nrm{v-u}_H$), then the
level method either certifies $\min_Kf>\ell$ or returns $y_N$ with
$f(y_N)\le\ell+48L(P^H_N)^2/N^3=\ell+384L\bar\omega^2\bN/N^3$,
and the bracket search returns an $\varepsilon$-solution after at most
$1+24N_H(\varepsilon)=O(1+[\frac{L\bar\omega^2}{\varepsilon}\log(2+\frac{L\bar\omega^2}{\varepsilon})]^{1/3})$
first-order queries; with a budget of $T\ge25$ queries the error is
$O(L\bar\omega^2\log(2T)/T^3)$.
\end{enumerate}
\end{restatable}

For $K=x_0+B_1^n(R)$ one has $\gw{S(K-x_0)}=R(\E\nrm{H^{1/2}Z}_\infty^2)^{1/2}\le R\sqrt{2(1+\log(2n))}\max_i\sqrt{H_{ii}}$,
which recovers \cref{thm:main} at $H=I$ and allows anisotropic smoothness. For
an ellipsoid $K=x_0+\{u:u^\top Qu\le R^2\}$, $\gw{S(K-x_0)}^2=R^2\tr(HQ^{-1})$,
so the rate is $O(LR^2\tr(HQ^{-1})\log(2N)/N^3)$; for the Euclidean
ball it is $O(LR^2n\log(2N)/N^3)$. The proposition does not show that the width
is the right parameter for every $K$.

\begin{proof}
\emph{Change of variables.} Put $\widetilde K=S(K-c)$, $\widetilde f(v)=f(S^{-1}v+c)$,
so that $\nabla\widetilde f(v)=S^{-1}\nabla f(S^{-1}v+c)$, $\nrm{S(u-v)}_2=\nrm{u-v}_H$
and $\nrm{S^{-1}g}_2=\nrm g_{H^{-1}}$. The hypotheses of (a) and (b) become,
for $\widetilde f$ on $\widetilde K$ with the same $G$ and $L$, the properties
used in the proofs of \cref{thm:spcut,thm:splevel,thm:main}: the subgradient
inequality with $\nrm{g}_2\le G$, and the descent inequality on $\widetilde K$
(the proofs use smoothness only through \cref{lem:level-step} and the
bracket). A query of $\widetilde f$ at $v\in\widetilde K$ is one query
of $f$ at $S^{-1}v+c\in K$; the Steiner point is translation equivariant
(\cref{lem:steiner-basic}), so the image $S(\operatorname{st}_H(C)-c)$ of $\operatorname{st}_H(C)$ is the
Euclidean Steiner point of the image $S(C-c)$, and
$\nrm{z-\operatorname{st}_H(C)}_H=\nrm{S(z-c)-\st{S(C-c)}}_2$.
So it suffices to prove (a) and (b) for the Euclidean methods on a general
compact convex $\widetilde K$ with $\gw{\widetilde K}\le\bar\omega$; the
translation $c$ never enters the method.

\emph{Path budget.} For a nested chain $\widetilde C_0\supseteq\dots\supseteq\widetilde C_T$
of nonempty compact convex subsets of $\widetilde K$ with $T\le N$,
\cref{lem:steiner-path-general} (with $c=0$) gives
$\sum_t\nrm{\st{\widetilde C_t}-\st{\widetilde C_{t-1}}}_2\le2\sqrt{2(1+\log T)}\,\bar\omega\le2\sqrt2\,\bar\omega\sqrt{\bN}=P^H_N$,
and with selector accuracy $P^H_N/(4N)$ the selected path has length at most
$\frac32P^H_N$. This is the only property of $\PN$ used in the proofs of
\cref{thm:spcut} and \cref{thm:splevel}, so both apply without change to
$P^H_N$; this gives (a) and the level bound in (b), since $48L(P^H_N)^2/N^3=384L\bar\omega^2\bN/N^3$.

\emph{Diameter and bracket.} Let $a,b\in\widetilde K$ with $\nrm{a-b}_2=\diam_2(\widetilde K)=:d$.
Then $\hs{\widetilde K}(Z)\ge\max\{\ip Za,\ip Zb\}=\frac12\ip Z{a+b}+\frac12|\ip Z{a-b}|$,
so $\E\hs{\widetilde K}(Z)\ge\frac12\E|\ip Z{a-b}|=d/\sqrt{2\pi}$, and by Jensen's
inequality $\gw{\widetilde K}\ge\E\hs{\widetilde K}(Z)$ (the mean is nonnegative because
$\E\hs{\widetilde K}(Z)\ge\E\ip Zu=0$ for any $u\in\widetilde K$). Hence
$d^2\le2\pi\bar\omega^2$ (the diameter is translation invariant). One query
at $x_0\in K$ gives $\flo\le\min_Kf$ as in \eqref{eq:bracket} and, by the descent
inequality at $x_0$, $f(v)\le\flo+\frac L2d^2\le\flo+W_0$ for the linear
minimizer $v$. Thus $[\flo,\flo+W_0]$ is a valid bracket of width $W_0=\pi L\bar\omega^2$
that needs no maximization over $K$. If $\varepsilon\ge W_0$ one query suffices.

\emph{Bracket search.} The horizon rule $N_H(w)$ is the condition
$48L(P^H_N)^2/N^3\le w/4$, so each phase behaves as in \cref{app:bracket}.
Every horizon exceeds $1$: at $N=1$ the rule requires $w\ge1536L\bar\omega^2$,
while $w\le W_0=\pi L\bar\omega^2$. The geometric summation of
\cref{app:bracket} applies to any rule of the form
$N(w)=\min\{N=2^j:N^3\ge c_0\bN/w\}$ with a constant $c_0>0$ for which every
horizon exceeds $1$, and gives the count $1+24N_H(\varepsilon)$. The explicit inversion with the scalar $\Theta=1536L\bar\omega^2/\varepsilon$
gives $N_H(\varepsilon)^3<32\Theta(3+\log_2\Theta)$, hence the stated order. The fixed-budget
statement follows with $\varepsilon_Q=1536L\bar\omega^2\mathsf b_Q/Q^3$ for the
largest power of two $Q$ with $24Q\le T-1$, whose error is at most
$\min\{W_0,\varepsilon_Q\}=O(L\bar\omega^2\log(2T)/T^3)$.
\end{proof}

The identity $\gw{S(K-x_0)}^2=R^2\tr(HQ^{-1})$ for the ellipsoid
$K=x_0+\{u:u^\top Qu\le R^2\}$ follows from
$\hs{S(K-x_0)}(z)=R\sqrt{z^\top SQ^{-1}Sz}$ and $\E Z^\top AZ=\tr A$; for
$K=x_0+B_1^n(R)$, $\hs{S(K-x_0)}(z)=R\nrm{Sz}_\infty$, where each $(SZ)_i$ is
centered Gaussian with variance $H_{ii}$, so the union bound over the $2n$ tails
and the tail integration of \cref{lem:gaussian-max} give
$\E\nrm{SZ}_\infty^2\le2\max_iH_{ii}(1+\log(2n))$; and for $K=x_0+B_2^n(R)$, $\gw{S(K-x_0)}^2=R^2\tr H$. (The
root-mean-square width is not translation invariant, since
$\gw{C+c}^2=\gw C^2+2\ip{\st C}c+\nrm c_2^2$; for this reason the proposition
is stated with a center.)

The two width bounds at the beginning of this subsection are proved in the same
way. If $K$ is the
convex hull of $V$ points $v_1,\dots,v_V$ of $B_2(c,R)$ and $c\in K$, then
$\hs{K-c}(z)=\max_i\ip{v_i-c}z\ge0$ and each $\ip{v_i-c}Z$ is Gaussian with
variance at most $R^2$, so the union bound and the tail integration of
\cref{lem:gaussian-max} give $\E\hs{K-c}(Z)^2\le2R^2(1+\log V)$. For
$K=B_p^n(R)$, $\hs K(Z)=R\nrm Z_q$ with $q\ge2$, and Jensen's inequality for the
concave function $t\mapsto t^{2/q}$ gives
$\E\nrm Z_q^2\le(\E\nrm Z_q^q)^{2/q}=(n\E|Z_1|^q)^{2/q}$.

\section{Lower bounds}
\label{app:lower}

All lower bounds in this section allow queries anywhere in $\R^n$ and require
only that the output is feasible. For deterministic methods, each proof builds,
for a given method, one fixed function that reproduces the whole transcript
(\cref{app:lower-det,app:lower-higher,app:lower-simplex}); for randomized methods
the hard function is drawn at random before the interaction
(\cref{app:lower-rand}).

\subsection{Deterministic first-order and nonsmooth bounds}
\label{app:lower-det}

The deterministic bounds use the signed-coordinate resisting oracle of
\citet{guzman2015lower}. At query $x_t$ the adversary reveals the signed
coordinate $v_t=\sign(x_{t,i_t})e_{i_t}$, where $i_t$ maximizes $|x_{t,i}|$
over the coordinates not yet revealed, and answers, for an offset $\tau>0$ fixed in the proofs, with the prefix maximum
\[
 \varphi_t(x)=\max_{j\le t}\bigl\{\ip{v_j}{x}-(j-1)\tau\bigr\}.
\]
Every future form $\ip{v_j}{x}-(j-1)\tau$, $j>t$, is at least $\tau$
below the current one at $x_t$, and differences of forms have $\ell_1$-norm at
most two; hence all completions of $\varphi_t$ coincide with $\varphi_t$ on the
$\ell_\infty$ ball of radius $\tau/2$ around $x_t$. After $N$ queries one
coordinate is still unrevealed; aligning its sign with the output $\widehat x$
gives $\varphi(\widehat x)\ge-N\tau$, while the comparator $x^\circ=-d\sum_jv_j$ with
$d=R/(N+1)^{1/p}$ has $\nrm{x^\circ}_p=R$ and $\varphi(x^\circ)=-d$.
\citet{guzman2015lower} smooth this function by an infimal convolution; for
first-order oracles we use the Moreau envelope, which gives explicit constants.

\begin{restatable}[First-order and nonsmooth lower bounds]{proposition}{thmfirstorderlower}
\label{thm:first-order-lower}
Let $N\ge1$, $n\ge N+1$, $1\le p\le2$ and $R>0$. For every deterministic method
that makes at most $N$ global queries and returns $\widehat x\in B_p^n(R)$:
\begin{enumerate}[label=(\alph*)]
\item with the subgradient oracle (values and subgradients), there is a convex
$G$-Lipschitz function and a fixed subgradient selection such that
$f(\widehat x)-\min_{B_p^n(R)}f\ge GR/(N+1)^{1/p}$;
\item with the first-order oracle, for every $0<\nu\le1$ there is a convex
function with $\nu$-H\"older gradient, $\nrm{\nabla f(x)-\nabla f(y)}_2\le
L_{1,\nu}\nrm{x-y}_2^\nu$, such that
\begin{equation}\label{eq:holder-lower}
 f(\widehat x)-\min_{B_p^n(R)}f\ \ge\ \frac{2^{\nu-2}}{8^{\nu}}\cdot
 \frac{L_{1,\nu}R^{1+\nu}}{(N+1)^{\nu+(1+\nu)/p}} .
\end{equation}
\end{enumerate}
\end{restatable}

At $\nu=1$ the bound is $LR^2/[16(N+1)^{1+2/p}]$, that is,
$LR^2/[16(N+1)^3]$ on the $\ell_1$ ball.

\begin{proof}
Set $m=N+1$ and use only the first $m$ coordinates. For the nonsmooth
claim consider $f_s(x)=G\max_{i\le m}s_ix_i$, with signs revealed
adaptively. At a query $x$, let $a$ be the largest value at $x$ of the
previously revealed signed forms ($-\infty$ initially), and let $u$ be the largest $|x_i|$
among unused coordinates. If $a\ge u$, return the earliest-revealed
maximizer. Otherwise reveal a coordinate attaining $u$, with its sign
aligned with $x_i$, and return its form and slope. Each answer reveals
at most one sign. Every still-unknown signed form at that query is at
most $u$, so all future sign completions preserve the value and the
validity of the returned subgradient.

After the output, align one unrevealed sign with the corresponding
coordinate of $\widehat x$ and complete the other signs arbitrarily.
Then $f_s(\widehat x)\ge0$, whereas
$x_i^\circ=-s_iR/m^{1/p}$ for $i\le m$, zero otherwise, is feasible
and has value $-GR/m^{1/p}$. Every slope has Euclidean norm $G$.
To obtain one fixed final oracle, order the coordinates by time of revelation,
with unused ones last, and always return the earliest maximizing
form in this order. Earlier maximizing forms precede later ties, and a newly
revealed form strictly exceeds every earlier one, so this rule
reproduces the full transcript, including repeated queries and ties.
This proves part (a) without using the feasibility of the output. (An
oracle that returns the whole subdifferential is not covered.)

For the H\"older claim put
\[
 d=R/m^{1/p},\qquad \tau=d/(2m),\qquad \lambda=\tau/4,
 \qquad \varphi(x)=\max_{i\le m}\{\langle v_i,x\rangle-(i-1)\tau\}.
\]
Now a fresh coordinate is revealed at \emph{every} query $x_t$: $i_t$
maximizes $|(x_t)_i|$ among the unused coordinates, and
$v_t=\sign((x_t)_{i_t})e_{i_t}$, with sign $+1$ at zero.
The oracle uses the prefix maximum through $t$. Every later vector
$v_j$, $j>t$, satisfies $\ip{v_j}{x_t}\le\ip{v_t}{x_t}$, so
the offset puts its affine form at least $\tau$ below
the current one. After the output, align the remaining signed coordinate
$v_m$ with that output. The final $\varphi$ and all prefixes are convex
and Euclidean $1$-Lipschitz.

Let $e_\lambda \varphi(x)=\min_z\{\varphi(z)+\nrm{z-x}_2^2/(2\lambda)\}$.
The unique proximal point $P(x)$ satisfies
$(x-P(x))/\lambda\in\partial \varphi(P(x))$, so
$\nrm{x-P(x)}_2\le\lambda$.
Monotonicity of $\partial \varphi$ shows that $I-P$ is nonexpansive.
Comparing the defining minima at $x$ and at $x+h$ gives
$\nabla e_\lambda \varphi(x)=(x-P(x))/\lambda$.
Thus the envelope has gradient norm at most one and gradient Lipschitz
constant at most $1/\lambda$.  Moreover,
\begin{equation}\label{eq:envelope-bias-lower}
 \varphi(x)-\lambda/2\le e_\lambda \varphi(x)\le \varphi(x),
\end{equation}
by minimizing $\varphi(x)-\nrm{z-x}_2+\nrm{z-x}_2^2/(2\lambda)$
over the distance for the lower inequality, and using $z=x$ for the upper.

Take $b=2^{\nu-1}L_{1,\nu}\lambda^\nu$ and $f=b e_\lambda \varphi$.
Its gradient difference is bounded by
$\min\{2b,b\nrm{x-y}_2/\lambda\}$, and
the inequality $\min\{u,v\}\le u^{1-\nu}v^\nu$ gives the required
H\"older condition.  It remains to check that the oracle stays consistent
after smoothing.
At query $x_t$ each future affine form is at least $\tau$ below the
current one.  A difference of two forms is $2$-Lipschitz, so no future
form can exceed the current one within distance $\tau/2$ of $x_t$.
The proximal points of the full and of the prefix envelope are within
$\lambda=\tau/4$ of their arguments, so for arguments within
$\tau/8$ of $x_t$ they remain in the region where
full and prefix maxima coincide.  Evaluating both minima at each other's
proximal points shows that the envelopes are equal in that neighborhood;
in particular their values and gradients agree at the query.

At the final output the last affine form has value at least $-N\tau$,
while $x^\circ=-d\sum_i v_i$ has $p$-norm exactly $R$ and
$\varphi(x^\circ)=-d$. Hence \eqref{eq:envelope-bias-lower} yields
\[
 f(\widehat x)-\min_{B_p^n(R)} f
 \ge b(d-N\tau-\lambda/2)\ge bd/2
 =2^{\nu-2}L_{1,\nu}\frac{d^{1+\nu}}{(8m)^\nu}.
\]
Substituting $d=R(N+1)^{-1/p}$ gives \eqref{eq:holder-lower}. The comparator
need not minimize the envelope, and this hard function need not have an
unconstrained minimizer.
\end{proof}

\subsection{Equal smoothness constants}
\label{app:equal-constants}

The following corollary shows that the gap between $N^{-2}$ and $N^{-3}$ does not
come from functions whose Euclidean constant greatly overestimates their
$\ell_1$ constant.

\begin{restatable}[Equal smoothness constants]{corollary}{corequal}
\label{cor:equal-constants}
Let $N\ge1$, $n\ge N+2$, $1\le p\le2$ and $R,L>0$. For every deterministic
first-order method with at most $N$ global queries and output
$\widehat x\in B_p^n(R)$ there is a convex $C^1$ function $f$ with
\[
 f(\widehat x)-\min_{B_p^n(R)}f\ \ge\ \frac{LR^2}{16(N+1)^{1+2/p}},
\]
whose smallest gradient Lipschitz constants in $\ell_2$ and in $\ell_1$ are both
equal to $L$, also when restricted to pairs of points in the interior of
$B_p^n(R)$.
\end{restatable}

\begin{proof}
Put $m=N+1$ and write each point of $\R^n$ as $x=(u,z,w)$, where
$u\in\R^m$, $z\in\R$, and $w\in\R^{n-m-1}$ (the last block is omitted
when $n=m+1$).  Given a function $f_0:\R^m\to\R$, define
\[
 f(u,z,w)=f_0(u)+\frac L2z^2.
\]
We simulate the prescribed algorithm with a first-order oracle for $f_0$.
A query $(u,z,w)$ uses one query of $f_0$ at $u$, and the algorithm receives
\[
 \left(f_0(u)+\frac L2z^2,\; (\nabla f_0(u),Lz,0)\right).
\]
The simulation reproduces all future query coordinates of the algorithm,
including $z$ and $w$; queries can be arbitrarily large and need not lie in
any previous linear span.  The simulated method returns the
first block $\widehat u$ of the final output, a feasible point of
$B_p^m(R)$, so \cref{thm:first-order-lower}, at $\nu=1$, supplies
one fixed convex function $f_0$ with Euclidean gradient Lipschitz constant at
most $L$ and
\[
 f_0(\widehat u)-\min_{B_p^m(R)}f_0
 \ge \frac{LR^2}{16(N+1)^{1+2/p}}.
\]
The function $f$ built from this $f_0$ reproduces all simulated answers.
Projection onto the first block maps $B_p^n(R)$ into $B_p^m(R)$, and every
point of the latter embeds by setting $z=w=0$.  Since the added quadratic
is nonnegative, $\min_{B_p^n(R)}f=\min_{B_p^m(R)}f_0$, and the same observation
at the output proves the displayed lower bound.

For any $x=(u,z,w)$ and $x'=(u',z',w')$, block orthogonality gives
\[
 \nrm{\nabla f(x)-\nabla f(x')}_2^2
 \le L^2\bigl(\nrm{u-u'}_2^2+|z-z'|^2\bigr)
 \le L^2\nrm{x-x'}_2^2.
\]
Moreover, Euclidean smoothness of $f_0$ implies
\[
 \nrm{\nabla f_0(u)-\nabla f_0(u')}_\infty
 \le L\nrm{u-u'}_1.
\]
Consequently
\[
 \nrm{\nabla f(x)-\nabla f(x')}_\infty
 \le L\max\{\nrm{u-u'}_1,|z-z'|\}
 \le L\nrm{x-x'}_1.
\]
Both constants are therefore at most $L$.  For $x=0$ and
$x'=t e_{m+1}$ with $0<t<R$, both Lipschitz ratios are equal to
$L$, and these points lie in the interior of the feasible ball.
\end{proof}

The extra coordinate fixes both smallest constants and carries no information
about $f_0$.

\subsection{Higher-order oracles}
\label{app:lower-higher}
\label{app:product-smoothing}

For oracles of order $k$ we smooth the same maximum of signed coordinates by
convolution with a compactly supported product kernel on the $m=N+1$ active
coordinates (\cref{lem:product-smoothing} below). Its support is an
$\ell_\infty$ ball, so smoothing preserves the local agreement exactly, and
since the score of the one-dimensional factor of the kernel has Fisher
information $10$, the directional derivatives of the product kernel are bounded
in $L^1$ uniformly in the dimension. Repeated smoothing and offset maxima are standard in lower-bound
constructions \citep{agarwal2018higher,contreras2025}; the point of the kernel
below is that its \emph{Euclidean} derivative bounds do not depend on the
dimension.

\begin{lemma}[Compact product smoothing]
\label{lem:product-smoothing}
Let $k\ge1$, $\rho>0$, and let $\varphi:\R^m\to\R$ be globally Euclidean
$G$-Lipschitz. Define
\[
 \chi(t)=\begin{cases}\frac{15}{16}(1-t^2)^2,&|t|\le1,\\0,&|t|>1,
 \end{cases}\quad
 \chi_h(z)=h^{-m}\prod_{i=1}^m\chi(z_i/h),\quad
 \mathcal S_{\rho,k}\varphi=\varphi*\chi_h^{*k},\qquad h=\rho/k.
\]
Here $\chi_h^{*k}$ denotes convolution of $k$ copies of $\chi_h$. Put
$\theta=k\sqrt{10}/\rho$. Then $\mathcal S_{\rho,k}\varphi\in C^{k+1}$ and
\begin{align}
 \nrm{D^j\mathcal S_{\rho,k}\varphi(x)}&\le G\theta^{j-1}
       &&(1\le j\le k+1),\label{eq:product-derivative-bound}\\
 \nrm{D^k\mathcal S_{\rho,k}\varphi(x)-D^k\mathcal S_{\rho,k}\varphi(y)}
 &\le2^{1-\nu}G\theta^{k-1+\nu}\nrm{x-y}_2^\nu
       &&(0<\nu\le1).\label{eq:product-holder-bound}
\end{align}
Convexity is preserved. If $\varphi$ is convex and $G_\infty$-Lipschitz in $\ell_\infty$, then
\begin{equation}\label{eq:product-bias}
 \varphi\le\mathcal S_{\rho,k}\varphi\le \varphi+G_\infty\rho.
\end{equation}
If two functions agree on $B_\infty(x,a)$ with $a>\rho$, their smoothings
agree on $B_\infty(x,a-\rho)$.
\end{lemma}

\begin{proof}
The density $\chi$ is even, nonnegative and $C^1$ on $\R$, with both
its value and first derivative zero at $\pm1$. Direct integration gives
\[
 \int\chi=1,\qquad \int t\chi(t)\,dt=0,\qquad
 \int\chi'=0,\qquad
 \int_{-1}^1\frac{\chi'(t)^2}{\chi(t)}\,dt
   =\int_{-1}^1 15t^2\,dt=10.
\]
The ratio is taken on $(-1,1)$, where $\chi>0$; the score need not be
bounded near the endpoints. For independent $Z_i$ with density
$\chi$, the scores $S_i=\chi'(Z_i)/\chi(Z_i)$ have mean zero and
second moment ten. Thus, for every $u\in\R^m$, independence and
Cauchy--Schwarz yield
\begin{equation}\label{eq:product-directional-variation}
 \nrm{D_u\chi_h}_{L^1}
 =h^{-1}\mathbb E\left|\sum_i u_iS_i\right|
 \le\frac{\sqrt{10}}h\nrm{u}_2.
\end{equation}
The total convolution noise has mean zero and infinity norm at most $kh=\rho$.
This gives locality, convexity preservation, and \eqref{eq:product-bias} by
Jensen's inequality and the infinity-norm Lipschitz bound. (The Euclidean norm
of the noise may be as large as $\rho\sqrt m$, hence the constant $G_\infty$ in
\eqref{eq:product-bias}.)

For the derivative bounds, put one weak derivative on $\varphi$ and each remaining
derivative on a different kernel factor. For $1\le j\le k+1$,
\[
 D_{u_1}\cdots D_{u_j}\mathcal S_{\rho,k}\varphi
 =(D_{u_1}\varphi)*(D_{u_2}\chi_h)*\cdots*(D_{u_j}\chi_h)*\chi_h^{*(k-j+1)}.
\]
At $j=1$ all $k$ factors are undifferentiated; at $j=k+1$ the last factor is
omitted. Since $\nrm{D_{u_1}\varphi}_{L^\infty}\le G\nrm{u_1}_2$, Young's
inequality and \eqref{eq:product-directional-variation} bound this expression
by $G\theta^{j-1}\prod_{i=1}^j\nrm{u_i}_2$.
These distributional derivatives have continuous representatives: if $J$ is
the combined $L^1$ kernel and $v$ tends to zero, their translation difference
is bounded by $G\nrm{u_1}_2\nrm{J(\cdot+v)-J}_{L^1}$, which tends to zero.
Integrating continuous weak partial derivatives along coordinate lines shows
that they are classical derivatives, so $\mathcal S_{\rho,k}\varphi\in C^{k+1}$.
Since the directions $u_1,\dots,u_j$ are arbitrary, this is the tensor-norm
bound \eqref{eq:product-derivative-bound}, without polarization or
coordinate-expansion loss. The case $j=k$ at $x$ and at $y$, and the case
$j=k+1$ integrated along the segment, give
\[
 \nrm{D^k\mathcal S_{\rho,k}\varphi(x)-D^k\mathcal S_{\rho,k}\varphi(y)}
 \le\min\{2G\theta^{k-1},G\theta^k\nrm{x-y}_2\}.
\]
Applying $\min\{a,b\}\le a^{1-\nu}b^\nu$ proves
\eqref{eq:product-holder-bound}, with $\nu=1$ read directly.
\end{proof}

\begin{restatable}[Lower bound for order-$k$ methods]{proposition}{thmcoordlower}
\label{thm:coordinate-tensor-lower}
Let $k,N\ge1$ be integers, $0<\nu\le1$, $s=k+\nu$, $1\le p\le2$,
$R,L_{k,\nu}>0$, and $n\ge N+1$. For every deterministic method that makes at most
$N$ global order-$k$ queries and returns $\widehat x\in B_p^n(R)$ there is a convex
function $f\in C^{k+1}(\R^n)$ satisfying \eqref{eq:holder-tensor} such that
\begin{equation}\label{eq:coordinate-tensor-lower}
 f(\widehat x)-\min_{B_p^n(R)}f\ \ge\
 \frac{2^{\nu-2}}{(8k\sqrt{10})^{s-1}}\cdot
 \frac{L_{k,\nu}R^{s}}{(N+1)^{s(1+1/p)-1}}.
\end{equation}
\end{restatable}

\begin{proof}
If the method stops early, pad its transcript with repeated queries. Set
$m=N+1$ and work on the first $m$ coordinates. Define
\[
 d=R/m^{1/p},\qquad \tau=d/(2m),\qquad \rho=\tau/4,\qquad
 b=2^{\nu-1}L_{k,\nu}\left(\frac{\rho}{k\sqrt{10}}\right)^{s-1}.
\]
At query $x_t$, select an unused coordinate $i_t\le m$ maximizing
$|(x_t)_{i_t}|$; use the smallest index to break ties. Choose
$v_t=\sign((x_t)_{i_t})e_{i_t}$, with positive sign at zero, and set
\[
 \varphi_t(x)=\max_{1\le j\le t}\{\ip{v_j}{x}-(j-1)\tau\}.
\]
The oracle returns the value and the derivatives through order $k$ of
$b\mathcal S_{\rho,k}\varphi_t$ at $x_t$, the smoothing
always acting on the first $m$ coordinates. A later vector $v_j$, $j>t$,
uses a coordinate unused at time $t$, so $\ip{v_j}{x_t}\le\ip{v_t}{x_t}$. Its shifted
form is therefore at least $(j-t)\tau\ge\tau$ below form $t$ there.
The difference of two signed coordinate slopes has $\ell_1$ norm at most two.
Consequently every future completion $\varphi$ of the maximum agrees with $\varphi_t$
throughout $B_\infty(x_t,\tau/2)$, and \cref{lem:product-smoothing}
gives the identity
\[
 \mathcal S_{\rho,k}\varphi=\mathcal S_{\rho,k}\varphi_t
       \quad\text{on }B_\infty(x_t,\tau/4).
\]
Hence the value and all requested derivatives at $x_t$ are those of every
completion, for queries of any size and position, repeated ones included.

After the output, align the remaining coordinate's sign with its coordinate
in $\widehat x$ and call the resulting vector $v_m$. Let
$\varphi(x)=\max_{j\le m}\{\ip{v_j}{x}-(j-1)\tau\}$ and
$f=b\mathcal S_{\rho,k}\varphi$. This one fixed function reproduces all answers.
Its slopes have Euclidean and $\ell_1$ norm one, so the smoothing lemma with
$G=G_\infty=1$ proves the claimed regularity and convexity.
The aligned final sign gives $\varphi(\widehat x)\ge-N\tau$, while
$x^\circ=-d\sum_{j=1}^m v_j$ has $p$-norm exactly $R$ and $\varphi(x^\circ)=-d$.
Jensen's inequality at the output and the upper bias bound at the comparator give
\[
 f(\widehat x)-\min_{B_p^n(R)}f
 \ge b(d-N\tau-\rho)
 =bd\left(\frac12+\frac{3}{8m}\right)\ge bd/2.
\]
Substitution proves \eqref{eq:coordinate-tensor-lower}.
\end{proof}

For $p=2$ the exponent $s(1+1/p)-1=(3s-2)/2$ is the optimal Euclidean exponent
of methods of order $k$ \citep{arjevani2019,grapiglia2020,kovalev2022tensor};
for $p<2$ it is larger by $s(1/p-1/2)$, and for $p=1$ it equals $2s-1$, the
exponent attained by \SPT{} (\cref{thm:tensor}). For $k=1$, smoothing by a
Moreau envelope (\cref{thm:first-order-lower}) gives better constants and also
covers the nonsmooth case.

\subsection{The simplex and prox steps}
\label{app:lower-simplex}

\begin{proposition}[Exact oracle reduction to a simplex]
\label{prop:simplex-lift}
The linear map $A:\Delta_{2d}\to B_1^d(1)$,
$A(w)=w_+-w_-$ for $w=(w_+,w_-)\in\R^d\times\R^d$, is onto and has Euclidean operator norm $\sqrt2$.
If $\bar f=f\circ A$, a query of $\bar f$ through order $k$ is simulated with one
query of $f$ through order $k$ and known linear transformations, and
\[
 L_{k,\nu}(\bar f)\le2^{(k+\nu)/2}L_{k,\nu}(f),
 \qquad G(\bar f)\le\sqrt2\, G(f).
\]
The constrained minima coincide, and the error of $w$ for $\bar f$ equals that of
$Aw$ for $f$.
\end{proposition}
\begin{proof}
For $x\in B_1^d(1)$ take its positive and negative parts and add half
the remaining mass $1-\nrm{x}_1$ to both entries of one signed pair.
This gives a preimage in $\Delta_{2d}$.
Conversely $\nrm{w_+-w_-}_1\le\sum_i(w_{+,i}+w_{-,i})=1$.
Since $AA^\top=2I$, its operator norm is $\sqrt2$.
The chain rule gives
$D^k\bar f(w)[h_1,\ldots,h_k]=D^kf(Aw)[Ah_1,\ldots,Ah_k]$.
The change of the base point contributes $\nrm{A}^\nu$, and the
$k$ directions contribute $\nrm{A}^k$.
Surjectivity proves the remaining claims.
\end{proof}

Thus the deterministic lower bounds for $p=1$ transfer to the simplex, with
twice the dimension and the stated constant loss; a larger simplex is mapped
onto $B_1^d(1)$ by sending the additional vertices to zero.

The lower bounds count oracle queries and allow any computation with known
quantities between them. They therefore also apply to methods built on a
non-Euclidean prox-setup: a prox step with a known prox-function, for instance
the entropy on the simplex, whose inputs are chosen from the transcript (and, for
a randomized method, from its random bits) is such a computation. In
particular, \cref{thm:coordinate-tensor-lower,thm:rand-tensor} give the exponent
$2s-1$ on $B_1^n(R)$ also for these methods.

\subsection{Randomized methods}
\label{app:lower-rand}

The resisting oracle above adapts the function to the queries and therefore
gives no bound for randomized methods. The bounds below draw one random function
\emph{before} the interaction. All three proofs follow the same scheme. For a
\emph{deterministic} method we bound the error averaged over the random
function; to do so we pass to a more informative oracle, whose answers determine
the true ones, and we keep track of the conditional law of the unrevealed
parameters given the history. This law stays explicit: a uniformly random subset
of the unchecked indices for a hidden support (by exchangeability), and
independent exponential residuals above known thresholds for random shifts (by
the product structure and the memoryless property). One query reveals only a
few parameters, so with dimension linear in $N$ much of the function remains
unknown. For a \emph{randomized} method we fix its random seed, including the
randomization of the output, apply the deterministic bound and average over the
seed; exchanging the two expectations gives one fixed function whose expected
error over the seed is at least the bound. Queries are global and return values
together with (sub)gradients, and a method that stops early is padded with dummy
queries.

\begin{restatable}[Randomized nonsmooth lower bound]{proposition}{thmrandnonsmooth}
\label{thm:rand-nonsmooth}
Let $G,R>0$, $N\ge1$, $1\le p\le2$ and $n\ge2^{p+2}N$. For every randomized
method with at most $N$ global queries of values and subgradients and output in
$B_p^n(R)$ there is a set $S\subseteq\{1,\dots,n\}$ with $|S|=2N$ such that
$f_S(x)=G\max_{i\in S}x_i$, with the fixed subgradient selection
$Ge_{\min\argmax_{i\in S}x_i}$, satisfies
\[
 \E\bigl[f_S(\widehat x)-\min_{B_p^n(R)}f_S\bigr]\ \ge\ \frac{GR}{2(2N)^{1/p}} ;
\]
for $p=1$ the bound is $GR/(4N)$, and it also holds on $\Delta_n(R)$ for
$\widetilde f_S(z)=G\max_{i\in S}(-z_i)$ with the selection
$-Ge_{\min\argmax_{i\in S}(-z_i)}$.
\end{restatable}

\begin{restatable}[Randomized smooth lower bound]{proposition}{thmrandsmooth}
\label{thm:rand-smooth}
Let $L,R>0$, $N\ge1$ and $n\ge32N+1$. For every randomized method with at most
$N$ global first-order queries and output in $K$, where $K=B_1^n(R)$ or
$K=\Delta_n(R)$, there is a convex $f\in C^1(\R^n)$ with $L$-Lipschitz gradient
such that $\E[f(\widehat x)-\min_Kf]\ge LR^2/(2^{50}N^3)$. If in addition
$\mu_1>0$ and the slack $\Lambda=L/\mu_1-n$ satisfies $\Lambda\ge2^{51}N^3$, the
function can be taken $\mu_1$-strongly convex in $\ell_1$ on $\R^n$ with the
bound $LR^2/(2^{52}N^3)$.
\end{restatable}

\begin{restatable}[Randomized higher-order lower bound]{proposition}{thmrandtensor}
\label{thm:rand-tensor}
Let $k\ge1$, $0<\nu\le1$, $s=k+\nu$, $L_{k,\nu},R>0$, $N\ge1$ and $n\ge32N+1$.
For every randomized method with at most $N$ global order-$k$ queries and
output in $K$, where $K=B_1^n(R)$ or $K=\Delta_n(R)$, there is a convex
$f\in C^{k+1}(\R^n)$
satisfying \eqref{eq:holder-tensor} such that
\[
 \E\bigl[f(\widehat x)-\min_Kf\bigr]\ \ge\ c_{k,\nu}'\,\frac{L_{k,\nu}R^{s}}{N^{2s-1}},
\]
where $c'_{k,\nu}>0$ depends only on $k$ and $\nu$ and is explicit in
\cref{app:rand-tensor}.
\end{restatable}

The hard functions of \cref{thm:rand-smooth,thm:rand-tensor} are smoothed
coordinate maxima $\max_i(x_i-a_i)$ with independent exponential shifts $a_i$.
The smoothing is a Moreau envelope for $k=1$ and a compact product kernel for
general $k$. The proofs use that a full order-$k$ answer at $x$ is determined by
the shifts of the coordinates near the maximum (the active coordinates of the
Moreau envelope, or those within $2\rho$ of the maximum for the kernel with
smoothing parameter $\rho$), and that on average at most two of these are
revealed for the first time at each query. The constants are not optimized.
Randomized lower bounds of this type are classical
\citep{diakonikolas2020lower,agarwal2018higher,garg2021lower}; the statements
above are for Euclidean smoothness on an $\ell_1$ domain, with global queries
and dimension linear in $N$, matching the upper bounds of
\cref{app:level,app:classes}.

\subsubsection{Hidden supports: proof of \cref{thm:rand-nonsmooth}}
\label{app:rand-nonsmooth}

\begin{proof}[Proof of \cref{thm:rand-nonsmooth}]
\emph{The family.} Let $0\le N<m\le n$, let $S$ be uniform among the
$m$-element subsets of $\{1,\dots,n\}$, and $f_S(x)=G\max_{i\in S}x_i$ on
$\R^n$. Then $|f_S(x)-f_S(y)|\le G\nrm{x-y}_\infty\le G\nrm{x-y}_2$, every
subgradient lies in $G\conv\{e_i:i\in S\}$ and has norm at most $G$. The oracle
returns $(f_S(x),Ge_{i(x,S)})$ with $i(x,S)=\min\argmax_{i\in S}x_i$, a fixed
rule. The minimum over $K=B_p^n(R)$ is $-GRm^{-1/p}$: if $t=\max_{i\in S}x_i<0$
then $R^p\ge\sum_{i\in S}|x_i|^p\ge m|t|^p$, and the value is attained at
$x_i=-Rm^{-1/p}$ on $S$ and $0$ elsewhere.

\emph{A more informative oracle.} Fix a deterministic method. The extended
history consists of the elements of $S$ found so far, the indices excluded so
far, and the set $U$ of unchecked indices. At a query $x$, order $U$ by
decreasing $x_i$ (ties by increasing index) and test membership in $S$ in this
order up to and including the first positive answer; the tested indices leave
$U$ and their status is revealed. Each query thus reveals exactly one new
element of $S$, even when a previously found element already determines the
true answer. The true answer $(f_S(x),Ge_{i(x,S)})$ is determined by these
data: the new positive index is the best among the previously unknown elements
of $S$ (with the same tie rule), and comparing it with the known elements gives
$i(x,S)$. Hence a lower bound for this oracle also holds for the original one.

\emph{Conditional distribution.} After any extended history, $S\cap U$ is
uniform among the subsets of $U$ of its size. Each membership test is chosen
by the revealed data, and its answer only includes or excludes the tested
index, so all remaining admissible subsets stay equally likely; the query
points are functions of the history and add no constraint. After $N$ queries
put $r=m-N>0$ and $q=|U|\ge r$. We claim
\begin{equation}\label{eq:reciprocal}
 \E\frac1q=\frac{m}{n(m-N)} .
\end{equation}
Let $r'/q'$ be the density of positive indices among the $q'$ unchecked ones
before an individual test. Before the $N$th positive index is found,
$r'\ge m-N+1\ge2$, so $q'\ge2$; a test succeeds with probability $r'/q'$, and the
conditional expectation of the density after the test is
$\frac{r'}{q'}\frac{r'-1}{q'-1}+(1-\frac{r'}{q'})\frac{r'}{q'-1}=\frac{r'}{q'}$.
After the $N$th positive index is found, we keep the density constant. The
number of tests is at most $n-(m-N)$, so every branch can be padded to the same
finite length, and repeated application of the identity shows that the expected
final density equals the initial one, $m/n$. The final density is $(m-N)/q$,
which proves \eqref{eq:reciprocal}.

\emph{Conditional risk.} Let $\widehat x\in K$ be the output, a fixed point given
the history. Since $S\cap U$ is nonempty and uniform,
\begin{equation}\label{eq:cond-risk}
 \E[f_S(\widehat x)\mid\mathrm{history}]
 \ge G\,\E\Bigl[\frac1r\sum_{i\in S\cap U}\widehat x_i\Bigm|\mathrm{history}\Bigr]
 =\frac Gq\sum_{i\in U}\widehat x_i\ \ge\ -GRq^{-1/p},
\end{equation}
where the last step is H\"older's inequality $\sum_{i\in U}|\widehat x_i|\le q^{1-1/p}R$.
For $1\le p\le2$ the function $t\mapsto t^{1/p}$ is concave, so Jensen's
inequality and \eqref{eq:reciprocal} give
$\E[q^{-1/p}]\le(m/(n(m-N)))^{1/p}$. Hence, for every deterministic method,
\begin{equation}\label{eq:nonsmooth-det}
 \E_S\bigl[f_S(\widehat x)-\min_Kf_S\bigr]\ge GR\Bigl[m^{-1/p}-\Bigl(\frac{m}{n(m-N)}\Bigr)^{1/p}\Bigr].
\end{equation}
For a randomized method, \eqref{eq:nonsmooth-det} holds for every fixed seed,
and exchanging the two expectations ($S$ takes finitely many values and the
errors are bounded on $K$) gives a fixed $S$
whose expected error over the seed is at least the right-hand side.

\emph{Parameters.} With $m=2N$ and $n\ge2^{p+2}N$,
$(m/(n(m-N)))^{1/p}=(2/n)^{1/p}\le\frac12(2N)^{-1/p}$, so the bound is
$GR/(2(2N)^{1/p})$; for $p=1$, $n\ge8N$ it is $GR/(4N)$.

\emph{Simplex.} For $p=1$ use $\widetilde f_S(z)=G\max_{i\in S}(-z_i)$ on
$\Delta_n(R)$ with the oracle obtained from the one above at $-z$ (sign of the
subgradient reversed). An output $z\in\Delta_n(R)$ gives $-z\in B_1^n(R)$, and the
point $z_i=R/m$ on $S$, $0$ elsewhere, lies in $\Delta_n(R)$ with value $-GR/m$,
the minimum over the ball, so the argument applies without change.
\end{proof}

\subsubsection{Exponential shifts and the Moreau envelope: proof of \cref{thm:rand-smooth}}
\label{app:rand-smooth}

\begin{proof}[Proof of \cref{thm:rand-smooth}]
In this proof $n$ denotes the number of \emph{shifted} coordinates and $d=n+1$
the ambient dimension, which is the $n$ of the statement; the hypothesis
$n\ge32N+1$ of \cref{thm:rand-smooth} thus reads $n\ge32N$. Let
$c_\star=2^{17}=131072$. We prove the bound $LR^2/(c_\star^2n(n+1)^2)$ for
$K=B_1^d(R)$; for $n=32N$ it is at least $LR^2/(2^{50}N^3)$ because
$32\cdot33^2<2^{16}$. Larger dimensions and the
simplex are treated at the end.

\emph{The family.} Let $\delta=R/d$, $a_0=0$, and let $a_1,\dots,a_n$ be
independent exponential with mean $\delta$ (density $\delta^{-1}e^{-a/\delta}$
on $[0,\infty)$). Put
\[
\begin{aligned}
 \varphi_a(x)&=\max_{0\le i\le n}(x_i-a_i),&\quad\lambda&=\frac{\delta}{c_\star n},\\
 h_a(x)&=\inf_z\Bigl\{\varphi_a(z)+\frac{\nrm{z-x}_2^2}{2\lambda}\Bigr\},& f_a&=L\lambda h_a .
\end{aligned}
\]
The distribution is known to the method; only the realization $a$ is unknown.

\emph{Smoothness.} The function $\varphi_a$ is convex and $1$-Lipschitz in $\ell_2$. Its conjugate
equals $\ip ap$ on $\Delta_d$ and $+\infty$ elsewhere, so
$h_a(x)=\max_{p\in\Delta_d}\{\ip p{x-a}-\frac\lambda2\nrm p_2^2\}$ has a unique
maximizer $p(x)$, and the variational inequalities of the maximizers at $x$ and
$y$ give $\lambda\nrm{p(x)-p(y)}_2^2\le\ip{p(x)-p(y)}{x-y}$. Hence
$\nabla h_a=p$ is $1/\lambda$-Lipschitz, $f_a$ is convex with $L$-Lipschitz
gradient, and, since $\nrm p_2\le1$ on $\Delta_d$, $0\le \varphi_a-h_a\le\lambda/2$.

\emph{Information.} The maximizer is $p_i=(x_i-a_i-t)_+/\lambda$ with the unique
threshold $t$ such that $\sum_i(x_i-a_i-t)_+=\lambda$, and
$h_a(x)=t+\frac\lambda2\nrm p_2^2$. The value and gradient of $f_a$ determine
$(p,t)$; on a coordinate with $p_i>0$ they reveal $a_i=x_i-t-\lambda p_i$
exactly, and on a coordinate with $p_i=0$ they reveal only $a_i\ge x_i-t$.
Conversely, these data determine the answer, so the answer contains no further
information on the unrevealed parameters. By induction on the queries,
conditionally on any history the unrevealed parameters are independent with
$a_i=b_i+E_i$, where $b_i$ is the maximum of $0$ and all past thresholds
$x_i^{(j)}-t_j$, and $E_i$ is exponential with mean $\delta$. For the induction
step, note that given the past the query point is fixed. On the event that the
active set is $J$, the answer $(t,p)$ is an affine injective function of the
unrevealed active parameters (the revealed ones being fixed), and each inactive
unrevealed parameter enters only through the constraint $a_i\ge x_i-t$. Hence, given
the answer, the density of the inactive unrevealed parameters is proportional to
$\prod_i\pi_i(a_i)\1\{a_i\ge x_i-t\}$, where $\pi_i$ is the current conditional
density of $a_i$; it keeps its product form, and the memoryless property gives the
exponential law of the excess. (Equality of an unknown parameter with the
threshold has probability zero.)

\emph{Number of revelations.} Let $s_i=x_i-a_i$ and $M=\max_is_i$; since
$t\ge M-\lambda$, every active coordinate satisfies $s_i>M-\lambda$. Condition on
the history, on the index attaining $M$ and on the value of $M$. For any other
unrevealed index the condition of not exceeding $M$ truncates $E_i$ from below;
after this truncation, $\Prob\{s_i>M-\lambda\mid\cdot\}\le1-e^{-\lambda/\delta}$
by the memoryless property (if the window below $M$ is shorter than $\lambda$
the probability is smaller). Counting the maximizer as one revelation,
$\E[\text{new revelations}\mid\mathrm{history}]\le1+n(1-e^{-\lambda/\delta})\le1+n\lambda/\delta\le2$.

\emph{Uncertain coordinates.} After $N$ queries let $B$ be the set of unrevealed
indices $i\ge1$ with $b_i\le\delta/4$. Since $b_i\le a_i$,
$|B|\ge\#\{i\ge1:a_i\le\delta/4\}-\#\{\text{revealed}\}$, and with
$1-e^{-1/4}\ge1/5$ and $n\ge32N$,
\begin{equation}\label{eq:uncertain}
 \E|B|\ge n(1-e^{-1/4})-2N\ge n/8 .
\end{equation}

\emph{The nonsmooth optimum.} The equation $\sum_{i=0}^n(t_*-a_i)_+=R$ has a
unique positive root (as $a_0=0$), $\min_K\varphi_a=-t_*$, and the unique minimizer is
$x_i^\star=-(t_*-a_i)_+$: the smallest $\ell_1$-norm of a point with $\varphi_a\le-t$ is
$\sum_i(t-a_i)_+$ (each coordinate with $a_i<t$ must be at most $a_i-t$, the
others are set to $0$), and at $t=t_*$ this norm equals $R$. Since $a_i\ge0$,
$t_*\ge\delta$. For
$x\in K$ let $r=\varphi_a(x)+t_*\ge0$ and $w_i=(t_*-a_i)_+$; if $w_i>0$ then
$x_i+w_i\le r$. Using $\sum_iw_i=R$ and $\nrm x_1\le R$,
\begin{equation}\label{eq:l1-distance}
 \nrm{x-x^\star}_1=\nrm x_1+R-2\sum_i\min\{(-x_i)_+,w_i\}\le2\sum_i(w_i-(-x_i)_+)_+\le2dr ,
\end{equation}
because each term is $0$ when $w_i=0$, is $(w_i+x_i)_+\le r$ when $x_i<0$, and is
$w_i\le w_i+x_i\le r$ when $x_i\ge0$. Next, change one $a_i$, $i\ge1$, from $u$
to $v$ with $0\le u<v<\delta$; the roots satisfy $t_v\ge t_u\ge\delta$, so the
coordinates $0$ and $i$ are active for both, and subtracting the two equations,
$0=2(t_v-t_u)-(v-u)+\sum_{j\ne0,i}[(t_v-a_j)_+-(t_u-a_j)_+]$ with a nonnegative
sum; hence $0\le t_v-t_u\le(v-u)/2$ and $x_i^\star(v)-x_i^\star(u)\ge(v-u)/2$.

\emph{Conditional risk.} Fix the history, an index $i\in B$, and then all other
parameters; $a_i=b_i+E_i$ keeps its conditional law. The events
$A_1=\{E_i\le\delta/8\}$ and $A_2=\{3\delta/8\le E_i\le\delta/2\}$
have probabilities at least $1/9$ and $e^{-3/8}(1-e^{-1/8})\ge\frac58\cdot\frac19>\frac1{16}$;
on both, $a_i\le3\delta/4<\delta$, and any two values from the two events differ
by at least $\delta/4$, so the corresponding $x_i^\star$ differ by at least
$\delta/8$. For two independent conditional copies $X_i^\star,\widetilde X_i^\star$ and
any number $z$ (in particular the output coordinate, fixed given the history),
$\E|z-X_i^\star|\ge\frac12\E|X_i^\star-\widetilde X_i^\star|\ge\Prob(A_1)\Prob(A_2)\delta/8\ge\delta/2048$.
Averaging over the other parameters and summing over the history-dependent set
$B$, \eqref{eq:uncertain} gives $\E\nrm{\widehat x-x^\star}_1\ge n\delta/16384$, and
\eqref{eq:l1-distance} with $n/d\ge1/2$ gives
$\E[\varphi_a(\widehat x)-\min_K\varphi_a]\ge n\delta/(32768d)\ge\delta/65536$. By
$0\le \varphi_a-h_a\le\lambda/2$ and $\min_Kh_a\le\min_K\varphi_a$,
$\E[h_a(\widehat x)-\min_Kh_a]\ge\delta/65536-\lambda/2\ge\delta/c_\star$, and
multiplying by $L\lambda$ gives $LR^2/(c_\star^2nd^2)$.

\emph{Randomized methods, larger dimension, simplex.} Randomized methods are
handled as described at the beginning of \cref{app:lower-rand}; the errors are uniformly bounded
by $2L\lambda R$ on $K$ (as $\nrm{\nabla h_a}_2\le1$), so the expectations can
be exchanged. In dimension $d'>d$ let the function depend on the first $d$
coordinates only. The projection of a feasible output has $\ell_1$-norm at most
$R$, the minimum is unchanged, and the extra coordinates of a global query
reveal nothing. For the simplex use $\widetilde f_a(z)=f_a(-z)$: $z\in\Delta_d(R)$
gives $-z\in B_1^d(R)$, the minimizer $x^\star$ corresponds to the feasible point
$z_i=(t_*-a_i)_+$, the reflection preserves smoothness and the oracle
information, and \eqref{eq:l1-distance} and the risk argument apply without
change with the minimum of $\varphi_a$ over $-\Delta_d(R)$, which equals $-t_*$.

\emph{Strong convexity.} From $(t-a_i)_+\ge t-a_i$ and the root equation,
$t_*\le R/d+\frac1d\sum_{i\ge1}a_i$, so $\E\nrm{x^\star}_2^2\le R\,\E t_*\le2R^2/d$.
Let $d'\ge d$ be the ambient dimension (the $n$ of the statement), while the number
of shifted coordinates is $n=32N$; let
$\Lambda=L/\mu_1-d'\ge0$ be the slack, and write $x=(u,v)\in\R^d\times\R^{d'-d}$. Put
$\theta=Ld/(d+\Lambda)\in(0,L]$ and
\[
 \bar f_a(u,v)=(L-\theta)\lambda h_a(u)+\frac\theta2\nrm u_2^2+\frac L2\nrm v_2^2
\]
(no last term if $d'=d$). It is convex and $L$-smooth on $\R^{d'}$, and its
curvature is at least $\theta$ in the first $d$ coordinates and $L$ in the
others, so by \cref{lem:weighted-strong-norm} in the $\ell_1$ case it is
$\mu_1$-strongly convex in $\ell_1$, since $d/\theta+(d'-d)/L=(d+\Lambda+d'-d)/L=1/\mu_1$.
The known quadratic part is subtracted from the oracle answers and the extra
coordinates of a global query reveal nothing, so the information argument is
unchanged, and the first block
$\widehat u$ of a feasible output satisfies $\nrm{\widehat u}_1\le R$. Comparing
with the feasible point $(x^\star,0)$, where $x^\star$ is the minimizer of $\varphi_a$ on
$B_1^d(R)$ (not necessarily a minimizer of $\bar f_a$), and dropping the nonnegative
quadratic at the output,
\begin{align*}
 \E[\bar f_a(\widehat x)-\min_K\bar f_a]
 &\ge(L-\theta)\lambda\,\E[h_a(\widehat u)-h_a(x^\star)]-\frac\theta2\E\nrm{x^\star}_2^2\\
 &\ge\frac{(L-\theta)R^2}{c_\star^2nd^2}-\frac{\theta R^2}{d}
 =\frac{LR^2}{d+\Lambda}\Bigl[\frac{\Lambda}{c_\star^2nd^2}-1\Bigr],
\end{align*}
using $h_a(\widehat u)-h_a(x^\star)\ge \varphi_a(\widehat u)-\varphi_a(x^\star)-\lambda/2$ and
$L-\theta=L\Lambda/(d+\Lambda)$. If $\Lambda\ge2c_\star^2nd^2$, then $\Lambda\ge d$, so
$d+\Lambda\le2\Lambda$ and the bracket is at least $\Lambda/(2c_\star^2nd^2)$, whence
the expected error is at least $LR^2/(4c_\star^2nd^2)$. With $n=32N$, $d=32N+1\le33N$
and $32\cdot33^2<2^{16}$, the condition holds when $\Lambda\ge2^{51}N^3$ and the
bound is at least $LR^2/(2^{52}N^3)$. The reflection to the simplex preserves
the quadratic part. (Padding with curvature $\theta$ instead of $L$ would not
give $\mu_1$-strong convexity in $\ell_1$.)
\end{proof}

\subsubsection{Compact product kernels: proof of \cref{thm:rand-tensor}}
\label{app:rand-tensor}

\begin{proof}[Proof of \cref{thm:rand-tensor}]
As in \cref{app:rand-smooth}, $n$ denotes the number of shifted coordinates and
$d=n+1$ the ambient dimension, and $c_\star=2^{17}$.

\emph{Smoothing.} Let $\rho>0$, let $\mathcal S_{\rho,k}$ be the compact product
smoothing of \cref{lem:product-smoothing} on $\R^d$, and put
$H_{k,\nu}=2^{1-\nu}(k\sqrt{10})^{s-1}$. For a convex $\varphi:\R^d\to\R$ that is
$1$-Lipschitz in $\ell_2$ and in $\ell_\infty$, the function $\mathcal S_{\rho,k}\varphi$
is convex and $C^{k+1}$, and by \eqref{eq:product-holder-bound} and
\eqref{eq:product-bias},
\begin{equation}\label{eq:holder-smoothed}
 \nrm{D^k\mathcal S_{\rho,k}\varphi(x)-D^k\mathcal S_{\rho,k}\varphi(y)}\le
 H_{k,\nu}\rho^{-(s-1)}\nrm{x-y}_2^\nu,\qquad 0\le\mathcal S_{\rho,k}\varphi-\varphi\le\rho .
\end{equation}
Moreover $\mathcal S_{\rho,k}\varphi=\varphi*\chi_h^{*k}$, where the density $\chi_h^{*k}$ is
supported in $[-\rho,\rho]^d$ and its derivatives up to order $k$ are integrable
(each derivative can be put on a different factor). Hence the value and the
derivatives up to order $k$ of $\mathcal S_{\rho,k}\varphi$ at $x$ are integrals of
$\varphi(x-w)$, $w\in[-\rho,\rho]^d$, against fixed kernels.

\emph{Locality.} Let $\varphi_a(x)=\max_{0\le i\le n}(x_i-a_i)$, which is
$1$-Lipschitz in $\ell_2$ and in $\ell_\infty$. For a query $x$ let
$m_x=\max_i(x_i-a_i)$ and $I_x=\{i:x_i-a_i\ge m_x-2\rho\}$. For
$w\in[-\rho,\rho]^d$, $\varphi_a(x-w)=\max_{i\in I_x}(x_i-w_i-a_i)$: an excluded
coordinate is below $m_x-\rho$ after the perturbation, while a coordinate
attaining $m_x$ stays at least $m_x-\rho$. Hence the order-$k$ answer for
$\mathcal S_{\rho,k}\varphi_a$ at $x$ is determined by $x$, $I_x$ and the values $a_i$ for
$i\in I_x$. We pass to the stronger oracle that returns $I_x$, $m_x$ and the exact
$a_i$ for $i\in I_x$; the other parameters then satisfy $a_i>x_i-m_x+2\rho$.

\emph{Exponential family.} Let $a_0=0$ and $a_1,\dots,a_n$ be independent
exponential with mean $\delta=R/(n+1)$, and set $\rho=\delta/(c_\star n)$, $n\ge32N$.
By the argument of \cref{app:rand-smooth}, conditionally on any extended history
the unrevealed parameters are independent, with $a_i=b_i+E_i$, where
$b_i$ is the maximum of $0$ and the past thresholds $x_i-m_x+2\rho$ and $E_i$ is
exponential with mean $\delta$. The expected number of newly revealed
parameters at a query is at most $1+n(1-e^{-2\rho/\delta})\le1+2n\rho/\delta\le2$:
conditionally on the maximizer and on $m_x$, and after the truncation from
below, any other unrevealed score $x_i-b_i-E_i$ falls in the window of length
$2\rho$ below the maximum with probability at most $1-e^{-2\rho/\delta}$. Hence after $N$ queries
$\E\#\{i\text{ unrevealed}:b_i\le\delta/4\}\ge n(1-e^{-1/4})-2N\ge n/8$.

\emph{Residual uncertainty.} The minimizer of $\varphi_a$ over $B_1^d(R)$, the
$\ell_1$-distance bound \eqref{eq:l1-distance}, and the sensitivity of $x_i^\star$ to
$a_i$ are as in \cref{app:rand-smooth} (the root $t_*(a_i)$ is
continuous, piecewise linear with slope $1/\sigma\le1/2$, where $\sigma\ge2$ is
the number of active coordinates); with the events $A_1,A_2$ of that proof,
$\E|\widehat x_i-X_i^\star|\ge\delta/2048$ for every unrevealed $i$ with $b_i\le\delta/4$,
so $\E\nrm{\widehat x-x^\star}_1\ge n\delta/16384$ and
$\E[\varphi_a(\widehat x)-\min_K\varphi_a]\ge\delta/65536$.

\emph{The bound.} Let $f_a=\frac{L_{k,\nu}\rho^{s-1}}{H_{k,\nu}}\mathcal S_{\rho,k}\varphi_a$;
by \eqref{eq:holder-smoothed} it is convex, $C^{k+1}$ and satisfies
\eqref{eq:holder-tensor} on $\R^d$. By the second inequality in
\eqref{eq:holder-smoothed},
$\E[\mathcal S_{\rho,k}\varphi_a(\widehat x)-\min_K\mathcal S_{\rho,k}\varphi_a]\ge\E[\varphi_a(\widehat x)-\min_K\varphi_a]-\rho\ge\delta/65536-\delta/(c_\star n)\ge\delta/c_\star$,
and multiplying by $L_{k,\nu}\rho^{s-1}/H_{k,\nu}$ with $\delta=R/(n+1)$ and
$\rho=\delta/(c_\star n)$ gives
\[
 \E[f_a(\widehat x)-\min_Kf_a]\ \ge\ \frac{L_{k,\nu}R^s}{c_\star^sH_{k,\nu}\,n^{s-1}(n+1)^s}
 \ \ge\ \frac{L_{k,\nu}R^s}{c_\star^sH_{k,\nu}32^{s-1}33^s}\,N^{-(2s-1)}\qquad(n=32N),
\]
that is, $c'_{k,\nu}=(c_\star^sH_{k,\nu}32^{s-1}33^s)^{-1}$. The passage from the
average over $a$ (for each fixed seed) to a fixed function and the extension to
larger dimensions (the function depends on the first $32N+1$ coordinates) are
as in \cref{app:rand-smooth}. So is the simplex, through
$\widetilde f_a(u)=f_a(-u)$: the jet at $u$ is the jet at $-u$ with the factor
$(-1)^j$ on the order-$j$ derivative, $-x^\star\in\Delta_d(R)$, and
$\mathcal S_{\rho,k}\varphi_a(-u)-\min_{v\in\Delta_d(R)}\mathcal S_{\rho,k}\varphi_a(-v)\ge \varphi_a(-u)-\varphi_a(x^\star)-\rho$.
\end{proof}

\section{Nonsmooth, H\"older and higher-order objectives}
\label{app:classes}

This section treats the minimization classes of \cref{tab:intro} other than the
smooth one with a first-order oracle: nonsmooth objectives with deep cuts (\cref{app:nonsmooth}),
H\"older continuous gradients (\cref{app:holder}), oracles of higher order
(\cref{app:tensor}), and gradient oracles without function values
(\cref{app:gradient-only,app:cog}). As before, $K\subseteq c+B_1^n(R)$ is a known
polytope, all queries are feasible, and $\bN$, $\cn$, $\PN$, $\eta_N=\PN/(4N)$ and
$\delta_N=4\PN/N$ are as in \cref{sec:method}.

\subsection{Nonsmooth objectives: deep cuts and a certificate}
\label{app:nonsmooth}

Let $f$ be convex and $G$-Lipschitz on $K$ and let the oracle return, at a
feasible $x$, one subgradient $g\in\partial f(x)$ with $\nrm g_2\le G$; function
values are not used. \Cref{alg:spcut} queries an approximate Steiner point of the
current set and cuts \emph{deep}: it keeps only the points at which the
supporting hyperplane predicts an improvement of at least $\varepsilon$. Such cuts
may remove every minimizer, but the first \emph{empty} intersection
yields, by linear-programming duality, a convex combination of the queried
points whose error is below $\varepsilon$, certified from the stored oracle
answers alone.

\begin{algorithm}[t]
\caption{\SPC: deep cuts at approximate Steiner points}
\label{alg:spcut}
\begin{algorithmic}[1]
\Require known polytope $K\subseteq c+B_1^n(R)$; subgradient oracle with
$\nrm g_2\le G$; cut depth $\varepsilon>0$; selector accuracy $\eta=\varepsilon/(2G)$
\State $C_0\gets K$
\For{$t=1,2,\dots$}
 \State $x_t\gets$ a point of $C_{t-1}$ with $\nrm{x_t-\st{C_{t-1}}}_2\le\eta$
 \Comment{\cref{app:steiner-computation}; $x_1=0$ if $K=B_1^n(R)$}
 \State query $g_t\in\partial f(x_t)$
 \State $C_t\gets C_{t-1}\cap\{y:\ip{g_t}{y-x_t}\le-\varepsilon\}$
 \If{$C_t=\varnothing$}
  \State solve the linear program
  $\widehat\varepsilon=\min_{\lambda\in\Delta_t}\bigl\{\sum_{j\le t}\lambda_j\ip{g_j}{x_j}
  +\hs{K}\bigl(-\sum_{j\le t}\lambda_jg_j\bigr)\bigr\}$ with minimizer $\lambda$
  \State\Return $\widehat x=\sum_{j\le t}\lambda_jx_j$ and the certificate $\widehat\varepsilon$
 \EndIf
\EndFor
\end{algorithmic}
\end{algorithm}

\begin{restatable}[\SPC]{proposition}{thmspcut}
\label{thm:spcut}
Let $f$ be convex on $K\subseteq c+B_1^n(R)$ with all returned subgradients of
Euclidean norm at most $G$, let $N\ge1$, and run \cref{alg:spcut} with
$\varepsilon=\varepsilon_N=4G\PN/N=16GR\sqrt{\bN\cn}/N$. The method stops after at
most $N$ queries and returns $\widehat x\in K$ with
\[
 f(\widehat x)-\min_Kf\ \le\ \widehat\varepsilon\ <\ \varepsilon_N,
 \qquad\text{and moreover }\widehat\varepsilon\le2GR
 \ (\widehat\varepsilon\le GR\text{ if }K=B_1^n(R)).
\]
Consequently, for every $0<\varepsilon\le GR$ an $\varepsilon$-solution of a
$G$-Lipschitz convex function on $B_1^n(R)$ is found with
$O\bigl(\frac{GR}{\varepsilon}\sqrt{\log(2n)\log(2+\frac{GR}{\varepsilon}\sqrt{\log(2n)})}\bigr)$
subgradient queries.
\end{restatable}

\begin{proof}
\emph{Certificate.} Suppose $C_t=\varnothing$. Then
$\max_{y\in K}\min_{j\le t}\ip{g_j}{x_j-y}<\varepsilon$: the maximum is attained
on the compact set $K$, and a point $y$ with all $\ip{g_j}{x_j-y}\ge\varepsilon$
would belong to $C_t$. The function $(y,\lambda)\mapsto\sum_j\lambda_j\ip{g_j}{x_j-y}$
is bilinear on the compact convex sets $K\times\Delta_t$, so the minimax theorem
gives
\begin{align*}
 \max_{y\in K}\min_{j\le t}\ip{g_j}{x_j-y}
 &=\min_{\lambda\in\Delta_t}\max_{y\in K}\sum_j\lambda_j\ip{g_j}{x_j-y}\\
 &=\min_{\lambda\in\Delta_t}\Bigl\{\sum_j\lambda_j\ip{g_j}{x_j}
 +\hs K\Bigl(-\sum_j\lambda_jg_j\Bigr)\Bigr\}=\widehat\varepsilon ,
\end{align*}
a linear program over the known polytope $K$ (for $K=B_1^n(R)$ its last term
is $R\nrm{\sum_j\lambda_jg_j}_\infty$). For every $y\in K$, convexity and the
subgradient inequalities give
$f(\widehat x)-f(y)\le\sum_j\lambda_j(f(x_j)-f(y))\le\sum_j\lambda_j\ip{g_j}{x_j-y}\le\widehat\varepsilon<\varepsilon$.
Concentrating $\lambda$ on $j=1$ gives $\widehat\varepsilon\le\max_{y\in K}\ip{g_1}{x_1-y}\le G\cdot\diam_2(K)\le2GR$,
and $\widehat\varepsilon\le R\nrm{g_1}_\infty\le GR$ when $x_1=0$ and $K=B_1^n(R)$.

\emph{Budget.} Suppose that $C_0,\dots,C_N$ are all nonempty. Since
$\st{C_t}\in C_t$, the cut gives $\ip{g_t}{\st{C_t}-x_t}\le-\varepsilon$, so
$\nrm{\st{C_t}-x_t}_2\ge\varepsilon/G$ and, by the choice of $\eta$,
$\nrm{\st{C_t}-\st{C_{t-1}}}_2\ge\varepsilon/G-\eta=\varepsilon/(2G)$ for
$t=1,\dots,N$. Summing and applying \cref{lem:steiner-path} yields
$N\varepsilon/(2G)\le\PN$, which contradicts $\varepsilon=4G\PN/N$. Hence some
$C_t$ with $t\le N$ is empty, and the method stops after at most $N$ queries;
the point $\st{C_N}$ is used only in this argument and is never computed.

\emph{Query count.} Given $\varepsilon$, the smallest $N$ with
$4G\PN/N\le\varepsilon$ has the stated order by \cref{lem:spcut-count} below.
\end{proof}

The bound $\Omega(GR/N)$ of \cref{thm:first-order-lower}(a) (and its randomized
version, \cref{thm:rand-nonsmooth}) shows that the exponent is optimal; we do
not know whether the logarithmic factor is necessary. The proof uses only
$\nrm{g_t}_2\le G$ and $f(y)\ge f(x_t)+\ip{g_t}{y-x_t}$ on $K$, so the oracle may
return any subgradient satisfying this norm bound, and $f$ need be Lipschitz only on $K$; the same
certificate therefore also covers bounded monotone operators and
convex--concave saddle problems (\cref{app:vi}).

\begin{lemma}\label{lem:spcut-count}
Let $0<\varepsilon\le GR$, $Q=GR/\varepsilon\ge1$ and $A=1+\log(2n)$. If
$N=\lceil32Q\sqrt{A[1+\log(16Q\sqrt A)]}\rceil$, then
$16GR\sqrt{(1+\log N)(1+\log(2n))}/N\le\varepsilon$; in particular \cref{alg:spcut}
with depth $\varepsilon$ stops after at most $N=O\bigl(Q\sqrt{\log(2n)\log(2+Q\sqrt{\log(2n)})}\bigr)$
queries; the same holds, with $N$ doubled, for $\bN\cn$ in place of
$(1+\log N)(1+\log(2n))$.
\end{lemma}

\begin{proof}
Put $c=16Q\sqrt A\ge16$ and $u=2c\sqrt{1+\log c}=32Q\sqrt{A[1+\log(16Q\sqrt A)]}$.
Since $\log(1+\log c)\le\log c$,
\[
 1+\log u=1+\log2+\log c+\tfrac12\log(1+\log c)<4(1+\log c),
\]
so $u/\sqrt{1+\log u}>c$. The function $t\mapsto t/\sqrt{1+\log t}$ is
increasing for $t\ge1$, so $N=\lceil u\rceil$ satisfies
$N/\sqrt{1+\log N}\ge c=16GR\sqrt A/\varepsilon$, which is equivalent to
$16GR\sqrt{(1+\log N)A}/N\le\varepsilon$. For the
condition with $\bN\cn$, note that $\bN\le2(1+\log N)$ and $\cn\le2(1+\log(2n))$,
since $1+\lceil\log_2 t\rceil\le2(1+\log t)$ for $t\ge2$, so it suffices that
$N/\sqrt{1+\log N}\ge2c$. This holds for $N=\lceil2u\rceil$, because
$1+\log(2u)=1+2\log2+\log c+\frac12\log(1+\log c)\le2.4+1.5\log c<4(1+\log c)$
and hence $2u/\sqrt{1+\log(2u)}>2c$.
\end{proof}

\subsection{H\"older continuous gradients}
\label{app:holder}

We use the following consequence of \cref{lem:steiner-path}. If in a run of a level method the sets
$C_0\supseteq\dots\supseteq C_N$ are nonempty and every $z_t$ is an
$\eta_N$-approximate Steiner point of $C_t$, then, with $d_t=\nrm{z_t-z_{t-1}}_2$,
\begin{equation}\label{eq:app-path}
 \sum_{t=1}^Nd_t\le\tfrac32\PN,\qquad
 \#\{t:d_t>\delta_N\}<\tfrac{3N}{8},\qquad
 \#\{t:d_t\le\delta_N\}>\tfrac N2 .
\end{equation}
The first bound is \eqref{eq:actual-path}; the other two follow as in the
proof of \cref{thm:splevel}.

Let $0\le\nu\le1$ and let $f$ be convex on $K$ such that the oracle returns, at
each feasible $x$, a subgradient $g_x\in\partial f(x)$ with
\begin{equation}\label{eq:holder-remainder}
 0\le f(y)-f(x)-\ip{g_x}{y-x}\le\frac{L_{1,\nu}}{1+\nu}\nrm{y-x}_2^{1+\nu}
 \qquad(x,y\in K).
\end{equation}
For $0<\nu\le1$ this holds with $g_x=\nabla f(x)$ whenever $\nabla f$ is
$\nu$-H\"older with constant $L_{1,\nu}$; for $\nu=0$ it holds with
$L_{1,0}=2G$ for every $G$-Lipschitz convex $f$ and any subgradient selection with $\nrm{g_x}_2\le G$.
We run \cref{alg:splevel} with $L$ replaced by $1$ in the weight recursion,
$a_t=\frac12(1+\sqrt{1+4A_{t-1}})$, and with the certificate
\eqref{eq:holder-level} below in place of \eqref{eq:level-error}; the cuts, the
selector accuracy and the acceptance rule are unchanged.

\begin{restatable}[H\"older gradients]{proposition}{thmholder}
\label{thm:holder}
Under \eqref{eq:holder-remainder}, the level method with horizon $N$ either
certifies $\min_Kf>\ell$ or returns $y_N\in K$ with
\begin{equation}\label{eq:holder-level}
 f(y_N)\le\ell+\frac{24\cdot4^\nu}{1+\nu}\cdot\frac{L_{1,\nu}\PN^{1+\nu}}{N^{1+2\nu}}
 \le\ell+\frac{48\,L_{1,\nu}\PN^{1+\nu}}{N^{1+2\nu}} .
\end{equation}
With the bracket search of \cref{alg:spaccel}, its horizon rule
\eqref{eq:horizon-rule} replaced by $N(w)=\min\{2^j:e_{2^j}\le w/4\}$, where
$e_N$ is the computable bound \eqref{eq:holder-eN} below, which dominates the
right-hand side of \eqref{eq:holder-level} and avoids real powers, and with a
computable initial width of at most $4L_{1,\nu}(2R)^{1+\nu}$, an
$\varepsilon$-solution is obtained with
\[
 O\Bigl(1+\Bigl[\frac{L_{1,\nu}R^{1+\nu}}{\varepsilon}\Bigl\{\log(2n)\log\Bigl(2+\frac{L_{1,\nu}R^{1+\nu}}{\varepsilon}\log(2n)\Bigr)\Bigr\}^{\frac{1+\nu}2}\Bigr]^{\frac1{1+2\nu}}\Bigr)
\]
queries, uniformly in $\nu\in[0,1]$, and with a budget of $T\ge385$ queries the
error is $O\bigl(L_{1,\nu}R^{1+\nu}[\log(2T)\log(2n)]^{(1+\nu)/2}/T^{1+2\nu}\bigr)$.
\end{restatable}

The exponent $1+2\nu$ interpolates between the nonsmooth rate $N^{-1}$ of
\cref{thm:spcut} and the smooth rate $N^{-3}$ of \cref{thm:main}; it matches the
lower bound $\nu+(1+\nu)/p=1+2\nu$ of \cref{thm:first-order-lower}(b) at $p=1$,
and the randomized bound of \cref{thm:rand-tensor} with $k=1$. In the matched
Euclidean setting the exponent is $(1+3\nu)/2$ \citep{nesterov2015}. The level
solver does not use $\nu$ or $L_{1,\nu}$; only the horizon rule does.

\begin{proof}[Proof of \cref{thm:holder}]
We write $M=L_{1,\nu}$, $s=1+\nu$ and $\sigma=1+2\nu$.

\emph{Level method.} The weights satisfy $a_t=\frac12(1+\sqrt{1+4A_{t-1}})$, so
$a_t^2=A_{t-1}+a_t=B_t$. Consider an accepted trial with $A=A_{t-1}$, $a=a_t$,
$B=B_t$, $x=x_t$, $z=z_{t-1}$, $z^+=z_t$, $y=y_{t-1}$, $y^+=y_t$, and
$d=\nrm{z^+-z}_2$. Since $y^+-x=(a/B)(z^+-z)$, the remainder bound
\eqref{eq:holder-remainder} at $x$ gives
$Bf(y^+)\le Bf(x)+a\ip{g_x}{z^+-z}+B\frac Ms(a/B)^sd^s$. As in
\cref{lem:level-step}, convexity and $z^+\in C_t$ (the level cut) bound the
first two terms by $Af(y)+a\ell$. Hence, using $B=a^2$,
\begin{equation}\label{eq:holder-step}
 B[f(y^+)-\ell]\le A[f(y)-\ell]+\frac Ms\frac{a^s}{B^{\nu}}d^s
 =A[f(y)-\ell]+\frac Msa^{1-\nu}d^s .
\end{equation}
Number the accepted updates $j=1,\dots,k$. Then $a_1=1$, $A_j=a_j^2$, and
$a_{j+1}=\frac12(1+\sqrt{1+4a_j^2})\le a_j+1$ because $1+4a_j^2\le(2a_j+1)^2$;
hence $a_j\le j\le N$. Also $\sqrt{A_j}=a_j\ge\sqrt{A_{j-1}}+\frac12$, so
$A_k\ge k^2/4\ge N^2/16$ by \eqref{eq:app-path}. Rejected trials do not change
$(A,y)$. Telescoping \eqref{eq:holder-step} from $A_0=0$ and using $d_t\le\delta_N$
on accepted trials and \eqref{eq:app-path},
\[
 A_k[f(y_N)-\ell]\le\frac MsN^{1-\nu}\sum_{t\ \mathrm{acc.}}d_t^{1+\nu}
 \le\frac MsN^{1-\nu}\delta_N^{\nu}\sum_{t=1}^Nd_t
 \le\frac{3\cdot4^\nu M}{2s}\PN^{s}N^{1-2\nu}
\]
(for $\nu=0$ read $\delta_N^0=1$). Dividing by $A_k\ge N^2/16$ gives the first
bound in \eqref{eq:holder-level}; the second follows because $\nu\mapsto4^\nu/(1+\nu)$
is increasing on $[0,1]$ with value $2$ at $\nu=1$. The certificate branch is as in
\cref{thm:splevel}.

\emph{Finite parameters.} For $u>0$ let $\Dnu(u)=2^{\lceil\nu\lceil\log_2u\rceil\rceil}$,
computed by binary scaling and integer comparisons. Since
$\nu\lceil\log_2u\rceil\in[\nu\log_2u,\nu\log_2u+\nu]$, we have
$u^\nu\le\Dnu(u)<2^{\nu+1}u^\nu\le4u^\nu$, and $\Dnu$ is nondecreasing. Put
\begin{equation}\label{eq:holder-eN}
 e_N=\frac{48M\PN}{N}\Dnu\Bigl(\frac{\PN}{N^2}\Bigr),\qquad\text{so that}\qquad
 \frac{48M\PN^s}{N^\sigma}\le e_N\le\frac{192M\PN^s}{N^\sigma},
\end{equation}
because $(\PN/N)(\PN/N^2)^\nu=\PN^s/N^\sigma$. Thus a nonempty run returns
$f(y_N)\le\ell+e_N$, and the test $e_N\le w/4$ needs only finite arithmetic.

\emph{Bracket.} Query $(f(x_0),g_0)$ at $x_0\in K$ and let
$v\in\argmin_{u\in K}\ip{g_0}u$. Then $\flo=f(x_0)+\ip{g_0}{v-x_0}\le\min_Kf$ by
the lower bound in \eqref{eq:holder-remainder}, and by its upper bound
$f(v)\le\flo+\frac Ms\nrm{v-x_0}_2^s\le\flo+\frac Ms(2R)^s\le\flo+W_0$ with the
computable width $W_0=M(2R)\Dnu(2R)\le4M(2R)^s$ (for $K=B_1^n(R)$, $x_0=0$, one may
take $W_0=MR\Dnu(R)\le4MR^s$). If $\varepsilon\ge W_0$, the method returns $v$.

\emph{Bracket search.} With width $w>\varepsilon$ take $\ell=(\flo+\fhi)/2$ and
$N(w)=\min\{2^j:e_{2^j}\le w/4\}$, which exists since $e_N\to0$ along powers
of two. An empty run raises $\flo$ to $\ell$ (width $w/2$); a nonempty run
stores $y_N$ and lowers $\fhi$ to $\ell+w/4$ (width $3w/4$). Let
$w_0>\dots>w_{m-1}>\varepsilon$ be the widths, $N_j=N(w_j)\le N_{m-1}\le N(\varepsilon)$.
All $N_j>1$, since $P_1\ge4R$ gives
$e_1\ge48M(4R)^s\ge12\cdot4M(2R)^s\ge12W_0>w_j/4$.
Since $N_j/2$ violates the rule and $P_{N_j/2}\le P_{N_j}$,
\[
 \frac{w_j}4<e_{N_j/2}\le\frac{192\cdot2^\sigma MP_{N_j}^s}{N_j^\sigma}\le\frac{1536\,MP_{N_j}^s}{N_j^\sigma},
 \qquad
 N_{m-1}^\sigma\ge\frac{192\,MP_{N_{m-1}}^s}{w_{m-1}},
\]
where $P_N$ is the path budget at horizon $N$ and we used $2^\sigma\le8$.
As $P_{N_j}\le P_{N_{m-1}}$ and $w_j\ge(4/3)^{m-1-j}w_{m-1}$, dividing gives
$N_j<32^{1/\sigma}(3/4)^{(m-1-j)/\sigma}N_{m-1}$, hence, including the initial query,
\[
 N_{\rm total}\le1+\frac{32^{1/\sigma}}{1-(3/4)^{1/\sigma}}N(\varepsilon)<1+384N(\varepsilon),
\]
uniformly for $1\le\sigma\le3$, since $32^{1/\sigma}\le32$ and $(3/4)^{1/\sigma}\le(3/4)^{1/3}<11/12$.

\emph{Order of $N(\varepsilon)$.} Let $Q=MR^s/\varepsilon$; in the active branch
$\varepsilon<W_0\le4M(2R)^s$, so $Q$ is bounded below by a constant. Put
$H=768\cdot4^sQ\cn^{s/2}\ge2$ and let $N$
be the smallest power of two with $N\ge N_0:=[16H(1+\log_2H)^{s/2}]^{1/\sigma}$.
Then $N<2N_0$, so $\log_2N<5+\log_2H+\log_2(1+\log_2H)\le5+2\log_2H$ and
$\bN\le7(1+\log_2H)$; hence $N^\sigma\ge16H(1+\log_2H)^{s/2}\ge H\bN^{s/2}$ because
$7^{s/2}\le7<16$; this is equivalent to $192M(4R)^s(\bN\cn)^{s/2}/N^\sigma\le\varepsilon/4$
and gives $e_N\le\varepsilon/4$. Therefore $N(\varepsilon)\le N<2N_0$, and since
$\cn=O(\log(2n))$ and $\log H=O(\log(2+Q\log(2n)))$, this is the stated order,
with a constant that can be chosen uniformly in $\nu\in[0,1]$.

\emph{Fixed horizon.} For $T\ge385$ let $Q_T$ be the largest power of two with
$384Q_T\le T-1$ and $\varepsilon_T=4e_{Q_T}$. If $\varepsilon_T\ge W_0$ the method
returns the initial $v$. Otherwise $N(\varepsilon_T)\le Q_T$, the count is at most
$1+384Q_T\le T$, and the error is at most
$\min\{W_0,4e_{Q_T}\}\le\min\{4M(2R)^s,768MP_{Q_T}^s/Q_T^\sigma\}$. Since $Q_T>(T-1)/768$
and $\mathsf b_{Q_T}\le\mathsf b_T$, this is $O(MR^s[\log(2T)\log(2n)]^{s/2}/T^\sigma)$.
\end{proof}

\subsection{Oracles of higher order}
\label{app:tensor}

Let $k\ge2$, $0\le\nu\le1$, $s=k+\nu$, and let $f$ be convex and $C^k$ on a
neighborhood of $K$ with $\nu$-H\"older $k$th derivative in the sense of
\eqref{eq:holder-tensor} with constant $L_{k,\nu}$. The oracle returns the
Taylor polynomial $T_k(x;u)=\sum_{j\le k}\frac1{j!}D^jf(x)[(u-x)^j]$ at a
feasible $x$. Put
\begin{equation}\label{eq:tensor-constants}
 c_{k,\nu}=\prod_{j=1}^{k-1}\frac{1}{\nu+j}=\frac{\Gamma(\nu+1)}{\Gamma(k+\nu)},
 \qquad
 m_x(u)=T_k(x;u)+\frac{2c_{k,\nu}L_{k,\nu}}{s}\nrm{u-x}_2^s ,
\end{equation}
so that $|f(u)-T_k(x;u)|\le\frac{c_{k,\nu}L_{k,\nu}}{s}\nrm{u-x}_2^s$ and
$\nrm{\nabla f(u)-\nabla_uT_k(x;u)}_2\le c_{k,\nu}L_{k,\nu}\nrm{u-x}_2^{s-1}$ on $K$.
The model $m_x$ is known and $C^1$ but need not be convex; it is never used to
form a cut. A trial of \SPT{} (\cref{alg:sptaylor}) makes
two queries. The first is the jet at the mixed point $x=(Ay+az)/B$; from it a
finite search over the known model computes a point $v\in K$ with the
\emph{checkable} approximate stationarity
$\max_{u\in K}\ip{\nabla m_x(v)}{v-u}\le\eta_{\mathrm{mod}}$. The second is the
jet at $v$, of which only $f(v)$ and $\nabla f(v)$ are used; they give the level
cut $\{u:f(v)+\ip{\nabla f(v)}{u-v}\le\ell\}$, a supporting half-space of $f$.
If the approximate Steiner point $z^+$ of the cut set
is within $\delta_N$ of $z$, the state becomes $(y,A)\leftarrow(v,B)$ with the integer
weights $a_j=j^k$; otherwise the cut is kept and the state is not updated.

\begin{restatable}[Higher-order oracles]{proposition}{thmtensor}
\label{thm:tensor}
Let $\Xi_k=3^{k+2}8^k(k+1)^{k+1}$. With the parameters above, the
level method of \SPT{} makes at most $2N$ order-$k$ queries and either certifies
$\min_Kf>\ell$ or returns $y\in K$ with
\begin{equation}\label{eq:tensor-level}
 f(y)\le\ell+e_N,\qquad e_N\text{ as in }\eqref{eq:tensor-eN},\qquad
 2\Xi_kc_{k,\nu}L_{k,\nu}\frac{\PN^s}{N^{2s-1}}\le e_N\le8\Xi_kc_{k,\nu}L_{k,\nu}\frac{\PN^s}{N^{2s-1}} .
\end{equation}
The bracket search, with the initial width obtained from one jet at $x_0$
as in the proof below, returns an $\varepsilon$-solution after at most
$1+32(2k+1)N(\varepsilon)$ queries, where
$N(\varepsilon)=\min\{2^j:e_{2^j}\le\varepsilon/4\}
=O_k\bigl(1+[\frac{L_{k,\nu}R^s}{\varepsilon}\{\log(2n)\log(2+\frac{L_{k,\nu}R^s}{\varepsilon}\log(2n))\}^{s/2}]^{1/(2s-1)}\bigr)$
uniformly in $\nu\in[0,1]$, and with a budget of $T\ge32(2k+1)+1$ queries the
error is
\begin{equation}\label{eq:tensor-rate}
 O_k\Bigl(\frac{L_{k,\nu}R^{k+\nu}\,[\log(2T)\log(2n)]^{(k+\nu)/2}}{T^{2(k+\nu)-1}}\Bigr).
\end{equation}
\end{restatable}

For $0<\nu\le1$ the exponent $2(k+\nu)-1$ equals the lower bound
$s(1+1/p)-1$ of \cref{thm:coordinate-tensor-lower} at $p=1$ and the randomized
bound of \cref{thm:rand-tensor}. For a Hessian with Lipschitz constant $L_{2,1}$
the rate is $\tO(L_{2,1}R^3/N^5)$, compared with the Euclidean $N^{-7/2}$
\citep{arjevani2019,gasnikov2019,kovalev2022tensor}. Regularized Taylor models
and the mixing of two states are also used in non-Euclidean higher-order
acceleration \citep{contreras2025}; here no inner problem is solved to
optimality against an unknown function, the model may be nonconvex, and the
geometry enters only through \cref{lem:steiner-path}.

In this subsection $M=L_{k,\nu}$, $s=k+\nu$ and $\sigma=2s-1$.

\begin{lemma}[Taylor remainders]\label{lem:taylor-remainder}
Let $k\ge2$, $0\le\nu\le1$, and let $f$ be $C^k$ on a convex set containing
$x,u$ with \eqref{eq:holder-tensor} on the segment $[x,u]$. Then
$|f(u)-T_k(x;u)|\le\frac{c_{k,\nu} M}s\nrm{u-x}_2^s$ and
$\nrm{\nabla f(u)-\nabla_uT_k(x;u)}_2\le c_{k,\nu} M\nrm{u-x}_2^{s-1}$.
\end{lemma}

\begin{proof}
With $h=u-x$, the integral form of the remainder gives
$f(u)-T_k(x;u)=\int_0^1\frac{(1-t)^{k-1}}{(k-1)!}\bigl(D^kf(x+th)-D^kf(x)\bigr)[h^k]\,dt$,
whose absolute value is at most
$M\nrm h_2^{s}\int_0^1\frac{(1-t)^{k-1}t^\nu}{(k-1)!}dt=M\nrm h_2^s\frac{\Gamma(\nu+1)}{\Gamma(k+\nu+1)}=\frac{c_{k,\nu} M}s\nrm h_2^s$.
Similarly, $\nabla f(u)-\nabla_uT_k(x;u)=\int_0^1\frac{(1-t)^{k-2}}{(k-2)!}\bigl(D^kf(x+th)-D^kf(x)\bigr)[h^{k-1},\cdot]\,dt$,
of norm at most $M\nrm h_2^{s-1}\frac{\Gamma(\nu+1)}{\Gamma(k+\nu)}=c_{k,\nu} M\nrm h_2^{s-1}$.
Here we used the Beta integral $\int_0^1(1-t)^{a-1}t^\nu dt=\Gamma(\nu+1)\Gamma(a)/\Gamma(a+\nu+1)$
and the recursion $\Gamma(\nu+1)/\Gamma(k+\nu)=\prod_{j=1}^{k-1}(\nu+j)^{-1}$;
the argument also covers $\nu=0$.
\end{proof}

\begin{algorithm}[t]
\caption{\SPT: level method with a regularized Taylor model (level $\ell$, horizon $N$)}
\label{alg:sptaylor}
\begin{algorithmic}[1]
\Require known polytope $K\subseteq c+B_1^n(R)$; order-$k$ oracle; $M=L_{k,\nu}$;
level $\ell$; horizon $N$; geometric accuracy $\eta_N$, threshold $\delta_N$;
model accuracy $\eta_{\mathrm{mod}}=e_N/(2N)$, where the certificate
$e_N$ is defined in \eqref{eq:tensor-eN} below
\State $C\gets K$; $z\gets$ an $\eta_N$-approximate Steiner point of $K$; $y\gets$ any point of $K$; $A\gets0$; $j\gets1$
\For{$t=1,\dots,N$}
 \State $a\gets j^k$; $B\gets A+a$; $x\gets(Ay+az)/B$
 \State query the jet at $x$; form $m_x$ from \eqref{eq:tensor-constants}
 \State find $v\in K$ with $\max_{u\in K}\ip{\nabla m_x(v)}{v-u}\le\eta_{\mathrm{mod}}$ (finite search, \cref{lem:model-search})
 \State query the jet at $v$; use $f(v)$ and $g=\nabla f(v)$
 \State $C\gets C\cap\{u:f(v)+\ip g{u-v}\le\ell\}$; \textbf{if} $C=\varnothing$ \textbf{return} ``$\min_Kf>\ell$''
 \State $z^+\gets$ an $\eta_N$-approximate Steiner point of $C$
 \If{$\nrm{z^+-z}_2\le\delta_N$} $y\gets v$; $A\gets B$; $j\gets j+1$ \Comment{accept}
 \EndIf
 \State $z\gets z^+$
\EndFor
\State\Return $y$ with the certificate $f(y)\le\ell+e_N$
\end{algorithmic}
\end{algorithm}

\begin{lemma}[Finite model search]\label{lem:model-search}
For every $\eta>0$ a point $v\in K$ with
$\cG(v):=\max_{u\in K}\ip{\nabla m_x(v)}{v-u}\le\eta$ can be found by a finite
enumeration of rational convex combinations of the vertices of $K$, using only
certified approximations of known scalar powers.
\end{lemma}

\begin{proof}
The model $m_x$ is $C^1$ because $s\ge2$, so it attains its minimum over the
compact set $K$ at some $v^\star$. Optimality along every segment
$[v^\star,u]\subseteq K$ gives $\cG(v^\star)=0$, and $\cG$ is continuous. Rational
convex combinations of the vertices are dense in $K$, and we enumerate them. At
each candidate $v$ we compute a vector $h$ with $\nrm{h-\nabla m_x(v)}_2\le\eta/(8R')$,
where $R'$ bounds the Euclidean diameter of $K$ (here $R'=2R$), compute
$\widehat\cG=\max_{u\in K}\ip h{v-u}$ by a linear program, and accept $v$ if
$\widehat\cG\le\eta/2$. Since $\max_{u\in K}\nrm{v-u}_2\le R'$, the two objectives
differ by at most $\eta/8$, so $\cG(v)\le\widehat\cG+\eta/8\le\eta$. Conversely,
for all candidates close enough to $v^\star$, $\cG(v)<\eta/8$, hence
$\widehat\cG<\eta/4<\eta/2$, so the enumeration stops. For rational $\nu$ the
powers in $m_x$ are algebraic; for a general known $\nu\in[0,1]$, certified
rational enclosures of $\nrm{v-x}_2^{s-2}$ suffice because the test has
positive slack.
\end{proof}

\begin{lemma}[One trial]\label{lem:tensor-step}
Let $A\ge0$, $a>0$, $B=A+a$, $y,z\in K$, $x=(Ay+az)/B$, let $v\in K$ satisfy
$\cG(v)\le\eta$, let $g=\nabla f(v)$, and let $z^+\in K$ satisfy
$f(v)+\ip g{z^+-v}\le\ell$. Then, with $r=v-x$ and $d=z^+-z$,
\begin{equation}\label{eq:tensor-step}
 B[f(v)-\ell]\le A[f(y)-\ell]+B\eta+3^{k+1}c_{k,\nu} M\frac{a^s}{B^{s-1}}\nrm d_2^s .
\end{equation}
\end{lemma}

\begin{proof}
Put $W=Bv-Ay-az^+=Br-ad$. Convexity at $v$, $Af(y)\ge Af(v)+A\ip g{y-v}$, and
the cut, $a\ell\ge af(v)+a\ip g{z^+-v}$, give
$B[f(v)-\ell]-A[f(y)-\ell]=Bf(v)-Af(y)-a\ell\le\ip g{W}$. Since
$(Ay+az^+)/B\in K$ and $W=B(v-(Ay+az^+)/B)$, the model condition gives
$\ip{\nabla m_x(v)}W\le B\eta$. By \cref{lem:taylor-remainder} and
\eqref{eq:tensor-constants}, $g-\nabla m_x(v)=e-2c_{k,\nu} M\nrm r_2^{s-2}r$ with
$\nrm e_2\le c_{k,\nu} M\nrm r_2^{s-1}$ (for $r=0$ the second term is $0$). Hence
\begin{align*}
 \ip gW&\le B\eta-2c_{k,\nu} M\nrm r_2^{s-2}\ip rW+\ip eW\\
 &\le B\eta-2c_{k,\nu} M\nrm r_2^{s-2}(B\nrm r_2^2-a\nrm r_2\nrm d_2)+c_{k,\nu} M\nrm r_2^{s-1}(B\nrm r_2+a\nrm d_2),
\end{align*}
that is, $\ip gW\le B\eta-c_{k,\nu} MB\nrm r_2^s+3c_{k,\nu} Ma\nrm r_2^{s-1}\nrm d_2$. For
$\tau\ge0$ the maximum of $-B\tau^s+3a\nrm d_2\tau^{s-1}$ is attained at
$\tau=3(s-1)a\nrm d_2/(sB)$ (it is $0$ if $d=0$) and equals
$3^s(s-1)^{s-1}s^{-s}a^s\nrm d_2^s/B^{s-1}\le3^{k+1}a^s\nrm d_2^s/B^{s-1}$.
\end{proof}

\begin{proof}[Proof of \cref{thm:tensor}]
\emph{Level method.} Feasibility and the certificate branch are as in
\cref{thm:splevel}, since every cut is a supporting half-space of $f$ and
therefore contains every $u$ with $f(u)\le\ell$. Suppose all $N$ sets are nonempty. By
\eqref{eq:app-path} at least $m>N/2$ trials are accepted. The accepted weights
are $a_j=j^k$ and $A_j=\sum_{i\le j}i^k\ge j^{k+1}/(k+1)$, so
\[
\begin{aligned}
 \frac{a_j^s}{A_j^{s-1}}&\le(k+1)^{s-1}j^{ks-(k+1)(s-1)}
 =(k+1)^{s-1}j^{k+1-s}\\
 &\le(k+1)^kN^{1-\nu},\qquad
 A_m\ge\frac{N^{k+1}}{2^{k+1}(k+1)} .
\end{aligned}
\]
Telescoping \eqref{eq:tensor-step} over accepted trials (rejected ones do not
change $(A,y)$), using $B_t\le A_m$ for the model terms, $\nrm{d_t}_2\le\delta_N$
on accepted trials and \eqref{eq:app-path},
\[
 A_m[f(y)-\ell]\le NA_m\eta_{\mathrm{mod}}+3^{k+1}c_{k,\nu} M(k+1)^kN^{1-\nu}\Bigl(\frac{4\PN}N\Bigr)^{s-1}\frac{3\PN}2 ,
\]
and dividing by $A_m$ gives, with $4^{s-1}\le4^k$ and $2^{k+1}\cdot\frac32\cdot3^{k+1}\le3^{k+2}2^k$,
\begin{equation}\label{eq:tensor-sum}
 f(y)-\ell\le \Xi_kc_{k,\nu} M\frac{\PN^s}{N^{2s-1}}+N\eta_{\mathrm{mod}}.
\end{equation}
Define, with $\Dnu$ from \cref{app:holder},
\begin{equation}\label{eq:tensor-eN}
 e_N=2\Xi_kc_{k,\nu} M\frac{\PN^k}{N^{2k-1}}\Dnu\Bigl(\frac{\PN}{N^2}\Bigr),\qquad
 \eta_{\mathrm{mod}}=\frac{e_N}{2N},
\end{equation}
so that $2\Xi_kc_{k,\nu} M\PN^s/N^\sigma\le e_N\le8\Xi_kc_{k,\nu} M\PN^s/N^\sigma$; both terms of
\eqref{eq:tensor-sum} are at most $e_N/2$, which proves \eqref{eq:tensor-level}.
Each trial makes two queries, so the level method makes at most $2N$.

\emph{Bracket from one jet.} Query the jet at $x_0\in K$ and put
$q_-(u)=T_k(x_0;u)-\frac{c_{k,\nu} M}s\nrm{u-x_0}_2^s\le f(u)$ on $K$
(\cref{lem:taylor-remainder}). Let $E=\frac{c_{k,\nu} M}s(2R)^k\Dnu(2R)$, so
$\frac{c_{k,\nu} M}s(2R)^s\le E\le\frac{4c_{k,\nu} M}s(2R)^s$ (for $K=B_1^n(R)$, $x_0=0$, use $R$
in place of $2R$). By a finite grid search over rational convex combinations
of the vertices (with the known Lipschitz constant of $q_-$ on $K$, obtained
from the coefficients of $T_k(x_0;\cdot)$) one finds $\underline q\le\min_Kq_-$ and
$v\in K$ with $q_-(v)\le\underline q+E$. Then $\underline q\le\min_Kf\le f(v)\le q_-(v)+\frac{2c_{k,\nu} M}s\nrm{v-x_0}_2^s\le\underline q+3E$,
so the initial width is $W_0=3E\le\frac{12c_{k,\nu} M}s(2R)^s$, without any query at $v$.

\emph{Bracket search.} We proceed as in \cref{app:holder} with $N(w)=\min\{2^j:e_{2^j}\le w/4\}$;
every phase uses at most $2N(w)$ queries. All $N_j>1$ because
$e_1\ge2\Xi_kc_{k,\nu} M(4R)^s>W_0/4$. Since $N_j/2$ violates the rule,
$w_j/4<e_{N_j/2}\le8\Xi_kc_{k,\nu} M\,2^\sigma P_{N_j}^s/N_j^\sigma$, while
$N_{m-1}^\sigma\ge8\Xi_kc_{k,\nu} MP_{N_{m-1}}^s/w_{m-1}$; hence
$N_j<2\cdot4^{1/\sigma}(3/4)^{(m-1-j)/\sigma}N_{m-1}$ and
\[
 N_{\rm total}\le1+2\sum_jN_j\le1+\frac{4\cdot4^{1/\sigma}}{1-(3/4)^{1/\sigma}}N(\varepsilon)\le1+32(2k+1)N(\varepsilon),
\]
using $\sigma\ge3$, $4^{1/\sigma}<2$ and $(3/4)^{1/\sigma}\le1-1/(4\sigma)$ (Bernoulli's
inequality), so that the fraction is at most $8\cdot4\sigma=32(2s-1)\le32(2k+1)$. The
order of $N(\varepsilon)$ follows from $e_N\le8\Xi_kc_{k,\nu} M\PN^s/N^\sigma$ as in
\cref{app:holder}, with $Q=MR^s/\varepsilon$ bounded below by a constant
depending on $k$ in the active branch; the numerical constants $7$ and $16$ used
there for $s\le2$ are replaced by constants depending on $k$.

\emph{Fixed horizon.} For $T\ge32(2k+1)+1$ let $Q_T$ be the largest power of two
with $32(2k+1)Q_T\le T-1$ and $\varepsilon_T=4e_{Q_T}$. If $\varepsilon_T\ge W_0$
the method returns the initial $v$, and otherwise it runs the bracket search
with accuracy $\varepsilon_T$. The count is at most $T$ and the error at most
$\min\{W_0,4e_{Q_T}\}=O_k(MR^s[\log(2T)\log(2n)]^{s/2}/T^{2s-1})$ since
$Q_T>(T-1)/(64(2k+1))$ and $\mathsf b_{Q_T}\le\mathsf b_T$.
\end{proof}

\subsection{Gradient oracles: proximal deep cuts}
\label{app:gradient-only}

Suppose now that a query at $x\in K$ returns $\nabla f(x)$ only, for a convex
$f$ with $L$-Lipschitz gradient on $\R^n$. Level cuts are unavailable, and deep
cuts $\{u:\ip{\nabla f(x)}{x-u}\ge\delta\}$ alone give the nonsmooth rate. The
method \SPP{} (\cref{alg:spprox}) cuts deep at approximate \emph{proximal
points} $w$ of $f$, taken with respect to points $x$ of the segment
$[y,z_t]\subseteq K$ between the incumbent $y$ and the current approximate
Steiner point $z_t$. The analysis uses the Moreau envelope
\begin{equation}\label{eq:envelope}
\begin{aligned}
 E(x)&=\min_{u\in K}f_x(u),&\qquad p(x)&=\argmin_{u\in K}f_x(u),\\
 f_x(u)&=f(u)+\frac L2\nrm{u-x}_2^2,& \nabla E(x)&=L\bigl(x-p(x)\bigr).
\end{aligned}
\end{equation}
The envelope is never evaluated. The method calls two gradient-only subroutines
(\cref{lem:prox,lem:search} below). \sub{Prox}$(x,k)$ runs $k$ projected-gradient steps
with step $1/(2L)$ on the strongly convex function $f(u)+\frac L2\nrm{u-x}_2^2$
starting from $u_0=x$, queries the gradient at the last point $w=u_k$, and
returns $w$ together with $g=\nabla f(w)$, a vector $v$ of the normal cone
$N_K(w)=\{u':\ip{u'}{u-w}\le0\ \forall u\in K\}$ read off the last projection,
and the residual $r=g+L(w-x)+v$. The projected-gradient
map is a contraction with factor $\frac12$, so $\nrm r_2\le6LR\,2^{-k}$ after $k+1$
queries. \sub{Search}$(y,z,\eta)$ bisects the segment $[y,z]\subseteq K$ on the
sign of the approximate directional derivative $\ip{L(x-w)}{z-y}$ of $E$ and
returns $x\in[y,z]$ with the output of \sub{Prox} at $x$, such that
\begin{equation}\label{eq:search-guarantee}
 E(x)\le E(y)+\eta,\qquad \ip{\nabla E(x)}{x-z}\le\eta,\qquad
 \nrm{L(x-w)-\nabla E(x)}_2\le\rho:=\frac{\eta}{4R},
\end{equation}
using at most $(\bar h+1)(k+1)$ gradients, $\bar h=\min\{h\ge0:4LR^2\cdot2^{-h}\le\eta\}$,
$k=\min\{j\ge1:6LR\cdot2^{-j}\le\rho\}$. Function values appear only in the
analysis; neither routine queries them.

\begin{algorithm}[t]
\caption{\SPP: proximal deep cuts at approximate Steiner points (gradient oracle)}
\label{alg:spprox}
\begin{algorithmic}[1]
\Require known polytope $K\subseteq c+B_1^n(R)$; gradient oracle of an
$L$-smooth convex $f$; accuracy $\varepsilon>0$
\State query $g_a=\nabla f(a)$ at a point $a\in K$; $y\gets\proj_K(a-g_a/L)$;
$U_0\gets2LR^2$ \Comment{$a=0$, $U_0=LR^2/2$ if $K=B_1^n(R)$}
\For{$j=1,\dots,J=\lceil\log_2(U_0/\varepsilon)\rceil$}\Comment{stage $j$: $f(y)-\min_Kf\le2\delta$ is known}
 \State $\delta\gets U_02^{-j}$;\ $m\gets\min\{2^i:\ 2^{3i}\ge64LR^2\mathsf b_{10\cdot2^i}\cn/\delta\}$;\ $T\gets10m$;\ $\eta\gets\delta/(16m)$
 \State $C_0\gets K$;\quad $z_0\gets$ a point of $K$ within $\PT{T}/(4T)$ of $\st K$
 \For{$t=0,1,\dots,T-1$}
  \State $(x_t,w_t,g_t)\gets$\sub{Search}$(y,z_t,\eta)$
  \Comment{$x_t\in[y,z_t]$, $w_t\in K$, $g_t=\nabla f(w_t)$}
  \State $C_{t+1}\gets C_t\cap\{u:\ip{g_t}{w_t-u}\ge\delta\}$
  \If{$C_{t+1}=\varnothing$}
   \State $\lambda\gets$ minimizer of the certificate program of \cref{alg:spcut}
   for $(w_i,g_i)_{i\le t}$;\ $y\gets\sum_{i\le t}\lambda_iw_i$;\ \textbf{break}
   \Comment{$f(y)-\min_Kf<\delta$}
  \EndIf
  \State $z_{t+1}\gets$ a point of $C_{t+1}$ within $\PT{T}/(4T)$ of $\st{C_{t+1}}$
  \If{$\nrm{x_t-w_t}_2^2\ge\delta/(Lm)$} $y\gets w_t$ \Comment{accept: $f$ decreased by $\delta/(4m)$} \EndIf
 \EndFor
\EndFor
\State\Return $y$ \Comment{$f(y)-\min_Kf\le\varepsilon$}
\end{algorithmic}
\end{algorithm}

\begin{restatable}[Gradient-only oracle]{proposition}{thmgradonly}
\label{thm:gradient-only}
Let $f$ be convex with $L$-Lipschitz gradient on $\R^n$, $L>0$, and let
$K\subseteq c+B_1^n(R)$ be a known polytope. For every $\varepsilon>0$,
\cref{alg:spprox} makes feasible gradient queries only, never uses a function
value, and returns $y\in K$ with $f(y)-\min_Kf\le\varepsilon$. Every stage ends
with an empty cut within its $T=10m$ trials, and the total number of gradient
queries is at most
\begin{equation}\label{eq:gradient-only-count}
 1+60\,m_JB_J(k_J+1)
 =O\Bigl(1+\Bigl[\frac{LR^2\cn}{\varepsilon}\Bigr]^{1/3}
 \Bigl[\log\Bigl(2+\frac{LR^2\cn}{\varepsilon}\Bigr)\Bigr]^{7/3}\Bigr),
\end{equation}
where $\delta_J=U_02^{-J}$ and $m_J$ are the parameters of the last stage,
$B_J=1+\lceil\log_2(64LR^2m_J/\delta_J)\rceil$ bounds the number of \sub{Prox}
calls of one trial and $k_J+1$, $k_J=\max\{1,\lceil\log_2(384LR^2m_J/\delta_J)\rceil\}$,
the number of gradients of one \sub{Prox} call, so that $B_J(k_J+1)$ bounds
the gradients of one trial. For $\varepsilon\ge U_0$ the point $y$ of the first
line suffices and the count is $1$.
\end{restatable}

Unlike in \SPL, a large approximate proximal step guarantees a decrease of $f$
that the method does not observe, and a small one is compatible with a nonempty
deep cut only if the Steiner point moves far. Compared with \eqref{eq:main-count},
the count \eqref{eq:gradient-only-count} has an extra factor $\log^2$ from the
inner searches. Since \cref{thm:first-order-lower,thm:rand-smooth} apply a
fortiori to gradient oracles, $N^{-3}$ is also the minimax exponent for the
gradient oracle on $B_1^n(R)$ and on $\Delta_n(R)$, and for $f$ that is
$L$-smooth on $\R^n$ the strongly convex reductions of \cref{app:restart} apply to
this oracle as well, with the counts of \cref{thm:gradient-only}.

In the proofs below, $\diam_2(K)\le2R$ because $K\subseteq c+B_1^n(R)$, and we
use the co-coercivity of the gradient,
\begin{equation}\label{eq:cocoercive}
 \nrm{\nabla f(a)-\nabla f(b)}_2^2\le L\ip{\nabla f(a)-\nabla f(b)}{a-b},
 \ \text{ hence }\ 
 \Bigl\|a-b-\frac{\nabla f(a)-\nabla f(b)}L\Bigr\|_2\le\nrm{a-b}_2 ,
\end{equation}
which follows by adding the descent lemma for $f-\ip{\nabla f(b)}{\cdot}$
(minimized at $b$) to the same inequality with $a$ and $b$ exchanged; expanding
the square gives the second inequality.

\paragraph{The envelope.}
Let $E$ and $p$ be as in \eqref{eq:envelope}. The function
$f_x(u)=f(u)+\frac L2\nrm{u-x}_2^2$ is $L$-strongly convex, so $p(x)$ is unique;
$E$ is the Moreau envelope of $f+\iota_K$ ($\iota_K$ the indicator of $K$) with parameter $1/L$, hence convex and
differentiable with $\nabla E(x)=L(x-p(x))$ and $L$-Lipschitz gradient
(for two points $x,x'$ the monotonicity of $\partial(f+\iota_K)$ at $p(x),p(x')$
gives $\ip{q_x-q_{x'}}{x-x'}\ge\nrm{q_x-q_{x'}}_2^2/L$ for $q_x=L(x-p(x))$).
For $x\in K$, $\nrm{\nabla E(x)}_2=L\nrm{x-p(x)}_2\le2LR$.

\begin{lemma}[Approximate proximal point]\label{lem:prox}
Fix $x\in K$ and $k\ge1$, let $u_0=x$ and
$u_{j+1}=\proj_K\bigl(u_j-[\nabla f(u_j)+L(u_j-x)]/(2L)\bigr)$ for $j<k$, query
$\nabla f(u_k)$, and put $w=u_k$, $g=\nabla f(w)$,
\begin{align*}
 v&=2L(u_{k-1}-u_k)-\nabla f(u_{k-1})-L(u_{k-1}-x),\\
 r&=g+L(w-x)+v=\nabla f(u_k)-\nabla f(u_{k-1})-L(u_k-u_{k-1}).
\end{align*}
This uses $k+1$ gradient queries, all in $K$; $v\in N_K(w)$,
$\nrm r_2\le6LR\cdot2^{-k}$, and if $\nrm r_2\le\rho$ then
\begin{equation}\label{eq:prox-accuracy}
 \nrm{w-p(x)}_2\le\rho/L,\qquad
 \nrm{L(x-w)-\nabla E(x)}_2\le\rho,\qquad
 0\le f(w)+\tfrac L2\nrm{w-x}_2^2-E(x)\le\frac{\rho^2}{2L}.
\end{equation}
\end{lemma}

\begin{proof}
The map $u\mapsto\frac12(u-\nabla f(u)/L)+\frac x2$ is a contraction with
factor $\frac12$ by \eqref{eq:cocoercive}, and the projection is nonexpansive,
so the iteration is a contraction with factor $\frac12$; its fixed point is the
unique minimizer $p(x)$ of $f_x$ over $K$ (the fixed-point equation is the
optimality condition of a projected-gradient step with step $1/(2L)$ on $f_x$).
Hence $\nrm{u_j-p(x)}_2\le2^{-j}\nrm{x-p(x)}_2\le2R\,2^{-j}$ and
$\nrm{u_k-u_{k-1}}_2\le6R\,2^{-k}$; \eqref{eq:cocoercive} applied to the pair
$u_k,u_{k-1}$ gives $\nrm r_2\le L\nrm{u_k-u_{k-1}}_2\le6LR\cdot2^{-k}$. The
optimality condition of the last projection says that
$u_{k-1}-[\nabla f(u_{k-1})+L(u_{k-1}-x)]/(2L)-u_k\in N_K(u_k)$, which is
$v/(2L)\in N_K(w)$; the two expressions for $r$ agree by substitution.

For \eqref{eq:prox-accuracy}, note that $\nabla f_x(w)=g+L(w-x)=r-v$, so $L$-strong
convexity of $f_x$ and $\ip v{u-w}\le0$ give, for all $u\in K$,
$f_x(u)\ge f_x(w)+\ip r{u-w}+\frac L2\nrm{u-w}_2^2$. Taking $u=p(x)$ and using
$f_x(w)\ge f_x(p(x))+\frac L2\nrm{w-p(x)}_2^2$ (minimality of $p(x)$ and strong
convexity) yields $L\nrm{w-p(x)}_2^2\le\rho\nrm{w-p(x)}_2$, i.e.\ the first
bound, and then $f_x(w)-E(x)\le\rho\nrm{w-p(x)}_2-\frac L2\nrm{w-p(x)}_2^2\le\rho^2/(2L)$;
the second bound is $L\nrm{w-p(x)}_2\le\rho$.
\end{proof}

\begin{lemma}[Bisection without values]\label{lem:search}
Let $y,z\in K$, $\eta>0$, $\rho=\eta/(4R)$, $k=\min\{j\ge1:6LR\cdot2^{-j}\le\rho\}$
and $\bar h=\min\{h\ge0:4LR^2\cdot2^{-h}\le\eta\}$. Put $d=z-y$ and
$\varphi(t)=E(y+td)$. Starting from $[a,b]=[0,1]$, while $b-a>2^{-\bar h}$: let
$t=(a+b)/2$, run \cref{lem:prox} at $x_t=y+td$ with $k$ steps, and compute
$\widetilde s=\ip{L(x_t-w)}{d}$; if $\widetilde s>\eta/2$ set $b=t$, if
$\widetilde s<-\eta/2$ set $a=t$, otherwise stop and return $t$ with the output
$(w,g,v,r)$ of \cref{lem:prox} at $x_t$. If the loop ends by length, return $t=b$
with the output stored at $b$ (running \cref{lem:prox} once at $z$ if $b=1$). The
returned $x=x_t\in[y,z]$ and $(w,g,v,r)$ satisfy $\nrm r_2\le\rho$,
\eqref{eq:prox-accuracy}, and
\begin{equation}\label{eq:search-proved}
 E(x)\le E(y)+\eta,\qquad\ip{\nabla E(x)}{x-z}\le\eta,\qquad
 \ip{L(x-w)}{x-z}\le\eta/2,
\end{equation}
at a cost of at most $(\bar h+1)(k+1)$ gradient queries.
\end{lemma}

\begin{proof}
The function $\varphi$ is convex and $C^1$ with $\varphi'(t)=\ip{\nabla E(x_t)}d$ and
$|\varphi'|\le2LR\cdot2R=4LR^2$. By \eqref{eq:prox-accuracy},
$|\widetilde s-\varphi'(t)|\le\rho\nrm d_2\le2R\rho=\eta/2$. Hence a positive
test certifies $\varphi'(t)>0$ and a negative one $\varphi'(t)<0$, so
$[a,b]$ always contains a minimizer $t_*$ of $\varphi$ on $[0,1]$, and
$\varphi(a)\le\varphi(0)$ because $\varphi$ is nonincreasing on $[0,t_*]$.
If the loop stops at $t$, then $|\widetilde s|\le\eta/2$, so $|\varphi'(t)|\le\eta$,
and convexity gives $\varphi(t)-\varphi(0)\le t\varphi'(t)\le\eta$ and
$\ip{\nabla E(x_t)}{x_t-z}=(t-1)\varphi'(t)\le\eta$; likewise
$\ip{L(x_t-w)}{x_t-z}=(t-1)\widetilde s\le\eta/2$. If it ends by length,
$\varphi(b)-\varphi(0)\le\varphi(b)-\varphi(t_*)\le4LR^2(b-a)\le4LR^2\cdot2^{-\bar h}\le\eta$;
if $b<1$ the last test at $b$ was positive, so $\varphi'(b)>0$, $\widetilde s(b)>0$
and both inner products with $x_b-z=(b-1)d$ are negative, while for $b=1$
they vanish. The loop performs at most $\bar h$ iterations and at most one further
call is made.
\end{proof}

\begin{lemma}[One stage]\label{lem:stage}
Let $0<\delta\le LR^2$ and $y\in K$ with $f(y)-\min_Kf\le2\delta$. Let $m$ be the
smallest power of two with $m^3\ge64LR^2\mathsf b_{10m}\cn/\delta$, $T=10m$,
$\eta=\delta/(16m)$, $\rho=\eta/(4R)$, and run the inner loop of
\cref{alg:spprox} with these parameters. Then some cut $C_{t+1}$ with
$t\le9m-1$ is empty, the returned mixture $y$ satisfies $f(y)-\min_Kf<\delta$,
and the stage uses at most $T(\bar h+1)(k+1)$ gradient queries, where
$\bar h+1\le B=1+\lceil\log_2(64LR^2m/\delta)\rceil$ and
$k=\max\{1,\lceil\log_2(384LR^2m/\delta)\rceil\}$.
\end{lemma}

\begin{proof}
The values of $\bar h$ and $k$ follow from \cref{lem:search} with
$\eta=\delta/(16m)$ and $\rho=\delta/(64Rm)$ ($4LR^2/\eta=64LR^2m/\delta$ and
$6LR/\rho=384LR^2m/\delta$). Consider a trial $t$ with a nonempty cut, write
$x=x_t$, $w=w_t$, $g=g_t$, $q=x-w$, $z=z_t$, $z^+=z_{t+1}$, $d=z^+-z$.

\emph{The dichotomy.} Since $z^+\in C_{t+1}$,
$\delta\le\ip g{w-z^+}$. Substituting $g=Lq+r-v$ and using
$\ip v{z^+-w}\le0$ ($v\in N_K(w)$, $z^+\in K$) and $\ip r{w-z^+}\le2R\rho$,
\[
 \delta\le L\ip q{w-z^+}+2R\rho
 =L\ip q{x-z}-L\nrm q_2^2-L\ip qd+2R\rho .
\]
By \eqref{eq:search-proved} and \eqref{eq:prox-accuracy},
$L\ip q{x-z}\le\ip{\nabla E(x)}{x-z}+\rho\nrm{x-z}_2\le\eta+2R\rho$. Hence
$\delta\le L\nrm q_2\nrm d_2-L\nrm q_2^2+\eta+4R\rho$, and $\eta+4R\rho=2\eta=\delta/(8m)\le\delta/8$
gives
\begin{equation}\label{eq:dichotomy}
 \tfrac34\delta+L\nrm q_2^2\ \le\ L\nrm q_2\nrm d_2 .
\end{equation}
In particular $q\ne0$.

\emph{Large steps.} By \eqref{eq:prox-accuracy}, \eqref{eq:search-proved} and
$E(y)\le f(y)$ (take $u=y$ in \eqref{eq:envelope}),
$f(w)\le E(x)+\rho^2/(2L)-\frac L2\nrm q_2^2\le f(y)-\frac L2\nrm q_2^2+\eta+\rho^2/(2L)$.
If $\nrm q_2^2\ge\delta/(Lm)$, then, since $\rho^2/(2L)=\delta^2/(8192LR^2m^2)\le\delta/(8192m)$
by $\delta\le LR^2$,
\[
 f(w)\le f(y)-\frac{\delta}{2m}+\frac{\delta}{16m}+\frac{\delta}{8192m}<f(y)-\frac{\delta}{4m}.
\]
The algorithm replaces $y$ by $w$ if and only if $\nrm q_2^2\ge\delta/(Lm)$, so
each accepted trial decreases $f(y)$ by more than $\delta/(4m)$. Since
$f(y)\ge\min_Kf$ and $f(y)-\min_Kf\le2\delta$ initially, there are at most $8m$
accepted trials.

\emph{Small steps.} If $\nrm q_2^2<\delta/(Lm)$, \eqref{eq:dichotomy} gives
$\nrm d_2\ge\frac{3\delta}{4L\nrm q_2}+\nrm q_2>\frac34\sqrt{m\delta/L}$. If
$s$ trials have a nonempty cut, the sets $C_0\supseteq\dots\supseteq C_s$ are
nonempty and nested in $c+B_1^n(R)$ with $s\le T$, so \cref{lem:steiner-path}
and the selector accuracy $\PT T/(4T)$ give
$\sum_{t<s}\nrm{z_{t+1}-z_t}_2\le\PT T+2T\cdot\PT T/(4T)=\frac32\PT T$.
The number of small nonempty trials is therefore less than
$\frac32\PT T\big/\bigl(\frac34\sqrt{m\delta/L}\bigr)=2\PT T\sqrt{L/(m\delta)}\le m$,
where the last inequality is $4\PT T^2L\le m^3\delta$, i.e.\
$64LR^2\mathsf b_{T}\cn\le m^3\delta$ with $\mathsf b_T=\mathsf b_{10m}$,
the defining property of $m$.

\emph{Termination and certificate.} At most $8m+(m-1)$ trials have a nonempty
cut, so the cut of some trial $t\le9m-1<T$ is empty. Then
$\max_{u\in K}\min_{i\le t}\ip{g_i}{w_i-u}<\delta$, and the minimax argument
of \cref{thm:spcut} gives $\lambda\in\Delta_{t+1}$ with
$\max_{u\in K}\sum_i\lambda_i\ip{g_i}{w_i-u}<\delta$, computed by the linear
program of \cref{alg:spprox}; for $\bar w=\sum_i\lambda_iw_i\in K$ and any $u\in K$,
$f(\bar w)-f(u)\le\sum_i\lambda_i[f(w_i)-f(u)]\le\sum_i\lambda_i\ip{g_i}{w_i-u}<\delta$.
Each trial makes at most $(\bar h+1)(k+1)$ queries by \cref{lem:search}, and the
mixture needs no query.
\end{proof}

\begin{proof}[Proof of \cref{thm:gradient-only}]
\emph{Initialization.} The point $y_0=\proj_K(a-g_a/L)$ minimizes the model
$\psi(u)=f(a)+\ip{g_a}{u-a}+\frac L2\nrm{u-a}_2^2$ over $K$, and for a minimizer $\opt$,
$f(y_0)\le\psi(y_0)\le\psi(\opt)\le f(\opt)+\frac L2\nrm{\opt-a}_2^2$ by smoothness at
$a$ and convexity; hence $f(y_0)-\min_Kf\le\frac L2\diam_2(K)^2\le2LR^2=U_0$,
and for $K=B_1^n(R)$, $a=0$, $f(y_0)-\min_Kf\le\frac L2\nrm\opt_2^2\le LR^2/2$.
No value is computed.

\emph{Stages.} Let $\delta_j=U_02^{-j}$ and $J=\lceil\log_2(U_0/\varepsilon)\rceil$,
so that $\delta_J\le\varepsilon<2\delta_J$ and $\delta_j\le U_0/2\le LR^2$ for
$j\ge1$. Stage $1$ receives $y_0$ with $f(y_0)-\min_Kf\le U_0=2\delta_1$; by
\cref{lem:stage} stage $j$ returns a point with error below $\delta_j=2\delta_{j+1}$,
which is the hypothesis of stage $j+1$. The output of stage $J$ has error below
$\delta_J\le\varepsilon$. All queries are at points of $K$.

\emph{Count.} Stage $j$ uses at most $10m_jB_j(k_j+1)$ queries, and $B_j$,
$k_j$ are nondecreasing in $j$ because $m_j/\delta_j$ is. Let
$\phi(m)=m^3/(64LR^2\mathsf b_{10m}\cn)$ for powers of two $m$; $m_j$ is the
smallest $m$ with $\phi(m)\ge1/\delta_j$, and since
$\phi(2m)/\phi(m)=8\mathsf b_{10m}/\mathsf b_{20m}=8\mathsf b_{10m}/(\mathsf b_{10m}+1)\in(4,8)$,
$\phi$ is increasing and $m_j\ge2$ (as $\phi(1)<1/(LR^2)\le1/\delta_j$). A
given $m$ is used if and only if $1/\phi(m)\le\delta_j<1/\phi(m/2)$; the
endpoints of this interval have ratio less than $8$, so it contains at most
three of the values $\delta_j$, which halve from stage to stage. Consequently
$\sum_jm_j\le3(m_J+m_J/2+\dots)\le6m_J$ and the total count is at most
$1+10B_J(k_J+1)\sum_jm_j\le1+60m_JB_J(k_J+1)$.

\emph{Order.} Put $Q=LR^2\cn/\varepsilon$. From the minimality of $m_J$ and
$\delta_J>\varepsilon/2$, $(m_J/2)^3<64LR^2\mathsf b_{5m_J}\cn/\delta_J<128Q\,\mathsf b_{5m_J}$,
and $\mathsf b_{5m_J}\le5+\log_2m_J$; as $m_J^3<1024Q(5+\log_2m_J)$ forces
$m_J\le C(1+Q)$, we get $m_J=O(1+[Q\log(2+Q)]^{1/3})$. Also
$LR^2m_J/\delta_J\le2LR^2m_J/\varepsilon\le2Qm_J$ gives
$B_J,k_J=O(\log(2+Q))$. Multiplying proves \eqref{eq:gradient-only-count}.
\end{proof}

\subsection{A center-of-gravity method for small dimension}
\label{app:cog}

When $n$ is small compared with the count \eqref{eq:gradient-only-count}, the
classical center-of-gravity method with deep cuts and the same mixture
certificate needs fewer queries, and it also uses gradients only. We state it
for an operator $F$, so that it also applies to variational inequalities
(\cref{app:vi}); for $F=\nabla f$ the certificate bounds $f(\bar w)-\min_Kf$.

\begin{restatable}[Center-of-gravity branch]{proposition}{propcog}
\label{prop:cog}
Let $F:B_1^n(R)\to\R^n$ be $L$-Lipschitz in $\ell_2$ and $\varepsilon>0$.
\CGC{} (\cref{alg:cgcut}) queries $F(0)$, then $F(y)$ at a
vertex $y$ minimizing $\ip{F(0)}{u}$ over the ball, and then $F$ at the
centroid of a known polytope that it cuts by the half-space
$\{u:\ip{F(w)}{w-u}\ge\varepsilon/2\}$, until the deeper cuts
$\{u:\ip{F(w_i)}{w_i-u}\ge\varepsilon\}$ have empty intersection with the
ball. It stops after at most
\begin{equation}\label{eq:cog-count}
 3+\Bigl\lfloor\frac{n\log(2+8LR^2/\varepsilon)}{-\log(1-1/e)}\Bigr\rfloor
 =O\bigl(1+n\log(2+LR^2/\varepsilon)\bigr)
\end{equation}
calls and returns weights $\lambda\in\Delta_{t+1}$ on the queried points
$w_0=y,w_1,\dots,w_t$ (the centroids) and $\bar w=\sum_i\lambda_iw_i\in B_1^n(R)$
with $\max_{u\in B_1^n(R)}\sum_i\lambda_i\ip{F(w_i)}{w_i-u}\le\varepsilon$. If
$F=\nabla f$ for a convex $f$, then $f(\bar w)-\min_{B_1^n(R)}f\le\varepsilon$; if
$F$ is monotone, the weak gap of $\bar w$ is at most $\varepsilon$. No bound on
$\nrm{F(0)}$ is needed, and $L>0$ is assumed (for $L=0$ the operator is
constant and the point $y$ above is exact).
\end{restatable}

\begin{algorithm}[t]
\caption{\CGC: centroid cuts with a global certificate (operator or gradient oracle)}
\label{alg:cgcut}
\begin{algorithmic}[1]
\Require $K=B_1^n(R)$; oracle $F$ that is $L$-Lipschitz on $K$; accuracy $\varepsilon>0$
\State query $b=F(0)$; \textbf{if} $b=0$ \textbf{return} $\bar w=0$, $\lambda=(1)$
\State $y\gets-R\,\sign(b_i)e_i$ for an index $i$ with $|b_i|=\nrm b_\infty$;\quad query $h=F(y)$;\quad $w_0\gets y$
\State $\widehat C\gets\{u\in K:\ip b{u-y}\le2LR^2\}$;\quad
$a_w(u):=\ip{F(w)}{w-u}$ for stored pairs
\For{$t=0,1,2,\dots$}
 \State $C'_t\gets K\cap\bigcap_{i\le t}\{u:a_{w_i}(u)\ge\varepsilon\}$;\quad
 $C_t\gets \widehat C\cap\bigcap_{i\le t}\{u:a_{w_i}(u)\ge\varepsilon/2\}$
 \If{$C'_t=\varnothing$}
  \State solve $\min_{\lambda\in\Delta_{t+1}}\{\sum_{i\le t}\lambda_i\ip{F(w_i)}{w_i}+R\nrm{\sum_{i\le t}\lambda_iF(w_i)}_\infty\}$;
  \Return $\bar w=\sum_{i\le t}\lambda_iw_i$, $\lambda$
 \EndIf
 \State $w_{t+1}\gets$ centroid of $C_t$;\quad query $F(w_{t+1})$
\EndFor
\end{algorithmic}
\end{algorithm}

\begin{proof}[Proof of \cref{prop:cog}]
\emph{Two initial queries.} If $b=0$ the certificate of the single point $0$ is
$\max_u\ip{0}{0-u}=0$. Otherwise $y$ minimizes $\ip bu$ over $K$, so
$\ip b{u-y}\ge0$ on $K$, and $\nrm{h-b}_2\le L\nrm y_2\le LR$. For $u\in K$,
\begin{equation}\label{eq:cog-initial}
 a_y(u)=\ip h{y-u}=-\ip b{u-y}+\ip{h-b}{y-u}\le-\ip b{u-y}+2LR^2 ,
\end{equation}
so $a_y(u)\ge\varepsilon$ implies $u\in \widehat C$, and hence $C'_t\subseteq \widehat C$ for every $t$.
Since $2LR^2>0$, the known polytope $\widehat C$ is full-dimensional: it contains a
neighborhood of the interior point $(1-s)y$ of $K$ for small $s>0$, where
$\ip b{(1-s)y-y}=sR\nrm b_\infty<2LR^2$. For $w,v\in \widehat C$,
\begin{equation}\label{eq:cog-lower}
 a_w(v)=\ip b{w-v}+\ip{F(w)-b}{w-v}\ge-2LR^2-LR\cdot2R=-4LR^2=:-U ,
\end{equation}
because $\ip b{w-y}\ge0$, $\ip b{v-y}\le2LR^2$ and $\nrm{F(w)-b}_2\le L\nrm w_2\le LR$.

\emph{Volume.} Let $\theta=\varepsilon/(2(U+\varepsilon))$. If $C'_t\ne\varnothing$,
pick $u\in C'_t\subseteq \widehat C$; for $v\in \widehat C$ and $i\le t$ the affine function
$a_{w_i}$ satisfies $a_{w_i}(u)\ge\varepsilon$ and $a_{w_i}(v)\ge-U$, so
$a_{w_i}((1-\theta)u+\theta v)\ge(1-\theta)\varepsilon-\theta U=\varepsilon/2$. Hence
$(1-\theta)u+\theta \widehat C\subseteq C_t$ and $\vol(C_t)\ge\theta^n\vol(\widehat C)>0$;
in particular the centroid $w_{t+1}$ is defined. If $F(w_{t+1})=0$ the new
constraint $a_{w_{t+1}}\ge\varepsilon/2$ cannot be satisfied, so
$C_{t+1}=C'_{t+1}=\varnothing$ and the method stops. Otherwise the half-space $\{u:a_{w_{t+1}}(u)\ge\varepsilon/2\}$
is contained in the half-space $\{u:\ip{F(w_{t+1})}{w_{t+1}-u}\ge0\}$ through the
centroid of $C_t$, and Gr\"unbaum's inequality \citep{grunbaum1960} gives
$\vol(C_{t+1})\le(1-1/e)\vol(C_t)$; also $C_0\subseteq \widehat C$. Let
$T=1+\lfloor n\log(2+8LR^2/\varepsilon)/(-\log(1-1/e))\rfloor$. If $C'_T\ne\varnothing$,
then $\theta^n\vol(\widehat C)\le\vol(C_T)\le(1-1/e)^T\vol(\widehat C)$, i.e.\
$T\le n\log(1/\theta)/(-\log(1-1/e))$ with $1/\theta=2+8LR^2/\varepsilon$, which
contradicts the choice of $T$. So the method stops with $C'_t=\varnothing$ for
some $t\le T$, after the two initial queries and at most $T$ centroid queries;
this is \eqref{eq:cog-count}.

\emph{Certificate.} If $C'_t=\varnothing$ then $\max_{u\in K}\min_{i\le t}a_{w_i}(u)<\varepsilon$
by compactness, and the minimax theorem, as in \cref{thm:spcut}, gives
$\lambda\in\Delta_{t+1}$ with $\max_{u\in K}\sum_i\lambda_ia_{w_i}(u)<\varepsilon$;
the maximum equals $\sum_i\lambda_i\ip{F(w_i)}{w_i}+R\nrm{\sum_i\lambda_iF(w_i)}_\infty$
on the ball, a linear program. Every $w_i$ is a point of $K$ at which $F$ was
queried ($w_0=y$ with $h=F(y)$; $F(0)$ is not used in the mixture), so
$\bar w\in K$. For $F=\nabla f$, convexity gives
$f(\bar w)-f(u)\le\sum_i\lambda_i\ip{\nabla f(w_i)}{w_i-u}\le\varepsilon$ for all
$u\in K$; for monotone $F$, $\ip{F(u)}{\bar w-u}=\sum_i\lambda_i\ip{F(u)}{w_i-u}\le\sum_i\lambda_i\ip{F(w_i)}{w_i-u}\le\varepsilon$
(\cref{lem:mixture-certificate}).
\end{proof}

The proof uses only the diameter, the linear minimization over $K$, the
emptiness tests and the centroids of known polytopes. On a general known
polytope $K$ of dimension $d$ the same argument therefore gives
$O(1+d\log(2+LR^2/\varepsilon))$ calls, with the centroid taken in the affine
hull of $K$, an arbitrary initial point $a\in K$ in place of $0$, and the
absolute constants in the diameter bounds doubled.

Choosing between \cref{thm:gradient-only} and \cref{prop:cog} by their
guaranteed query counts gives, for the gradient oracle on $B_1^n(R)$,
$O\bigl(1+\min\{[\frac{LR^2\cn}\varepsilon]^{1/3}\log^{7/3}(2+\frac{LR^2\cn}{\varepsilon}),\,
n\log(2+\frac{LR^2}\varepsilon)\}\bigr)$ queries.

\section{Strong convexity}
\label{app:strong}

This section gives the strongly convex versions of the upper bounds, first by
restarts with the condition number $\kappa_1=L/\mu_1$
(\cref{app:restart}), and then, for smooth functions on the ball, a sharper
analysis in terms of the slack $\Lambda=\kappa_1-n$ that determines the
deterministic complexity (\cref{app:slack,app:strong-map}). We begin with the
constraint that relates the two constants.

\subsection{Norm compatibility}
\label{app:compatibility}

Euclidean smoothness and $\ell_p$ strong convexity cannot be combined
arbitrarily: comparing the two quadratic bounds at $x$ and $x+tv$ gives the
following inequality.

\begin{restatable}[Compatibility inequality]{proposition}{propcompatibility}
\label{prop:compatibility}
If $f$ is $L$-smooth in $\ell_2$ and $\mu_p$-strongly convex in $\ell_p$ on a
nonempty open convex subset of $\R^n$, then $L\ge\mu_pn^{\beta_p}$, with
equality for $f(x)=\frac L2\nrm x_2^2$.
\end{restatable}

\begin{proof}
For $x$ in the set, $v\in\R^n$ and small $t>0$, the upper and the lower
quadratic model at $x$, evaluated at $x+tv$, give
$\mu_p\nrm{v}_p^2\le L\nrm{v}_2^2$ after the affine terms cancel. The largest
value of $\nrm v_p^2/\nrm v_2^2$ is $n^{2/p-1}=n^{\beta_p}$, attained at vectors
with equal absolute coordinates. For $f=\frac L2\nrm x_2^2$ the inequality
$\nrm v_2^2\ge n^{-\beta_p}\nrm v_p^2$ shows that $\mu_p=Ln^{-\beta_p}$ is
admissible.
\end{proof}

In particular $\kappa_1=L/\mu_1\ge n$, and we call $\Lambda=\kappa_1-n\ge0$ the
\emph{slack}.

\subsection{Restarts}
\label{app:restart}

Suppose that $f$ is in addition $\mu_1$-strongly convex with respect to
$\nrm\cdot_1$ on $K$ (a boundary minimizer is allowed). Then
$f(x)-f(\opt)\ge\frac{\mu_1}2\nrm{x-\opt}_1^2$ for all $x\in K$, so an
$\varepsilon$-solution $x_j$ localizes $\opt$ in the known polytope
$K_{j+1}=K\cap(x_j+B_1^n(r_{j+1}))$ with $r_{j+1}=\sqrt{2\varepsilon/\mu_1}$.
\SPR{} solves the problem on $K_0=K$ to accuracy $\mu_1R^2/8$ and then, with
$r_j=R2^{-j}$, on $K_j$ to accuracy $\mu_1r_j^2/8$, each time by
\cref{alg:spaccel} restricted to $K_j$. The radius cancels in the required
horizon, which is therefore the same in every phase.

\begin{restatable}[Restarts]{proposition}{thmrestart}
\label{thm:restart}
Let $f$ be convex, $L$-smooth and $\mu_1$-strongly convex in $\ell_1$ on
$K\subseteq c+B_1^n(R)$, $\kappa_1=L/\mu_1$, and let $N_\star$ be the smallest
power of two with $N_\star^3\ge24576\,\kappa_1\cn(1+\log_2N_\star)$. For every
$\varepsilon>0$, \SPR{} returns an $\varepsilon$-solution after at most
$(J+1)(1+24N_\star)$ first-order queries, where $J=\ceilp{\log_4(\mu_1R^2/(8\varepsilon))}$; that is,
\begin{equation}\label{eq:restart-count}
 O\Bigl(\bigl[1+\{\kappa_1\log(2n)\log(2+\kappa_1\log(2n))\}^{1/3}\bigr]
 \bigl[1+\log_+\tfrac{\mu_1R^2}{\varepsilon}\bigr]\Bigr).
\end{equation}
If $f$ is $G$-Lipschitz and $\mu_1$-strongly convex in $\ell_1$ on $K$ (subgradient
oracle without values), the same restarts with \cref{alg:spcut} return an
$\varepsilon$-solution after
$O\bigl([1+\log_+\frac{\mu_1R^2}{\varepsilon}]+\frac{G\sqrt{\log(2n)}}{\sqrt{\mu_1\varepsilon}}
\sqrt{\log(2+\frac{G\sqrt{\log(2n)}}{\sqrt{\mu_1\varepsilon}})}\bigr)$ queries.
If $\nabla f$ is $\nu$-H\"older with constant $M=L_{1,\nu}$, $0<\nu<1$, the
restarts with the level method of \cref{thm:holder} need
$\tO_\nu\bigl(1+(M/(\mu_1R^{1-\nu}))^{1/(1+2\nu)}+\log_+\frac{\mu_1R^2}{\varepsilon}
+M^{1/(1+2\nu)}\mu_1^{-(1+\nu)/(2+4\nu)}\varepsilon^{-(1-\nu)/(2+4\nu)}\bigr)$
queries.
\end{restatable}

The Euclidean bound is $O(\sqrt{\kappa_2}\log(1/\varepsilon))$ with
$\kappa_2=L/\mu_2$, where $\mu_2\ge\mu_1$ is the Euclidean modulus, so the new
bound is smaller when $\mu_1$ is comparable to $\mu_2$ and $\kappa_1$ is large;
it requires no regularity of the Hessian. For deterministic methods and
$n\ge\Lambda^{1/3}$, \cref{thm:strong-map} below shows that the power
$\kappa_1^{1/3}$ cannot be improved when $\Lambda$ is comparable to $\kappa_1$.
The same restarts apply to the method of higher order of \cref{thm:tensor}.

\begin{proof}[Proof of \cref{thm:restart}]
\emph{Localization.} Let $\opt\in\argmin_Kf$. Strong convexity along the
segment from $\opt$ to $x\in K$ and minimality of $\opt$ give, for $0<t<1$,
$f(\opt+t(x-\opt))\le(1-t)f(\opt)+tf(x)-\frac{\mu_1}2t(1-t)\nrm{x-\opt}_1^2$ and
$f(\opt+t(x-\opt))\ge f(\opt)$, hence
$f(x)-f(\opt)\ge\frac{\mu_1}2(1-t)\nrm{x-\opt}_1^2$; letting $t\downarrow0$ gives
$f(x)-f(\opt)\ge\frac{\mu_1}2\nrm{x-\opt}_1^2$, so the minimizer is unique (possibly
on the boundary of $K$). Consequently, if $f(x_j)-f(\opt)\le\mu_1r_j^2/8$ then
$\nrm{x_j-\opt}_1\le r_j/2=r_{j+1}$, and $K_{j+1}=K\cap(x_j+B_1^n(r_{j+1}))$ is a
known polytope with $\opt\in K_{j+1}\subseteq x_j+B_1^n(r_{j+1})$.

\emph{Smooth case.} Apply \cref{alg:spaccel} on $K_j\subseteq x_j'+B_1^n(r_j)$
($x_j'=x_{j-1}$ for $j\ge1$, $x_0'=c$, $r_0=R$) to accuracy $\delta_j=\mu_1r_j^2/8$.
By \cref{thm:main} for general polytopes (\cref{app:bracket}), this uses at
most $1+24N_j$ queries, where $N_j$ is the smallest power of two with
$768Lr_j^2\mathsf b_{N}\cn/N^3\le\delta_j/4=\mu_1r_j^2/32$, i.e.\
$N^3\ge24576\kappa_1\mathsf b_N\cn$; the radius cancels, so $N_j=N_\star$ for all
$j$ (for a power of two, $\mathsf b_N=1+\log_2N$). After $J+1$ phases, where $J$
is the smallest index with $\mu_1R^24^{-J}/8\le\varepsilon$, the output $x_J$
satisfies $f(x_J)-\min_Kf\le\delta_J\le\varepsilon$; the total is at most
$(J+1)(1+24N_\star)$. Finally $N_\star=O(1+[\kappa_1\cn\log(2+\kappa_1\cn)]^{1/3})$,
because for a large absolute constant $C$ the value
$N=C\max\{1,[\kappa_1\cn\log(2+\kappa_1\cn)]^{1/3}\}$ satisfies the defining
inequality (as $1+\log_2N=O(\log(2+\kappa_1\cn))$ for such $N$), and rounding up
to a power of two at most doubles it. This gives \eqref{eq:restart-count}.

\emph{Nonsmooth case.} Apply \cref{alg:spcut} on $K_j$ with depth $\delta_j$;
by \cref{thm:spcut} it stops after at most $Q(r_j,\delta_j)$ queries, where
$Q(r,\delta)$ is the smallest $m\ge1$ with $16Gr\sqrt{\mathsf b_m\cn}/m\le\delta$,
and returns a point with error at most $\delta_j$ on $K_j$. Since
$\mathsf b_m\cn\le4(1+\log m)(1+\log(2n))$, it suffices that
$m/\sqrt{1+\log m}\ge A_j/2$, where
$A_j=64Gr_j\sqrt{1+\log(2n)}/\delta_j=2^jA_0$ and
$A_0=512G\sqrt{1+\log(2n)}/(\mu_1R)$; the argument in the proof of
\cref{lem:spcut-count} with $c=\max\{16,A_j/2\}$ gives
$Q(r_j,\delta_j)=O(\max\{1,A_j\}\sqrt{1+\log(2+A_j)})$. These bounds grow
geometrically in $j$, so
$\sum_{j\le J}Q(r_j,\delta_j)=O(J+1+A_J\sqrt{1+\log(2+A_J)})$, and since
$\varepsilon/4<\delta_J\le\varepsilon$ we have $r_J\asymp\sqrt{\varepsilon/\mu_1}$
and $A_J\asymp G\sqrt{1+\log(2n)}/\sqrt{\mu_1\varepsilon}$, which gives the
stated count. If $\varepsilon\ge\mu_1R^2/8$, a single run of \cref{alg:spcut} with depth
$\varepsilon$ suffices.

\emph{H\"older case.} With the level method of \cref{thm:holder}, $M=L_{1,\nu}$,
$s=1+\nu$ and $\sigma=1+2\nu$, phase $j$ uses
$\tO(1+(Mr_j^s/\delta_j)^{1/\sigma})=\tO(1+(M/(\mu_1r_j^{2-s}))^{1/\sigma})$
queries. For $\nu<1$ this grows geometrically with the factor $2^{(2-s)/\sigma}>1$, so the
sum is dominated by the last phase, where $r_J\asymp\sqrt{\varepsilon/\mu_1}$, and the
total is $\tO_\nu\bigl(1+(M/(\mu_1R^{2-s}))^{1/\sigma}+\log_+(\mu_1R^2/\varepsilon)+M^{1/\sigma}\mu_1^{-s/(2\sigma)}\varepsilon^{-(2-s)/(2\sigma)}\bigr)$.
For $\nu=1$ all phases have the same bound and one recovers the smooth case.
\end{proof}

\begin{restatable}[Restarts with higher-order oracles]{corollary}{cortensorobjectiverestart}
\label{cor:tensor-objective-restart}
Let $k\ge2$, $0\le\nu\le1$, $s=k+\nu$ and $\sigma=2s-1$. Let $f$ satisfy the
assumptions of \cref{thm:tensor} and be $\mu_1$-strongly convex in $\ell_1$ on
$K\subseteq c+B_1^n(R)$, and put $\bar\kappa=L_{k,\nu}R^{s-2}/\mu_1$ and
$\Pi_0=[\cn\log(2+\bar\kappa\cn^{s/2})]^{s/(2\sigma)}$. Restarting the
bracket search of \cref{thm:tensor} on $K_0=K$ and
$K_{j+1}=K\cap(x_j+B_1^n(R2^{-j-1}))$, with accuracy $\mu_1R^24^{-j}/8$ in
phase $j=0,\dots,J$, $J=\ceilp{\log_4(\mu_1R^2/(8\varepsilon))}$, returns
$x_J\in K$ with $f(x_J)-\min_Kf\le\varepsilon$ after
\[
\begin{aligned}
 &O_k\Bigl(\Bigl[1+\log_+\frac{\mu_1R^2}{\varepsilon}\Bigr]\bigl[1+\bar\kappa^{1/\sigma}\Pi_0\bigr]\Bigr)
 &&\text{if }s=2,\\
 &O_k\Bigl(1+\log_+\frac{\mu_1R^2}{\varepsilon}+\frac{\bar\kappa^{1/\sigma}\Pi_0}{1-2^{-(s-2)/\sigma}}\Bigr)
 &&\text{if }s>2
\end{aligned}
\]
order-$k$ queries, where $O_k$ hides a factor that depends only on $k$.
\end{restatable}

For a Lipschitz Hessian ($s=3$) the count is
$\tO(1+\log_+(\mu_1R^2/\varepsilon)+(L_{2,1}R/\mu_1)^{1/5})$. In Euclidean
geometry the condition-number term of second-order methods is of order
$(L_{2,1}R/\mu_2)^{2/7}$ \citep{arjevani2019}, and near the solution Newton steps
reduce the dependence on $\varepsilon$ to $\log\log(1/\varepsilon)$; the
corollary improves the first term when $\mu_1$ is comparable to $\mu_2$, but not
the second.

\begin{proof}[Proof of \cref{cor:tensor-objective-restart}]
Let $\opt$ be the minimizer of $f$ on $K$; strong convexity gives
$f(x)-f(\opt)\ge\frac{\mu_1}2\nrm{x-\opt}_1^2$ on $K$. Let
$J=\ceilp{\log_4(\mu_1R^2/(8\varepsilon))}$, $\rho_j=R2^{-j}$ and
$\delta_j=\mu_1\rho_j^2/8$. If $K_j$ contains $\opt$, the bracket search of
\cref{thm:tensor} on $K_j$, started from one jet on $K_j$ as in the proof of \cref{thm:tensor},
returns $x_j\in K_j$ with $f(x_j)-f(\opt)\le\delta_j$, hence
$\nrm{x_j-\opt}_1^2\le2\delta_j/\mu_1=\rho_{j+1}^2$, and $K_{j+1}$ contains
$\opt$. After the phases $j=0,\dots,J$ the gap is at most
$\delta_J\le\varepsilon$.

On a domain of radius $\rho_j$ the certificate \eqref{eq:tensor-eN} satisfies
$e_N\le8\Xi_kc_{k,\nu}L_{k,\nu}(4\rho_j)^s(\mathsf b_N\cn)^{s/2}/N^\sigma$. Let
$N_j$ be the smallest power of two with
$N_j^\sigma\ge256\Xi_kc_{k,\nu}4^s(L_{k,\nu}/\mu_1)\rho_j^{s-2}(\mathsf b_{N_j}\cn)^{s/2}$;
then $e_{N_j}\le\mu_1\rho_j^2/32=\delta_j/4$, so the horizon used at accuracy
$\delta_j$ is at most $N_j$ and phase $j$ uses at most $1+32(2k+1)N_j$ queries,
for a total of $(J+1)+32(2k+1)\sum_{j\le J}N_j$. With
$\bar\kappa_j=L_{k,\nu}\rho_j^{s-2}/\mu_1=\bar\kappa\,2^{-j(s-2)}$, the dyadic horizon
estimate in the proof of \cref{thm:tensor} gives
$N_j=O_k(1+\bar\kappa_j^{1/\sigma}\{\cn\log(2+\bar\kappa_j\cn^{s/2})\}^{s/(2\sigma)})
=O_k(1+\bar\kappa^{1/\sigma}\Pi_0\,2^{-j(s-2)/\sigma})$, since $s\ge2$.
For $s=2$ the bound does not depend on $j$; for $s>2$ the geometric series is at
most $(1-2^{-(s-2)/\sigma})^{-1}$. Since $J+1\le2+\log_+(\mu_1R^2/\varepsilon)$,
this proves the corollary.
\end{proof}

\subsection{The slack and localization by one gradient}
\label{app:slack}

By \cref{prop:compatibility}, $\kappa_1=n+\Lambda\ge n$, so the factor
$\kappa_1^{1/3}$ of \cref{thm:restart} is at least $n^{1/3}$ even when the slack
$\Lambda=\kappa_1-n$ is small. On the ball the appropriate parameter is
$\Lambda$, because one gradient localizes the minimizer in an $\ell_1$ ball
whose radius is controlled by $\Lambda$ instead of $\kappa_1$.

\begin{restatable}[Localization by one gradient]{lemma}{lemlocalization}
\label{lem:localization}
Let $K=B_1^n(R)$, let $f$ be convex and differentiable on an open set
containing $K$, $L$-smooth in $\ell_2$ and $\mu_1$-strongly convex in $\ell_1$
on $K$, with $\Lambda=L/\mu_1-n>0$, and let $\opt$ be its minimizer on $K$. For
$x\in K$ put $x^+=\proj_K(x-\nabla f(x)/L)$. Then
\begin{equation}\label{eq:localization}
 \nrm{x^+-\opt}_1^2\le\frac{16\Lambda\,[1+\log(1+\Lambda)]}{L}\,\bigl[f(x)-f(\opt)\bigr],
\end{equation}
and if smoothness and strong convexity hold on all of $\R^n$, then
$\nrm{x^+-\opt}_1^2\le\frac{8\Lambda}{L}[f(x)-f(\opt)]$.
\end{restatable}

\begin{restatable}[Restarts governed by the slack]{proposition}{thmslackrestart}
\label{thm:slack-restart}
Under the assumptions of \cref{lem:localization}, let $\Lambda_*=8\Lambda$ in the
global case and $\Lambda_*=16\Lambda(1+\log(1+\Lambda))$ in the domain case, and
let $N_\Lambda$ be the smallest power of two with
$N_\Lambda^3\ge6144\Lambda_*\,\mathsf b_{N_\Lambda}\cn$. For every
$\varepsilon>0$, the restarted method returns an $\varepsilon$-solution after at
most $1+J(2+24N_\Lambda)$ first-order queries, $J=\ceilp{\log_2(LR^2/(2\varepsilon))}$;
that is,
\begin{equation}\label{eq:slack-restart}
 O\Bigl(\bigl[1+\{\Lambda_*\log(2n)\log(2+\Lambda_*\log(2n))\}^{1/3}\bigr]
 \bigl[1+\log_+\tfrac{LR^2}{\varepsilon}\bigr]\Bigr).
\end{equation}
In the global case, with the gradient-only method \SPP{} in each phase, the
count is
$O([1+(\Lambda\log(2n))^{1/3}\log^{7/3}(2+\Lambda\log(2n))][1+\log_+\frac{LR^2}\varepsilon])$
gradient queries, and no function value is used. If instead $n\le\Lambda^{1/3}$,
\CGC{} with $F=\nabla f$ gives $O(1+n\log(2+LR^2/\varepsilon))$ gradient queries.
\end{restatable}

\begin{lemma}[The $\ell_1$ ball projection in $\ell_1$]\label{lem:proj-l1}
$\nrm{\proj_K(u)-\proj_K(v)}_1\le2\nrm{u-v}_1$ for all $u,v\in\R^n$.
\end{lemma}

\begin{proof}
For $\nrm u_1\le R$, $\proj_K(u)=u$. For $\nrm u_1>R$, $\proj_K(u)_i=\sign(u_i)(|u_i|-\lambda)_+$
with the unique $\lambda>0$ such that $\sum_i(|u_i|-\lambda)_+=R$. The map is
continuous and piecewise affine on a finite polyhedral subdivision of $\R^n$.
On a piece with active set $I$ ($|I|=m$, the coordinates with $|u_i|>\lambda$)
and signs $s_i=\sign u_i$ we have $\lambda=(\sum_{i\in I}|u_i|-R)/m$, so the Jacobian
has the block $I_m-ss^\top/m$ on the active coordinates and zero columns for
the inactive ones. The $\ell_1$-norm of every column of $I_m-ss^\top/m$ is
$(1-1/m)+(m-1)/m\le2$, so the Jacobian has $\ell_1\to\ell_1$ operator norm at
most $2$ on every piece. Splitting the segment $[u,v]$ into finitely many
pieces and adding gives the claim; on the ball itself the constant is $1$.
\end{proof}

\begin{proof}[Proof of \cref{lem:localization}, global case]
Let $f$ be $L$-smooth and $\mu_1$-strongly convex in $\ell_1$ on $\R^n$. Then
$f$ is coercive and strictly convex, $f^*(g)=\sup_u\{\ip gu-f(u)\}$ is finite,
and $\nabla f:\R^n\to\R^n$ is a homeomorphism with $\nabla f^*(\nabla f(x))=x$;
moreover $\nabla f^*$ is Lipschitz, because $\nabla f$ is strongly monotone in
$\ell_2$ (with modulus at least $\mu_1$, since $\nrm v_1\ge\nrm v_2$).
For $x=\nabla f^*(g)$ and every $h$, the smoothness upper model and the
strong-convexity lower model of $f$ at $x$ give
\begin{equation}\label{eq:conj-two-sided}
 f^*(g)+\ip xh+\frac{\nrm h_2^2}{2L}\ \le\ f^*(g+h)\ \le\ f^*(g)+\ip xh+\frac{\nrm h_\infty^2}{2\mu_1}
\end{equation}
(for the lower bound take $u=x+h/L$ in the supremum; for the upper one use
$f(u)\ge f(x)+\ip g{u-x}+\frac{\mu_1}2\nrm{u-x}_1^2$ and $\sup_v\{\ip hv-\frac{\mu_1}2\nrm v_1^2\}=\nrm h_\infty^2/(2\mu_1)$).
Put $\psi(g)=f^*(g)-\nrm g_2^2/(2L)$. By \eqref{eq:conj-two-sided}, $\psi$ is
convex with $\nabla\psi(g)=x-g/L$, and $\nabla f^*$ is Lipschitz, so by
Rademacher's theorem $\nabla f^*$ is differentiable almost everywhere. At such
a point $g$ the second-order expansion of \eqref{eq:conj-two-sided} gives
$B:=\nabla^2f^*(g)-I/L\succeq0$ and $s^\top Bs\le1/\mu_1-n/L=\Lambda/L$ for every
sign vector $s\in\{-1,1\}^n$. Since $h\mapsto h^\top Bh$ is convex, its maximum
over the cube $[-1,1]^n$ is attained at a vertex, so $h^\top Bh\le(\Lambda/L)\nrm h_\infty^2$
for all $h$. Mollifying $\psi$ with a smooth compactly supported kernel gives
$C^\infty$ functions $\psi_\epsilon$ whose Hessians are averages of the matrices
$B$ and satisfy the same two-sided bound. Hence
$\psi_\epsilon(g+h)\le\psi_\epsilon(g)+\ip{\nabla\psi_\epsilon(g)}h+\frac{\Lambda}{2L}\nrm h_\infty^2$,
and letting $\epsilon\to0$ (local uniform convergence of $\psi_\epsilon$ and of
$\nabla\psi_\epsilon$, as $\nabla\psi$ is continuous) shows that $\psi$ is
$(\Lambda/L)$-smooth with respect to $\nrm\cdot_\infty$ on $\R^n$.

For a convex $\psi$ that is $\gamma$-smooth with respect to a norm $\nrm\cdot$ with
dual norm $\nrm\cdot_*$, $D_\psi(u,v)\ge\nrm{\nabla\psi(u)-\nabla\psi(v)}_*^2/(2\gamma)$
(apply the smoothness upper model of $\psi-\ip{\nabla\psi(v)}\cdot$, which is
minimized at $v$, at $u$ in a direction attaining the dual norm). With
$g=\nabla f(x)$, $g^\star=\nabla f(\opt)$, the conjugate identity
$D_{f^*}(g^\star,g)=D_f(x,\opt)$ and $D_{\nrm\cdot_2^2/(2L)}(g^\star,g)=\nrm{g-g^\star}_2^2/(2L)$
give $D_\psi(g^\star,g)=D_f(x,\opt)-\nrm{g-g^\star}_2^2/(2L)\le D_f(x,\opt)\le f(x)-f(\opt)$,
the last step because $\ip{g^\star}{x-\opt}\ge0$ for $x\in K$ by optimality. Since
$\nabla\psi(g)-\nabla\psi(g^\star)=x-\opt-(g-g^\star)/L$, we obtain
\begin{equation}\label{eq:localization-core}
 \Bigl\|x-\opt-\frac{\nabla f(x)-\nabla f(\opt)}L\Bigr\|_1^2\le\frac{2\Lambda}L\bigl[f(x)-f(\opt)\bigr].
\end{equation}
Finally $\opt=\proj_K(\opt-\nabla f(\opt)/L)$ by the optimality condition, so
\cref{lem:proj-l1} gives $\nrm{x^+-\opt}_1\le2\nrm{x-\opt-(\nabla f(x)-\nabla f(\opt))/L}_1$,
and squaring proves the global bound $8\Lambda/L$.
\end{proof}

\begin{proof}[Proof of \cref{lem:localization}, domain case]
First let $f$ be $C^2$ on a neighborhood of the segment $[z,x]$ with
$\mu_1\nrm v_1^2\le v^\top H_tv\le L\nrm v_2^2$ for $H_t=\nabla^2f(z+td)$,
$d=x-z$, $t\in[0,1]$ (only these bounds are used, so the argument applies to the
mollified functions below). Put $J_t=I-H_t/L$ and
$A_t=H_t^{-1}-I/L\succeq0$. Maximizing $\ip sv-\frac12v^\top H_tv$ over $v$
and using $\ip sv\le\nrm s_\infty\nrm v_1$ gives $\frac12s^\top H_t^{-1}s\le\nrm s_\infty^2/(2\mu_1)$,
so $s^\top A_ts\le\Lambda/L$ for every sign vector $s$; averaging over uniformly
random signs gives $\tr A_t\le\Lambda/L$, hence $\lambda_{\max}(H_t^{-1})\le(1+\Lambda)/L$
and $H_t\succeq\frac L{1+\Lambda}I$. The matrices $H_t,J_t,A_t$ commute, and on
an eigenvalue $0<\lambda\le L$ of $H_t$ one checks
$(1-\lambda/L)^2/\lambda\le1/\lambda-1/L$ and $(1-\lambda/L)^2\le L(1/\lambda-1/L)$;
hence $J_tH_t^{-1}J_t\preceq A_t$ and $J_t^2\preceq LA_t$. For a sign vector $s$,
the Cauchy--Schwarz inequality in the $H_t$-metric and in the Euclidean metric
gives $\ip s{J_td}\le\sqrt{s^\top J_tH_t^{-1}J_ts}\sqrt{d^\top H_td}$ and
$\ip s{J_td}\le\sqrt{s^\top J_t^2s}\,\nrm d_2$, so
\begin{equation}\label{eq:jt-bounds}
 \nrm{J_td}_1^2\le\frac\Lambda L\,d^\top H_td,\qquad \nrm{J_td}_1^2\le\Lambda\nrm d_2^2 .
\end{equation}
Now $\beta:=D_f(x,z)=\int_0^1(1-t)\,d^\top H_td\,dt$ and
$x-z-(\nabla f(x)-\nabla f(z))/L=\int_0^1J_td\,dt$. Let $\eta=1/(1+\Lambda)$. On
$[0,1-\eta]$ the first bound of \eqref{eq:jt-bounds} and the Cauchy--Schwarz
inequality with the weight $(1-t)$ give
\[
 \Bigl\|\int_0^{1-\eta}J_td\,dt\Bigr\|_1\le\sqrt{\frac\Lambda L}\int_0^{1-\eta}\frac{\sqrt{(1-t)d^\top H_td}}{\sqrt{1-t}}dt
 \le\sqrt{\frac{\Lambda\beta}{L}\log\frac1\eta},
\]
and on $[1-\eta,1]$ the second bound and $\beta\ge\frac{L}{2(1+\Lambda)}\nrm d_2^2$ give
$\|\int_{1-\eta}^1J_td\,dt\|_1\le\eta\sqrt\Lambda\nrm d_2\le\sqrt{2\Lambda\beta/(L(1+\Lambda))}$.
By $(a+b)^2\le2a^2+2b^2$,
\begin{equation}\label{eq:localization-domain}
 \Bigl\|x-z-\frac{\nabla f(x)-\nabla f(z)}L\Bigr\|_1^2
 \le\frac\Lambda L\Bigl[2\log(1+\Lambda)+\frac4{1+\Lambda}\Bigr]D_f(x,z)
 \le\frac{4\Lambda[1+\log(1+\Lambda)]}{L}D_f(x,z).
\end{equation}
For a $C^1$ function $f$ satisfying the two inequalities on $K$ only, we apply
\eqref{eq:localization-domain} to mollifications of $f$. For interior $x,z$ the
segment has positive distance to the boundary, and the convolution with a
smooth kernel of smaller radius is $C^\infty$, $L$-smooth and $\mu_1$-strongly
convex near the segment (both inequalities are averaged over translates that
stay in $K$). Its values and gradients converge locally uniformly to those of
$f$, so \eqref{eq:localization-domain} holds for interior points. For boundary
points we replace $x,z$ by $(1-\sigma)x,(1-\sigma)z$ and let $\sigma\downarrow0$,
using the continuity of $f$ and $\nabla f$ on $K$.
With $z=\opt$, $D_f(x,\opt)\le f(x)-f(\opt)$ by optimality, and
\cref{lem:proj-l1} with $\opt=\proj_K(\opt-\nabla f(\opt)/L)$ gives
\eqref{eq:localization}.
\end{proof}

\begin{proof}[Proof of \cref{thm:slack-restart}]
\emph{Initialization.} One query at $0$ gives $g_0=\nabla f(0)$ and a vertex
$v\in\argmin_{u\in K}\ip{g_0}u$. Smoothness at $0$ gives
$f(v)\le f(0)+\ip{g_0}v+\frac L2\nrm v_2^2$, and convexity with the choice of $v$
gives $f(\opt)\ge f(0)+\ip{g_0}{\opt}\ge f(0)+\ip{g_0}v$; hence
$f(v)-f(\opt)\le\frac L2\nrm v_2^2=LR^2/2=:\Delta_0$ without evaluating $f$.
If $\Delta_0\le\varepsilon$, the method returns $v$.

\emph{One phase.} Let $x\in K$ satisfy the certified bound $f(x)-f(\opt)\le\Delta$.
Query $\nabla f(x)$ and put $x^+=\proj_K(x-\nabla f(x)/L)$,
$r=\min\{2R,\sqrt{\Lambda_*\Delta/L}\}$ and $K'=K\cap(x^++B_1^n(r))$, a known
polytope contained in $x^++B_1^n(r)$. By \cref{lem:localization} (the global or
the domain bound, according to the case), $\opt\in K'$, so
$\min_{K'}f=f(\opt)$. We run \cref{alg:spaccel} on $K'$ with accuracy $\Delta/2$. By
\cref{thm:main} for general polytopes, this uses at most $1+24N'$ queries,
where $N'$ is the smallest power of two with
$N'^3\ge3072Lr^2\mathsf b_{N'}\cn/(\Delta/2)$ (if the initial bracket width
of the subproblem is at most $\Delta/2$, one query suffices). Since
$Lr^2/(\Delta/2)\le2\Lambda_*$, every $N$ satisfying
$N^3\ge6144\Lambda_*\mathsf b_N\cn$ satisfies the defining inequality of
$N'$, so $N'\le N_\Lambda$. The phase returns a point $x'\in K'\subseteq K$ with
$f(x')-f(\opt)\le\Delta/2$ and uses at most $2+24N_\Lambda$ queries.

\emph{Phases.} Starting from $x=v$, $\Delta=\Delta_0$, we run $J=\ceilp{\log_2(LR^2/(2\varepsilon))}$
phases, halving $\Delta$ each time; the last certified bound is
$\Delta_02^{-J}\le\varepsilon$. The total is at most $1+J(2+24N_\Lambda)$, and
$N_\Lambda=O(1+[\Lambda_*\cn\log(2+\Lambda_*\cn)]^{1/3})$ by the argument in
the proof of \cref{thm:restart} (substitute $N=C\max\{1,[\Lambda_*\cn\log(2+\Lambda_*\cn)]^{1/3}\}$),
which gives \eqref{eq:slack-restart}.

\emph{Gradient oracle (global case).} Here $f$ is $L$-smooth on $\R^n$, as
\cref{thm:gradient-only} requires. We replace \cref{alg:spaccel} by
\cref{alg:spprox} on $K'$ with accuracy $\Delta/2$, started at $a=x^+$. Its
initial certified bound is $U_0=2Lr^2\le2\Lambda_*\Delta$, and
\cref{thm:gradient-only} on the polytope $K'\subseteq x^++B_1^n(r)$ with the ratio
$Lr^2\cn/(\Delta/2)\le2\Lambda_*\cn$ bounds its cost by
$O(1+(\Lambda_*\cn)^{1/3}\log^{7/3}(2+\Lambda_*\cn))$, with $\Lambda_*=8\Lambda$;
the localization step uses one gradient and no value.

\emph{Dimension branch.} \Cref{prop:cog} with $F=\nabla f$ on $K$ gives the
last claim, since $f(\bar w)-\min_Kf\le\varepsilon$ for its output.
\end{proof}

\subsection{The deterministic complexity for $\Lambda\ge8$}
\label{app:strong-map}

The lower bounds need two ingredients: the randomized bound of
\cref{thm:rand-smooth} with strong convexity, which gives the cubic-rate phase,
and a transfer of a classical finite-dimensional bound, which gives the
high-accuracy phase.

\begin{lemma}[Weighted norm comparison]
\label{lem:weighted-strong-norm}
Let $1\le p<2$, $\beta_p=2/p-1$, and $a_i>0$. Then
\begin{equation}\label{eq:weighted-strong-norm}
 \nrm{h}_p^2\le
 \left(\sum_{i=1}^n a_i^{-1/\beta_p}\right)^{\beta_p}
 \sum_{i=1}^n a_i h_i^2.
\end{equation}
The constant is sharp for the diagonal quadratic form.
\end{lemma}
\begin{proof}
Apply H\"older with exponents $2/p$ and $2/(2-p)$ to
$|h_i|^p=(a_i h_i^2)^{p/2}a_i^{-p/2}$, then raise to $2/p$.
Equality holds when $|h_i|$ is proportional to $a_i^{-1/(2-p)}$.
\end{proof}

\begin{restatable}[High accuracy at fixed dimension]{proposition}{propnytransfer}
\label{prop:ny-transfer}
Let $n\ge2$, $L,\mu_1,R>0$, $\Lambda=L/\mu_1-n$, and let $d$ be an integer with
$2\le d\le n$ and $d^3-d\le\Lambda$. Consider the class of convex $C^{1,1}$
functions on $\R^n$ that are $L$-smooth in $\ell_2$ and $\mu_1$-strongly convex
in $\ell_1$, with global first-order queries and output in $B_1^n(R)$. If a
deterministic method with at most $N$ queries guarantees error at most
$\varepsilon$ on this class, where $0<\varepsilon<LR^2/(2d^5)$, then
\begin{equation}\label{eq:ny-transfer}
 N+2\ \ge\ c_{\rm NY}\,\frac{d}{\log d}\,\log\frac{LR^2}{2d^5\varepsilon},
\end{equation}
where $c_{\rm NY}>0$ is the absolute constant of
\citet[Theorem~7.2.7]{nemirovski1983}: every deterministic method that solves all
Euclidean $L$-smooth and $(L/d^2)$-strongly convex functions on $\R^d$ from the
start $0$ to relative accuracy $\varepsilon_{\rm rel}\in(0,1)$ has laboriousness
(number of queries plus one for the output) at least
$c_{\rm NY}(d/\log d)\log(1/\varepsilon_{\rm rel})$.
\end{restatable}

The proof normalizes an unknown $d$-dimensional instance by one query, so that
its minimizer lies in the ball, and pads the remaining $n-d$ coordinates with
the known curvature $L$. The padded function is $\mu_1$-strongly convex in
$\ell_1$ because $d/(L/d^2)+(n-d)/L=(d^3+n-d)/L\le1/\mu_1$. With
$d\asymp\min\{n,\Lambda^{1/3}\}$ this gives the high-accuracy term of
\cref{thm:strong-map} below.

\begin{proof}[Proof of \cref{prop:ny-transfer}]
Let $\mu_2=L/d^2$ and let $h:\R^d\to\R$ be any convex $C^{1,1}$ function that is
$L$-smooth and $\mu_2$-strongly convex in $\ell_2$ on $\R^d$, with unique minimizer
$u^\star$. We describe a deterministic method $\cB$ for $h$ that simulates the
given method $\cA$ for the mixed class and has laboriousness at most $N+2$ in
the convention of \citet{nemirovski1983}, in which the queries and the final
output are counted.

\emph{Normalization by one query.} $\cB$ queries $h(0)$ and $g_0=\nabla h(0)$. If
$g_0=0$ it outputs $0$, which is exact. Otherwise it puts
$a=R\mu_2/(\sqrt d\nrm{g_0}_2)>0$. Strong monotonicity of $\nabla h$ between $0$ and
$u^\star$ gives $\mu_2\nrm{u^\star}_2^2\le\ip{g_0}{-u^\star}\le\nrm{g_0}_2\nrm{u^\star}_2$, so
$\nrm{u^\star}_2\le\nrm{g_0}_2/\mu_2$ and $\nrm{au^\star}_1\le\sqrt d\,a\nrm{u^\star}_2\le R$.

\emph{Padding.} Write $x=(u,v)\in\R^d\times\R^{n-d}$ and define
\[
 f(u,v)=a^2h(u/a)+\tfrac L2\nrm v_2^2
\]
(no second term if $d=n$). Then $f$ is convex and $L$-smooth on $\R^n$, and for
increments $(\xi_u,\xi_v)$ its strong-convexity remainder is at least
$\frac12(\mu_2\nrm{\xi_u}_2^2+L\nrm{\xi_v}_2^2)$. By \cref{lem:weighted-strong-norm} in
the $\ell_1$ case ($\beta_1=1$),
$\nrm{(\xi_u,\xi_v)}_1^2\le(d/\mu_2+(n-d)/L)(\mu_2\nrm{\xi_u}_2^2+L\nrm{\xi_v}_2^2)$, and
$d/\mu_2+(n-d)/L=(d^3+n-d)/L\le(n+\Lambda)/L=1/\mu_1$ by the hypothesis $d^3-d\le\Lambda$.
Hence $f$ is $\mu_1$-strongly convex in $\ell_1$ and belongs to the mixed
class; its global minimizer $(au^\star,0)$ lies in $B_1^n(R)$, so
$\min_{B_1^n(R)}f=f(au^\star,0)=a^2h(u^\star)$.

\emph{Simulation.} A query of $\cA$ at any $(u,v)\in\R^n$ is answered from one
query of $h$ at $u/a$: $f(u,v)=a^2h(u/a)+\frac L2\nrm v_2^2$ and
$\nabla f(u,v)=(a\nabla h(u/a),Lv)$. Thus $\cB$ makes at most $N+1$ queries of $h$
in total, and it is deterministic. When $\cA$ outputs $(\widehat u,\widehat v)\in B_1^n(R)$,
$\cB$ outputs $\widehat u/a$.

\emph{Accuracy.} Smoothness at $0$ gives $h(0)-h(u^\star)\ge\nrm{g_0}_2^2/(2L)$
(compare the value of the upper model at $-g_0/L$ with the minimum), so
$a^2[h(0)-h(u^\star)]\ge R^2\mu_2^2/(2Ld)=LR^2/(2d^5)$. The guarantee of $\cA$ and the
nonnegativity of the padding give
$a^2[h(\widehat u/a)-h(u^\star)]\le f(\widehat u,\widehat v)-\min_{B_1^n(R)}f\le\varepsilon$,
hence the relative error of $\cB$ is at most
$\varepsilon_{\rm rel}:=2d^5\varepsilon/(LR^2)\in(0,1)$. The cited theorem, applied to $\cB$
with the condition number $L/\mu_2=d^2$, gives
$N+2\ge c_{\mathrm{NY}}(d/\log d)\log(1/\varepsilon_{\rm rel})$, which is \eqref{eq:ny-transfer}.
\end{proof}

\Cref{thm:slack-restart} and these lower bounds determine the deterministic
complexity of the smooth strongly convex class on the ball for every $n\ge2$
and $\Lambda\ge8$, up to factors polylogarithmic in $n$ and $\Lambda$, and for
both first-order oracles. To state the result, let
\begin{equation}\label{eq:map-parameters}
 {\bar n}=\min\{n,\Lambda^{1/3}\},\ \
 W(\varepsilon)=\min\Bigl\{{\bar n},\Bigl(\frac{LR^2}\varepsilon\Bigr)^{1/3}\Bigr\},\ \
 T(\varepsilon)=\log_+\frac{LR^2}{{\bar n}^3\varepsilon},\ \
 \vartheta_{\bar n}=\log(2+\cn {\bar n}^3),
\end{equation}
and let $N_{\rm J}(\varepsilon)$ and $N_{\rm G}(\varepsilon)$ be the smallest
numbers of queries with which a deterministic method, using global queries of
$(f,\nabla f)$ or of $\nabla f$ respectively, guarantees
$f(\widehat x)-\min_{B_1^n(R)}f\le\varepsilon$ on every convex $C^{1,1}$ function
on $\R^n$ that is $L$-smooth in $\ell_2$ and $\mu_1$-strongly convex in
$\ell_1$, with output $\widehat x\in B_1^n(R)$.

\begin{restatable}[Deterministic complexity for $\Lambda\ge8$]{proposition}{thmstrongmap}
\label{thm:strong-map}
There are absolute constants $c,C>0$ such that for all $n\ge2$, $L,\mu_1,R>0$
with $\Lambda=L/\mu_1-n\ge8$, $K=B_1^n(R)$ and all $\varepsilon>0$, with
$c_{\rm NY}$ from \cref{prop:ny-transfer},
\begin{equation}\label{eq:strong-map}
\begin{aligned}
 c\,\min\{1,c_{\rm NY}\}\,\frac{1+W(\varepsilon)+{\bar n}T(\varepsilon)}{\log(2+{\bar n})}
 \ &\le\ N_{\rm J}(\varepsilon)\ \le\ N_{\rm G}(\varepsilon)\ \le\
 C\,\Pi_{\rm G}\,\bigl[1+W(\varepsilon)+{\bar n}T(\varepsilon)\bigr],\\
 N_{\rm J}(\varepsilon)\ &\le\ C\,\Pi_{\rm J}\,\bigl[1+W(\varepsilon)+{\bar n}T(\varepsilon)\bigr],
\end{aligned}
\end{equation}
where $\Pi_{\rm J}=1+\log(2+n)+(\cn\vartheta_{\bar n})^{1/3}$ and
$\Pi_{\rm G}=1+\log(2+n)+\cn^{1/3}\vartheta_{\bar n}^{7/3}$ are polylogarithmic in $n$ and
$\Lambda$.
The lower bound $c(1+W(\varepsilon))$ holds for randomized methods as well.
\end{restatable}

The two terms of \eqref{eq:strong-map} correspond to two regimes: a
cubic-rate phase of $\min\{{\bar n},(LR^2/\varepsilon)^{1/3}\}$ queries, which lasts
until the accuracy $LR^2/{\bar n}^3$, and a high-accuracy phase of
${\bar n}\log(LR^2/({\bar n}^3\varepsilon))$ queries, linear in the logarithm of the accuracy
with the factor $\min\{n,\Lambda^{1/3}\}$. In particular, when
$n\ge\Lambda^{1/3}$ and $\Lambda\ge n$, the count $\tO(\kappa_1^{1/3}\log(1/\varepsilon))$ of
\cref{thm:restart} is optimal for deterministic methods up to polylogarithmic
factors. For randomized methods at fixed dimension, and for $\Lambda<8$, the
complexity is not determined here.

\begin{proof}[Proof of \cref{thm:strong-map}]
Write $A=LR^2$ and $S=1+W+{\bar n}T$ with $W=W(\varepsilon)$, $T=T(\varepsilon)$.
The inequality $N_{\rm J}\le N_{\rm G}$ holds because the joint oracle is more
informative. All upper bounds below use feasible queries only, and
$\Lambda\ge8$ with $n\ge2$ gives ${\bar n}\ge2$.

\emph{Upper bounds, $\varepsilon\ge A/{\bar n}^3$.} Then $T=0$ and $W=(A/\varepsilon)^{1/3}\le {\bar n}$.
\Cref{thm:main} (joint oracle) gives
$N_{\rm J}\le1+24N(\varepsilon)=O(1+[(A/\varepsilon)\cn\log(2+(A/\varepsilon)\cn)]^{1/3})\le C(1+W(\cn\vartheta_{\bar n})^{1/3})\le C\Pi_{\rm J}S$,
because $A/\varepsilon\le {\bar n}^3$; and \cref{thm:gradient-only} gives
$N_{\rm G}=O(1+[(A/\varepsilon)\cn]^{1/3}\log^{7/3}(2+(A/\varepsilon)\cn))\le C(1+W\cn^{1/3}\vartheta_{\bar n}^{7/3})\le C\Pi_{\rm G}S$.

\emph{Upper bounds, $\varepsilon<A/{\bar n}^3$.} Then $W={\bar n}$ and $T>0$. If
${\bar n}=\Lambda^{1/3}\le n$, we first reach the accuracy $e_0=A/\Lambda=A/{\bar n}^3$ with the
convex methods, with the count just computed for $W={\bar n}$, i.e.\ $O(\Pi_{\rm J}{\bar n})$ or
$O(\Pi_{\rm G}{\bar n})$. We then run the phases of \cref{thm:slack-restart} (global case,
$\Lambda_*=8\Lambda$) from the certified bound $e_0$ until it is at most
$\varepsilon$. This takes $\lceil\log_2(e_0/\varepsilon)\rceil\le1+T/\log2$
phases, each using $O(1+(\Lambda\cn\log(2+\Lambda\cn))^{1/3})=O(1+{\bar n}(\cn\vartheta_{\bar n})^{1/3})\le C\Pi_{\rm J}{\bar n}$
joint queries or $O(1+(\Lambda\cn)^{1/3}\log^{7/3}(2+\Lambda\cn))\le C\Pi_{\rm G}{\bar n}$
gradient queries. The total is $O(\Pi_{\rm J}(1+{\bar n}+{\bar n}T))$, resp.\ $O(\Pi_{\rm G}(1+{\bar n}+{\bar n}T))$.
If ${\bar n}=n\le\Lambda^{1/3}$, \cref{prop:cog} with $F=\nabla f$ gives
$N_{\rm G}\le C[1+n\log(2+A/\varepsilon)]$, and
$\log(2+A/\varepsilon)\le\log(2+{\bar n}^3)+T\le3\log(2+{\bar n})+T$, so
$N_{\rm G}\le C[1+{\bar n}(3\log(2+{\bar n})+T)]\le C'(1+\log(2+n))S$, which is at most
$C'\Pi_{\rm J}S$ and $C'\Pi_{\rm G}S$.

\emph{Lower bound $N_{\rm J}\ge1$.} The class contains
$f_b(x)=\frac L2\nrm x_2^2+\ip bx$ for every $b$ (it is $L/n$-strongly convex in
$\ell_1$, and $L/n\ge\mu_1$). A method with no query returns a point $\widehat x$
(random or not) that does not depend on $b$. For $b=\mp te_1$, $t>0$,
$\min_Kf_b\le f_b(\pm Re_1)=LR^2/2-tR$ and $f_b(\widehat x)\ge\mp t\widehat x_1$, so
$f_b(\widehat x)-\min_Kf_b\ge t(R\mp\widehat x_1)-LR^2/2$; the sum over the two
signs is $2tR-LR^2$, so for one of them the (expected) error is at least
$tR-LR^2/2$, which exceeds $\varepsilon$ for large $t$. Hence $N_{\rm J}\ge1$,
also for randomized methods.

\emph{Lower bound of order $W$.} Let $c_0=2^{-19}$ and
$N=\lfloor c_0W\rfloor$. If $N\ge1$, then $W\ge2^{19}$, so $n\ge W\ge2^{19}$ and
$32N+1\le32c_0n+1\le n$, and $2^{51}N^3\le2^{51}c_0^3W^3\le2^{-6}\Lambda\le\Lambda$
since $W^3\le {\bar n}^3\le\Lambda$. Hence \cref{thm:rand-smooth} applies with this $N$
and gives, for every method with at most $N$ queries (randomized methods
included), a function of the class on which the expected error is at least
$A/(2^{52}N^3)\ge A/(2^{52}c_0^3W^3)=32A/W^3\ge32\varepsilon>\varepsilon$, using
$W^3\le A/\varepsilon$. Therefore every method that guarantees $\varepsilon$
makes more than $\lfloor c_0W\rfloor$ queries, and with the previous
paragraph $N_{\rm J}\ge\max\{1,c_0W\}\ge\frac12(1+c_0W)$.

\emph{Lower bound of order ${\bar n}T/\log(2+{\bar n})$.} Let $d=\lfloor {\bar n}\rfloor$, so
$2\le d\le {\bar n}\le n$, ${\bar n}/2\le d$, and $d^3-d\le {\bar n}^3\le\Lambda$. Suppose first
$T\ge6\log(2+{\bar n})$. Then $A/(2d^5\varepsilon)=(A/({\bar n}^3\varepsilon))\,{\bar n}^3/(2d^5)\ge e^T{\bar n}^3/(2{\bar n}^5)=e^T/(2{\bar n}^2)>1$,
so \cref{prop:ny-transfer} applies, and
$\log\frac{A}{2d^5\varepsilon}\ge T-2\log {\bar n}-\log2\ge T-3\log(2+{\bar n})\ge T/2$. Hence
$N_{\rm J}+2\ge c_{\rm NY}\frac{d}{\log d}\cdot\frac T2\ge c_{\rm NY}\frac{{\bar n}T}{4\log(2+{\bar n})}$,
and since $N_{\rm J}\ge1$, $N_{\rm J}\ge\frac13(N_{\rm J}+2)\ge c_{\rm NY}\frac{{\bar n}T}{12\log(2+{\bar n})}$.
If instead $0<T<6\log(2+{\bar n})$, then $W={\bar n}$ (as $\varepsilon<A/{\bar n}^3$) and
${\bar n}T/\log(2+{\bar n})<6{\bar n}=6W$, so the previous paragraph gives
$N_{\rm J}\ge\frac{c_0}{6}\cdot\frac{{\bar n}T}{\log(2+{\bar n})}$. If $T=0$ there is nothing
to prove. Thus $N_{\rm J}\ge\max\{1,\,c_0W,\,\min\{c_{\rm NY}/12,c_0/6\}\,{\bar n}T/\log(2+{\bar n})\}$,
so $N_{\rm J}\ge\frac13\min\{1,c_{\rm NY}\}\frac{c_0}{12}\bigl(1+W+{\bar n}T/\log(2+{\bar n})\bigr)$,
and $S/\log(2+{\bar n})\le1+W+{\bar n}T/\log(2+{\bar n})$ because $\log(2+{\bar n})\ge1$. This proves the
left-hand side of \eqref{eq:strong-map} with $c=2^{-25}$.
\end{proof}

\section{Variational inequalities and saddle-point problems}
\label{app:vi}

This section extends the deep-cut methods to monotone operators: Lipschitz
operators (\cref{app:vi-lipschitz}), operators with a model whose error is a
power of the distance, including H\"older operators and Taylor models of higher
order (\cref{app:vi-models}), and the matching lower bound for Lipschitz
operators (\cref{app:vi-lower}). As before, $K\subseteq c+B_1^n(R)$ is a known
polytope.

\subsection{Lipschitz operators}
\label{app:vi-lipschitz}

Let $F:K\to\R^n$ and consider the variational inequality of finding $\opt\in K$
with $\ip{F(\opt)}{x-\opt}\ge0$ for all $x\in K$. For a monotone $F$ we measure
accuracy by the \emph{weak (Minty) gap} $\sup_{u\in K}\ip{F(u)}{\bar w-u}$; for a
convex--concave $\Phi$ on $K=X\times Y$ with $F=(\nabla_x\Phi,-\nabla_y\Phi)$ we
measure the \emph{saddle gap} $\max_{y\in Y}\Phi(\bar x,y)-\min_{x\in X}\Phi(x,\bar y)$.
Both are certified by the mixture certificate of \cref{alg:spcut}. If
$\alpha\in\Delta_T$ satisfies $\max_{u\in K}\sum_t\alpha_t\ip{F(w_t)}{w_t-u}\le\varepsilon$,
then $\bar w=\sum_t\alpha_tw_t$ has weak gap at most $\varepsilon$ when $F$ is
monotone, and saddle gap at most $\varepsilon$ in the convex--concave case
(\cref{lem:mixture-certificate}). Two methods produce such certificates.

\begin{lemma}[Mixture certificate]\label{lem:mixture-certificate}
Let $w_1,\dots,w_T\in K$, $\alpha\in\Delta_T$, $\bar w=\sum_t\alpha_tw_t$, and
suppose $\sum_t\alpha_t\ip{F(w_t)}{w_t-u}\le\varepsilon$ for all $u\in K$.
\begin{enumerate}[label=(\alph*)]
\item If $F$ is monotone on $K$, then $\ip{F(u)}{\bar w-u}\le\varepsilon$ for all
$u\in K$.
\item If $K=X\times Y$, $F(x,y)=(g,-h)$ with $g\in\partial_x\Phi(x,y)$ and
$-h\in\partial_y(-\Phi)(x,y)$ for a function $\Phi$ that is convex in $x$ and
concave in $y$, then
$\sup_{y\in Y}\Phi(\bar x,y)-\inf_{x\in X}\Phi(x,\bar y)\le\varepsilon$.
\end{enumerate}
\end{lemma}

\begin{proof}
(a) Monotonicity gives $\ip{F(u)}{w_t-u}\le\ip{F(w_t)}{w_t-u}$; multiplying by
$\alpha_t$ and summing, with $\sum_t\alpha_t(w_t-u)=\bar w-u$, gives the claim.
(b) For $(x,y)\in K$ and $w_t=(x_t,y_t)$, convexity in $x$ and concavity in
$y$ give $\Phi(x_t,y)-\Phi(x,y_t)\le\ip{g_t}{x_t-x}-\ip{h_t}{y_t-y}=\ip{F(w_t)}{w_t-(x,y)}$.
Multiplying by $\alpha_t$, summing, and using convexity in $x$ and concavity in
$y$ on the left ($\sum_t\alpha_t\Phi(x_t,y)\ge\Phi(\bar x,y)$ and
$\sum_t\alpha_t\Phi(x,y_t)\le\Phi(x,\bar y)$ respectively) gives
$\Phi(\bar x,y)-\Phi(x,\bar y)\le\varepsilon$ for all $(x,y)\in K$.
\end{proof}

\begin{restatable}[Bounded and Lipschitz operators]{proposition}{thmvi}
\label{thm:vi}
Let $K\subseteq c+B_1^n(R)$ be a known polytope.
\begin{enumerate}[label=(\alph*)]
\item If $\nrm{F(x)}_2\le G$ on $K$, then \cref{alg:spcut} with the cuts
$\{y:\ip{F(x_t)}{x_t-y}\ge\varepsilon_N\}$, $\varepsilon_N=4G\PN/N$, stops after
at most $N$ operator calls and returns $\bar w\in K$ with the certificate above.
No monotonicity is needed for the certificate.
\item If $F$ is $L$-Lipschitz in $\ell_2$ on $K$, then \SPV{}
(\cref{alg:spvi}), which at the current approximate Steiner point $z_t$ computes
$w_t=\proj_K(z_t-F(z_t)/(2L))$ and cuts with
$\{u:\ip{F(w_t)}{w_t-u}\ge\varepsilon_N\}$, $\varepsilon_N=6L\PN^2/N^2=96LR^2\bN\cn/N^2$,
stops after at most $2N$ operator calls with the certificate above. With a
budget of $Q\ge2$ calls the weak gap of a monotone operator, and the saddle gap
of a convex--concave problem, is $O(LR^2(1+\log Q)(1+\log(2n))/Q^2)$.
\item If moreover $\ip{F(x)-F(y)}{x-y}\ge\mu_1\nrm{x-y}_1^2$ on $K$, then
restarting (b) on the localized polytopes $K\cap(\bar w_j+B_1^n(R2^{-j-1}))$
returns a point at $\ell_1$-distance at most $d_\star$ from the unique solution after
$O\bigl([1+\{\kappa_1\log(2n)\log(2+\kappa_1\log(2n))\}^{1/2}][1+\log_+(R/d_\star)]\bigr)$
calls, $\kappa_1=L/\mu_1$.
\end{enumerate}
\end{restatable}

\begin{algorithm}[t]
\caption{\SPV: extragradient step with deep cuts (horizon $N$)}
\label{alg:spvi}
\begin{algorithmic}[1]
\Require known polytope $K\subseteq c+B_1^n(R)$; operator oracle, $F$
$L$-Lipschitz on $K$; horizon $N$; $\varepsilon_N=6L\PN^2/N^2$; $\eta_N=\PN/(4N)$
\State $C_0\gets K$; $z_0\gets$ an $\eta_N$-approximate Steiner point of $K$
\For{$t=0,1,\dots$}
 \State query $F(z_t)$;\quad $w_t\gets\proj_K\bigl(z_t-F(z_t)/(2L)\bigr)$
 \State query $F(w_t)$;\quad $C_{t+1}\gets C_t\cap\{u:\ip{F(w_t)}{w_t-u}\ge\varepsilon_N\}$
 \If{$C_{t+1}=\varnothing$}
  \State solve $\min_{\alpha\in\Delta_{t+1}}\max_{u\in K}\sum_{j\le t}\alpha_j\ip{F(w_j)}{w_j-u}$ (a linear program)
  \State\Return $\bar w=\sum_{j\le t}\alpha_jw_j$
 \EndIf
 \State $z_{t+1}\gets$ an $\eta_N$-approximate Steiner point of $C_{t+1}$
\EndFor
\end{algorithmic}
\end{algorithm}

The method of part (b) is an extragradient step \citep{korpelevich1976} whose
second point is used as a deep cut. In the mismatched setting its rate
$\tO(LR^2/N^2)$ improves on the Euclidean $O(LR^2/N)$ of Mirror-Prox
\citep{nemirovski2004prox}. For $F=\nabla f$ the weak gap of $\bar w$ is at most
$f(\bar w)-\min_Kf$, so lower bounds for minimization do not transfer to the
weak gap. The skew-symmetric construction of \cref{thm:vi-lower} gives the
matching deterministic bound $LR^2/(8(2N+3)^2)$ for $n\ge4N+6$; hence the
exponent $N^{-2}$ of (b) is optimal up to the factor $\log(2N)\log(2n)$. The
saddle gap of a bilinear game on $\Delta_{n_x}\times\Delta_{n_y}$ with a
Euclidean operator bound is a special case of (b), and \cref{thm:vi-lower}(b)
shows that its rate is also optimal in this sense. In small dimension the
center-of-gravity branch \CGC{} of \cref{prop:cog} applies to operators without
change. It returns a mixture with the same certificate after
$O(1+n\log(2+LR^2/\varepsilon))$ operator calls, without any bound on
$\nrm{F(0)}$, so the weak gap of a monotone $L$-Lipschitz operator on
$B_1^n(R)$ can be brought below $\varepsilon$ with
\[
 O\Bigl(1+\min\Bigl\{\Bigl[\frac{LR^2\cn}{\varepsilon}\log\Bigl(2+\frac{LR^2\cn}{\varepsilon}\Bigr)\Bigr]^{1/2},\
 n\log\Bigl(2+\frac{LR^2}{\varepsilon}\Bigr)\Bigr\}\Bigr)
\]
calls.

\begin{proof}[Proof of \cref{thm:vi}]
(a) We run \cref{alg:spcut} with the cut $C_t=C_{t-1}\cap\{y:\ip{F(x_t)}{x_t-y}\ge\varepsilon_N\}$
and selector accuracy $\eta=\varepsilon_N/(2G)$. If $C_t\ne\varnothing$ then
$\st{C_t}\in C_t$ gives $\ip{F(x_t)}{x_t-\st{C_t}}\ge\varepsilon_N$, so
$\nrm{\st{C_t}-x_t}_2\ge\varepsilon_N/G$ and
$\nrm{\st{C_t}-\st{C_{t-1}}}_2\ge\varepsilon_N/(2G)$; as in \cref{thm:spcut},
\cref{lem:steiner-path} forces an empty set within $N$ queries. Emptiness of
$C_T$ means $\max_{u\in K}\min_{t\le T}\ip{F(x_t)}{x_t-u}<\varepsilon_N$, and the
minimax theorem gives $\alpha\in\Delta_T$ with
$\max_{u\in K}\sum_t\alpha_t\ip{F(x_t)}{x_t-u}<\varepsilon_N$, computed by the
linear program of \cref{alg:spcut}. \Cref{lem:mixture-certificate} gives the gaps.

(b) \emph{One nonempty cut.} We drop the index $t$ and write $r=w-z$, $d=z'-z$
for the next selected point $z'\in C_{t+1}$, and $v=-F(z)-2Lr$. By the
optimality condition of the projection, $v$ belongs to the normal cone
$N_K(w)=\{g:\ip g{u-w}\le0\ \forall u\in K\}$. With $e=F(w)-F(z)$,
$\nrm e_2\le L\nrm r_2$, we have $F(w)=-2Lr+e-v$. Since $z'\in K$,
$\ip{-v}{w-z'}\le0$, and the cut gives
\begin{align*}
 \varepsilon_N\le\ip{F(w)}{w-z'}\le\ip{-2Lr+e}{r-d}
 &\le-2L\nrm r_2^2+2L\nrm r_2\nrm d_2+L\nrm r_2^2+L\nrm r_2\nrm d_2\\
 &=-L\nrm r_2^2+3L\nrm r_2\nrm d_2\le\tfrac94L\nrm d_2^2 ,
\end{align*}
using $w-z'=r-d$ and $-\tau^2+3\tau\sigma\le\frac94\sigma^2$. Hence every
nonempty transition has $\nrm{z_{t+1}-z_t}_2\ge\frac23\sqrt{\varepsilon_N/L}$.
If $C_1,\dots,C_N$ were all nonempty, \eqref{eq:app-path} would give
$\frac{2N}3\sqrt{\varepsilon_N/L}\le\frac32\PN$, i.e.\ $\varepsilon_N\le\frac{81}{16}L\PN^2/N^2$,
contradicting $\varepsilon_N=6L\PN^2/N^2$. Hence the method stops within $N$ trials
and $2N$ calls, and the certificate and the gaps follow as in (a). For a
budget $Q\ge2$ we take $N=\lfloor Q/2\rfloor$; then
$\varepsilon_N=96LR^2\bN\cn/N^2=O(LR^2(1+\log Q)(1+\log(2n))/Q^2)$.

(c) Existence of a solution $\opt$ follows from continuity and compactness,
and uniqueness from strong monotonicity. Let $K'\subseteq K$ be a known polytope
containing $\opt$ and let the certificate of (b) on $K'$ hold with right-hand
side $\delta$. Testing it at $u=\opt$ and using
$\ip{F(w_t)}{w_t-\opt}\ge\ip{F(\opt)}{w_t-\opt}+\mu_1\nrm{w_t-\opt}_1^2\ge\mu_1\nrm{w_t-\opt}_1^2$
(by strong monotonicity and the variational inequality at $\opt$, as
$w_t\in K'\subseteq K$) and the convexity of $\nrm\cdot_1^2$,
\[
 \delta\ge\sum_t\alpha_t\ip{F(w_t)}{w_t-\opt}\ge\mu_1\sum_t\alpha_t\nrm{w_t-\opt}_1^2\ge\mu_1\nrm{\bar w-\opt}_1^2 .
\]
With $\rho_j=R2^{-j}$, $K_0=K$ and $\delta_j=\mu_1\rho_j^2/4$, the output $\bar w_j$ of
phase $j$ satisfies $\nrm{\bar w_j-\opt}_1\le\rho_j/2=\rho_{j+1}$, so
$K_{j+1}=K\cap(\bar w_j+B_1^n(\rho_{j+1}))$ contains $\opt$; the projection in
phase $j$ is onto $K_j$. The horizon $N_\star$, the smallest power of two with
$96L\rho_j^2\mathsf b_N\cn/N^2\le\mu_1\rho_j^2/4$, i.e.\ $N^2\ge384\kappa_1\mathsf b_N\cn$,
does not depend on $j$, and $N_\star=O(1+[\kappa_1\cn\log(2+\kappa_1\cn)]^{1/2})$
as in the proof of \cref{thm:restart}. For the target $d_\star>0$ let $J$ be the smallest index
with $R2^{-J-1}\le d_\star$; after $J+1$ phases $\nrm{\bar w_J-\opt}_1\le d_\star$ with at
most $2(J+1)N_\star$ calls.
\end{proof}

\subsection{Operator models with power-law error}
\label{app:vi-models}

The deep-cut argument of \cref{thm:vi}(b) applies to any operator model whose
error is a power of the distance. Let $F:K\to\R^n$ be continuous and monotone,
and suppose that one oracle call at $z\in K$ returns $F(z)$ and a continuous
model $T_z:K\to\R^n$, which can be evaluated to any prescribed accuracy, such
that
\begin{equation}\label{eq:vi-model-remainder}
 \nrm{F(w)-T_z(w)}_2\le M\nrm{w-z}_2^r\qquad(z,w\in K)
\end{equation}
for some $M,r>0$; computations with $T_z$ are not counted as calls. Two classes
satisfy \eqref{eq:vi-model-remainder}. If $F$ is $\nu$-H\"older with constant
$L_{0,\nu}$, $0<\nu\le1$, take $T_z\equiv F(z)$, $r=\nu$ and $M=L_{0,\nu}$. If
$k\ge1$, $F$ is $C^k$ on a neighborhood of $K$ and $D^kF$ is $\nu$-H\"older with
constant $L_{k,\nu}$ in the Euclidean operator norm, $0\le\nu\le1$, an order-$k$
call returns $(F(z),DF(z),\dots,D^kF(z))$, and the Taylor model
$T_z(w)=\sum_{j\le k}\frac1{j!}D^jF(z)[(w-z)^j]$ satisfies
\eqref{eq:vi-model-remainder} with $r=k+\nu$ and
$M=L_{k,\nu}\prod_{j=1}^k(\nu+j)^{-1}$ (see the paragraph on Taylor models
below).

\begin{restatable}[Operator models]{proposition}{thmholdertensorvi}
\label{thm:holder-tensor-vi}
Under \eqref{eq:vi-model-remainder}, the variant of \SPV{} described below
makes at most $2N$ feasible calls and returns the points $w_1,\dots,w_T\in K$ of
its $T\le N$ trials and weights $\alpha\in\Delta_T$ with
\[
 \max_{u\in K}\sum_{t=1}^T\alpha_t\ip{F(w_t)}{w_t-u}\le\varepsilon_N
 =2M\Bigl(\frac{6\PN}N\Bigr)^{r+1}
 =O_r\Bigl(\frac{MR^{r+1}[\log(2N)\log(2n)]^{(r+1)/2}}{N^{r+1}}\Bigr).
\]
Hence $\bar w=\sum_t\alpha_tw_t$ has weak gap at most $\varepsilon_N$, and saddle
gap at most $\varepsilon_N$ for a convex--concave saddle operator. Each trial
uses a finite model search whose arithmetic cost is not bounded.
\end{restatable}

For $r=1$ this recovers the rate $N^{-2}$ of \cref{thm:vi}(b), with a worse
constant. For $\nu$-H\"older operators ($k=0$) the Euclidean weak gap of
Mirror-Prox-type methods is $O(N^{-(1+\nu)/2})$ \citep{dvurechensky2018}, so on
the $\ell_1$ ball the exponent doubles. For Lipschitz $D^kF$, $k\ge1$,
extragradient-type methods reach weak gap $O(N^{-(k+2)/2})$
\citep{monteiro2012npe,bullins2022higher,lin2025perseus}, and
\citet{zhang2026matching} obtain the tangent residual
$\tO(N^{-(3k+2)/2})$, with a matching lower bound in Euclidean geometry; after
multiplication by the diameter this bounds the weak gap, and against it the
exponent $k+2$ of \cref{thm:holder-tensor-vi} is larger only for $k=1$. A
matching lower bound for \cref{thm:holder-tensor-vi} is known only for
$T_z\equiv F(z)$ and $r=1$ (\cref{thm:vi-lower}); for $r=\nu<1$ the affine
instance of \cref{thm:vi-lower} gives only $\Omega(N^{-2})$, and for Taylor
models it is identified by one call that returns $F$ and $DF$.

\begin{restatable}[Restarts under $\ell_1$-strong monotonicity]{proposition}{thmholdertensorvirestart}
\label{thm:holder-tensor-vi-restart}
Assume \eqref{eq:vi-model-remainder} and
$\ip{F(x)-F(y)}{x-y}\ge\mu_1\nrm{x-y}_1^2$ on $K$ with $\mu_1>0$, and let $\opt$
be the solution of the variational inequality. Let $0<d_\star<R$,
$J=\lceil\log_2(R/d_\star)\rceil$ and $\rho_j=R2^{-j}$, and let $N_j$ be the
smallest power of two with
\begin{equation}\label{eq:vi-holder-restart-horizon}
 N_j^{r+1}\ge8\cdot24^{r+1}\frac{M}{\mu_1}\rho_j^{r-1}(\mathsf b_{N_j}\cn)^{(r+1)/2}.
\end{equation}
Restarting \cref{thm:holder-tensor-vi} on $K_0=K$ and
$K_{j+1}=K\cap(\bar w_j+B_1^n(\rho_{j+1}))$ returns $x\in K$ with
$\nrm{x-\opt}_1\le d_\star$ after at most $2\sum_{j<J}N_j$ feasible calls. With
the condition number $\bar\kappa(\rho)=M\rho^{r-1}/\mu_1$ at scale $\rho$,
$\bar\kappa_\star=\bar\kappa(d_\star)$ for $r<1$ and
$\bar\kappa_\star=\bar\kappa(R)$ for $r\ge1$, and
$\mathcal H_\star=[\cn\log(2+\bar\kappa_\star\cn^{(r+1)/2})]^{1/2}$, this number is
$O_r(J+\mathcal H_\star\bar\kappa(d_\star)^{1/(r+1)})$ for $r<1$,
$O(J[1+\mathcal H_\star(M/\mu_1)^{1/2}])$ for $r=1$, and
$O_r(J+\mathcal H_\star\bar\kappa(R)^{1/(r+1)})$ for $r>1$.
\end{restatable}

\paragraph{Taylor models.} Let $k\ge1$ and let $F$ be $C^k$ on a neighborhood of $K$ with
$\nrm{D^kF(w)-D^kF(z)}\le L_{k,\nu}\nrm{w-z}_2^\nu$, and let $h=w-z$. The
integral form of the remainder,
\[
 F(w)-T_z(w)=\int_0^1\frac{(1-t)^{k-1}}{(k-1)!}\bigl(D^kF(z+th)-D^kF(z)\bigr)[h^k]\,dt,
\]
and the beta integral of \cref{lem:taylor-remainder} give
$\nrm{F(w)-T_z(w)}_2\le L_{k,\nu}\nrm h_2^{k+\nu}\int_0^1\frac{(1-t)^{k-1}t^\nu}{(k-1)!}\,dt
=L_{k,\nu}\prod_{j=1}^k(\nu+j)^{-1}\nrm h_2^{k+\nu}$, which is
\eqref{eq:vi-model-remainder} with $r=k+\nu$.

\paragraph{The model step.} Let $\psi_r(a)=\nrm a_2^{r-1}a$ for $a\ne0$ and
$\psi_r(0)=0$; this map is continuous for every $r>0$. At a selector point $z$
put $\widetilde T_z(w)=T_z(w)+2M\psi_r(w-z)$, the regularized model of higher-order
methods for variational inequalities
\citep{monteiro2012npe,bullins2022higher,lin2025perseus}, and, for $\eta>0$,
find $w\in K$ with
\begin{equation}\label{eq:vi-model-residual}
 \max_{u\in K}\ip{\widetilde T_z(w)}{w-u}\le\eta .
\end{equation}
This condition concerns the model only and is checked by one linear program
over $K$. The method then queries $F(w)$ and uses $F(w)$, not $\widetilde T_z(w)$, in the
cut.

\begin{lemma}[Model search]\label{lem:vi-model-search}
For every $z\in K$ and $\eta>0$, a point satisfying \eqref{eq:vi-model-residual}
exists and is found by a finite enumeration of rational convex combinations of
the vertices of $K$. Monotonicity of $\widetilde T_z$ is not needed.
\end{lemma}

\begin{proof}
The map $w\mapsto\proj_K(w-\widetilde T_z(w))$ sends the compact convex set $K$
continuously into itself, so by Brouwer's theorem it has a fixed point $w^*$.
The optimality condition of the projection gives $\ip{\widetilde T_z(w^*)}{u-w^*}\ge0$ for
all $u\in K$, so the left-hand side of \eqref{eq:vi-model-residual} vanishes at
$w^*$. This left-hand side is the maximum over the vertices $u$ of $K$ of the
continuous functions $w\mapsto\ip{\widetilde T_z(w)}{w-u}$, hence continuous, and rational
convex combinations of the vertices are dense in $K$. Enumerate them, evaluate
the residual with absolute error at most $\eta/8$, and accept the first point
whose computed residual is at most $\eta/2$. An accepted point has residual at
most $5\eta/8$, and every candidate close enough to $w^*$ has residual below
$\eta/8$ and is accepted, so the enumeration stops.
\end{proof}

\begin{lemma}[Displacement forced by a cut]\label{lem:vi-holder-step}
Assume \eqref{eq:vi-model-remainder} and \eqref{eq:vi-model-residual}. If
$z'\in K$ and $\ip{F(w)}{w-z'}\ge\varepsilon$, then
\begin{equation}\label{eq:vi-holder-step}
 \varepsilon\le\eta+M\bigl(3\nrm{z'-z}_2\bigr)^{r+1}.
\end{equation}
\end{lemma}

\begin{proof}
Put $a=w-z$, $b=z'-z$ and $e=F(w)-T_z(w)$, so that
$F(w)=\widetilde T_z(w)-2M\psi_r(a)+e$, $\nrm e_2\le M\nrm a_2^r$ and $w-z'=a-b$. By
\eqref{eq:vi-model-residual} with $u=z'$ and the Cauchy--Schwarz inequality,
\begin{align*}
 \varepsilon\le\ip{F(w)}{a-b}
 &\le\eta-2M\nrm a_2^{r+1}+2M\nrm a_2^r\nrm b_2+M\nrm a_2^r(\nrm a_2+\nrm b_2)\\
 &=\eta+M\nrm a_2^r\bigl(3\nrm b_2-\nrm a_2\bigr).
\end{align*}
If $\nrm a_2\ge3\nrm b_2$, the last term is nonpositive. Otherwise
$\nrm a_2^r\le(3\nrm b_2)^r$ and $3\nrm b_2-\nrm a_2\le3\nrm b_2$.
\end{proof}

\begin{proof}[Proof of \cref{thm:holder-tensor-vi}]
We run \cref{alg:spvi} with the same nested sets and selector accuracy
$\PN/(4N)$, with the projected step replaced by a point $w_t$ satisfying
\eqref{eq:vi-model-residual} at $z=z_t$ with $\eta=\varepsilon_N/2$, and with the
deep cut $C_{t+1}=C_t\cap\{u:\ip{F(w_t)}{w_t-u}\ge\varepsilon_N\}$. Each trial
makes one call at $z_t$, which returns the model, and one at $w_t$; the model
search and the selector make no calls. If $C_1,\dots,C_N$ were all nonempty,
\cref{lem:vi-holder-step} with $z'=z_{t+1}$ would give
$M(6\PN/N)^{r+1}=\varepsilon_N/2\le M(3\nrm{z_{t+1}-z_t}_2)^{r+1}$, so every
displacement would be at least $2\PN/N$ and their sum at least $2\PN$, contrary
to the bound $\frac32\PN$ of \eqref{eq:app-path}. Hence some $C_T$, $T\le N$, is
empty, and the minimax theorem gives $\alpha\in\Delta_T$ with the stated
certificate, computed by a linear program as in \cref{thm:vi}(a). Substituting
$\PN=4R\sqrt{\bN\cn}$ gives the rate, and \cref{lem:mixture-certificate} gives
the weak and the saddle gap.
\end{proof}

\begin{proof}[Proof of \cref{thm:holder-tensor-vi-restart}]
By Brouwer's theorem, applied as in \cref{lem:vi-model-search} with $F$ in place
of $\widetilde T_z$, the variational inequality has a solution, which is unique by strong
monotonicity. Suppose that $K_j\subseteq K$ is a known polytope containing
$\opt$ and contained in an $\ell_1$ ball of radius $\rho_j$ (for $j=0$, the ball
$c+B_1^n(R)$). \Cref{thm:holder-tensor-vi} on $K_j$ with horizon $N_j$ gives a
certificate with right-hand side $2M(24\rho_j\sqrt{\mathsf b_{N_j}\cn}/N_j)^{r+1}$,
which is at most $\mu_1\rho_j^2/4$ by \eqref{eq:vi-holder-restart-horizon}. As in
the proof of \cref{thm:vi}(c), testing the certificate at $u=\opt$ gives
$\mu_1\nrm{\bar w_j-\opt}_1^2\le\mu_1\rho_j^2/4$, so
$\nrm{\bar w_j-\opt}_1\le\rho_{j+1}$ and $K_{j+1}$ contains $\opt$. After $J$
phases the distance is at most $\rho_J\le d_\star$, and phase $j$ uses at most
$2N_j$ calls.

For the explicit count, put $\bar\kappa_j=\bar\kappa(\rho_j)$. Bounding
$\mathsf b_N=O(1+\log N)$ as in the proof of \cref{thm:holder} shows that
\eqref{eq:vi-holder-restart-horizon} holds for every power of two above a
suitable multiple, depending only on $r$, of
$1+\bar\kappa_j^{1/(r+1)}(\cn\log(2+\bar\kappa_j\cn^{(r+1)/2}))^{1/2}$; since
$\bar\kappa_j\le\bar\kappa_\star$ (for $r<1$ because $\rho_{J-1}>d_\star$), this gives
$N_j=O_r(1+\bar\kappa_j^{1/(r+1)}\mathcal H_\star)$. For $r=1$ all $\bar\kappa_j$ equal
$M/\mu_1$. For $r<1$ the sequence $\bar\kappa_j^{1/(r+1)}$ increases geometrically
with ratio $2^{(1-r)/(r+1)}$ and its sum is at most a constant times
$\bar\kappa(d_\star)^{1/(r+1)}$; for $r>1$ it decreases geometrically and its sum
is at most a constant times $\bar\kappa(R)^{1/(r+1)}$. The constants contain
$(1-2^{-|r-1|/(r+1)})^{-1}$ and are not uniform as $r\to1$.
\end{proof}

\subsection{A lower bound for monotone operators}
\label{app:vi-lower}

As noted after \cref{thm:vi}, lower bounds for minimization do not transfer to
the weak gap, so a separate construction is needed. It is affine and
skew-symmetric, so the weak and the strong gap
coincide, and it uses a difference matrix whose rotations are revealed two
columns per query.

\begin{restatable}[Lower bound for monotone operators]{proposition}{thmvilower}
\label{thm:vi-lower}
Let $L,R,R_x,R_y>0$, $N\ge1$ and $m=2N+3$.
\begin{enumerate}[label=(\alph*)]
\item Let $n\ge2m=4N+6$. For every deterministic method that makes at most $N$
global queries of the operator and returns $\widehat z\in B_1^n(R)$ there is an
affine operator $F(z)=Hz+b$ with $H^\top=-H$ and $\nrm H_{2\to2}\le L$, hence
monotone and $L$-Lipschitz on $\R^n$, with a zero $z^\star$ of $F$ in $B_1^n(R/2)$,
such that
\[
 \max_{u\in B_1^n(R)}\ip{F(u)}{\widehat z-u}=\max_{u\in B_1^n(R)}\ip{F(\widehat z)}{\widehat z-u}
 \ \ge\ \frac{LR^2}{8(2N+3)^2}\ \ge\ \frac{LR^2}{200N^2}.
\]
\item Let $n_x,n_y\ge m$. For every deterministic method that makes at most $N$
global queries of $(\Phi,\nabla_x\Phi,\nabla_y\Phi)$ and returns
$(\widehat x,\widehat y)\in B_1^{n_x}(R_x)\times B_1^{n_y}(R_y)$ there is a
bilinear-affine $\Phi(x,y)=\frac L2\ip x{Ay-\gamma e_1}$ with
$\nrm A_{2\to2}\le2$, whose operator $(\nabla_x\Phi,-\nabla_y\Phi)$ is
$L$-Lipschitz, such that the saddle gap
$\max_{y\in B_1^{n_y}(R_y)}\Phi(\widehat x,y)-\min_{x\in B_1^{n_x}(R_x)}\Phi(x,\widehat y)$
is at least
$LR_xR_y/(4(2N+3)^2)$. On $\Delta_{2m}(R_x)\times\Delta_{2m}(R_y)$ the function
$\widetilde\Phi(p,q)=\frac12\Phi(Pp,Pq)$ with $P=(I_m,-I_m)$ has an $L$-Lipschitz
operator and saddle gap at least $LR_xR_y/(8(2N+3)^2)$ for every deterministic
method with at most $N$ global queries.
\end{enumerate}
\end{restatable}

Part (a) shows that the exponent of the rate $\tO(LR^2/N^2)$ of \SPV{}
(\cref{thm:vi}) is optimal for deterministic methods when $n\ge4N+6$. The
logarithmic gap remains open, and the construction gives no randomized bound,
since its last columns depend on the whole transcript of the method. Part (b)
gives the same for the saddle gap of bilinear games on products of $\ell_1$
balls and of simplices, with the operator measured in the spectral norm. The
constant $8$ on simplices comes from the factor $\frac12$ in $\widetilde\Phi$,
which keeps the lifted operator $L$-Lipschitz because
$\nrm{(I,-I)}_{2\to2}=\sqrt2$. At fixed dimension and large
$N$ no such bound can hold, since an affine operator is identified by $n+1$
queries.

\begin{proof}[Proof of \cref{thm:vi-lower}]
Let $m=2N+3$ and let $B\in\R^{m\times m}$ be the lower bidiagonal difference
matrix, $(Bv)_i=v_i-v_{i-1}$ with $v_0=0$. Then $B=I-S$ for the shift $S$
with $\nrm S_{2\to2}\le1$, so $\nrm B_{2\to2}\le2$; $B\1=e_1$; and
$\sum_i(Bv)_i=v_m$ by telescoping.

\emph{A resisting pair of rotations.} Fix a deterministic method with $N$
queries $(x_t,y_t)\in\R^m\times\R^m$, which receives answers of the form
$\bigl(c(Ay_t-\gamma e_1),\,-cA^\top x_t,\,c\ip{x_t}{Ay_t-\gamma e_1}\bigr)$ with
known $c,\gamma>0$ and $A=UBV^\top$, and returns $(\widehat x,\widehat y)$. We
construct orthogonal $U,V$ with $Ue_1=Ve_1=e_1$ column by column. Put
$u_1=v_1=e_1$. Before query $t$ the columns $u_1,\dots,u_{2t-1}$ and
$v_1,\dots,v_{2t-1}$ are fixed, and the transcript so far determines $x_t,y_t$.
Let $u_{2t}$ be the normalized component of $x_t$ orthogonal to
$\spn\{u_1,\dots,u_{2t-1}\}$ (any unit vector orthogonal to that span if the
component vanishes), and $v_{2t}$ likewise for $y_t$; then
$x_t\in\spn\{u_1,\dots,u_{2t}\}$ and $y_t\in\spn\{v_1,\dots,v_{2t}\}$. Choose
$u_{2t+1},v_{2t+1}$ as arbitrary further orthonormal vectors. Now
$V^\top y_t$ is supported on the first $2t$ coordinates, so $BV^\top y_t$ is
supported on the first $2t+1$ coordinates and $Ay_t=U(BV^\top y_t)$ uses only
$u_1,\dots,u_{2t+1}$; similarly $U^\top x_t$ is supported on the first $2t$
coordinates, $B^\top U^\top x_t$ as well, and $A^\top x_t$ uses only
$v_1,\dots,v_{2t}$. Hence the answer to query $t$, including the value, is
determined by the columns already fixed and is unchanged by any later
completion. After $N$ queries, $2N+1=m-2$ columns of each matrix are fixed
and the output $(\widehat x,\widehat y)$ is determined; choose $u_{m-1},v_{m-1}$
from the orthogonal components of $\widehat x,\widehat y$ as above, and complete
with $u_m,v_m$. The resulting single matrix $A$ reproduces the entire transcript
(the method, being deterministic, makes the same queries and the same output
against $A$), and $\widehat y\in\spn\{v_1,\dots,v_{m-1}\}$, i.e.\ $(V^\top\widehat y)_m=0$.
A method that stops early is padded with ignored queries.

\emph{The key inequality.} Put $\widehat v=V^\top\widehat y$. Since $\widehat v_m=0$,
$\sum_i(B\widehat v-\gamma e_1)_i=\widehat v_m-\gamma=-\gamma$, so
$\nrm{B\widehat v-\gamma e_1}_1\ge\gamma$ and $\nrm{B\widehat v-\gamma e_1}_2\ge\gamma/\sqrt m$.
As $Ue_1=e_1$, $A\widehat y-\gamma e_1=U(B\widehat v-\gamma e_1)$, and therefore
\begin{equation}\label{eq:vi-key}
 \nrm{A\widehat y-\gamma e_1}_\infty\ \ge\ \frac{\nrm{A\widehat y-\gamma e_1}_2}{\sqrt m}
 \ \ge\ \frac\gamma m .
\end{equation}

\emph{(b) Saddle gap on two balls.} Let $c=L/2$, $\gamma=R_y/(2m)$ and
$\Phi(x,y)=c\ip x{Ay-\gamma e_1}$ on $\R^m\times\R^m$, which is bilinear-affine,
hence convex--concave. Its operator is $F(x,y)=H(x,y)+b$ with
$H=c\bigl(\begin{smallmatrix}0&A\\-A^\top&0\end{smallmatrix}\bigr)$, $H^\top=-H$
and $\nrm H_{2\to2}=c\nrm A_{2\to2}\le L$, so $F$ is monotone and $L$-Lipschitz.
The point $(x^\star,y^\star)=(0,\gamma V\1)$ satisfies $Ay^\star=\gamma UBV^\top V\1=\gamma UB\1=\gamma e_1$,
so $\Phi(x,y^\star)=0=\Phi(0,y)$ for all $x,y$: it is a saddle point of value $0$,
and $\nrm{y^\star}_1\le\sqrt m\nrm{y^\star}_2=m\gamma=R_y/2$, so it is feasible. For the
output of the method (whose queries receive the answers used in the
construction),
\begin{align*}
 \max_{y\in B_1^m(R_y)}\Phi(\widehat x,y)-\min_{x\in B_1^m(R_x)}\Phi(x,\widehat y)
 &\ge\Phi(\widehat x,y^\star)-\min_{x\in B_1^m(R_x)}c\ip x{A\widehat y-\gamma e_1}\\
 &=cR_x\nrm{A\widehat y-\gamma e_1}_\infty\ \ge\ \frac{LR_xR_y}{4m^2}
\end{align*}
by \eqref{eq:vi-key}. If $n_x>m$ or $n_y>m$, let $\Phi$ depend on the first
$m$ coordinates of each block only: a global query is answered from its first
$m$ coordinates, the padded operator is still $L$-Lipschitz, and the output's
first $m$ coordinates are feasible. Since $2N+3\le5N$, the
bound is at least $LR_xR_y/(100N^2)$.

\emph{Simplices.} Let $P=(I_m,-I_m)$, so that $P\Delta_{2m}(r)=B_1^m(r)$ for
every $r>0$ and $\nrm P_{2\to2}=\sqrt2$, and let
$\widetilde\Phi(p,q)=\frac12\Phi(Pp,Pq)=\frac L4\ip{Pp}{APq-\gamma e_1}$. Its
operator is $\frac L4\bigl(\begin{smallmatrix}0&P^\top AP\\-P^\top A^\top P&0\end{smallmatrix}\bigr)(p,q)+\text{const}$,
of spectral norm at most $\frac L4\cdot2\cdot2=L$. A method for $\widetilde\Phi$
with global queries $(p_t,q_t)$ defines a method for $\Phi$ with the
queries $(Pp_t,Pq_t)$, since the answers for $\widetilde\Phi$ are those for
$\Phi$ multiplied by known matrices; an output
$(\widehat p,\widehat q)$ on the simplices gives the feasible output
$(P\widehat p,P\widehat q)$ on the balls. Because $P$ maps each simplex onto the
corresponding ball, the saddle gap of $(\widehat p,\widehat q)$ for
$\widetilde\Phi$ equals half the saddle gap of $(P\widehat p,P\widehat q)$ for
$\Phi$, which is at least $LR_xR_y/(8m^2)$ by the previous paragraph applied to
the simulating method.

\emph{(a) Weak gap on one ball.} Let $K=B_1^{2m}(R)$, $c=L/2$, $\gamma=R/(2m)$,
and let $F(z)=Hz+b$ be the operator of $\Phi$ above, now on the single ball.
The point $z^\star=(0,\gamma V\1)$ has $F(z^\star)=0$ and $\nrm{z^\star}_1\le R/2$. For any
affine skew-symmetric operator, $\ip{Hu}u=0$ and $\ip{Hu}z=-\ip{Hz}u$ give
\[
 \max_{u\in K}\ip{F(u)}{z-u}=\ip bz+\max_{u\in K}\ip{-(Hz+b)}u=\ip bz+R\nrm{F(z)}_\infty
 =\max_{u\in K}\ip{F(z)}{z-u},
\]
so the weak and the strong gap coincide. As $b=-Hz^\star$ and $\ip{z^\star}{b}=-\ip{z^\star}{Hz^\star}=0$,
$\ip bz=\ip{z^\star}{Hz}=\ip{z^\star}{F(z)}$, hence
$\max_{u\in K}\ip{F(u)}{z-u}\ge(R-\nrm{z^\star}_1)\nrm{F(z)}_\infty\ge\frac R2\nrm{F(z)}_\infty$.
The method queries the full operator, which is the answer form used in the
construction with $x_t,y_t$ the two blocks of the query; its output
$\widehat z=(\widehat x,\widehat y)$ satisfies \eqref{eq:vi-key}, and the first
block of $F(\widehat z)$ is $c(A\widehat y-\gamma e_1)$, so
$\nrm{F(\widehat z)}_\infty\ge c\gamma/m$ and the weak gap is at least
$\frac R2\cdot\frac L2\cdot\frac{R}{2m^2}=\frac{LR^2}{8m^2}\ge\frac{LR^2}{200N^2}$.
For $n>2m$ let $F$ act on the first $2m$ coordinates and vanish on the others;
the operator remains affine, skew-symmetric, monotone and $L$-Lipschitz, the
support function of $B_1^n(R)$ is still $R\nrm\cdot_\infty$, so the gap identity
and the bound hold with the first $2m$ coordinates of the output.
\end{proof}

\section{Convex quadratics}
\label{app:quadratics}

The Steiner-point methods settle the general smooth class on the $\ell_1$ ball
up to logarithmic factors. For convex quadratics a different mechanism,
\emph{curvature learning}, gives the rate $O_p(LR^2/N^{1+2/p})$ without
logarithms, on every $\ell_p$ ball with $1\le p<2$ and in every dimension. The
matching lower bound, and with it optimality of the exponent, needs $n\ge3N+2$
(\cref{thm:power-chain}); some such condition is necessary, since in dimension
$n$ a quadratic is identified by $n+1$ gradient queries \citep{arjevani2015bounds}. Throughout this section $1\le p<2$,
$K\subseteq B_p^n(1)$ (the general radius is recovered by scaling), and
\begin{equation}\label{eq:quadratic-model}
 f(x)=\tfrac12\ip{x}{Qx}+\ip{b}{x},\qquad 0\preceq Q\preceq LI,
\end{equation}
with $L$ known and $Q$ unknown. The method needs $b$ and products $Qd$, both
obtained from gradients: $b=\nabla f(0)$, and $Qd=\nabla f(d)-\nabla f(0)$ for
the global quadratic. Thus $T$ products require $T+1$ first-order queries. The
query at $0$ and the probe points $d$ need not lie in $K$; these are the only
infeasible queries used by an upper bound in this paper; on $B_p^n(R)$ itself the
query at $0$ and the scaled probe $hd$ with $h\nrm d_p\le R$ are feasible.

\subsection{Mechanism}
\label{sec:cla-mechanism}

An accelerated estimate sequence needs a regularizer with quadratic curvature,
while the comparison with the minimizer needs a regularizer that is bounded on
$B_p^n(1)$. For $p<2$ the function $\frac12\nrm{x}_2^2$ is bounded on the ball but
gives only the Euclidean rate. We use, for a threshold $0<\tau\le1$ and
$c_p=\frac{2-p}{2p}$, the \emph{power--Huber function}
\begin{equation}\label{eq:power-huber}
 \psi_{\tau,p}(s)=\begin{cases}\frac12\tau^{p-2}s^2,&|s|\le\tau,\\[2pt]
 \frac1p|s|^p-c_p\tau^p,&|s|>\tau,\end{cases}
 \qquad
 \Psi_I(x)=\sum_{j\notin I}\psi_{\tau,p}(x_j)+\frac1p\sum_{j\in I}|x_j|^p ,
\end{equation}
where $I\subseteq\{1,\dots,n\}$ is the set of \emph{promoted} coordinates.

\begin{lemma}[Power--Huber regularizer]\label{lem:power-huber}
The function $\psi_{\tau,p}$ is convex and continuously differentiable, $\Psi_I$ is
convex, and on $B_p^n(1)$:
\textup{(i)} $0\le \Psi_I(x)\le 1/p$;
\textup{(ii)} $\psi_{\tau,p}(s)\ge s^2/2$ for $|s|\le1$;
\textup{(iii)} $\psi_{\tau,p}$ has curvature $\tau^{p-2}$ on $[-\tau,\tau]$;
\textup{(iv)} if $|x_j|\ge\tau$, moving $j$ into $I$ increases $\Psi_I(x)$ by exactly
$c_p\tau^p$ and does not change $\partial \Psi_I(x)$.
\end{lemma}

Acceleration uses property (iii) and the comparison with the minimizer uses
property (i); the two are compatible only if coordinates that leave
$[-\tau,\tau]$ are promoted. By (iv) a promotion keeps the current model
minimizer optimal and increases the potential by $c_p\tau^p$, so at most
$O(\tau^{-p})$ promotions can occur. The method \emph{learns} curvature: at every
stage it queries the product $Qx_{t-1}$ at the previous iterate and, after a
promotion, the product $Qe_j$ for every newly promoted $j$, and it updates a lower model $0\preceq H\preceq Q$
by the positive semidefinite SR1 step
\begin{equation}\label{eq:psd-sr1}
 v=(Q-H)d,\qquad
 H^+=\begin{cases}H+\dfrac{vv^\top}{\ip{d}{v}},&\ip{d}{v}>0,\\[4pt]
 H,&\ip{d}{v}=0.\end{cases}
\end{equation}
The update preserves $0\preceq H\preceq H^+\preceq Q$, makes $(Q-H^+)d=0$, and
keeps all previously learned null directions of $Q-H$. Consequently the residual
$f-h_t$, $h_t(x)=\frac12\ip{x}{H_tx}+\ip bx$, is a positive semidefinite
quadratic that vanishes on the span of the history, including the accelerated
averages at which the estimate sequence evaluates it. Every stage without a full
accelerated step promotes a coordinate, and every promoted coordinate costs one
product, so $T$ products and $m$ promoted coordinates leave at least $T-2m-1$
accelerated steps; the threshold is tuned so that $\tau^{p}\asymp1/N$ and
$m\le N/3$:
\begin{equation}\label{eq:exponent-arithmetic}
 \underbrace{N^{-2}}_{\text{acceleration}}\cdot
 \underbrace{\tau^{2-p}}_{\text{inverse curvature of }\psi_{\tau,p}}
 \asymp N^{-2}\cdot N^{-(2-p)/p}=N^{-(1+2/p)} .
\end{equation}
At $p=1$ this is the inverse-cubic mechanism of \citet{ouyang2026cubic}, whose
regularizer is the Huber function; for $p>1$ the power branch $|s|^p/p$ makes
(i) and (iv) hold simultaneously.

\subsection{The method and its rate}

For a stage $t\ge1$, a step parameter $\gamma\in[0,1)$ and the current promoted set,
the \emph{model} is
\begin{equation}\label{eq:cla-model}
 m_t(\gamma;x)=L_\tau\,\Gamma_{t-1}(1-\gamma)\,\Psi_{I_{t-1}}(x)+h_t(x),
 \qquad r_t(x)=\max_{j\notin I_{t-1}}|x_j| ,
\end{equation}
with $\max\varnothing=0$, $L_\tau=L\tau^{2-p}$, and we write
$X_t(\gamma)=\argmin_K m_t(\gamma;\cdot)$ and
$Y_t(\gamma)=\{x\in X_t(\gamma):r_t(x)\le\tau\}$.
\Cref{alg:cla} states the method.

\begin{algorithm}[t]
\caption{\CLA: curvature-learning acceleration for quadratics on $B_p^n(1)$}
\label{alg:cla}
\begin{algorithmic}[1]
\Require compact convex $K\subseteq B_p^n(1)$, $1\le p<2$; $L$ and $b$; product
oracle $d\mapsto Qd$; budget $T$ with $T+1\ge 6/(2-p)$
\State $\tau\gets[6/((2-p)(T+1))]^{1/p}$, $L_\tau\gets L\tau^{2-p}$, $\Gamma_0\gets\tau^{p-2}$,
$H_0\gets0$, $\mathsf{calls}\gets0$
\State $x_0\gets$ a minimizer of $L_\tau\Gamma_0\Psi_\varnothing(x)+\ip bx$ over $K$;
$\bar x_0\gets x_0$; $I_0\gets U_0\gets\{j:|x_{0,j}|\ge\tau\}$
\For{$t=1,2,\dots$}
 \State $H\gets H_{t-1}$
 \For{$d=e_j$, $j\in U_{t-1}$, and then $d=x_{t-1}$}
  \State query $Qd$; $\mathsf{calls}\gets\mathsf{calls}+1$
   \Comment{redundant directions are also counted}
  \If{$\mathsf{calls}=T$} \Return $\bar x_{t-1}$ \EndIf
  \State update $H$ by \eqref{eq:psd-sr1}
 \EndFor
 \State $H_t\gets H$; \ $\bar\gamma_t\gets$ the root in $(0,1)$ of
 $\bar\gamma^2=\Gamma_{t-1}(1-\bar\gamma)$
 \If{some $x\in X_t(0)$ has $r_t(x)\ge\tau$}
  \State $\gamma_t\gets0$, $x_t\gets x$ \Comment{promotion step}
 \ElsIf{$Y_t(\bar\gamma_t)\neq\varnothing$}
  \State $\gamma_t\gets\bar\gamma_t$, $x_t\in Y_t(\bar\gamma_t)$ \Comment{full accelerated step}
 \Else
  \State
$\gamma_t\gets\max\{\gamma\in[0,\bar\gamma_t]:Y_t(\gamma)\ne\varnothing\}$;
  $x_t\in Y_t(\gamma_t)$ with $r_t(x_t)=\tau$ \Comment{partial step}
 \EndIf
 \State $\Gamma_t\gets(1-\gamma_t)\Gamma_{t-1}$;\quad
 $\bar x_t\gets(1-\gamma_t)\bar x_{t-1}+\gamma_tx_t$
 \State $U_t\gets\{j\notin I_{t-1}:|x_{t,j}|\ge\tau\}$ if $\gamma_t<\bar\gamma_t$,
 else $U_t\gets\varnothing$;\quad $I_t\gets I_{t-1}\cup U_t$
\EndFor
\end{algorithmic}
\end{algorithm}

The step selection is well defined (\cref{app:cla-selection}), and every
stage with $\gamma_t<\bar\gamma_t$ promotes at least one new coordinate. The
algorithm assumes exact minimization of the known models and an exact choice of
the largest admissible $\gamma$; both assumptions are removed below.

\begin{restatable}[Rate of \CLA]{proposition}{thmcla}
\label{thm:cla}
Let $1\le p<2$ and let $T$ be an integer with $T+1\ge6/(2-p)$. After $T$
products, \cref{alg:cla} returns $\bar x\in K$ with
\begin{equation}\label{eq:cla-rate}
 f(\bar x)-\min_Kf\ \le\ \frac{C_pL}{(T+1)^{2/p}(T-2)},\qquad
 C_p=\frac{9}{2p}\Bigl(\frac{6}{2-p}\Bigr)^{(2-p)/p}.
\end{equation}
For $K\subseteq B_p(c,R)$ and a quadratic with Euclidean smoothness $L$, the method
applied to $u\mapsto f(c+Ru)$ uses at most $T+1$ first-order queries and the
bound becomes $C_pLR^2/[(T+1)^{2/p}(T-2)]$. The quadratic lower bound of
\cref{thm:power-chain} (dimension $n\ge3N+2$) shows that the exponent $1+2/p$
is optimal for every fixed $p<2$.
\end{restatable}

At $p=1$, $C_1=27$ and \eqref{eq:cla-rate} is the bound of
\citet{ouyang2026cubic}. The constant tends to $9/4$ as $p\to2$, while the
smallest admissible $T$, $\lceil6/(2-p)\rceil-1$, diverges; at $p=2$ Euclidean acceleration
gives the matching $N^{-2}$. \Cref{tab:quadratic-exponents} lists some exponents: the
upper bounds are \cref{thm:cla,cor:cla-restart} (Euclidean acceleration for $p=2$),
and the matching lower bounds are \cref{thm:power-chain,thm:sc-constant-reduction}.

\begin{table}[t]
\centering
\caption{Quadratics on $B_p^n(R)$: error after $N$ queries, and queries per
halving of the error for strongly convex quadratics.}
\label{tab:quadratic-exponents}
\begin{tabular}{lcccl}
\toprule
$p$ & $1$ & $4/3$ & $3/2$ & $2$\\
\midrule
convex rate & $N^{-3}$ & $N^{-5/2}$ & $N^{-7/3}$ & $N^{-2}$\\
queries per halving & $\kappa_1^{1/3}$ & $\kappa_{4/3}^{2/5}$ &
$\kappa_{3/2}^{3/7}$ & $\kappa_2^{1/2}$\\
\bottomrule
\end{tabular}
\end{table}

\paragraph{A certified implementation.}
Exact model minimization and the exact largest admissible step can be replaced by
finite procedures without changing the exponent. The certified version of \CLA{}
(\cref{alg:cla-certified} in \cref{app:cla-certified}) makes three changes:
(i) each model, stabilized by $\frac\zeta2\nrm{x}_2^2$, is solved only to an
$\eta$-certificate (\cref{def:inner-certificate}); (ii) the step is selected by a
bisection that stops as soon as the largest unpromoted coordinate lies in the
band $[(1-\sigma)\tau,\tau]$, $\sigma=c_p\tau/16$, and coordinates are promoted
from $(1-\sigma)\tau$ on; (iii) the SR1 matrix is replaced by the ridge model
\begin{equation}\label{eq:ridge-model}
 H_\delta(S)=QV_S\bigl(V_S^\top QV_S+\delta I\bigr)^{-1}V_S^\top Q ,
\end{equation}
where $V_S$ is a matrix whose columns form an orthonormal basis of the learned
span $S$. Unlike \eqref{eq:psd-sr1}, the ridge model needs no exact test
$\ip dv=0$; it satisfies $0\preceq H_\delta(S)\preceq Q$, increases with $S$, and
$\ip v{(Q-H_\delta(S))v}\le\delta\nrm v_2^2$ for $v\in S$ (\cref{lem:ridge}). The
clipped-coordinate inequality \eqref{eq:clipped-huber} shows that an objective
certificate controls the clipped coordinates, which are all the threshold test
needs, although $\psi_{\tau,p}$ has little curvature far from the origin.

For inexact inner minimization we use the following certificate, which can be
checked by one linear minimization over $K$.

\begin{definition}[Certified inner solution]\label{def:inner-certificate}
Let $\mathcal J$ be a known convex function on $K$ and $\eta\ge0$. A triple
$(x,g,\underline s)$ with $x\in K$, $g\in\partial\mathcal J(x)$ and
$\underline s\le\min_{y\in K}\ip{g}{y}$ is an $\eta$-certificate for $\mathcal J$ if
$\ip{g}{x}-\underline s\le\eta$. It implies $\mathcal J(x)-\min_K\mathcal J\le\eta$.
\end{definition}

\begin{restatable}[Certified implementation]{proposition}{thmclacertified}
\label{thm:cla-certified}
Let $1\le p<2$, let $T+1\ge24/(2-p)$, and assume the certificate interface of
\cref{def:inner-certificate} is available on $K$. With the parameters of
\eqref{eq:certified-parameters}, the certified version of \CLA{} performs finitely
many certified inner solves and exact-real comparisons, uses $T$ products, and
returns $\bar x\in K$ with
\begin{equation}\label{eq:cla-certified-rate}
 f(\bar x)-\min_Kf\le\frac{\widetilde C_pL}{(T+1)^{2/p}(T-2)},\qquad
 \widetilde C_p=\frac{18}{p}\Bigl(\frac{24}{2-p}\Bigr)^{(2-p)/p},
 \qquad \widetilde C_1=432.
\end{equation}
\end{restatable}

\subsection{Restarts and the quadratic assumption}

\paragraph{Restarts under quadratic growth.}
Suppose that $f$ satisfies the quadratic growth condition
\begin{equation}\label{eq:quadratic-growth}
 f(x)-\fstar\ge\frac{\mu_p}2\dist_p(x,X^\star)^2\qquad(x\in K),
\end{equation}
where $X^\star=\argmin_Kf$ and $\dist_p(x,X^\star)=\min_{y\in X^\star}\nrm{x-y}_p$;
strong convexity in $\ell_p$ implies it. If
$f(z)-\fstar\le\Delta$, every point of $X^\star$ closest to $z$ lies in
$B_p(z,r)$ with $r^2=2\Delta/\mu_p$, so \CLA{} can be restarted on
$K\cap B_p(z,r)$.

\begin{restatable}[Restarted \CLA]{corollary}{corclarestart}
\label{cor:cla-restart}
Assume \eqref{eq:quadratic-growth}, a point $z_0\in K$ with
$f(z_0)-\fstar\le\Delta_0$, and $0<\varepsilon<\Delta_0$. Let
$\alpha_p=1+2/p$ and
\[
 M_p=\Bigl\lceil\max\Bigl\{\frac{6}{2-p}-1,\
 \bigl(8C_pL/\mu_p\bigr)^{1/\alpha_p}\Bigr\}\Bigr\rceil .
\]
Running \cref{alg:cla} with budget $M_p$ on $K\cap B_p(z_j,r_j)$,
$r_j^2=2^{1-j}\Delta_0/\mu_p$, and setting $z_{j+1}$ to its output, gives an
$\varepsilon$-solution after at most
$(M_p+1)\lceil\log_2(\Delta_0/\varepsilon)\rceil
=O_p\bigl((1+\kappa_p^{p/(p+2)})\log(2\Delta_0/\varepsilon)\bigr)$
first-order queries. No uniqueness of the minimizer is needed.
\end{restatable}

By \cref{thm:sc-constant-reduction}, the power $\kappa_p^{p/(p+2)}$ cannot be
improved for a constant-factor reduction. Restarting under quadratic growth is
standard \citep{fercoq2019restart}.

\paragraph{Why only quadratics.}
The proof of \cref{thm:cla} uses that the residual $f-h_t$ is a fixed positive
semidefinite quadratic vanishing on every learned direction. Off the quadratic
class this invariant can fail in one dimension: for $f=\log\cosh$ the secant
slope learned between $0$ and $a>0$ is
$H=\tanh(a)/a$, while $f''(a)=\operatorname{sech}^2 a<H$, so $f-\frac12Hx^2$ is
not convex near $a$. The Steiner-point methods do not model curvature and store
only first-order cuts, so the difficulty does not arise for them.

\subsection{Proofs for curvature learning}
\label{app:cla}

Throughout this subsection $1\le p<2$, $K\subseteq B_p^n(1)$ is nonempty,
compact and convex, and $f(x)=\frac12\ip x{Qx}+\ip bx$ with $0\preceq Q\preceq LI$.
The vector $b$ and the bound $L$ are known, and one call returns a product
$Qd$. For a global quadratic one call takes one gradient query, and obtaining
$b$ takes one more (see the beginning of this section). Exact minimization of known models
and the exact step selection of \cref{alg:cla} are internal operations;
\cref{app:cla-certified} replaces them by finite procedures.

\subsubsection{The power--Huber regularizer}
\label{app:power-huber}

\begin{proof}[Proof of \cref{lem:power-huber}]
The two branches of $\psi_{\tau,p}$ have the same value and the same derivative
$\pm\tau^{p-1}$ at $s=\pm\tau$, and both are convex; hence $\psi_{\tau,p}$ is
convex and $C^1$, and $\Psi_I$ is convex as a sum of convex functions. Property
(iii) is immediate.

For (i), write $|s|=\tau u$ with $0\le u\le1$ inside the threshold; then
$\frac12\tau^{p-2}s^2=\frac12\tau^pu^2\le\frac1p\tau^pu^p=|s|^p/p$ because
$u^2\le u^p$ and $\frac12\le\frac1p$. Outside the threshold
$\psi_{\tau,p}(s)=|s|^p/p-c_p\tau^p\le|s|^p/p$. Hence
$0\le \Psi_I(x)\le\sum_j|x_j|^p/p\le1/p$ on $B_p^n(1)$; nonnegativity is clear on
both branches, since $\tau^p/p-c_p\tau^p=\tau^p/2>0$ and $|s|^p/p$ increases.

For (ii), note that $\tau^{p-2}\ge1$ inside the threshold. On $[\tau,1]$ the function
$\psi_{\tau,p}(s)-s^2/2$ has derivative $s^{p-1}-s\ge0$ and value
$(\tau^p-\tau^2)/2\ge0$ at $s=\tau$. The case $s<0$ follows by symmetry.

For (iv), if $|x_j|\ge\tau$ then $|x_j|^p/p-\psi_{\tau,p}(x_j)=c_p\tau^p$ (at
$|x_j|=\tau$ both formulas give $\tau^p/2$), and the derivatives of the two
functions coincide at $x_j$; the other terms of $\Psi_I$ do not change.
\end{proof}

\subsubsection{The positive semidefinite SR1 update}

\begin{lemma}\label{lem:psd-sr1}
Let $0\preceq H\preceq Q$, $d\in\R^n$, $v=(Q-H)d$, and define $H^+$ by
\eqref{eq:psd-sr1}. Then $0\preceq H\preceq H^+\preceq Q$, $(Q-H^+)d=0$, and
$(Q-H)u=0$ implies $(Q-H^+)u=0$.
\end{lemma}

\begin{proof}
Write $E=Q-H\succeq0$. If $\ip dv=\ip d{Ed}=0$, the Cauchy--Schwarz inequality
for the semidefinite form $E$ gives $Ed=0$, so $v=0$ and $H^+=H$.
Otherwise $E^+=E-Edd^\top E/\ip d{Ed}$. Cauchy--Schwarz gives
$(u^\top Ed)^2\le(u^\top Eu)(d^\top Ed)$, i.e.\ $E^+\succeq0$, and
$E^+\preceq E$ is clear; thus $H\preceq H^+\preceq Q$. Finally $E^+d=0$, and
$Eu=0$ implies $E^+u=0$.
\end{proof}

\subsubsection{Well-posedness of the step selection}
\label{app:cla-selection}

Fix a stage $t$ and write $X(\gamma)=X_t(\gamma)$, $Y(\gamma)=Y_t(\gamma)$ and
$r=r_t$. The graph of $X$ over $[0,\bar\gamma_t]$ is compact, because $K$ is
compact and the model is jointly continuous in $(\gamma,x)$. If one minimizer
$x\in X(\gamma)$ satisfies $r(x)<\tau$, then every minimizer has the same
unpromoted coordinates. Along the segment to another minimizer the model is
constant, so each of its convex terms is affine there; but the unpromoted
power--Huber terms are strictly convex near $x$ in every direction that
changes an unpromoted coordinate (their coefficient $L_\tau\Gamma_{t-1}(1-\gamma)$ is
positive). By compactness, the strict inequality $r<\tau$ on $X(\gamma)$ persists
for all $\gamma'$ close to $\gamma$. Consequently, if $X(0)$ contains no point with $r\ge\tau$, the set
$\{\gamma:Y(\gamma)\neq\varnothing\}$ is compact and contains a neighborhood of
$0$ in $[0,\bar\gamma_t]$. If it does not contain $\bar\gamma_t$, its largest
element $\gamma_t$ lies in $(0,\bar\gamma_t)$, and some $x\in Y(\gamma_t)$ has
$r(x)=\tau$, since otherwise the strict inequality would persist beyond
$\gamma_t$. Hence every stage with $\gamma_t<\bar\gamma_t$ promotes at least one
new coordinate.

\subsubsection{The estimate sequence and the proof of the rate}

Let $\lambda_t=L_\tau\Gamma_t$ with $L_\tau=L\tau^{2-p}$.

\begin{lemma}[Estimate sequence]\label{lem:cla-potential}
At every completed stage $t\ge0$ of \cref{alg:cla},
\begin{equation}\label{eq:cla-potential}
 f(\bar x_t)-\fstar\le L\tau^{2-p}\Gamma_t\Bigl(\frac1p-c_p\tau^p|I_t|\Bigr),
 \qquad |I_t|\le\frac{2}{(2-p)\tau^p}.
\end{equation}
\end{lemma}

\begin{proof}
By \cref{lem:power-huber}(iv), promotion does not change the subdifferential of
the model at the current point, so $x_t$ minimizes $\lambda_t\Psi_{I_t}+h_t$ over $K$.
Put $W_t=\lambda_t\Psi_{I_t}(x_t)+h_t(x_t)$. If $x$ and $x_t$ have all unpromoted
coordinates in $[-\tau,\tau]$, convex optimality and the quadratic branch of
$\psi_{\tau,p}$ give
\begin{equation}\label{eq:cla-local-curvature}
 \lambda_t\Psi_{I_t}(x)+h_t(x)-W_t\ \ge\ \frac{L\Gamma_t}{2}
 \nrm{(x-x_t)_{I_t^c}}_2^2 .
\end{equation}
Let $\gamma_t>0$; we apply \eqref{eq:cla-local-curvature} at stage $t-1$ with
$x=x_t$. The increase $H_t-H_{t-1}\succeq0$ does not decrease the quadratic part,
and promotion increases the regularizer at $x_t$ by $c_p\tau^p|U_t|$. Multiplying
by $1-\gamma_t$, adding $\gamma_th_t(x_t)$ and using $\Gamma_t\ge\gamma_t^2$,
we obtain
\begin{equation}\label{eq:cla-W}
 W_t\ \ge\ (1-\gamma_t)W_{t-1}+\gamma_th_t(x_t)
 +\frac{L\gamma_t^2}{2}\nrm{(x_t-x_{t-1})_{I_{t-1}^c}}_2^2
 +Lc_p\tau^2\Gamma_t|U_t| .
\end{equation}
For $\gamma_t=0$ the same inequality follows by minimizing the old
model, using $H_t\succeq H_{t-1}$ and adding the promotion increment.

The residual $Q-H_t$ annihilates $x_0,\dots,x_{t-1}$ and $e_j$ for
$j\in I_{t-1}$ (\cref{lem:psd-sr1}), hence also $\bar x_{t-1}$. Expanding the quadratic form of $Q-H_t$ at
$\bar x_t=\bar x_{t-1}+\gamma_t(x_t-x_{t-1})+\gamma_t(x_{t-1}-\bar x_{t-1})$ and
using convexity of $h_t$ gives
\begin{equation}\label{eq:cla-f}
 f(\bar x_t)\ \le\ (1-\gamma_t)f(\bar x_{t-1})+\gamma_th_t(x_t)
 +\frac{L\gamma_t^2}{2}\nrm{(x_t-x_{t-1})_{I_{t-1}^c}}_2^2 .
\end{equation}
(The promoted part of $x_t-x_{t-1}$ lies in the span of the learned coordinates
and does not contribute.) We subtract \eqref{eq:cla-f} from \eqref{eq:cla-W} and
divide by $\Gamma_t$. At initialization, \cref{lem:power-huber}(ii) and
$Q\preceq LI$ give
\[
 \frac{W_0-f(x_0)}{\Gamma_0}
 =\frac{L\Psi_\varnothing(x_0)-\frac12\ip{x_0}{Qx_0}}{\Gamma_0}+Lc_p\tau^2|I_0|
 \ \ge\ Lc_p\tau^2|I_0| .
\]
Telescoping yields $W_t-f(\bar x_t)\ge Lc_p\tau^2\Gamma_t|I_t|$. On the other
hand, optimality of $x_t$, $H_t\preceq Q$ and $\Psi_{I_t}(\opt)\le1/p$ give
$W_t\le\fstar+L\tau^{2-p}\Gamma_t/p$. The two inequalities give the gap bound,
and nonnegativity of its left side gives the cardinality bound.
\end{proof}

\begin{proof}[Proof of \cref{thm:cla}]
Let $J$ be the last completed stage, $m=|I_J|$, and let $r$ be the number of
completed stages with $\gamma_t=\bar\gamma_t$. Every other completed stage promotes
a new coordinate, so $J-r\le m$. Before stopping, at most $J+1$ iterate
directions and $m$ coordinate directions were queried; hence
\[
 T\le J+m+1,\qquad r\ge T-2m-1,\qquad m\le\frac{T+1}3 ,
\]
the last inequality by \eqref{eq:cla-potential} and the choice
$\tau^p=6/((2-p)(T+1))$. If $m=(T+1)/3$, then \eqref{eq:cla-potential} gives zero
error. Otherwise $m\le T/3$ and $T\ge5$ imply $r\ge1$. At a full stage,
$\Gamma_t=\bar\gamma_t^2$ and
\[
 \frac1{\sqrt{\Gamma_t}}=\frac{1+\sqrt{1+4/\Gamma_{t-1}}}{2}
 \ \ge\ \frac1{\sqrt{\Gamma_{t-1}}}+\frac12 ,
\]
and the first full stage gives $\Gamma_t<1$; at the other stages $\Gamma$
does not increase. Thus $\Gamma_J\le4/(r+1)^2\le4/(T-2m)^2$. Put $t_0=T+1-3m>0$
and $u=T-2>0$. By \eqref{eq:cla-potential},
\[
 f(\bar x_J)-\fstar\ \le\ \frac{4L\tau^{2-p}}{p(T+1)}\cdot\frac{t_0}{(T-2m)^2}
 \ \le\ \frac{9L\tau^{2-p}}{2p(T+1)(T-2)},
\]
where the last step uses $T-2m=(u+2t_0)/3$ and $(u+2t_0)^2\ge8ut_0$.
Substituting $\tau$ proves \eqref{eq:cla-rate}, since
$1+(2-p)/p=2/p$.

For $K\subseteq B_p(c,R)$ we substitute $x=c+Rz$; the Hessian becomes $R^2Q\preceq LR^2I$
and the linear coefficient $R\nabla f(c)$. One query at $c$ gives this
coefficient, and each product requires one query at $c+Rd$, since
$R(\nabla f(c+Rd)-\nabla f(c))=R^2Qd$. Optimality of the exponent follows from
\cref{thm:power-chain}, whose hard function is a convex quadratic.
\end{proof}

\subsubsection{A certified implementation}
\label{app:cla-certified}

In the certified version of \cref{alg:cla} the inner problems are solved only
up to the certificate of \cref{def:inner-certificate}, the step is chosen by a
finite bisection, and exact null directions are replaced by a ridge-regularized
model. Products $Qd$ are exact. ``Finite'' refers to the number of inner calls and
exact-real comparisons, not to the arithmetic cost of the certificates.

\paragraph{Inner certificates.}
If $(x,g,\underline s)$ is an $\eta$-certificate for a convex $\mathcal J$, then
for every $y\in K$, $\mathcal J(y)\ge\mathcal J(x)+\ip g{y-x}\ge\mathcal J(x)-\eta$.
For the stabilized objective $\widetilde{\mathcal J}=\mathcal J+\frac\zeta2\nrm x_2^2$, $\zeta>0$, the exact minimizer $X$ is
unique and a certified point satisfies $\nrm{x-X}_2\le\sqrt{2\eta/\zeta}$.
Removing the stabilizer increases the certificate by at most $\zeta/4$, because
$\ip x{y-x}\le\nrm x_2-\nrm x_2^2\le\frac14$ for $x,y\in B_2(1)$. On the full
balls and on the simplex such certificates can be found by finite enumeration of
rational points, because the continuous VI gap tends to zero along a sequence
converging to the minimizer. On a general compact convex $K$ the certificate
interface is assumed.

\paragraph{The ridge model.}
Let $S$ be the span of the learned directions, $V_S$ a matrix whose columns form
an orthonormal basis of $S$, and $\delta>0$ fixed for the whole run. Define
$H_\delta(S)$ by \eqref{eq:ridge-model} ($H_\delta(\{0\})=0$) and
$h_S(x)=\frac12\ip x{H_\delta(S)x}+\ip bx$. The model is determined by the
products $QV_S$, and each new orthogonalized direction needs one new product.

\begin{lemma}[Ridge model]\label{lem:ridge}
Let $0\preceq Q$ and $\delta>0$. Then $H_\delta(S)$ does not depend on the choice
of the basis $V_S$, $0\preceq H_\delta(S)\preceq Q$, $H_\delta(S)\preceq H_\delta(S')$
whenever $S\subseteq S'$, and $0\le\ip v{(Q-H_\delta(S))v}\le\delta\nrm v_2^2$ for
every $v\in S$.
\end{lemma}

\begin{proof}
Let $A=Q^{1/2}$ and $B=AV_S$. The push-through identity gives
\[
 H_\delta(S)=A^\top B\bigl(B^\top B+\delta I\bigr)^{-1}B^\top A
 =A^\top\bigl[I-\delta\bigl(BB^\top+\delta I\bigr)^{-1}\bigr]A ,
 \qquad BB^\top=A\Pi_SA ,
\]
where $\Pi_S=V_SV_S^\top$ is the orthogonal projector onto $S$. This depends on
$S$ only. The matrix in brackets lies between $0$ and $I$, which gives
$0\preceq H_\delta(S)\preceq A^\top A=Q$. If $S\subseteq S'$, then
$\Pi_S\preceq\Pi_{S'}$, hence $(A\Pi_SA+\delta I)^{-1}\succeq(A\Pi_{S'}A+\delta I)^{-1}$
and $H_\delta(S)\preceq H_\delta(S')$. Finally, for $v=V_Sy\in S$ put
$G=B^\top B\succeq0$; then $\nrm y_2=\nrm v_2$ and
\[
 \ip v{(Q-H_\delta(S))v}=y^\top\bigl[G-G(G+\delta I)^{-1}G\bigr]y
 =\delta\,y^\top G(G+\delta I)^{-1}y\ \le\ \delta\nrm y_2^2 ,
\]
since $0\preceq G(G+\delta I)^{-1}\preceq I$.
\end{proof}

\paragraph{Parameters.}
For an integer $T$ with $T+1\ge24/(2-p)$ we use
\begin{equation}\label{eq:certified-parameters}
\begin{aligned}
 \tau^p&=\frac{24}{(2-p)(T+1)},& L_\tau&=2L\tau^{2-p},&
 \sigma&=\frac{c_p\tau}{16},& a&=(1-\sigma)\tau,&
 \xi&=\frac{\sigma\tau}8,\\
 e&=\frac{L_\tau}{p(T+1)^3},& \zeta&=e,& \delta&=\frac e{18},&
 \eta&=\min\Bigl\{\frac e4,\frac{L\xi^2}{(T+1)^2}\Bigr\}. &&
\end{aligned}
\end{equation}
Here $a$ is the lower promotion threshold and $\xi$ a positional tolerance. For
the current learned model $h_t=h_S$ let
\[
 J_t(\gamma;x)=L_\tau\Gamma_{t-1}(1-\gamma)\Psi_{I_{t-1}}(x)+h_t(x)+\frac\zeta2\nrm x_2^2,
\]
and let $\sub{Solve}_t(\gamma)$ return an $\eta$-certified point for $J_t(\gamma;\cdot)$.
\Cref{alg:cla-certified} states the certified version of \CLA.

\begin{algorithm}[t]
\caption{Certified implementation of \CLA}
\label{alg:cla-certified}
\begin{algorithmic}[1]
\Require as in \cref{alg:cla}, with $T+1\ge24/(2-p)$; exact products; the
certificate interface for the known models
\State set the parameters \eqref{eq:certified-parameters}; $S\gets\{0\}$,
$\mathsf{calls}\gets0$, $\Gamma_0\gets\tau^{p-2}$
\State $x_0\gets$ an $\eta$-certified minimizer of
$L_\tau\Gamma_0\Psi_\varnothing+\ip b\cdot+\frac\zeta2\nrm\cdot_2^2$;
$\bar x_0\gets x_0$; $I_0\gets U_0\gets\{j:|x_{0,j}|\ge a\}$
\For{$t=1,2,\dots$}
 \For{$d=e_j$, $j\in U_{t-1}$, and then $d=x_{t-1}$}
  \State query $Qd$; $\mathsf{calls}\gets\mathsf{calls}+1$;
  \If{$\mathsf{calls}=T$} \Return $\bar x_{t-1}$ \EndIf
  \State $S\gets\spn(S\cup\{d\})$
 \EndFor
 \State build $h_t=h_S$ from the stored products; $\bar\gamma_t\gets$ root of
 $\bar\gamma^2=\Gamma_{t-1}(1-\bar\gamma)$
 \State $\gamma_t\gets0$; $x_t\gets\sub{Solve}_t(0)$
 \If{$r_t(x_t)<a$}
  \State $\gamma_t\gets\bar\gamma_t$; $x_t\gets\sub{Solve}_t(\bar\gamma_t)$
  \If{$r_t(x_t)>\tau$}
   \State $\gamma_-\gets0$, $\gamma_+\gets\bar\gamma_t$
   \Repeat
    \State $\gamma_t\gets(\gamma_-+\gamma_+)/2$; $x_t\gets\sub{Solve}_t(\gamma_t)$
    \State if $r_t(x_t)<a$ then $\gamma_-\gets\gamma_t$; if $r_t(x_t)>\tau$ then $\gamma_+\gets\gamma_t$
   \Until{$a\le r_t(x_t)\le\tau$}
  \EndIf
 \EndIf
 \State $\Gamma_t\gets(1-\gamma_t)\Gamma_{t-1}$; $\bar x_t\gets(1-\gamma_t)\bar x_{t-1}+\gamma_tx_t$
 \State $U_t\gets\{j\notin I_{t-1}:|x_{t,j}|\ge a\}$ if $\gamma_t<\bar\gamma_t$, else
 $U_t\gets\varnothing$; $I_t\gets I_{t-1}\cup U_t$
\EndFor
\end{algorithmic}
\end{algorithm}

\paragraph{Clipped coordinates.}
Let $\operatorname{clip}_\tau(s)=\max\{-\tau,\min\{s,\tau\}\}$. The Huber function
$\tau\psi_{\tau,1}(s)=s\operatorname{clip}_\tau(s)-\operatorname{clip}_\tau(s)^2/2$
has derivative $\operatorname{clip}_\tau$, and checking the cases $y$ inside,
above and below the threshold interval gives
\[
 \tau D_{\psi_{\tau,1}}(y,x)=\tfrac12\bigl(\operatorname{clip}_\tau(y)-\operatorname{clip}_\tau(x)\bigr)^2
 +\bigl(y-\operatorname{clip}_\tau(y)\bigr)\bigl(\operatorname{clip}_\tau(y)-\operatorname{clip}_\tau(x)\bigr),
\]
where the second term is nonnegative. The difference
$\psi_{\tau,p}-\tau^{p-1}\psi_{\tau,1}$ is convex, since it vanishes on
$[-\tau,\tau]$ and its derivative $\sign(s)(|s|^{p-1}-\tau^{p-1})$ is
nondecreasing outside. Hence
\begin{equation}\label{eq:clipped-huber}
 D_{\psi_{\tau,p}}(y,x)\ \ge\ \tfrac12\tau^{p-2}\,
 \bigl|\operatorname{clip}_\tau(y)-\operatorname{clip}_\tau(x)\bigr|^2 .
\end{equation}
Every tested step $\gamma$ has $\Gamma_{\rm trial}:=\Gamma_{t-1}(1-\gamma)\ge(T+1)^{-2}$, because
$\Gamma_{\rm trial}^{-1/2}$ increases by at most one per stage and is at most one
initially. Thus the coefficient of the unpromoted quadratic curvature is at least
$L_\tau\Gamma_{\rm trial}\tau^{p-2}\ge2L/(T+1)^2$, and \eqref{eq:clipped-huber}
together with the choice of $\eta$ shows that the clipped unpromoted coordinates of
a certified point are within Euclidean distance $\xi$ of those of the exact
stabilized minimizer $X(\gamma)$.

\paragraph{The bisection is finite.}
Subgradients of $\Psi_I$ on $B_p^n(1)$ have Euclidean norm at most
$n^{1/p-1/2}$. Strong monotonicity of the stabilized optimality condition gives
$\nrm{X(\gamma)-X(\gamma')}_2\le L_\tau\Gamma_{t-1}n^{1/p-1/2}|\gamma-\gamma'|/\zeta$, and the
clipped maximum $\widehat r(X)=\min\{r_t(X),\tau\}$ is Lipschitz with the same
constant. A sampled left endpoint has $\widehat r(X)<a+\xi$ and a sampled right
endpoint has $\widehat r(X)\ge\tau-\xi$; these differ by more than
$\sigma\tau-2\xi>0$, so an interval shorter than
$\zeta(\sigma\tau-2\xi)/(L_\tau\Gamma_{t-1}n^{1/p-1/2})$ cannot have endpoints of both types,
and the bisection stops. With rational thresholds
$a_{\rm rat}\in(a,a+\sigma\tau/8)$ and $t_{\rm rat}\in(\tau-\sigma\tau/8,\tau)$ all
comparisons become rational and the same argument applies.

\begin{proof}[Proof of \cref{thm:cla-certified}]
For $|s|\ge a$, promotion adds at least $\pi_\sigma\tau^p$ to the regularizer, where
$\pi_\sigma=\frac{(1-\sigma)^p}p-\frac{(1-\sigma)^2}2\ge c_p-\frac{\sigma^2}2$, because on
$[1-\sigma,1]$ the derivative of $r^p/p-r^2/2$ lies in $[0,1-r]$. Promotion does not
change the subgradient above $\tau$ and changes it by at most $\sigma\tau^{p-1}$
inside the band. Hence the returned point has VI residual, for the unstabilized promoted
model, at most
\begin{equation}\label{eq:certified-residual}
 \bar\eta_t=\alpha+2\lambda_t\sigma\tau^{p-1}|U_t|,\qquad
 \alpha=\eta+\zeta/4\le e/2,
\end{equation}
where $\lambda_t=L_\tau\Gamma_t$ and $\nrm{y-x}_\infty\le2$ was used.

Let $W_t=\lambda_t\Psi_{I_t}(x_t)+h_t(x_t)$ and $u_t=(x_t-x_{t-1})_{I_{t-1}^c}$. For a
positive step, the local curvature inequality with error $\bar\eta_{t-1}$ and the
monotonicity of the ridge models give
\[
 W_t\ \ge\ (1-\gamma_t)W_{t-1}+\gamma_th_t(x_t)+L\gamma_t^2\nrm{u_t}_2^2
 +L_\tau\Gamma_t\pi_\sigma\tau^p|U_t|-(1-\gamma_t)\bar\eta_{t-1},
\]
where we used $\Gamma_t\ge\gamma_t^2$; the curvature coefficient is now $2L$. For a zero
step the same inequality follows from the global VI residual.

The learned span contains all previous iterates and all $e_j$, $j\in I_{t-1}$.
Hence $v=\bar x_t-\gamma_tu_t\in S$ and $\nrm v_2\le1+2\gamma_t\le3$. Young's
inequality in the seminorm of $Q-H_\delta(S)$ and the defect bound of
\cref{lem:ridge} give
$\frac12\ip{\bar x_t}{(Q-H_\delta(S))\bar x_t}\le L\gamma_t^2\nrm{u_t}_2^2+9\delta$,
and therefore
\[
 f(\bar x_t)\ \le\ (1-\gamma_t)f(\bar x_{t-1})+\gamma_th_t(x_t)
 +L\gamma_t^2\nrm{u_t}_2^2+9\delta,\qquad 9\delta=\frac e2 .
\]
At initialization $\Psi_\varnothing(x_0)\ge\nrm{x_0}_2^2/2$ and $Q\preceq LI$ give
$W_0-f(x_0)\ge L_\tau\Gamma_0\pi_\sigma\tau^p|I_0|$. Subtracting the two inequalities,
dividing by $\Gamma_t$, telescoping and using $W_J\le\fstar+L_\tau\Gamma_J/p+\bar\eta_J$, we get
\begin{equation}\label{eq:certified-ledger}
 f(\bar x_J)-\fstar\ \le\ L_\tau\Gamma_J\Bigl(\frac1p-\pi_\sigma\tau^pm\Bigr)
 +\Gamma_J\sum_{j=0}^J\frac{\bar\eta_j}{\Gamma_j}
 +\Gamma_J\sum_{t=1}^J\frac{9\delta}{\Gamma_t},\qquad m=|I_J| .
\end{equation}
The promotion part of $\bar\eta_j$ contributes
$2L_\tau\Gamma_J\sigma\tau^{p-1}m$, and $\pi_\sigma-2\sigma/\tau\ge c_p/2$. For every allowed
step $1/\sqrt{\Gamma_t}\le1/\sqrt{\Gamma_{t-1}}+1$, so $\Gamma_j\ge(T+1)^{-2}$ for
$j\le J\le T$, and the remaining normalized errors sum to at most
$e(T+1)^3=L_\tau/p$. Hence
\[
 f(\bar x_J)-\fstar\ \le\ \frac{2L_\tau\Gamma_J}{p}\Bigl(1-\frac{3m}{T+1}\Bigr),
\]
so $m\le(T+1)/3$, and the counting argument of \cref{thm:cla} gives
$\Gamma_J\le4/(r+1)^2$ and the bound $9L_\tau/[p(T+1)(T-2)]$. Substituting $L_\tau$
and $\tau$ proves \eqref{eq:cla-certified-rate}.
\end{proof}

\subsubsection{Restarts under quadratic growth}
\label{app:cla-restart}

\begin{proof}[Proof of \cref{cor:cla-restart}]
Let $\Delta_j=2^{-j}\Delta_0$ and assume $f(z_j)-\fstar\le\Delta_j$. For the
analysis only, let $x_j^\star\in X^\star$ be a point closest to $z_j$ in $\ell_p$. By
\eqref{eq:quadratic-growth}, $\nrm{z_j-x_j^\star}_p\le r_j$ with
$r_j^2=2\Delta_j/\mu_p$, so $K_j=K\cap B_p(z_j,r_j)$ contains a global minimizer and
$\min_{K_j}f=\fstar$. \Cref{thm:cla} with $T=M_p$ applies on $K_j$, and
$M_p\ge6/(2-p)-1\ge5$. For $T\ge5$,
$(T+1)^{2/p}(T-2)\ge T^{\alpha_p}/2$, so the error is at most
\[
 \frac{2C_pLr_j^2}{M_p^{\alpha_p}}=\frac{4C_pL\Delta_j}{\mu_pM_p^{\alpha_p}}
 \ \le\ \frac{\Delta_j}2 .
\]
Each stage uses $M_p$ products and one query for the linear coefficient. After
$\lceil\log_2(\Delta_0/\varepsilon)\rceil$ stages the gap is at most $\varepsilon$.
Only the common value $\fstar$ enters the induction, so the minimizer need not be
unique.
\end{proof}

\subsection{Lower bounds for quadratics}
\label{app:quadratic-lower}

\begin{proposition}[Quadratic lower bound]
\label{thm:power-chain}
Let $N\ge1$, $n\ge3N+2$, $1\le p\le2$ and $R,L>0$, and put $n_0=3N+2$, $m=2N+1$ and
$S_m=\sum_{j=1}^mj^2$. For every deterministic method with at most $N$ global
first-order queries and output $\widehat x\in B_p^n(R)$ there is a convex
quadratic $f$ with $0\preceq\nabla^2f\preceq LI$ and a global minimizer in
$B_p^n(R)$ such that
\begin{equation}\label{eq:mixed-jet-power}
 f(\widehat x)-\min_{B_p^n(R)}f
 \ \ge\ \frac{L(N+1)R^2n_0^{1-2/p}}{8S_m}
 \ \ge\ \frac{LR^2}{8\cdot3^{2+2/p}(N+1)^{1+2/p}} .
\end{equation}
\end{proposition}

\begin{proof}
Let $B_mz=(z_1-z_2,\ldots,z_m-z_{m+1})$ with $z_{m+1}=0$; it is invertible and
$\nrm{B_m}_{2\to2}\le2$. For $a>0$ put
$\Phi(z)=\frac12\nrm{B_mz}_2^2-a z_1$. In the variables $w=B_mz$ we have
$z_1=\sum_iw_i$, so every term $\frac12w_i^2-aw_i$ is minimized at $w_i=a$, and
\begin{equation}\label{eq:chain-minimizers-values}
 z_i^\star=(m+1-i)a,\qquad
 \min\Phi=-\frac m2a^2,\qquad
 \min_{z_j=0\ (j>t)}\Phi=-\frac t2a^2 .
\end{equation}

\emph{Resisting oracle.} All directions are chosen in the span of the first $n_0$
coordinate vectors. At the query $x_t$, choose a unit vector $v_t$ in this span
perpendicular to $x_1,\ldots,x_t$ and $v_1,\ldots,v_{t-1}$; these are at most
$2t-1<n_0$ linear constraints, and later directions are also perpendicular to all
earlier queries. Hence $\ip{v_j}{x_t}=0$ for $j\ge t$, and since $B_m^\top B_m$ is
tridiagonal, the gradient of $\Phi$ at $V^\top x_t$ involves only the directions
$v_1,\dots,v_t$ chosen so far. The value and the gradient at $x_t$ can therefore
be returned consistently with every completion. After the output $\widehat x$,
complete $v_1,\ldots,v_N$ by $N+1$ orthonormal directions perpendicular to all
queries, to the output and to the first $N$ directions; the available dimension
is at least $n_0-(2N+1)=N+1$.

\emph{The hard function.} Let $V=[v_1\ \cdots\ v_m]$, $r=Rn_0^{1/2-1/p}$,
$a=r/\sqrt{S_m}$ and $f(x)=\frac L4\Phi(V^\top x)$. Since $V^\top V=I$ and
$\nrm{B_m}_{2\to2}\le2$, $f$ is a convex quadratic with
$0\preceq\nabla^2f\preceq LI$, and it reproduces every answer. Its minimizer
$\opt=Vz^\star$ has Euclidean norm $a\sqrt{S_m}=r$ and at most $n_0$ nonzero
coordinates, so $\nrm{\opt}_p\le n_0^{1/p-1/2}r=R$ by \eqref{eq:sparse-inclusion}.
The last $m-N=N+1$ chain coordinates $\ip{v_j}{\widehat x}$ of the output vanish,
so \eqref{eq:chain-minimizers-values} gives
\[
 f(\widehat x)-\min f\ \ge\ \frac L4\cdot\frac{(m-N)a^2}{2}
 =\frac{L(N+1)r^2}{8S_m},
\]
which is the first bound in \eqref{eq:mixed-jet-power}. Finally $n_0\le3(N+1)$,
$1-2/p\le0$ and $S_m\le m^3\le[3(N+1)]^3$ give the second.
\end{proof}

For strongly convex quadratics, a tridiagonal chain with an explicit inverse
gives an exponential tail.

\begin{proposition}[Explicit exponential first-order tail]
\label{cor:sc-exponential}
Let $L>\mu_2>0$, $R>0$, $n\ge3N+2$, $1\le p\le2$ and $N\ge1$.
Put $n_0=3N+2$ and
$\rho=(\sqrt{L/\mu_2}-1)/(\sqrt{L/\mu_2}+1)\in(0,1)$.
For every deterministic $N$-query first-order algorithm on $B_p^n(R)$
there is a globally $L$-smooth, globally Euclidean
$\mu_2$-strongly convex quadratic with a feasible global minimizer, on
which the error of the algorithm is at least
\begin{equation}\label{eq:mixed-strong-exponential}
 \frac{\mu_2 R^2n_0^{1-2/p}}4(1+\rho)^2\rho^{2N}
                  (1-\rho^{2N+2}).
\end{equation}
\end{proposition}
\begin{proof}
Set $\theta=(L-\mu_2)/4$, $m=2N+1$,
$H_j=\mu_2 I_j+\theta B_j^\top B_j$, and
$s_j=e_1^\top H_j^{-1}e_1$.
The diagonal of $H_j$ is $(\mu_2+\theta,\mu_2+2\theta,\ldots,
\mu_2+2\theta)$ and its neighboring off-diagonal entries are $-\theta$,
so $H_N$ is the leading principal block of $H_m$, also
when $N=1$.
The identity $\mu_2=\theta(1-\rho)^2/\rho$ and substitution into this
tridiagonal system give, for $1\le i\le j$,
\[
 (H_j^{-1}e_1)_i=
 \frac{c}{1+\rho^{2j+1}}
    (\rho^{i-1}-\rho^{2j-i+1}),
 \qquad c=\frac{\rho}{\theta(1-\rho)}.
\]
To check this, note that the two geometric sequences solve each interior
recurrence, their difference has zero coordinate $j+1$, and the first row
equals one with the displayed coefficient.  Thus
\[
 s_j=c\frac{1-\rho^{2j}}{1+\rho^{2j+1}},\qquad
 s_m-s_N=
 \frac{c(1+\rho)\rho^{2N}(1-\rho^{2(m-N)})}
      {(1+\rho^{2m+1})(1+\rho^{2N+1})}.
\]
Let $w=H_m^{-1}e_1$, $r=Rn_0^{1/2-1/p}$, and choose the linear
coefficient $b=r/\nrm{w}_2$ in
$f(x)=\theta\nrm{B_mV^\top x}_2^2/2-b\langle v_1,x\rangle
+\mu_2\nrm{x}_2^2/2$.
Its global minimizer is $bVw$, of norm $r$, hence feasible.
The resisting oracle is that of \cref{thm:power-chain}, and the hidden
chain coordinates of the output vanish.  The full and truncated
unconstrained minima are $-b^2s_m/2$ and $-b^2s_N/2$, respectively, so the
error is at least $r^2(s_m-s_N)/(2\nrm{w}_2^2)$.
Finally,
\[
 \nrm{w}_2^2\le
 \frac{c^2}{(1+\rho^{2m+1})^2(1-\rho^2)},\qquad
 \frac{1+\rho^{2m+1}}{1+\rho^{2N+1}}\ge\frac12.
\]
Combining these inequalities and using
$c^{-1}(1-\rho^2)(1+\rho)=\mu_2(1+\rho)^2$ proves
\eqref{eq:mixed-strong-exponential}, since $m-N=N+1$.
\end{proof}

\subsubsection{A matching condition-number power for constant reduction}
\label{app:sc-matching}

The explicit chain also bounds the initial gap, which gives a
condition-number power matching restarted \CLA{} for a constant-factor
reduction.

\begin{proposition}[Strong quadratic constant-reduction lower bound]
\label{thm:sc-constant-reduction}
Fix $1\leq p\leq2$, $N\geq1$, $R>0$, and $\mu_p>0$, and put
\[
 n_0=n=3N+2,\qquad L_N=\mu_pn_0^{\beta_p}(2N+1)^2 .
\]
For every deterministic algorithm using at most $N$ first-order queries
on $B_p^{n_0}(R)$, there is a globally $L_N$-smooth quadratic that is globally
$\mu_p$-strongly convex in $\nrm{\cdot}_p$, has a feasible unconstrained
minimizer, and satisfies
\begin{equation}\label{eq:strong-quadratic-constant-gap}
 f(0)-f^\star\leq\frac{100}{27}\mu_pR^2<4\mu_pR^2,
 \qquad
 f(\widehat x)-f^\star\geq\frac{3}{64}\mu_pR^2.
\end{equation}
The algorithm may be given the feasible start $0$ and the certificate
$\Delta_0=4\mu_pR^2$ in advance.  All gradient queries, including any
query at zero, are counted; queries anywhere in $\R^{n_0}$ are allowed.
Consequently $N$ queries cannot guarantee error at most $\Delta_0/128$.
\end{proposition}
\begin{proof}
Set $\mu_2=\mu_pn_0^{\beta_p}$, $m=2N+1$,
$\theta=(L_N-\mu_2)/4$, and $r=Rn_0^{-\beta_p/2}$.
Use the construction of \cref{cor:sc-exponential},
with $H_m=\mu_2I+\theta B_m^\top B_m$, $w=H_m^{-1}e_1$,
$b=r/\nrm{w}_2$, and
\[
 f(x)=\tfrac\theta2\nrm{B_mV^\top x}_2^2-b\ip{v_1}{x}
                                     +\tfrac{\mu_2}{2}\nrm{x}_2^2.
\]
Its minimizer $bVw$ has Euclidean norm $r$ and $p$-norm at most $R$.
Furthermore,
$\mu_2\nrm{h}_2^2\geq\mu_2n_0^{-\beta_p}\nrm{h}_p^2
=\mu_p\nrm{h}_p^2$, proving the claimed global strong convexity.
The Hessian upper bound is $\mu_2+4\theta=L_N$.

The parameter of \cref{cor:sc-exponential} is now
$\rho=N/(N+1)$.  For every integer $N\geq1$,
\begin{equation}\label{eq:rho-constant-bounds}
 \rho\geq\tfrac12,\qquad \rho^{2N}\geq\tfrac19,
 \qquad \rho^{2N+2}\leq\tfrac14.
\end{equation}
The binomial theorem and $k!\geq2^{k-1}$ for $k\geq1$ give
$(1+1/N)^N\leq\sum_{k=0}^N1/k!<3$, and
Bernoulli's inequality gives
$(1+1/N)^{N+1}\geq1+(N+1)/N>2$.
Substituting \eqref{eq:rho-constant-bounds} into
\eqref{eq:mixed-strong-exponential}, and using $\mu_2r^2=\mu_pR^2$,
yields
\[
 f(\widehat x)-f^\star
 \geq\frac{\mu_pR^2}{4}(1+\rho)^2\rho^{2N}
                                  (1-\rho^{2N+2})
 \geq\frac{3}{64}\mu_pR^2.
\]

It remains to bound the initial gap without a condition-number
factor.  The inverse formula in the preceding proof, with
$c=(1-\rho)/\mu_2$, gives
\[
 w_i=\frac{c(\rho^{i-1}-\rho^{2m-i+1})}{1+\rho^{2m+1}},
 \qquad s_m=w_1\leq c,
 \qquad f(0)-f^\star=\frac{r^2s_m}{2\nrm{w}_2^2}.
\]
For $1\leq i\leq N+1$, the ratio of the subtracted geometric term to
the first is $\rho^{2m-2i+2}\leq\rho^{2N+2}\leq1/4$.
Also $\rho^{2m+1}=\rho^{4N+3}\leq1/4$.
Consequently $w_i\geq(3/5)c\rho^{i-1}$, and
\[
 \nrm{w}_2^2\geq\frac9{25}c^2\sum_{j=0}^N\rho^{2j}
 \geq\frac{27c^2}{100(1-\rho^2)}.
\]
It follows that
\[
 f(0)-f^\star
 \leq\frac{50r^2(1-\rho^2)}{27c}
 =\frac{50}{27}\mu_2r^2(1+\rho)
 \leq\frac{100}{27}\mu_pR^2.
\]
This certificate is independent of the hidden resisting rotation.
Finally $\Delta_0/128=\mu_pR^2/32<3\mu_pR^2/64$, proving the
constant-reduction assertion.
\end{proof}

The condition $n=n_0$ is used here, because the conversion to $\ell_p$ strong
convexity holds only in $n_0$ ambient coordinates.

\section{Experimental details}
\label{app:experiments}

Except for the study of \cref{app:server-study}, all runs used Python~3.11 with
NumPy, SciPy (HiGHS) and Matplotlib on a two-core 2.8\,GHz Intel Xeon with
7\,GB of memory. The randomized runs use fixed seeds; only in the largest
shared-direction runs can a repeated run differ, by up to a few percent, because
HiGHS may return different optimal support points.

Each budget $N$ defines its own instance, including the dimension, so slopes
against $N$ describe budget-indexed families, not the trajectory of one method on
one objective. The sparse quadratic family is deterministic, the spread quadratic
family uses three rotations per budget, and the Steiner experiments use one seed
per budget; these runs do not estimate a worst-case expectation over the
randomized selector. The study of \cref{app:server-study} adds fixed instances,
repeated instance seeds and more baselines.

\subsection{Setting and further results}
\label{app:experiments-main}

Instances, constants and baselines are those of \cref{sec:experiments}; the
quadratic comparison also uses $p=4/3$ and $3/2$. \AGDE{} is FISTA
\citep{beck2009fista}, \AGDEnt{} is the accelerated method with the entropy
prox-function on the simplex lift \citep{dvurechensky2017triangles}, \PSG{} and
\MDE{} return their best iterate, and \EG{} is the extragradient method of
\citet{korpelevich1976}. ``\PSG{} + mixture LP'' and ``\EG{} + mixture LP'' apply
the certificate program of \cref{alg:spcut} to the points of \PSG{} and \EG.

\paragraph{Steiner paths.}
\Cref{fig:main-mechanism}(a) uses the sign chain
$C_t=K\cap\{x:x_i\le0,\ i\le t\}$, $t=0,\dots,n$, in the unit balls $K$ of $\ell_1$
and $\ell_2$, and the faces of \cref{prop:path-tight} in $B_1^n(1)$. For the sign
chain, $\st{C_t}=\alpha_t(e_1+\dots+e_t)$ by symmetry, with scalars $\alpha_t$ that
we estimate from $40\,000$ Gaussian samples; the path length is then
$\sum_t((t-1)(\alpha_t-\alpha_{t-1})^2+\alpha_t^2)^{1/2}$. For the faces the path is
computed exactly. Over $n=8,16,\dots,4096$ the sign chain travels between $0.687$
and $0.710$ in the $\ell_1$ ball and between $0.467\sqrt n$ and $0.491\sqrt n$ in the
$\ell_2$ ball. The faces travel $8.337$ at $n=4096$, between the bounds $7.90$ and
$8.89$ of \cref{prop:path-tight}, and the path budget at $N=n=4096$ is
$\PN=53.96$.

\paragraph{Smooth chains.}
The instances (\cref{app:experiments-instances}) are Nesterov's chains, quadratic
or nonquadratic, in dimension $n=2N+1$, rotated so that the minimizer is a fixed
$5$-sparse point. \Cref{fig:main-mechanism}(b) shows the quadratic chain. \SPL{}
runs at the exact level $\ell=\fstar$ with horizon $N$ and $m=256$ directions,
which separates the rate of the level method from the search for the level. On
both chains \AGDE, \AGDEnt{} and \PG{} remove a small fraction of the initial gap,
and so does \CLA{} on the quadratic chain at these budgets. \SPL{} falls below the
best error attainable in the span of the first $N$ chain coordinates at $N=64$.
Its fitted slope is $-2.32$ on the quadratic chain ($8\le N\le256$) and $-2.31$ on
the nonquadratic one ($8\le N\le128$), and at $N=256$ its error is
$6.90\cdot10^{-8}$, against $2.52\cdot10^{-7}$ for \AGDE. The nonquadratic chain is
nearly quadratic near its minimizer: its initial gaps and baseline errors differ
from those of the quadratic chain by a relative amount of at most $10^{-4}$, and
the errors of \SPL{} on the two chains agree within $5\%$. At $N=256$ the
guarantee of \cref{thm:splevel} is $1.7\cdot10^4$ times the initial gap.

\begin{figure}[t]
\centering
\includegraphics[width=\linewidth]{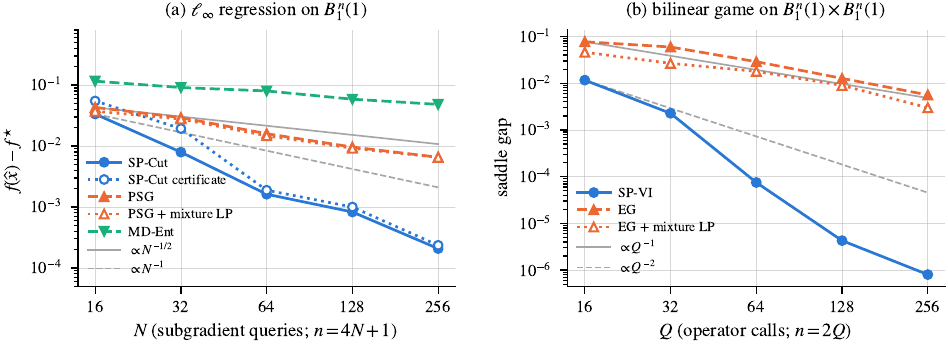}
\caption{(a) $\ell_\infty$ regression, $n=4N+1$: error after $N$ subgradient
queries. (b) Bilinear games, $n=2Q$: saddle gap after $Q$ operator calls.}
\label{fig:sp-nonsmooth-vi}
\end{figure}

\paragraph{Nonsmooth objectives.}
On $\ell_\infty$ regression (\cref{fig:sp-nonsmooth-vi}a) the error of \SPC{} with
$m=32$ directions decays with the exponent $-1.8$, against $-0.7$ for \PSG{} and
$-0.3$ for \MDE; at $N=256$ ($n=1025$) it is $2.1\cdot10^{-4}$, against
$6.6\cdot10^{-3}$ for \PSG. The mixture certificate improves \PSG{} by less than
$10\%$ for $N\ge32$, so the gain comes from the query points and not from the
certificate. The certified value $\widehat\varepsilon$ of \SPC{} bounds the true
error on every run and is within a factor $1.25$ of it for $N\ge64$. The
best-iterate baselines use additional function values to select their output.

\paragraph{Saddle problems.}
On the bilinear games (\cref{fig:sp-nonsmooth-vi}b) the saddle gap of \SPV{}
decays faster than $Q^{-2}$ for $16\le Q\le256$ (fitted exponent $-3.7$, local
slope $-2.4$ over the last doubling). At $Q=256$ it is $8\cdot10^{-7}$, against
$5.6\cdot10^{-3}$ for \EG, whose gap decays like $Q^{-1}$; the mixture
certificate improves \EG{} by a factor between $1.4$ and $2.2$. Localization
probably benefits from the unique interior saddle point of these games. Because
the objective is bilinear, the certificate equals the saddle gap of the returned
mixture.

\begin{figure}[t]
\centering
\includegraphics[width=\linewidth]{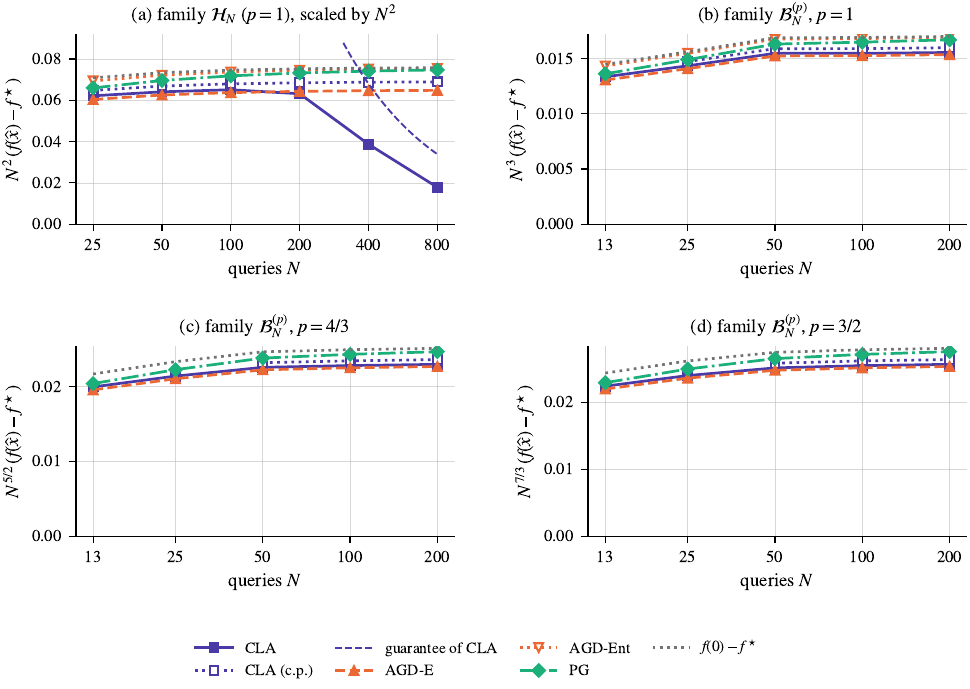}
\caption{Quadratics: error after $N$ gradient queries, multiplied by $N^2$ on
$\mathcal H_N$ (a) and by $N^{1+2/p}$ on $\mathcal B_N^{(p)}$ (b--d).}
\label{fig:rates}
\end{figure}

\paragraph{Quadratics on $\ell_p$ balls.}
\Cref{fig:rates} uses two quadratic families with $L=1$
(\cref{app:experiments-cla}): $\mathcal H_N$, a rotated chain with the $1$-sparse
minimizer $0.9e_1$ and $n=2N+1$, and $\mathcal B^{(p)}_N$, the chain of
\cref{thm:power-chain} in dimension $n=3N+2$ under a random rotation, with a
minimizer of $\ell_p$-norm $0.9$ spread over all coordinates. \CLA{} (c.p.), for
``certified parameters'', is \CLA{} with the parameters of
\cref{thm:cla-certified}. On $\mathcal H_N$ the Euclidean methods stay close to
the initial gap, and the entropic method does not help. The bound of
\cref{thm:cla}, which becomes $27/(N-3)$ after multiplication by $N^2$ (dashed in
\cref{fig:rates}(a)), falls below the plateau $N^2(f(0)-\fstar)\approx0.076$ near
$N=360$; the error of
\CLA{} is already slightly below that of \AGDE{} at $N=200$, and at $N=800$ it is
$0.275$ times that of \AGDE. On $\mathcal B^{(p)}_N$ no method removes more than
$10\%$ of the initial gap, which is dotted in \cref{fig:rates}(b--d), and the fitted
exponents of all methods agree with $1+2/p$ (\cref{tab:exponents}), as the lower
bound predicts. \CLA{} (c.p.) runs only for $N\ge24/(2-p)$, so its slopes for
$p=4/3$ and $3/2$ are fitted on $N\in\{50,100,200\}$; the relative errors in the
last row agree for the three values of $p$ up to $0.01$.

\begin{table}[t]
\centering
\caption{Family $\mathcal B_N^{(p)}$, $25\le N\le200$: fitted slopes of
$\log(f(\widehat x)-\fstar)$ against $\log N$, and the median of
$(f(\widehat x)-\fstar)/(f(0)-\fstar)$.}
\label{tab:exponents}
\begin{tabular}{lccccccc}
\toprule
 & predicted & \CLA & \CLA{} (c.p.) & \AGDE & \AGDEnt & \PG & $f(0)-\fstar$\\
\midrule
$p=1$ & $-3$ & $-2.97$ & $-2.97$ & $-2.97$ & $-2.96$ & $-2.95$ & $-2.97$\\
$p=4/3$ & $-2.5$ & $-2.47$ & $-2.49$ & $-2.47$ & -- & $-2.45$ & $-2.47$\\
$p=3/2$ & $-2.33$ & $-2.30$ & $-2.32$ & $-2.30$ & -- & $-2.29$ & $-2.30$\\
\midrule
relative error & & $0.92$ & $0.94$ & $0.90$ & $0.99$ & $0.97$ & $1$\\
\bottomrule
\end{tabular}
\end{table}

\subsection{Implementation of the Steiner-point methods}
\label{app:experiments-sp}

\paragraph{Support points.}
A localization set is $C=K\cap\{x:\ip{a_j}x\le\beta_j,\ j\le k\}$, where $K$ is the
unit $\ell_1$ ball or a product of two such balls. The support problem
$\max_{x\in C}\ip zx$ is solved through its dual linear program, which has one
variable per cut and per ball and $2n$ constraints; the support point is read off
the dual multipliers. HiGHS works in floating point, so support points and
certificates are numerical approximations, which we check against the primal
problem. An unbounded dual means that $C$ is empty; a primal feasibility problem
confirms it. One support problem takes about $40$\,ms at $n=257$ with $k\le128$
cuts and about $0.4$\,s at $n=1025$ with $k\le256$.

\paragraph{Selector.}
The approximate Steiner point is the average $\bar x=\frac1m\sum_iV_i$ of
\cref{thm:random-selector} over $m$ Gaussian directions. Within one level call,
or one run of \SPC{} or \SPV, all centers use the same $m$ directions. The average
stays feasible up to LP errors, and sharing the directions reduces fluctuations
between successive centers (\cref{app:selector-robustness}). Because later
localization sets then depend on the directions, the guarantee of
\cref{thm:random-selector}, which uses fresh samples, does not apply. Its schedule
$m_N$ would require at least $32N^2/(\bN\cn)\approx1.9\cdot10^4$ directions per
center at $N=256$; the level runs use $m=256$.

\paragraph{\SPL{} and \SPA.}
\SPL{} is \cref{alg:splevel} with the threshold $\delta_N=4\PN/N$; at $n=2N+1$ it
equals $10.6$, $6.3$, $3.7$, $2.1$, $1.2$ and $0.65$ for $N=8,\dots,256$, so it
exceeds the diameter $2$ of $K$ for $N\le64$. Every trial was accepted at every
budget. At $N=256$ the largest displacement of the center is $0.063$ and its total
path is $2.09$, against the path budget $\PN=41.6$. \SPA{} is \cref{alg:spaccel}
with three changes that keep every certificate valid. (i) After each level call
the value of the returned point is queried and counted; $\fhi$ is the best measured
value, and $\flo$ is the largest of the certificates and of the bundle lower bound
$\min_{u\in K}\max_j[f(x_j)+\ip{g_j}{u-x_j}]$ over the stored queries. (ii) The
horizon is doubled whenever a level call at width $w$ returns a point with value
above $\ell+w/4$. The level method succeeds once the horizon reaches
\eqref{eq:horizon-rule}, so this costs at most twice that horizon per phase plus
one query per retry. (iii) Each level call starts from $K$ intersected with the
level cuts of all stored queries, and a level whose initial polytope is empty is
certified without a query. The output is the best measured point.

\paragraph{\SPC{} and \SPV.}
Both methods choose the cut depth from the certificate. After every query the
mixture certificate $\widehat\varepsilon$ of the stored cuts is recomputed (a
linear program with one variable per stored point), the depth $\varepsilon$ is
replaced by $\min\{\varepsilon,\widehat\varepsilon/2\}$, and the cuts are rebuilt
when the depth changes. By the minimax identity of \cref{thm:spcut}, the
localization set is never empty at a depth below $\widehat\varepsilon$, and the
returned mixture is the certified one, so the certificate bounds the error or gap
whatever the accuracy of the selector. The depth $4G\PN/N$ of the analysis exceeds
$GR$ for $N\le128$ and cannot be used at these budgets. \SPV{} uses the step
$w=\proj_K(z-F(z)/(2L))$ of \cref{alg:spvi}, the cut $\ip{F(w)}{w-u}\ge\varepsilon$
and $m=32$ directions.

\subsection{Instances}
\label{app:experiments-instances}

\paragraph{Smooth chains.}
Let $B$ be the chain difference operator $(By)_i=y_i-y_{i+1}$, $i<n$, $(By)_n=y_n$,
used in the proof of \cref{thm:power-chain}, and $A=\frac14B^\top B$, so that
$0\prec A\preceq I$. Let $w$ be the unit vector proportional to $A^{-1}e_1$, the
chain minimizer, for which $Bw$ is a constant vector; let $u$ be the unit vector
proportional to $(1,-0.7,0.5,-0.3,0.2,0,\dots,0)$ and $U$ the Householder
reflection with $Uw=u$. With $\opt=0.9u/\nrm u_1$, the objective is
\[
 f(x)=\tfrac14\textstyle\sum_i\psi((BUx)_i)-\ip{g_{\mathrm{lin}}}x,\qquad
 g_{\mathrm{lin}}=\tfrac14U^\top B^\top\psi'(BU\opt),
\]
so that $\opt$ is the unconstrained minimizer and $\nrm{\opt}_1=0.9$. For
$\psi(t)=t^2/2$ this is Nesterov's chain $\frac12\ip x{Qx}-\ip{Q\opt}x$ with
$Q=UAU$; for $\psi(t)=\sqrt{1+t^2}-1$ the Hessian is at most $I$. In both cases
$\psi'(BU\opt)$ is a constant vector, so $g_{\mathrm{lin}}$ is proportional to
$U^\top e_1$, and for every method whose iterates lie in the span of the past
gradients, the $k$th gradient is supported on the first $k$ coordinates of $Ux$.
The minimum over the span of the first $N$ of these coordinates has the gap
$\frac{n-N}{n}(f(0)-\fstar)$, the dotted line of \cref{fig:main-mechanism}(b); it
bounds only methods whose outputs stay in that span. The nonquadratic objective is
nearly quadratic on the region explored: $BU\opt=t_\star\1$ with
$t_\star\approx2\cdot10^{-4}$ at $N=128$, where $\psi''(t_\star)\approx1-5\cdot10^{-8}$.

\paragraph{$\ell_\infty$ regression.}
$f(x)=\nrm{M(x-\opt)}_\infty$, where $M\in\R^{4n\times n}$ is Gaussian with unit
rows ($G=1$), $\opt=(0.5,-0.3,0.1,0,\dots,0)$ and $n=4N+1$; the oracle returns
$\pm M_i^\top$ for a row $i$ of largest residual. \PSG{} and \MDE{} use the steps
$R/(G\sqrt N)$ and $R\sqrt{2\log(2n)}/(G\sqrt N)$, the latter on the lift
$w\mapsto R(w_+-w_-)$ of $\Delta_{2n}$.

\paragraph{Bilinear games.}
$\Phi(x,y)=\ip x{My}+\ip bx-\ip cy$ on $B_1^n(1)\times B_1^n(1)$, where $M$ is
Gaussian scaled to $\nrm M_2=1$ (so $L=1$), $b=-My^\star$ and $c=M^\top x^\star$ for
the $3$-sparse pair $x^\star=(0.5,-0.3,0.1,0,\dots,0)$,
$y^\star=(-0.4,0.35,0.15,0,\dots,0)$, which is then the unique saddle point; $n=2Q$
for a budget of $Q$ calls. The saddle gap of $(\bar x,\bar y)$ is
$\ip b{\bar x}+\ip c{\bar y}+\nrm{M^\top\bar x-c}_\infty+\nrm{M\bar y+b}_\infty$. \EG{}
uses the step $1/(2L)$ and Euclidean projections and returns the average of its
extrapolated points.

\subsection{The complete bracket search}
\label{app:experiments-full}

\begin{table}[t]
\centering
\caption{\SPA{} with $m=64$ after $N$ joint queries on the chains, $n=2N+1$.}
\label{tab:sp-full}
\setlength{\tabcolsep}{5pt}\begin{tabular}{@{}llccccc@{}}
\toprule
chain & $N$ & $n$ & $f(0)-\fstar$ & \AGDE & \SPA{} (best) & last level call\\
\midrule
quadratic & 32 & 65 & $1.80\cdot10^{-5}$ & $1.54\cdot10^{-5}$ & $1.80\cdot10^{-5}$ & $5.68\cdot10^{-5}$\\
quadratic & 64 & 129 & $4.63\cdot10^{-6}$ & $3.95\cdot10^{-6}$ & $4.63\cdot10^{-6}$ & $9.64\cdot10^{-6}$\\
quadratic & 128 & 257 & $1.17\cdot10^{-6}$ & $1.00\cdot10^{-6}$ & $1.17\cdot10^{-6}$ & $1.52\cdot10^{-6}$\\
nonquadratic & 32 & 65 & $1.80\cdot10^{-5}$ & $1.54\cdot10^{-5}$ & $1.80\cdot10^{-5}$ & $5.77\cdot10^{-5}$\\
nonquadratic & 64 & 129 & $4.63\cdot10^{-6}$ & $3.95\cdot10^{-6}$ & $4.63\cdot10^{-6}$ & $9.93\cdot10^{-6}$\\
nonquadratic & 128 & 257 & $1.17\cdot10^{-6}$ & $1.00\cdot10^{-6}$ & $1.17\cdot10^{-6}$ & $1.52\cdot10^{-6}$\\
\bottomrule
\end{tabular}
\end{table}

\Cref{tab:sp-full} reports the errors of \SPA, which also searches for the level,
of \AGDE{} after the same number of queries, and of the output of the last level
call. The initial point has a small gap, while the bracket starts with the width
$R\nrm{g_0}_\infty$, several orders of magnitude larger. The phases with horizons
$2,4,\dots$ reduce the width, but the outputs of their level calls stay above the
initial gap, so the best measured point is the initial point, whose gap exceeds
the error of \AGDE{} by about $17\%$. The last level call, with horizon $N/2$, was
stopped by the budget after $12$, $27$ and $58$ trials for $N=32,64,128$; its
output approaches the initial gap as $N$ grows and is within $30\%$ of it at
$N=128$. With $m=64$ directions even the exact-level method stays near $0.75$ of
the initial gap (\cref{app:selector-robustness}), so the overhead of the search
and the effect of the sample size cannot be separated here; with $m=256$ at fixed
$n=65$ the complete method does separate from the baselines
(\cref{app:server-smooth}). Larger budgets were not run because of the cost of the
linear programs.

\subsection{The gradient-only method}
\label{app:experiments-prox}

\begin{table}[t]
\centering
\caption{\SPP{} on the nonquadratic chain with $n=65$, stages $j\ge12$.}
\label{tab:sp-prox}
\setlength{\tabcolsep}{4pt}
\begin{tabular}{rrrrrrcc}
\toprule
$j$ & $\delta_j$ & trials & accepted & gradients & total & certificate & $f(y)-\fstar$\\
\midrule
12 & $1.22\cdot10^{-4}$ & 2 & 0 & 1\,600 & 5\,709 & $1.03\cdot10^{-4}$ & $1.74\cdot10^{-5}$\\
13 & $6.10\cdot10^{-5}$ & 4 & 1 & 2\,574 & 8\,283 & $5.81\cdot10^{-5}$ & $1.67\cdot10^{-5}$\\
14 & $3.05\cdot10^{-5}$ & 15 & 14 & 10\,360 & 18\,643 & $3.01\cdot10^{-5}$ & $1.29\cdot10^{-5}$\\
15 & $1.53\cdot10^{-5}$ & 24 & 22 & 16\,164 & 34\,807 & $1.52\cdot10^{-5}$ & $6.32\cdot10^{-6}$\\
16 & $7.63\cdot10^{-6}$ & 27 & 26 & 21\,014 & 55\,821 & $7.30\cdot10^{-6}$ & $2.77\cdot10^{-6}$\\
17 & $3.81\cdot10^{-6}$ & 24 & 23 & 19\,539 & 75\,360 & $3.74\cdot10^{-6}$ & $1.29\cdot10^{-6}$\\
18 & $1.91\cdot10^{-6}$ & 27 & 25 & 23\,280 & 98\,640 & $1.87\cdot10^{-6}$ & $5.80\cdot10^{-7}$\\
19 & $9.54\cdot10^{-7}$ & 28 & 27 & 27\,132 & 125\,772 & $9.42\cdot10^{-7}$ & $2.47\cdot10^{-7}$\\
20 & $4.77\cdot10^{-7}$ & 31 & 30 & 30\,616 & 156\,388 & $4.63\cdot10^{-7}$ & $9.66\cdot10^{-8}$\\
\bottomrule
\end{tabular}
\end{table}

\Cref{tab:sp-prox} reports \SPP{} (\cref{alg:spprox}) with $32$ directions shared by
the centers of a stage and the Euclidean projection in the inner
projected-gradient loop, on the nonquadratic chain with $n=65$ and $L=R=1$,
started at $0$ ($f(0)-\fstar=1.80\cdot10^{-5}$). For each stage $j$ it lists the
cut depth $\delta_j=2^{-j-1}$, the numbers of trials and of accepted proximal steps,
the gradient queries of the stage and in total, and the certificate and the true
error of the stage output. The method uses no function values; the true errors are
computed afterward. The stages $j\le11$ end at their first trial with an empty
cut, which certifies that the current point is $\delta_j$-optimal without
comparing values. Every stage ended with an empty cut after at most $31$ trials,
far below the worst-case allowance, and almost every trial with a nonempty cut
was an accepted proximal step, so the small-step case of \eqref{eq:dichotomy},
which the path bound controls, was almost never used. The certificate bounds the
true error at every stage, by a factor between $2.3$ and $4.8$ for $j\ge14$. The
cost lies in the inner loops: each trial runs a bisection whose steps are loops of
projected-gradient steps, and at $n=65$ the stages with more than one trial use
between $6\cdot10^2$ and $10^3$ gradients per trial. The final error $10^{-7}$
takes $1.6\cdot10^5$ gradients at $n=65$; at $n=17$ and $33$, $16$ and $18$ stages
take $5.0\cdot10^4$ and $8.9\cdot10^4$ gradients and reach $1.5\cdot10^{-6}$ and
$3.8\cdot10^{-7}$. In gradient evaluations the method is therefore not competitive
with the joint-oracle methods or with \AGDE; the experiment supports the mechanism
and the certificates, not a practical advantage.

\subsection{Implementation of \CLA{} and the quadratic families}
\label{app:experiments-cla}

\paragraph{Instance families.}
Both families use the matrix $A$ of \cref{app:experiments-instances} in dimension
$n$. For $\mathcal H_N$, $n=2N+1$, $Q=UAU$, where $U$ is the Householder reflection
that maps $A^{-1}e_1/\nrm{A^{-1}e_1}_2$ to $e_1$, and
$f(x)=\frac12\ip x{Qx}-\ip{Q\opt}x$ with $\opt=0.9e_1$; in the coordinates $Ux$
the linear term is a multiple of $e_1$, so Krylov-type methods have error of order
$\nrm{\opt}_2^2/N^2$ for $N\le(n-1)/2$. For $\mathcal B_N^{(p)}$, $n=3N+2$,
$Q=UAU^\top$ with a random orthogonal $U$, and $\opt$ is proportional to
$UA^{-1}e_1$ with $\nrm{\opt}_p=0.9$; then $f(0)-\fstar$ decays like
$N^{-(1+2/p)}$. Three rotations are used for each pair $(p,N)$, and medians are
reported. All methods receive exactly $N$ gradient queries; for \CLA{} these are
the query at $0$ and $N-1$ products.

\paragraph{Learned model.}
The directions requested by \cref{alg:cla} are orthonormalized against the current
basis $V_S$ of the learned span $S$ by two Gram--Schmidt passes; a direction whose
residual has norm below $10^{-10}$ is redundant but still counted as a query. For
a new unit vector $d$, the product $Qd=(\nabla f(hd)-\nabla f(0))/h$ with
$h=1/\max\{1,\nrm d_p\}$ uses the feasible query $hd\in B_p^n(1)$. The learned
matrix is the ridge model \eqref{eq:ridge-model}, stored as $EE^\top$ with
$E\in\R^{n\times\dim S}$.

\paragraph{Model minimization.}
Every model has the form
$J(x)=\lambda\Psi_I(x)+\frac12\nrm{E^\top x}_2^2+\ip bx+\frac\zeta2\nrm x_2^2$ on
$B_p^n(1)$ with a weight $\lambda>0$. We maximize the concave dual function
\[
 \varphi(u)=-\tfrac12\nrm u_2^2+\min_{x\in B_p^n(1)}\bigl\{\lambda\Psi_I(x)+\ip{b+Eu}x+\tfrac\zeta2\nrm x_2^2\bigr\},
 \qquad u\in\R^{\dim S},
\]
by Newton's method with backtracking. The inner problem is separable up to the
ball constraint and is solved by a scalar search for the multiplier of the
constraint, with closed-form coordinate roots for $p\in\{1,\frac43,\frac32\}$;
the generalized Jacobian of its solution map gives the Newton matrix. At the pair
$(x(u),u)$ the duality gap equals $\frac12\nrm{E^\top x(u)-u}_2^2$, and it bounds
$J(x(u))-\min J$ when the inner minimization is exact. \CLA{} stops when this gap
is at most $10^{-13}\max\{1,|\varphi(u)|\}$, and \CLA{} (c.p.) when it is at most
$\eta$ of \eqref{eq:certified-parameters}; consecutive solves are warm started.
This rule does not check the residual required by \cref{def:inner-certificate},
so \CLA{} (c.p.) is not a certified realization of \cref{thm:cla-certified}.

\paragraph{Step selection and parameters.}
The step parameter is selected as in \cref{alg:cla,alg:cla-certified}: the model
at $\gamma=0$ is solved first; if its largest unpromoted coordinate is at least
$(1-\sigma)\tau$, the stage promotes; otherwise the model at $\bar\gamma_t$ is
solved, and if its largest unpromoted coordinate exceeds $\tau$, a bisection on
$\gamma$ stops once this coordinate lies in $[(1-\sigma)\tau,\tau]$. Both variants
use $\zeta=L_\tau/(p(T+1)^3)$ and the ridge parameter $\zeta/18$ of
\eqref{eq:certified-parameters}. \CLA{} (c.p.) uses all parameters of
\eqref{eq:certified-parameters}; \CLA{} uses $\tau^p=6/((2-p)(T+1))$ and
$L_\tau=L\tau^{2-p}$ of \cref{alg:cla} and $\sigma=10^{-6}$. The stabilizer then
changes the objective by at most $\zeta/2=O(L\tau^{2-p}T^{-3})$ on $B_p^n(1)$,
below the guarantee of \cref{thm:cla}. On $\mathcal H_N$, \CLA{} made only full
steps for $N\le100$; for $N\ge200$ one stage was a partial step that promoted one
coordinate. On one core, a run on $\mathcal H_{800}$ took about 21 minutes for
\CLA{} and 2.5 minutes for \CLA{} (c.p.).

\paragraph{Baselines.}
\AGDE{} uses the indicator of $B_p^n(1)$, the step $1/L$ with $L=\lambda_{\max}(Q)$,
and its last iterate as output. The projection onto $B_1^n(1)$ uses sorting
\citep{duchi2008projection}, and the projection onto $B_p^n(1)$, $p>1$, the
separable solver above. \PG{} is $x_{k+1}=\proj_K(x_k-\nabla f(x_k)/L)$ from
$x_0=0$. \AGDEnt{} is the similar-triangles method with $\theta_k=2/(k+2)$ on the
lift $w\mapsto w_+-w_-$ of $\Delta_{2n}$ onto $B_1^n(1)$, with entropic mirror steps
of size $1/(L_1\theta_k)$ from the uniform distribution, where
$L_1=\max_{ij}|Q_{ij}|$ is the $\ell_1$ smoothness constant of the lifted function;
on the nonquadratic chain it uses the constant of the quadratic chain, which is
larger.

\subsection{Sensitivity to selector sampling}
\label{app:selector-robustness}

We repeat the exact-level runs on both chains at $N\in\{16,32,64\}$, $n=2N+1$, with
$64$ directions per center and five seeds. Each seed is used twice: with directions
shared by all centers of the level call, and with a fresh batch at every center;
every run returns its last accelerated iterate. All $60$ runs use exactly $N$
queries, stay feasible to within $10^{-10}$ and accept every trial.

\begin{figure}[t]
\centering
\includegraphics[width=\linewidth]{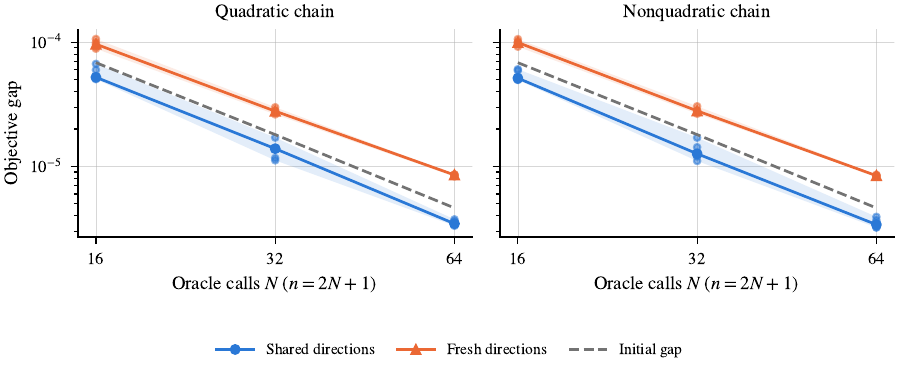}
\caption{Exact-level \SPL{} with $64$ shared or fresh directions per center:
medians, ranges and all runs over five seeds.}
\label{fig:selector-robustness}
\end{figure}

\Cref{fig:selector-robustness} shows a strong sampling effect. At $N=64$ on the
quadratic chain the median gap is $3.46\cdot10^{-6}$ with shared directions
(range $3.33$--$3.75\cdot10^{-6}$) and $8.53\cdot10^{-6}$ with fresh ones
($8.31$--$8.62\cdot10^{-6}$), for an initial gap of $4.63\cdot10^{-6}$; on the
nonquadratic chain the medians are $3.40\cdot10^{-6}$ and $8.42\cdot10^{-6}$. With
fresh directions the last iterate is worse than the initial point at every budget
and for every seed; with shared directions the median stays between $0.70$ and
$0.77$ of the initial gap. Fresh batches satisfy the independence assumption of
\cref{thm:random-selector}, but $64$ directions are far below its schedule. Shared
directions shorten the path of the centers (median $2.06$ against $4.98$ at
$N=64$) and couple later localization sets to the sample. At this sample size
neither variant shows the rate of the level method.

\paragraph{Larger fresh samples.}
\Cref{tab:selector-fresh} repeats the quadratic-chain runs with $m$ fresh
directions per center and three seeds. The relative error decreases with $m$, and
for $m=1024$ also with $N$: the fitted slope over $N=16,32,64$ is $-2.29$, against
$-2.08$ for $m=256$ and $-2.42$ for the shared runs with $m=256$ of
\cref{fig:main-mechanism}(b). The spread over seeds is at most $0.09$, and every
trial was accepted. With $N+1$ selector calls and failure probability $0.1$, the
schedule of \cref{thm:random-selector} requires about $1.3\cdot10^3$,
$4.1\cdot10^3$ and $1.4\cdot10^4$ directions for $N=16,32,64$, so only $m=4096$ at
$N=16$ meets it.

\begin{table}[t]
\centering
\caption{Exact-level \SPL{} on the quadratic chain: median of
$(f(\widehat x)-\fstar)/(f(0)-\fstar)$ with $m$ fresh directions per center.}
\label{tab:selector-fresh}
\begin{tabular}{lccc}
\toprule
directions per center & $N=16$ & $N=32$ & $N=64$\\
\midrule
$64$ (five seeds) & $1.41$ & $1.54$ & $1.84$\\
$256$ & $0.85$ & $0.81$ & $0.71$\\
$1024$ & $0.67$ & $0.57$ & $0.42$\\
$4096$ & $0.66$ & $0.50$ & --\\
\midrule
$256$, shared & $0.71$ & $0.55$ & $0.37$\\
\bottomrule
\end{tabular}
\end{table}

\subsection{Controlled study at fixed dimension}
\label{app:server-study}

This study complements the budget-indexed runs: it uses fixed instances, three
instance seeds at selected endpoints, more baselines and measured times. It ran
on a Linux server with Python~3.12, NumPy~2.4 and SciPy~1.17 (HiGHS), with twelve
tasks in parallel and one BLAS thread each; times are elapsed times of a task
under this load. We report medians with observed ranges, or single runs, not
confidence intervals. The relative error of an output $\widehat x$ is
$r=(f(\widehat x)-\fstar)/(f(0)-\fstar)$; for games, $r$ is the ratio of saddle
gaps.

\paragraph{Baselines and oracles.}
The study adds Frank--Wolfe (\alg{FW}) and its away-step and pairwise variants
(\alg{AFW}, \alg{PFW}) \citep{lacostejulien2015global}, all with the short step
clipped to $[0,1]$; \alg{AGD-E-feasible}, the similar-triangles method of
\citet{dvurechensky2017triangles} with the Euclidean prox-function, whose queries
are convex combinations of feasible points; and \alg{E-Bundle}, a Euclidean
bundle-level method \citep{lemarechal1995bundle} that projects onto a level set of
the stored cuts and returns its best evaluated point. \alg{PSG-mix} and
\alg{EG-mix} are the methods ``\PSG{} + mixture LP'' and ``\EG{} + mixture LP'' of
\cref{app:experiments-main}, where \CLA{} (c.p.) is also defined. \PG, the
accelerated methods, the Frank--Wolfe variants and \CLA{} use gradients only;
\alg{E-Bundle} and the Steiner-point methods also use values, and \SPL{} is given
$\fstar$. A smaller error at equal $N$ therefore does not mean a lower total cost.

\subsubsection{Fixed smooth instances}
\label{app:server-smooth}

The instances are the chains of \cref{app:experiments-instances} with $n=65$, a
$5$-sparse minimizer of $\ell_1$-norm $0.9$ and budgets $N=16,32,64$. The
nonquadratic chain of \cref{app:experiments-instances} is almost quadratic near its
minimizer, so here the nonquadratic chain uses
\[
 \psi_s(t)=s^2\bigl(\sqrt{1+(t/s)^2}-1\bigr),\qquad
 s=\tfrac1n\textstyle\sum_i|(BU\opt)_i|,
\]
whose curvature $\psi_s''(t)=(1+(t/s)^2)^{-3/2}$ varies on the scale of the
solution, while $L=1$ still holds. The Steiner-point methods use $m=256$ shared
directions. \Cref{fig:server-smooth} shows seed~$0$ at $N=16,32$ (open markers) and
the median and range over three seeds at $N=64$ (filled markers).

\begin{figure}[t]
\centering
\includegraphics[width=\linewidth]{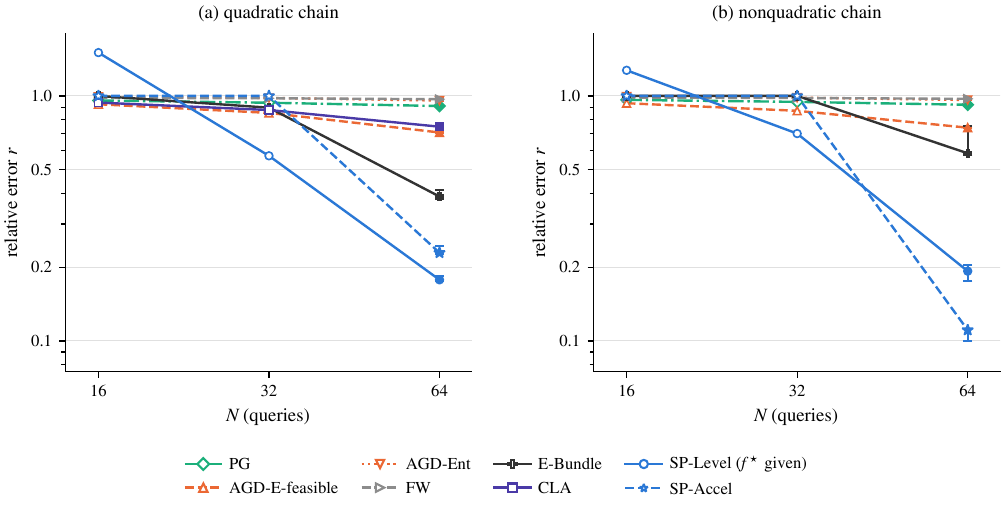}
\caption{Relative error $r$ on fixed chains with $n=65$.}
\label{fig:server-smooth}
\end{figure}

\begin{table}[t]
\centering\small
\caption{Median relative error $r$ and time at $n=65$, $N=64$, over three seeds
($\star$: given $\fstar$).}
\label{tab:server-smooth}
\begin{tabular}{lrrrr}
\toprule
Method & $r$, quadratic & Time (s) & $r$, nonquadratic & Time (s)\\
\midrule
PG & $0.909$ & $0.00413$ & $0.92$ & $0.00547$\\
AGD-E & $0.709$ & $0.00441$ & $0.742$ & $0.00555$\\
AGD-E-feasible & $0.709$ & $0.00584$ & $0.742$ & $0.00752$\\
AGD-Ent & $0.957$ & $0.00565$ & $0.963$ & $0.00708$\\
FW & $0.969$ & $0.00597$ & $0.973$ & $0.008$\\
AFW & $0.969$ & $0.00983$ & $0.973$ & $0.0107$\\
PFW & $0.97$ & $0.0092$ & $0.973$ & $0.0108$\\
E-Bundle & $0.388$ & $3.73$ & $0.583$ & $1.18$\\
CLA & $0.749$ & $0.234$ & --- & ---\\
SP-Level$^{\star}$ & $0.178$ & $92.4$ & $0.193$ & $90.6$\\
SP-Accel & $0.229$ & $74.1$ & $0.111$ & $93.6$\\
\bottomrule
\end{tabular}

\end{table}

On the nonquadratic chain \SPA{} leaves $r=0.111$ (range $0.100$--$0.111$), against
$0.742$ for \alg{AGD-E-feasible} and $0.583$ ($0.573$--$0.757$) for
\alg{E-Bundle}; on the quadratic chain the medians are $0.229$, $0.709$ and
$0.388$ (\cref{tab:server-smooth}). \SPL, although given $\fstar$, leaves $0.193$
on the nonquadratic chain, so knowing the level does not guarantee a smaller
error at a fixed budget; \CLA{} was run only on the quadratic chain. \SPA{} takes
about $94$\,s on the nonquadratic chain, against $1.18$\,s for \alg{E-Bundle} and milliseconds for the accelerated
methods. The advantage depends on the dimension: at $n=129$ and $N=64$, on one
nonquadratic instance, \SPA{} leaves $r=0.853$, \SPL{} $0.564$,
\alg{AGD-E-feasible} $0.870$ and \alg{E-Bundle} $1.00$.

\subsubsection{Sparsity and dimension in curvature learning}
\label{app:server-geometry}

\begin{figure}[t]
\centering
\includegraphics[width=\linewidth]{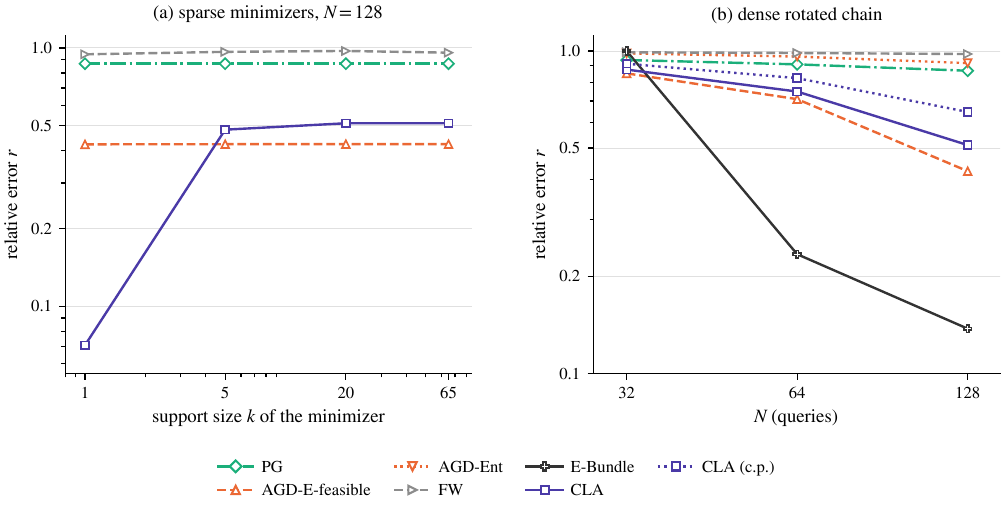}
\caption{Relative error at $n=65$, $p=1$: (a) against the support size $k$ of the
minimizer, $N=128$; (b) on the dense rotated chain.}
\label{fig:server-geometry}
\end{figure}

\Cref{fig:server-geometry} uses instance seed~$0$ and a minimizer of $\ell_1$-norm
$0.9$; each marker is one run. At $n=65$, $N=128$ and $p=1$, \CLA{} leaves
$r=0.0706$ for a $1$-sparse minimizer, against $0.424$ for \AGDE; for $5$-sparse
minimizers the medians over three instance seeds are $0.456$ and $0.424$, and on
the dense rotated chain
\CLA{} gives $0.511$, \AGDE{} $0.424$ and \alg{E-Bundle} $0.138$. The advantage of
\CLA{} is therefore tied to sparse minimizers. For a $1$-sparse minimizer,
increasing $n$ from $65$ to $129$ and $513$ changes $r$ from $0.0706$ to $0.400$ and
$0.936$ for \CLA, against $0.424$, $0.710$ and $0.927$ for \AGDE; for $p=4/3$ and
$3/2$ at $n=65$, \CLA{} gives $0.385$ and $0.481$, against $0.424$ for \AGDE.
At $n=65$, $N=128$ and a $5$-sparse minimizer, tightening the inner tolerance of
\CLA{} and \CLA{} (c.p.) from $10^{-11}$ to $10^{-15}$ changes $r$ by less than
$6\cdot10^{-8}$ for each $p$.

\begin{table}[t]
\centering\small
\setlength{\tabcolsep}{3.8pt}
\caption{\CLA{} and \AGDE{} on $\mathcal H_N$, $p=1$, $n=2N+1$.}
\label{tab:server-growing}
\begin{tabular}{rrrrrrr}
\toprule
$N$ & AGD-E gap & CLA gap & Ratio & AGD-E (s) & CLA (s) & Misses\\
\midrule
100 & $6.37\cdot10^{-6}$ & $6.51\cdot10^{-6}$ & $1.02$ & $0.00353$ & $0.854$ & 0\\
200 & $1.61\cdot10^{-6}$ & $1.58\cdot10^{-6}$ & $0.982$ & $0.0129$ & $9.08$ & 0\\
400 & $4.04\cdot10^{-7}$ & $2.42\cdot10^{-7}$ & $0.599$ & $0.0809$ & $96.9$ & 0\\
800 & $1.01\cdot10^{-7}$ & $2.79\cdot10^{-8}$ & $0.275$ & $0.7$ & $1.32\cdot10^{3}$ & 0\\
1600 & $2.54\cdot10^{-8}$ & $2.76\cdot10^{-9}$ & $0.109$ & $10.2$ & $2.56\cdot10^{4}$ & 47\\
\bottomrule
\end{tabular}

\end{table}

\Cref{tab:server-growing} lists the gaps of \CLA{} and \AGDE, their ratio, the times
and the number of inner solves of \CLA{} that miss their tolerance. It confirms the
separation seen in \cref{fig:rates} up to $N=800$. At $N=1600$ the gap of \CLA{} is
$0.109$ times that of \AGDE, but $47$ inner solves miss their tolerance, so this
point is not evidence for the exact method; the run took $25\,614$\,s, against
$10.2$\,s for \AGDE, and no fitted exponent uses it. The queries of \CLA{} stay in
the ball, while those of \AGDE{} reach the $\ell_1$-norm $1.000754$ at $N=1600$.

\subsubsection{Selector sample size, nonsmooth problems and games}
\label{app:server-diagnostics}

\paragraph{Selector sample size.}
On one nonquadratic instance of \cref{app:server-smooth} with $n=65$, a $5$-sparse
minimizer and $N=32$,
\SPL{} with $m=32$, $128$ and $512$ directions and three selector seeds reaches the
median relative errors $0.9994$, $0.7502$ and $0.6726$ with shared directions and
$1.9143$, $1.0236$ and $0.8331$ with fresh ones; the shared runs take $3.9$,
$15.1$ and $60.5$\,s. Larger samples help, at a cost linear in $m$, and shared
directions do better than fresh ones on this instance. All $18$ corresponding runs
of \SPA{} end with $r=1$: at this budget the overhead of the search for the level
outweighs the gain.

\paragraph{$\ell_\infty$ regression and bilinear games.}
\Cref{tab:server-other} reports medians over three instances with $64$ oracle
calls, and $128$ operator calls for games, where $n$ is the total dimension of
both blocks. On $\ell_\infty$ regression with $n=32$, \SPC{} reaches $r=0.00391$,
against $0.101$ for \PSG{} and $0.0268$ for \MDE, but \alg{E-Bundle}, which also
uses values, reaches $3.79\cdot10^{-5}$ in $0.368$\,s, against $15.4$\,s for
\SPC; at $n=128$ \alg{E-Bundle} is again more accurate ($0.006$ against
$0.0128$). On bilinear games with $n=128$, \SPV{} reaches $1.96\cdot10^{-5}$, against
$0.07$ for \EG{} and $0.0266$ for \alg{EG-mix}, in $46$\,s against $0.0129$\,s
for \EG; at $n=32$ it reaches $4.88\cdot10^{-7}$, against $0.0377$ for \EG. The
Steiner-point methods thus save oracle calls at a large cost in linear programs.

\begin{table}[!htb]
\centering\small
\caption{$\ell_\infty$ regression (first five rows) and bilinear games (last three
rows): median relative error $r$ and time over three instances.}
\label{tab:server-other}
\begin{tabular}{lrrrr}
\toprule
Method & $r$, $n=32$ & Time (s) & $r$, $n=128$ & Time (s)\\
\midrule
PSG & $0.101$ & $0.00407$ & $0.115$ & $0.00735$\\
PSG-mix & $0.0936$ & $0.0114$ & $0.115$ & $0.0269$\\
MD-Ent & $0.0268$ & $0.00352$ & $0.441$ & $0.00628$\\
E-Bundle & $3.79\cdot10^{-5}$ & $0.368$ & $0.006$ & $0.837$\\
SP-Cut & $0.00391$ & $15.4$ & $0.0128$ & $37$\\
EG & $0.0377$ & $0.0104$ & $0.07$ & $0.0129$\\
EG-mix & $0.0138$ & $0.0175$ & $0.0266$ & $0.0271$\\
SP-VI & $4.88\cdot10^{-7}$ & $15.9$ & $1.96\cdot10^{-5}$ & $46$\\
\bottomrule
\end{tabular}

\end{table}

\paragraph{Other diagnostics.}
With a memory of $32$ cuts, \alg{E-Bundle} was not monotone in its inner tolerance
on one instance with $n=33$ and $N=64$: the tolerances $10^{-7}$, $10^{-9}$ and
$10^{-11}$ gave $r=0.0468$, $0.0669$ and $0.136$. The four runs of \SPP{}, on both
chains with $n=33$ and budgets of $64$ and $256$ gradients, stopped at the budget
before completing their four stages: with $64$ gradients no stage was completed,
and with $256$ two, leaving $r=0.978$ on the quadratic chain and $0.990$ on the
nonquadratic one. These budgets are too small for the method; in
\cref{app:experiments-prox}, $18$ stages at $n=33$ take $8.9\cdot10^4$ gradients.

\section{Concurrent work}
\label{app:concurrent}

Four preprints posted on arXiv between 17 and 20 September 2026 study closely
related questions. We discuss them for completeness; our proofs do not use their
results or constructions. We have not verified the proofs of these preprints and
do not vouch for their correctness. Three of them consider the domain $B_p^n(R)$ with
Lipschitz or H\"older constants measured in the $\ell_q$ norm; in this appendix
$q$ denotes this norm, not the conjugate exponent of $p$. Our setting is $p=1$,
$q=2$. Their rates are stated through the exponent
$\rho_{p,q}=1/p-\max\{1/q-1/2,0\}$, which equals $1$ at $(p,q)=(1,2)$ and $1/2$
when $p=q\le2$.

\paragraph{Lipschitz convex optimization.}
For convex objectives that are $G$-Lipschitz in $\ell_q$, with values and
subgradients, \citet[Corollary~11]{martinezrubio2026nondual} prove the
deterministic oracle bound $\tO_{p,q}(GR/N^{\rho_{p,q}})$ for
$1\le p\le q\le\infty$; their implementation may take time exponential in the
dimension and the number of queries. \citet[Theorems~3 and~12]{martinezrubio2026stable}
attain the same rate for $p<\min\{q,2\}$ with feasible queries and, with high
probability, polynomial computation, and they bound the movement of their
centers along nested subsets of $B_p^n(R)$ (their Theorem~2). At $(p,q)=(1,2)$
both give $\tO(GR/N)$, the rate of \cref{thm:spcut}. Section~4 of
\citet{martinezrubio2026stable} uses, at $(p,q)=(1,2)$, the Steiner point of the
sublevel sets of a bundle model, deep cuts, a movement bound deduced from
\citet{bubeck2020chasing}, and a Monte Carlo implementation with one linear
program per sample; this is the mechanism of \SPC. Our version uses
subgradients without function values, returns a certificate computed from them,
and applies to every known polytope in $c+B_1^n(R)$, including the simplex.
Appendix~G of \citet{martinezrubio2026stable} proves randomized lower bounds for the Lipschitz class in
dimension $n\gtrsim\varepsilon^{-2+\delta}$ for any fixed $\delta>0$;
\cref{thm:rand-nonsmooth} needs $n\ge8N$.

\paragraph{Smooth and H\"older convex optimization.}
For gradients that are $\nu$-H\"older in $\ell_q$ with constant $L$,
\citet[Theorem~1]{martinezrubio2026smooth} obtain, after $N\le n$ queries of
values and gradients, a feasible point with error
$\tO_{\nu,p,q}(LR^{1+\nu}/N^{(1+\nu)(1+\rho_{p,q})-1})$ for $p<\min\{q,2\}$ and in
the matched cases $p=q\le2$, with high probability and in polynomial time; they
also state a deterministic version without an efficiency guarantee. At
$(p,q)=(1,2)$ the exponent is $1+2\nu$, so for $N\le n$ their result has the
exponents of \cref{thm:overview,thm:holder}, for a wider range of norm pairs.
Their method combines a stable center, a line search on the segment between the
incumbent and the center, and a Moreau envelope for feasibility. \SPL{} uses the
Steiner point in place of the prox step of an accelerated level method, and
\SPP{} uses a related envelope search without function values. Our upper bounds
hold for all $N$ and $n$; the bounds for oracles of higher order, the randomized
lower bounds for smooth functions and the results on strong convexity have no
counterpart there.

\paragraph{Variational inequalities of higher order.}
\citet{zhang2026optimalvi} study monotone variational inequalities in Euclidean
geometry with an oracle that returns $F$ and its derivatives up to order $k$,
where $D^kF$ is Lipschitz with constant $L$. For $k\ge1$ their Theorem~3.1
finds a point $x$ with $\dist(0,F(x)+N_K(x))\le\varepsilon$ after
$O_k((LD_0^{k+1}/\varepsilon)^{2/(3k+2)})$ calls, where $D_0$ is the initial
Euclidean distance to a solution and $N_K(x)$ is the normal cone of $K$ at $x$.
A point with this residual has weak gap at most $\varepsilon$ times the
Euclidean diameter of $K$, so on the $\ell_1$ ball their method gives the weak gap
$O(LR^{k+2}/N^{(3k+2)/2})$ after $N$ calls. For the same oracle,
\cref{thm:holder-tensor-vi} with $r=k+1$ gives $\tO(LR^{k+2}/N^{k+2})$. Our
exponent is larger for $k=1$ ($3$ against $5/2$), the two coincide for $k=2$,
and theirs is larger for $k\ge3$; the gain at $k=1$ comes from the $\ell_1$
geometry, while their method applies to general convex domains.
\citet{zhang2026matching} prove the same Euclidean rate with a matching lower
bound (\cref{app:vi-models}).

\end{document}